\documentclass[]{amsart}
\usepackage{amssymb, amsfonts, latexsym, amsthm, amsmath, verbatim}
\pdfoutput=1
\usepackage{marginnote} 
\usepackage{array,longtable}
\usepackage{color}
\usepackage{soul} 
\usepackage{paralist}
\usepackage[all]{xy}
\usepackage{amscd}
\usepackage{mathrsfs} 
\usepackage[english]{babel}
\usepackage{tikz}
\usepackage{graphicx} 
\usepackage{standalone}
\usepackage{epstopdf}
\definecolor{darkblue}{rgb}{0,0,0.4} 
\usepackage[colorlinks=true, citecolor=darkblue, filecolor=darkblue, linkcolor=darkblue,urlcolor=darkblue]{hyperref}
\usetikzlibrary{arrows,cd,decorations.pathmorphing,decorations.pathreplacing}
\usepackage{amscd}
\usepackage{enumitem}
\usepackage{transparent}
\usepackage[normalem]{ulem} 
\usepackage[color=blue!20!white]{todonotes}
\tikzstyle{crossing}=[circle,fill=white,minimum height=6pt,inner sep=0pt, outer sep=0pt, style={transform shape=false}]
\usepackage{hyperref}
\usepackage{subcaption}

\numberwithin{equation}{section}

\theoremstyle{plain}
\newtheorem{thm}[equation]{Theorem}

\newtheorem*{thm*}{Theorem}
\newtheorem{prop}[equation]{Theorem}
\newtheorem{corollary}[equation]{Corollary}
\newtheorem{lem}[equation]{Lemma}

\theoremstyle{definition}
\newtheorem{rmk}[equation]{Remark}
\newtheorem{example}[equation]{Example}

\newtheorem{defn}[equation]{Definition}

\numberwithin{figure}{section}

\usepackage{etoolbox}
\def\do#1{\csdef{c#1}{\mathcal{#1}}}
\docsvlist{A,B,C,D,E,F,G,H,I,J,K,L,M,N,O,P,Q,R,S,T,U,V,W,X,Y,Z}

\newcommand{\frs}{\mathfrak{s}}

\newcommand{\AbFunc}{\mathfrak{F}}

\newcommand{\f}{\mathbb{F}}
\newcommand{\F}{\mathbb{F}}
\newcommand{\gr}{\mathrm{gr}}

\newcommand{\Kh}{\mathit{Kh}}
\newcommand{\Id}{\mathrm{Id}}

\newcommand{\bKom}{\mathbf{Kom}}

\newcommand{\Khr}{\mathit{Khr}}

\newcommand{\R}{\mathbb{R}}
\newcommand{\wt}[1]{\widetilde{#1}}

\newcommand{\eg}{\smash{\widetilde{g}_4}}

\newcommand{\ieg}{\smash{\widetilde{ig}_4}}

\newcommand{\Z}{\mathbb{Z}}

\newcommand{\spa}{\mathrm{span}}

\DeclareMathOperator{\Cone}{Cone}
\DeclareMathOperator{\Fix}{Fix}

\DeclareMathOperator{\Torso}{Tors}

\DeclareMathOperator{\id}{\mathrm{id}}
\DeclareMathOperator{\lk}{lk}
\DeclareMathOperator{\Hom}{Hom}

\DeclareMathOperator{\Tor}{Tor}

\DeclareMathOperator{\Kc}{\mathit{Kc}}
\DeclareMathOperator{\MKc}{\mathit{MKc}}
\DeclareMathOperator{\Br}{\mathscr{B}}

\DeclareMathOperator{\Hh}{\mathscr{H}}

\newcommand{\Kcr}{\mathit{Kcr}}
\newcommand{\Kcrm}{\Kcr^{-}}

\newcommand{\Kcm}{\Kc^{-}}

\newcommand{\Kchat}{\widehat{\Kc}}
\newcommand{\Mkc}{\mathit{MKc}^-}
\newcommand{\Mkcr}{\mathit{MKcr}^-}
\newcommand{\Kcinf}{\Kc^{\infty}}
\newcommand{\CKh}{\mathit{CKh}}

\newcommand{\frf}{\mathfrak{f}}
\newcommand{\frg}{\mathfrak{g}}
\newcommand{\frH}{\mathfrak{H}}

\newcommand{\Kob}{\mathrm{Kob}}

\usepackage{imakeidx}

\makeatletter
\newcommand*{\wackyenum}[1]{%
  \expandafter\@wackyenum\csname c@#1\endcsname%
}

\newcommand*{\@wackyenum}[1]{%
\ifcase#1\or{IR-1}\or{IR-2}\or{IR-3}\or{R-2}\or{M-2}\or{M-3}\or{R-1}\or{M-1}\else\@ctrerr\fi%
}
\AddEnumerateCounter{\wackyenum}{\@wackyenum}{53.13}
\makeatother

\newcounter{nparcount}

\def\ypar#1{\stepcounter{nparcount}%
  \colorbox{yellow!10}{$!^{\arabic{nparcount}}$}%
  \marginnote{\colorbox{yellow!10}{\begin{minipage}{2cm}\fontsize{5}{6}\selectfont\color{green!20!black}
${}^{\arabic{nparcount}}$#1\end{minipage}}}}
\def\mpar#1{\stepcounter{nparcount}%
  \colorbox{orange!10}{$!^{\arabic{nparcount}}$}%
  \marginnote{\colorbox{red!10}{\begin{minipage}{2cm}\fontsize{5}{6}\selectfont\color{green!20!black}
${}^{\arabic{nparcount}}$#1\end{minipage}}}}

\def\myref{\colorbox{green!30}{[REF]}}

\makeatletter

\def\@setref#1#2#3{%
  \ifx#1\relax
   \protect\G@refundefinedtrue
   \nfss@text{\colorbox{green!30}{\myref}}%
   \@latex@warning{Reference '#3' on page \thepage \space
             undefined}%
  \else
   \expandafter#2#1\null
  \fi}
\def\@citex[#1]#2{\leavevmode
  \let\@citea\@empty
  \@cite{\@for\@citeb:=#2\do
    {\@citea\def\@citea{,\penalty\@m\ }%
     \edef\@citeb{\expandafter\@firstofone\@citeb\@empty}%
     \if@filesw\immediate\write\@auxout{\string\citation{\@citeb}}\fi
     \@ifundefined{b@\@citeb}{\colorbox{blue!30}{\reset@font REF}%
       \G@refundefinedtrue
       \@latex@warning
         {Citation `\@citeb' on page \thepage \space undefined}}%
       {\@cite@ofmt{\csname b@\@citeb\endcsname}}}}{#1}}
\makeatother

\title{Khovanov homology and equivariant surfaces}

\author{Maciej Borodzik}
\address{Institute of Mathematics\\ University of Warsaw
\\Warsaw, Poland}
\email{mcboro@mimuw.edu.pl}

\author{Irving Dai}
\address{Department of Mathematics\\The University of Texas at Austin\\ Austin, TX, USA}
\email{irving.dai@math.utexas.edu}

\author{Abhishek Mallick}
\address{Department of Mathematics\\Dartmouth College \\ Hanover, NH, USA}
\email{abhishek.mallick@dartmouth.edu}

\author{Matthew Stoffregen}
\address{Department of Mathematics\\Michigan State University\\ East Lansing, MI, USA}
\email{stoffre1@msu.edu}

\makeindex[title=Index of Notation, columns=2]
\begin{document}

\maketitle

\begin{abstract}
We introduce a refinement of Bar-Natan homology for involutive links, extending the work of Lobb-Watson and Sano. We construct a new suite of numerical invariants and derive bounds for the genus of equivariant cobordisms between strongly invertible knots. Our invariants show that the difference between the equivariant slice genus and isotopy-equivariant slice genus can be arbitrarily large, whereas previously these were not known to differ.
\end{abstract}
\tableofcontents

\pagebreak 
\section{Introduction}
In addition to being of intrinsic interest \cite{Murasugi, Sakuma, Naik, ChaKo, ND, BI}, equivariant knots have formed a robust source of examples in the study of many topological questions, including the investigation of exotic phenomena and the theory of corks \cite{DHM, hayden-sundberg, DMS, levine2023new}. Numerous authors have studied equivariant knots in the context of Khovanov homology; see for example \cite{Watson, stoffregen-zhang, lobb-watson, borodzik-politarczyk-silvero, LSinvertible, Sano}. The current work introduces a new Khovanov-type invariant for strongly invertible knots (or, more generally, involutive links) especially suited for analyzing equivariant cobordisms. We are particularly interested in the difference between the notion of an equivariant cobordism and that of an \textit{isotopy-}equivariant cobordism see for example \cite{DMS}. While approaches such as \cite{DMS, Sano} have generally produced invariants insensitive to this distinction, here our goal is to define an equivariant version of Khovanov homology capable of distinguishing the two.

In this paper, we record an involution on a link by introducing (and generalizing) a $\Z/2\Z$-equivariant Khovanov-type homology theory. This is in contrast to \cite{DHM, DMS, Sano}, in which symmetries of an object are encoded via the $\iota$-complex/mapping cone formalism of Hendricks-Manolescu \cite{HM} and Hendricks-Manolescu-Zemke \cite{HMZ}. We demonstrate that this distinction in Khovanov homology produces strikingly different qualitative results. In particular, we use our invariants to show that the difference between the \textit{bona fide} equivariant genus and the isotopy-equivariant genus of a strongly invertible knot can be arbitrarily large, whereas previously these were not known to differ.

\subsection{Involutions on Khovanov homology} Let $(L, \tau)$ be an involutive link. Associated to $L$, we may form the \textit{Bar-Natan complex} $\Kcm(L)$, which is a doubly-graded chain complex over the polynomial ring $\f[u]$ (where $\f = \f_2$, the field of two elements). The involution $\tau$ induces an involution on $\Kcm(L)$, which by abuse of notation we also denote by $\tau$. 

The action of $\tau$ was utilized by Lobb-Watson in \cite{lobb-watson}, who constructed an additional filtration on the Khovanov complex of $L$ coming from the axis of symmetry. The Lobb-Watson filtration is defined by weighting each resolution differently depending on its placement relative to the fixed-point axis $\Fix(\tau)$. In \cite{lobb-watson}, it is shown that this additional filtration gives rise to a triply-graded refinement $\mathbb{H}^{i, j, k}(L)$ of Khovanov homology which is an involutive link invariant. Lobb-Watson show (as a \textit{non-}equivariant application) that their triply-graded refinement is capable of detecting pairs of mutant knots, even though the usual Khovanov homology (with coefficients in $\f$) is mutation-invariant \cite{bloom, wehrli}.

Using the functoriality and naturality of Khovanov homology, it is also possible to show that the homotopy class of the pair $(\Kcm(L), \tau)$ is an invariant of $(L, \tau)$. This program was essentially carried out by Sano in \cite{Sano}. See \cite{DMS} for the analogous program in the knot Floer setting developed by the second, third, and fourth authors.

\subsection{The Borel construction}
In the present work, we leverage two principal differences which endow the Khovanov formalism with additional structure not present in the knot Floer setting. The first is the fact that the action of $\tau$ is an actual involution, rather than a homotopy involution. This allows us to define the \textit{Borel complex} $\Kcm_Q(L)$ of $L$, which is given by

\[
\Kcm_Q(L) = (\Kcm(L) \otimes \F[Q], \partial_Q)
\]
where 
\[
\partial_Q = \partial + Q \cdot (1 + \tau).
\]
Writing the $\Z \times \Z$-grading on $\Kcm_Q(L)$ by $(\gr_h, \gr_q)$, our conventions are that
\[
\gr(\partial_Q) = \gr(\partial) = (1, 0) \quad \quad \gr(u) = (0, -2) \quad \quad \gr(Q) = (1, 0).
\]

The reader should compare the above construction with \cite[Theorem 1.1]{Sano}, in which the mapping cone invariant $\mathit{KhI}(L, \tau)$ is constructed. The crucial distinction is that $\mathit{KhI}(L, \tau)$ is formed by taking the truncated tensor product of the Khovanov complex with $\f[Q]/Q^2$ -- as in the involutive Heegaard Floer formalism of Hendricks-Manolescu \cite{HM} -- whereas our invariant is over the full ring $\f[Q]$. As we will see, this will lead to a more refined invariant capable of capturing more subtle equivariant behavior. We also refer the reader to \cite{chen-yang}, which gives several conjectures relating the Borel construction to other aspects of Khovanov homology. 

Importantly, it is \textit{not} a formal consequence of naturality or functoriality of Khovanov homology that (the homotopy class of) $\Kcm_Q(L)$ is a well-defined invariant. This is in contrast to the approach of \cite{DMS, Sano}. These latter constructions mimic the framework developed by Hendricks-Manolescu \cite{HM} and Hendricks-Manolescu-Zemke \cite{HMZ} in their work on involutive Heegaard Floer homology. In general, a theory constructed in this latter manner can be shown to produce a well-defined invariant by appealing to certain naturality and functoriality results, as we discuss in Section~\ref{sec:involutions-on-bar-natan}. However, in our case the analogous formal properties do not suffice to establish invariance of the Borel complex. Instead, following the work of Lobb-Watson \cite{lobb-watson}, we utilize an explicit analysis of equivariant Reidemeister moves.

\begin{thm}\label{thm:full_borel_invariant}
The homotopy type of the Borel complex $\smash{\Kcm_Q(L)}$ is an invariant of $(L, \tau)$ up to Sakuma equivalence. Moreover, if $\Sigma$ is an equivariant cobordism from $(L_1, \tau_1)$ to $(L_2, \tau_2)$, then for appropriate choices of data there exists a cobordism map
\[
\Kcm_Q(\Sigma) \colon \Kcm_Q(L_1) \rightarrow \Kcm_Q(L_2).
\]
The homotopy type of the reduced Borel complex $\smash{\Kcrm_Q(L)}$ is likewise an invariant of $(L, \tau)$, and there exists a reduced cobordism map
\[
\Kcrm_Q(\Sigma) \colon \Kcrm_Q(L_1) \rightarrow \Kcrm_Q(L_2).
\]
Both $\Kcm_Q(\Sigma)$ and $\Kcrm_Q(\Sigma)$ have grading shift $(0, -2g(\Sigma))$. Moreover, if $L_1$ and $L_2$ are strongly invertible knots and $\Sigma$ is connected, then $\Kcrm_Q(\Sigma)$ is local.
\end{thm}

Here, the reduced Borel complex $\smash{\Kcrm_Q(K)}$ is a half-dimensional summand of $\smash{\Kcm_Q(K)}$ constructed in the usual Khovanov-theoretic manner. However, our treatment of $\smash{\Kcrm_Q(\Sigma)}$ will actually be rather subtle: usually, defining a cobordism map between reduced Bar-Natan complexes requires a choice of path $\gamma$ on $\Sigma$. This turns out to be inconvenient for studying equivariant cobordisms, as no \textit{equivariant} path need exist. We thus construct our cobordism maps in a manner which avoids the need for $\gamma$, at the cost of naturality. Indeed, it should be stressed that we do not prove either $\Kcm_Q(\Sigma)$ or $\Kcrm_Q(\Sigma)$ are natural; e.g.,\ independent of a handle decomposition or other data. 

The locality condition on $\Kcrm_Q(\Sigma)$ means that it induces an isomorphism on homology after inverting $u$; see Definition~\ref{def:knotlike-borel}. As a consequence of locality, we define a new suite of numerical invariants associated to the reduced Borel complex and show that these constrain the genus of equivariant cobordisms between pairs of strongly invertible knots.

\begin{thm}\label{thm:equivariant_s_borel_intro}
Let $(K, \tau)$ be a strongly invertible knot. Then there are numerical invariants 
\[
\wt{s}_Q(K) \quad \text{and} \quad \wt{s}_{Q, A, B}(K),
\]
with the latter being defined for any pair of integers $0 \leq A < B$. These are an invariant of the equivariant concordance class of $(K, \tau)$. Moreover, if $\Sigma$ is an equivariant cobordism from $K_1$ and $K_2$, then
\[
\wt{s}_Q(K_1)-2g(\Sigma)\leq \wt{s}_Q(K_2) \quad \text{and} \quad \wt{s}_{Q,A,B}(K_1)-2g(\Sigma)\leq \wt{s}_{Q,A,B}(K_2).
\]
\end{thm}
\noindent
We refer to Definitions~\ref{def:strong-s} and \ref{def:strong-s-general} for the construction of $\wt{s}_Q(K)$ and $\wt{s}_{Q, A, B}(K)$.
Note that $\tilde{s}_Q(K)$ and $\tilde{s}_{Q, A, B}(K)$ are distinct from -- and should be considered further refinements of -- the usual numerical invariants constructed via the standard involutive methods of \cite{HM}. (See for example \cite[Definition 3.23]{Sano}.) These latter invariants are, of course, themselves equivariant refinements of the usual $s$-invariant. It is not difficult to find examples in which $\tilde{s}_Q(K)$ and $\tilde{s}_{Q, A, B}(K)$ contain strictly more information than the invariants of \cite{Sano}; see Section~\ref{sec:strong-s-examples} and Theorem~\ref{thm:kyle_knot_intro} below.
 
 \subsection{Equivariant and isotopy-equivariant slice genus}
 Let $(K, \tau)$ be an equivariant knot. Recall that the \textit{equivariant slice genus} $\eg(K)$ of $K$ is defined to be the minimum genus of a smoothly embedded surface $\Sigma$ in $B^4$ which has boundary $K$ and is sent to itself by the obvious extension $\tau_{B^4}$ of $\tau$ over $B^4$.\footnote{There is a slight difference in this paper versus (for example) \cite{DMS}: both here and in general, we will only consider the standard extension of $\tau$ over $B^4$. See Section~\ref{sec:topological-preliminaries} for a more precise discussion.} Equivariant sliceness and the equivariant genus have been studied at length by several authors \cite{Sakuma, ChaKo, ND, BI, DMS, AB, MillerPowell}, among others. In particular, it is known that the equivariant slice genus can be arbitrarily different from the usual slice genus, both in the periodic \cite{BI} and strongly invertible \cite{DMS, MillerPowell} cases. 
 
In \cite[Definition 2.8]{DMS}, a slightly different notion was introduced: we say that a slice surface $\Sigma$ for $K$ is \textit{isotopy}-equivariant if $\Sigma$ is smoothly isotopic rel boundary to $\tau_{B^4}(\Sigma)$. We can then define the \textit{isotopy-equivariant slice genus} $\ieg(K)$ to be the minimum genus over all such surfaces. Obviously, $\eg(K) \geq \ieg(K)$. It turns out that many bounds on the equivariant genus are in fact bounds on the isotopy-equivariant genus. While in several applications this is helpful -- for example, in \cite{DMS} such bounds are used to show that knot Floer homology can be used to detect exotic pairs of disks -- it is an interesting question as to whether these notions are distinct. Here, we use our invariants to exhibit a family of knots whose equivariant slice genus grows arbitrarily large but whose isotopy-equivariant slice genus remains bounded. Previously, these two notions were not known to differ.

\begin{thm}\label{thm:kyle_knot_intro}
Let $J = 17nh_{74}$ be the strongly invertible knot shown in Figure~\ref{fig:kyle}. Then
\[
\eg(\#_mJ) \geq \lceil m/2 \rceil \quad \text{but} \quad \ieg(\#_mJ) \leq 1
\]
for all $m \in \mathbb{N}$.\footnote{In fact, it is not difficult to show using \cite{DMS} that $\ieg(\#_mJ) \geq 1$, so the second inequality is an equality.}
\end{thm}

\begin{figure}[h!]
\includegraphics[scale = 0.55]{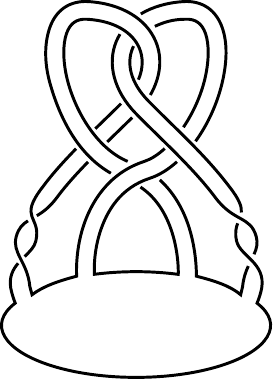}
\caption{The strongly invertible knot $J$ of Theorem~\ref{thm:kyle_knot_intro}. The involution $\tau$ is given by rotation around a vertical axis. Taken from the work of Hayden \cite{Hayden}.}\label{fig:kyle}
\end{figure}

Roughly speaking, the essential point here is that the invariants of \cite{DMS, Sano} are functorial with respect to isotopy-equivariant cobordisms. This will be the case for all equivariant invariants that follow the program of \cite{HM, HMZ}, which in fact only leverage the naturality and functoriality of the usual, non-equivariant TQFT structure. In contrast, only \textit{bona fide} equivariant cobordisms yield maps on the Borel complex. As we will see, invariants derived from the Borel construction are thus capable of distinguishing equivariant and isotopy-equivariant cobordisms. The proof of Theorem~\ref{thm:kyle_knot_intro} will require a connected sum formula for our invariants, which we give in Theorem~\ref{thm:connected-sum-tau}.

\begin{rmk}
  Here we show that $\eg$ and $\ieg$ differ in the smooth category. In the topological category, it turns out that $\eg$ and $\ieg$ can be shown to differ using Sakuma's $\eta$-polynomial \cite{Sakuma}. For this, we provide a proof that the $\eta$-polynomial is invariant under topological concordance in Appendix~\ref{sec:appendixa} and apply this to show that $\eg(J)\neq\ieg(J)$ for $J=17nh_{74}$, see Corollary~\ref{cor:J_topo}.
\end{rmk}

\subsection{Mixed complexes}

As discussed above, the second distinction between the Khovanov and knot Floer settings is the presence of the Lobb-Watson filtration coming from the axis of symmetry \cite{lobb-watson}. In the present work, we aim to incorporate this structure into our story by introducing an additional variable $W$ to the Borel construction in order to record the Lobb-Watson grading. 
This gives a refinement of $\Kcm_Q(L)$, which is a chain complex over variables $u,Q,W$. However, it turns out that the homotopy type of this complex is \emph{not} an invariant of a knot. This leads us to construct another, more complex refinemenet of $\Kcm_Q(L)$,
which we call the \textit{mixed complex} $\Mkc(L)$. The homotopy type of a mixed complex is still \textit{not} an invariant of $(L, \tau)$. Instead, it turns out that the appropriate notion of equivalence for mixed complexes is a relation called \textit{$Q$-equivalence}, which we introduce in Section~\ref{sec:3}:

\begin{thm}\label{thm:mixedcomplex}
The $Q$-equivalence class of the mixed complex $\Mkc(L)$ is an invariant of $(L, \tau)$ up to Sakuma equivalence. If $\Sigma$ is an equivariant cobordism from $(L_1, \tau_1)$ to $(L_2, \tau_2)$, then for appropriate choices of data there exists a cobordism map
\[
\Mkc(\Sigma) \colon \Mkc(L_1) \rightarrow \Mkc(L_2).
\]
The homotopy type of the reduced mixed complex $\Mkcr(K)$ is likewise an invariant of $(L, \tau)$, and there exists a reduced cobordism map
\[
\Mkcr(\Sigma) \colon \Mkcr(L_1) \rightarrow \Mkcr(L_2).
\]
Both $\Mkc(\Sigma)$ and $\Mkcr(\Sigma)$ have grading shift $(0, -2g(\Sigma), 0)$. Moreover, if $L_1$ and $L_2$ are strongly invertible knots and $\Sigma$ is connected, then $\Mkcr(\Sigma)$ is local.
\end{thm}

One motivation for considering the Lobb-Watson refinement is that for small knots, neither the action of $\tau$ nor the Borel construction yields any interesting information. For example, we are unable to distinguish the different involutions on the figure-eight knot or the different involutions on the stevedore knot using the Borel construction, and we are likewise unable to prove that the stevedore knot is not equivariantly slice. In contrast, all of these results can be obtained using knot Floer homology \cite{DMS}. Since the Lobb-Watson refinement is capable of distinguishing different involutions on small knots, it is natural to alloy the construction of \cite{lobb-watson} with the Borel formalism in the hopes that this will give improved equivariant genus bounds. 

\begin{rmk}
It is possible to use the locality condition of Theorem~\ref{thm:mixedcomplex} to define an equivariant $s$-invariant in the setting of mixed complexes. Unfortunately, the authors do not currently have any example in which this gives more information than the equivariant $s$-invariants $\wt{s}_Q$ and $\wt{s}_{Q, A, B}$ of Theorem~\ref{thm:equivariant_s_borel_intro}. 
\end{rmk}

\subsection{Computer calculations} A computer program that calculates the action of $\tau$ on Khovanov homology and the second page of the Bar-Natan spectral sequence may be found at \cite{Khocomp}. While this does not generally suffice to determine the Borel or mixed complex, such computations can still be used to determine or bound other equivariant invariants. These include a weaker numerical invariant $\tilde{s}(K)$ (essentially constructed in \cite[Definition 3.23]{Sano}), which we define in Section~\ref{sec:tau-complex}. An abbreviated summary of the program output for various knots up to thirteen crossings can be found in Appendix~\ref{sec:different}. We give a discussion of how to use \cite{Khocomp}
 to constrain the Borel complex in Section~\ref{sec:borel-bootstrapping}. 

\subsection{Further applications} We also list some assorted observations resulting from studying the action of $\tau$. First, we provide a slightly more straightforward argument that involutions on Khovanov homology can detect mutant pairs of knots. While Lobb-Watson leverage their additional filtration to establish this result, we point out that some of their examples can be distinguished simply by calculating the action of $\tau$ via \cite{Khocomp}:

\begin{thm}\cite[Theorem 1.2]{lobb-watson}\label{thm:lw_knots}
The mutant knots $K_1$ and $K_2$ of Figure~\ref{fig:lobbwatson} have identical Khovanov and knot Floer homologies. However,
both $K_1$ and $K_2$ admit two strong involutions. These involutions for $K_1$ and $K_2$ induce different actions on $\Kh$,
so $K_1$ and $K_2$ are not isotopic.
\end{thm}

\begin{figure}[h!]
  \begin{tikzpicture}
    \node at (-2.2,0){\includegraphics[width=3.5cm]{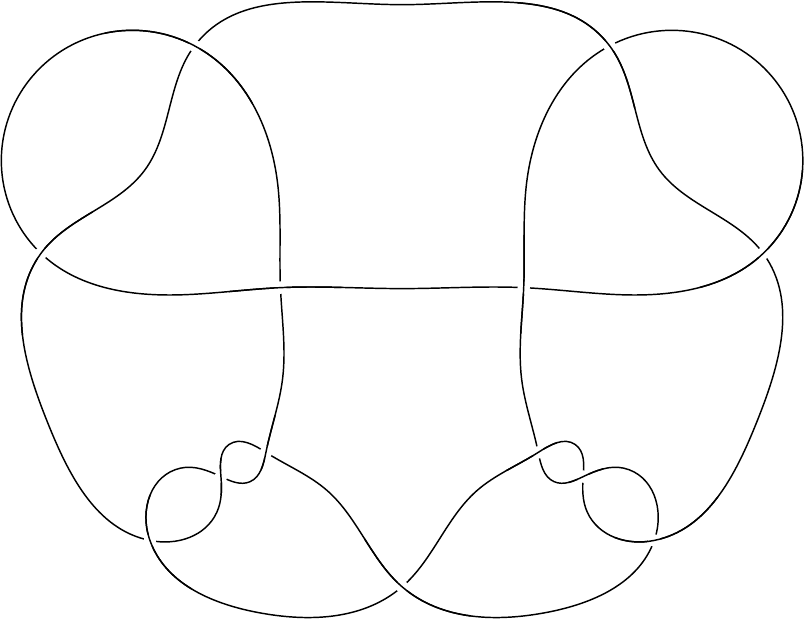}};
    \node at (3,0) {\includegraphics[width=3.4cm, angle=47]{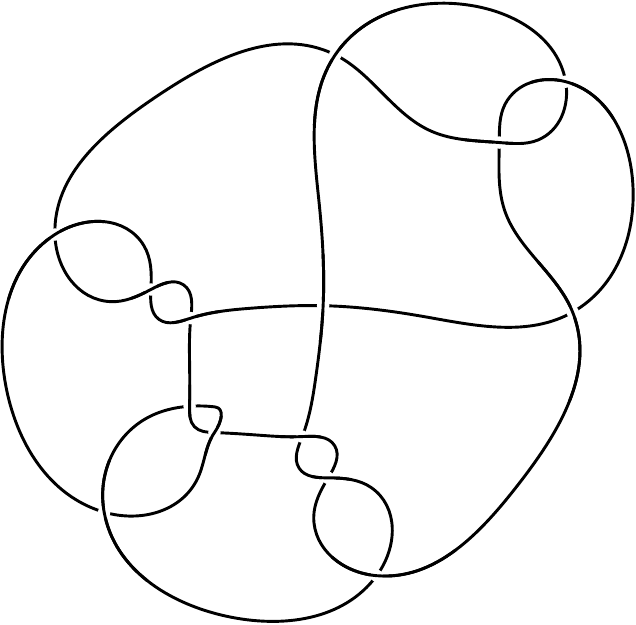}};
  \end{tikzpicture}
  \caption{The Lobb-Watson knots $K_1$ and $K_2$ of Theorem~\ref{thm:lw_knots}.}\label{fig:lobbwatson} 
\end{figure}

Next, we present an application to detecting \textit{equivariantly squeezed} knots. Recall from \cite{feller-lewark-lobb} that a knot $K \subset S^3$ is \emph{squeezed} if it appears as a slice of a genus-minimizing, connected cobordism from a positive torus knot $T^+$ to a negative torus knot $T^-$. It is natural to generalize this notion to equivariant knots by considering knots that arise as a slice of an equivariant genus-minimizing, connected cobordism from $T^+$ to $T^-$. We provide a straightforward obstruction as follows:

\begin{thm}\label{thm:equivariantly-squeezed}
Suppose $(K, \tau)$ is equivariantly squeezed. Then $(\Kcrm(K), \tau)$ is locally equivalent to $\f[u]_{s(K)}$. In particular, $\wt{s}(K) = s(K)$.
\end{thm}

Combined with \cite{Khocomp}, we easily obtain many examples of knots that are not equivariantly squeezed. In particular, even though the (regular) squeezedness of $10_{141}$ is unknown (see \cite[Prize]{feller-lewark-lobb}), we show: 

\begin{corollary}\label{cor:10_141}
  The strongly invertible knot $10_{141}$ of Figure~\ref{fig:10141} is not equivariantly squeezed.
\end{corollary}

\begin{figure}
  \includegraphics[width=3.3cm]{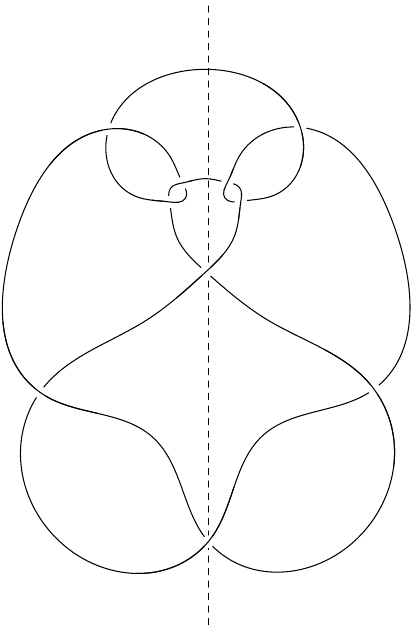}
  \caption{The strongly invertible knot $10_{141}$.}\label{fig:10141}
\end{figure}

\subsection{Comparison with other works} For the convenience of the reader, we close by explaining the relation of the present paper to other works involving involutions on Khovanov homology (in particular, \cite{lobb-watson} and \cite{Sano}). The invariant of \cite{lobb-watson} is obtained by setting $Q = 1$ in the Borel construction and then using the additional filtration described in Section~\ref{sec:lobb-watson-filtration} to define the triply-graded refinement $\mathbb{H}^{i,j,k}(L)$. (See also the discussion in \cite[Section 6]{chen-yang}.) The proof of invariance of $\smash{\Kcm_Q(L)}$ thus bears certain similarities to the arguments in \cite{lobb-watson}. However, as we will see, the Borel construction -- even without the additional filtration -- provides a more powerful formalism to define numerical invariants and obtain our results on the equivariant genus. Combining the Borel construction with the Lobb-Watson filtration (via our mixed complex formalism) likewise represents an attempt to extend \cite{lobb-watson} to applications involving equivariant cobordisms. For experts, we note that the additional complexity of the mixed complex formalism necessitates significant modifications to the proof of invariance from \cite{lobb-watson}; in particular, we present an entirely new proof of invariance under the M3 equivariant Reidemeister move.

The comparison with \cite{Sano} is more straightforward. As discussed previously, the mapping cone construction of \cite{Sano} may be viewed as a truncation of the Borel complex which forgets much of the homotopy commutation data. (We review this in Section~\ref{sec:involutions-on-bar-natan}.) From this point of view, the proof of invariance of $\Kcm_Q(L)$ requires a significantly finer analysis of the maps induced by equivariant Reidemeister moves. Remembering this information allows us to distinguish between the isotopy-equivariant genus and equivariant genus, as described above.



\subsection{Acknowledgments} The authors would like to thank Kristen Hendricks, Robert Lipshitz, Mark Powell, and Liam Watson for helpful conversations. We also thank Jacek Olesiak for assisting the development of the program which computes the action of symmetry on the Khovanov homology. Part of this work was completed when the authors were at the IMPAN Knots, Homology, and Physics Simons Semester and joint with Warsaw University IDUB programme. We also thank the Simons Center for Geometry and Physics for hosting some of the authors during the conference Gauge Theory and Floer Homology in Low Dimensional Topology. ID was supported by NSF grant DMS--2303823. MB was supported by the NCN grant OPUS 2024/53/B/ST1/03470. MS was partially supported by NSF grant DMS-2203828.

\section{Background and Preliminaries}\label{sec:topological-preliminaries}

In this section we review some definitions regarding involutive links and their diagrams. Our discussion is mostly taken from \cite[Section 2]{BI}, \cite[Section 2]{lobb-watson} and \cite[Section 2]{DMS}.

\subsection{Involutive links}
We begin with the definition of an involutive link.

\begin{defn}\label{def:involutive-link}
An \textit{involutive link} is a pair $(L, \tau)$, where $\tau$ is an orientation-preserving involution on $S^3$ with one-dimensional fixed-point set and $L$ is a link which is setwise fixed by $\tau$. In the case that $L$ is a knot, we refer to $(L, \tau)$ as an \textit{equivariant knot}. If $L$ is disjoint from the fixed-point set of $\tau$, then we say that $(L, \tau)$ is \textit{(2-)periodic}. If each component of $L$ intersects the fixed-point set of $\tau$ in two points, then we say that $(L, \tau)$ is \textit{strongly invertible}. 
\end{defn}

\begin{defn}\label{def:Sakuma-equivalent}
Two involutive links $(L, \tau)$ and $(L', \tau')$ are said to be \textit{equivariantly diffeomorphic} (or \textit{Sakuma equivalent}) if there exists a diffeomorphism $f \colon S^3 \rightarrow S^3$ which sends $L$ to $L'$ and intertwines $\tau$ and $\tau'$; that is, $f \circ \tau = \tau' \circ f$.
\end{defn}


 

The (resolution of the) Smith conjecture \cite{Waldhausen, MorganBass} states that the fixed-point set of any $\tau$ as in Definition~\ref{def:involutive-link} is an unknot. We fix a standard model for such an involution as follows:

\begin{defn}\label{def:std-involution}
Identify $S^3$ with the one-point compactification of $\mathbb{R}^3$. The \textit{standard involution $\tau$ on $S^3$} is defined by extending the involution on $\mathbb{R}^3$ given by $\pi$-rotation around the $z$-axis. 
\end{defn}

Up to Sakuma equivalence, any involutive link in $S^3$ may be regarded as an involutive link in this standard model. We thus often suppress writing $\tau$ when discussing $L$.

\begin{defn}\label{def:equivariant-isotopy}
Let $L$ and $L'$ be two involutive links in the same standard model for $S^3$. We say that $L$ and $L'$ are \textit{equivariantly isotopic in $S^3$} if they are isotopic through a family of involutive links. 
\end{defn}




We will also be concerned with cobordisms of involutive links. To set this up, note that since we have fixed our model of $\tau$ on $S^3$, there is a standard extension $\tau_{S^3 \times I}$ of $\tau$ over $S^3 \times I$, obtained by taking the product of $\tau$ with the identity map. Likewise, there is an obvious extension $\tau_{B^4}$ of $\tau$ over $B^4$.

\begin{defn}\label{def:stdeqcob}
Let $L_1$ and $L_2$ be two equivariant knots. We say that a cobordism $\Sigma \subseteq S^3 \times I$ from $L_1$ to $L_2$ is \textit{equivariant} if $\tau_{S^3 \times I}(\Sigma) = \Sigma$, setwise.
\end{defn}

In the special case of an equivariant knot, we can of course consider an equivariant slice surface:

\begin{defn}\label{def:stdeqgenus}
Let $K$ be an equivariant knot. A slice surface $\Sigma \subseteq B^4$ for $K$ is \textit{equivariant} if $\tau_{B^4}(\Sigma) = \Sigma$, setwise. The \textit{equivariant slice genus} of $K$ is defined to be the minimum genus over all equivariant slice surfaces for $K$ in $B^4$. We denote this quantity by $\eg(K)$, suppressing the dependence on $\tau$.
\end{defn}

\begin{rmk}\label{rem:standard}
In the literature, a surface $\Sigma$ is often said to be equivariant if it is fixed under \textit{any} extension of $\tau$ over $B^4$ as an involution, and  likewise for an equivariant concordance or cobordism. This is \textit{not} the notion which we will consider in this paper. Indeed, due to the diagrammatic nature of Khovanov homology, it will be convenient for us to assume that the extension of $\tau$ over $B^4$ or $S^3 \times I$ is the standard one. Thus the quantity $\eg(K)$ of Definition~\ref{def:stdeqgenus} is not \emph{a priori} the same as the quantity $\eg(K)$ defined in \cite[Definition 3.2]{BI} or \cite[Definition 2.1]{DMS}. However, since we will always deal with the standard extension of $\tau$, we hope this will cause no confusion. The authors do not know of any example showing that these notions differ.
\end{rmk}

In the setting of strongly invertible knots, it is furthermore possible to define an equivariant connected sum operation and form an equivariant concordance group. This requires some additional formalism due to the fact that the connected sum operation depends \textit{a priori} on choosing an equivariant basepoint on each summand. We refer the reader to \cite{Sakuma} for further discussion.

Finally, we discuss the isotopy-equivariant slice genus:

\begin{defn}\label{def:stdisotopyeqgenus}
We say that a knot cobordism $\Sigma$ is \textit{isotopy-equivariant} if $\tau_{S^3 \times I}(\Sigma)$ is isotopic to $\Sigma$ rel boundary. Likewise, we say that a slice surface $\Sigma$ for $K$ is an \textit{isotopy-equivariant slice surface} if $\tau_{B^4}(\Sigma)$ is isotopic to $\Sigma$ rel $K$. Define the \textit{isotopy-equivariant slice genus} of $K$ by:
\[
\ieg(K) = \min_{\text{all isotopy-equivariant slice surfaces $\Sigma$}} \{g(\Sigma)\}.
\]
Here $\ieg(K)$ depends on $\tau$, but we suppress this from the notation.
\end{defn}

Note that calculating $\ieg(K)$ allows us to easily construct pairs of slice surfaces which are not isotopic rel boundary. For example, if $K$ is an equivariant slice knot with $\ieg(K) > 0$, then we may take \textit{any} slice disk $\Sigma$ for $K$ and conclude that $\Sigma$ and $\tau(\Sigma)$ are not isotopic. This was used in \cite{DMS} to show that knot Floer homology detects exotic pairs of disks. Although useful in this context, it is an interesting question to ask whether the isotopy-equivariant genus differs from the usual equivariant genus. Most invariants which bound the equivariant genus of $K$ (for example, \cite{DMS, Sano}) actually bound the isotopy-equivariant genus; prior to the current work, these notions were not known to differ. Here we show that the equivariant genus can be arbitrarily larger than the isotopy-equivariant genus. While most of our work is in the smooth category, we also show that $\eg$ and $\ieg$ differ in the topological category. For this, we give a proof of the invariance of Sakuma's $\eta$-polynomial under topological concordance; see Appendix~\ref{sec:appendixa}.

\subsection{Projections and link diagrams}

We now discuss some conventions involving projections and link diagrams, based on \cite[Section 2]{lobb-watson}. To the best of our knowledge, many genericity results regarding diagrams of strongly invertible links have been stated, but rigorous proofs have not appeared in the literature. A detailed, Morse-theoretical approach to diagrams of strongly invertible links will be given in our forthcoming paper \cite{BDMS2}. 

Recall that we have fixed a model for $S^3$ as the one-point compactification of $\mathbb{R}^3$.
\begin{defn}\label{std-projection}
  Define the \textit{standard projection} to be projection from $\mathbb{R}^3$ onto the $xz$-plane. We say that a link $L$ is \textit{generic} with respect to the standard projection if it misses the point at $\infty$ in $S^3$ and the standard projection subsequently gives a link diagram for $L$. (That is, the image of $L$ has no cusps, no tangency points, and no triple points.)
\end{defn}

If $L$ is any link in $S^3$, then we may make $L$ generic with respect to the standard projection simply by applying a small isotopy. Moreover, if $L$ is an involutive link in our standard model for $S^3$, then we may isotope $L$ through involutive links so that it is generic. Note that the resulting diagram will then be invariant under rotation about the $z$-axis. An equivariant knot diagram with this property (where $\Fix(\tau)$ is visible in the plane of projection) is called a \textit{transvergent diagram}. Henceforth, we will usually assume that our involutive links are generic with respect to the standard projection and are represented by transvergent diagrams. See Figure~\ref{fig:transvergent-diagram}. 

\begin{figure}[h!]
\includegraphics[scale = 0.8]{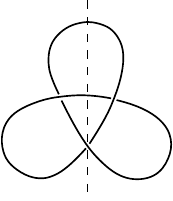}
	\caption{A transvergent diagram for the trefoil.}
	\label{fig:transvergent-diagram}
\end{figure}

Now let $L$ and $L'$ be a pair of involutive links with transvergent diagrams $D$ and $D'$, respectively. There are two slightly different notions of equivariant isotopy. The first is Definition~\ref{def:equivariant-isotopy}, where $L$ and $L'$ are isotopic through involutive links in $S^3$. The second is to consider equivariant isotopy in $\mathbb{R}^3$; that is, isotopies through involutive links which miss the point at $\infty$. In our fixed model of $S^3$, these two notions are not the same: Figure~\ref{fig:imove} may be interpreted as an equivariant isotopy which goes over the point at $\infty$. In contrast to the non-equivariant case, it is not possible to perturb this isotopy in an equivariant manner to lie in $\mathbb{R}^3$. 

As with ordinary link diagrams, there is a list of local moves such that two involutive links $L$ and $L'$ (with transvergent diagrams $D$ and $D'$) are equivalent if and only if $D$ and $D'$ are connected by a sequence of local moves and equivariant planar isotopies. However, this list of moves differs slightly depending on whether we consider our notion of equivalence to be equivariant isotopy in $\mathbb{R}^3$, equivariant isotopy in $S^3$, or Sakuma equivalence. We record this in Theorem \ref{thm:equivalences-of-diagram} below; see also \cite[Theorem 2.3]{lobb-watson}:




\begin{thm}\label{thm:equivalences-of-diagram}
	Let $L$ and $L'$ be two involutive links in the same standard model for $S^3$. Let $D$ and $D'$ be transvergent diagrams for $L$ and $L'$.  Then:
	\begin{enumerate}
		\item $L$ and $L'$ are isotopic as involutive links in $\mathbb{R}^3$ if and only if $D$ and $D'$ are related by the involutive Reidemeister moves (IR-1) through (M-3) in Figure~\ref{fig:all_equi_moves_lobb_watson}.
		\item $L$ and $L'$ are isotopic as involutive links in $S^3$ if and only if $D$ and $D'$ are related by the involutive Reidemeister moves, together with the I-move of passing a strand through the point at infinity, as in Figure~\ref{fig:imove}. 
		\item $L$ and $L'$ are Sakuma equivalent if and only if $D$ and $D'$ are related by the involutive Reidemeister moves, the I-move, and the R-move of rotating the transvergent diagram $\pi$ around the origin, as in Figure~\ref{fig:rmove}.
	\end{enumerate}
\end{thm}
\begin{proof}
Except for the I-move, the claim is shown in \cite[Theorem 2.3]{lobb-watson}. The I-move in the non-equivariant setting is discussed e.g. in \cite{MWW}, where it leads to the sweep-around move. To the best of our knowledge, the I-move has not been studied in the equivariant setting. We refer to \cite{BDMS2} for a detailed description.
\end{proof}

\begin{figure}[h!]
\includegraphics[width=6cm]{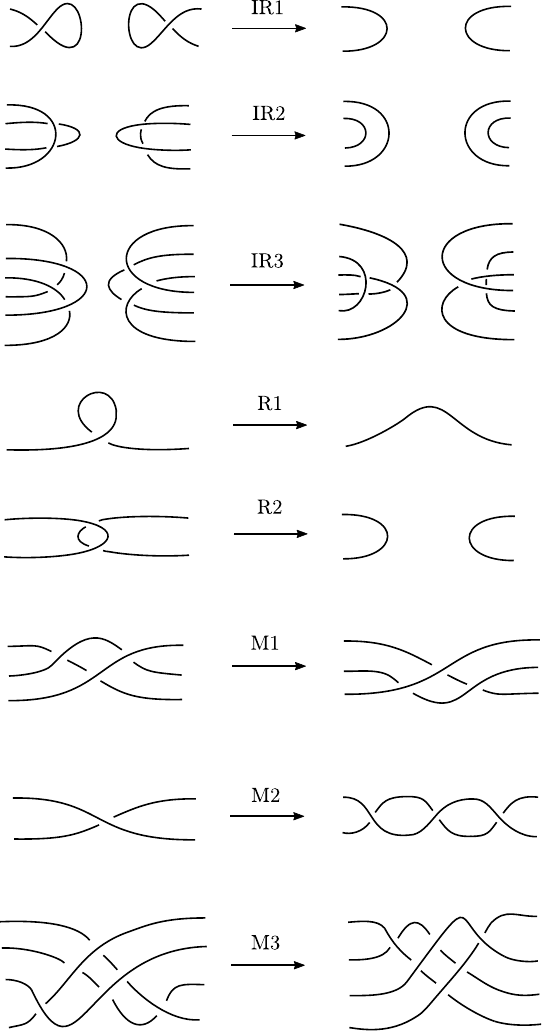}
	\caption{Set of equivariant Reidemeister moves.}
	\label{fig:all_equi_moves_lobb_watson}
\end{figure}

\begin{figure}[h!]
\includegraphics[scale = 0.75]{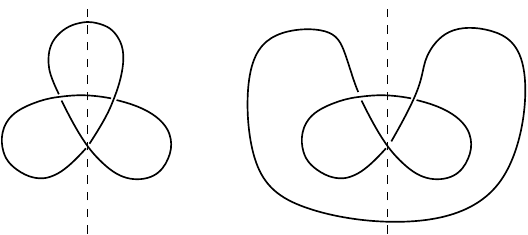}
	\caption{Illustration of the I-move. The strand at the top of the trefoil passes through the point at $\infty$.}
	\label{fig:imove}
\end{figure}

\begin{figure}[h!]
\includegraphics[scale = 0.75]{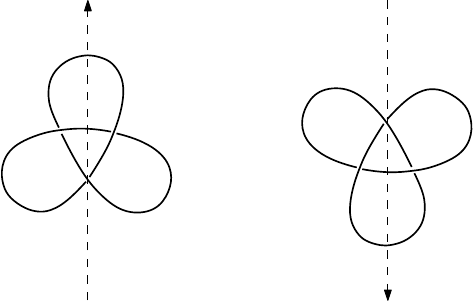}
	\caption{Illustration of the R-move. This consists of $\pi$ rotation about the origin. Note that the orientation on $\Fix(\tau)$ is reversed.}
	\label{fig:rmove}
\end{figure}

We will show that our Khovanov-type invariant for involutive links is an invariant up to Sakuma equivalence, which is the most flexible notion of the three listed above. We thus establish invariance of diagram under the set of involutive Reidemeister moves, the I-move, and the R-move.

\subsection{Cobordisms}

We now briefly discuss the diagrammatics of cobordisms and equivariant cobordisms. 

\begin{defn}\label{def:std-projection-cobordism}
  Define the \textit{standard projection} on $\mathbb{R}^3 \times I$ to be the standard projection in each $\mathbb{R}^3 \times \{t\}$. A cobordism $\Sigma$ in $S^3 \times I$ is \textit{generic} with respect to the standard projection if it misses the arc $\{\infty\} \times I$ and the standard projection subsequently gives a cobordism movie.
\end{defn}

Here, by a cobordism movie we mean a diagrammatic depiction of the cobordism as a sequence of Reidemeister moves, planar istopies, and \textit{elementary cobordisms}, which consist of births, deaths, and saddles. If $\Sigma$ is a cobordism between two links which are generic with respect to the standard projection, then it is a standard fact that we may perturb $\Sigma$ (rel boundary) to itself be generic; compare \cite{BDMS2}. In the equivariant setting, however, we need a slightly different notion:

\begin{defn}\label{def:std-projection-equivariant-cobordism}
An equivariant cobordism $\Sigma$ in $S^3 \times I$ is \textit{equivariantly generic} with respect to the standard projection if it misses the arc $\{\infty\} \times I$ and the standard projection subsequently gives an equivariant cobordism movie.
\end{defn}

Here, by an equivariant cobordism movie, we mean a diagrammatic depiction of the cobordism as a sequence of equivariant Reidemeister moves, equivariant planar isotopies, and \textit{elementary equivariant cobordisms}. Equivariant Reidemeister moves have already been discussed. The elementary equivariant cobordisms are depicted in Figure~\ref{fig:all_equi_moves_lobb_watson_cob}; these consist of on-axis births, deaths, and saddles, together with pairs of off-axis births, deaths, and saddles. See \cite{BDMS2}. Note that an equivariantly generic cobordism is not usually generic in the sense of Definition~\ref{def:std-projection-cobordism}.

\begin{figure}[h!]
\includegraphics[scale = 0.5]{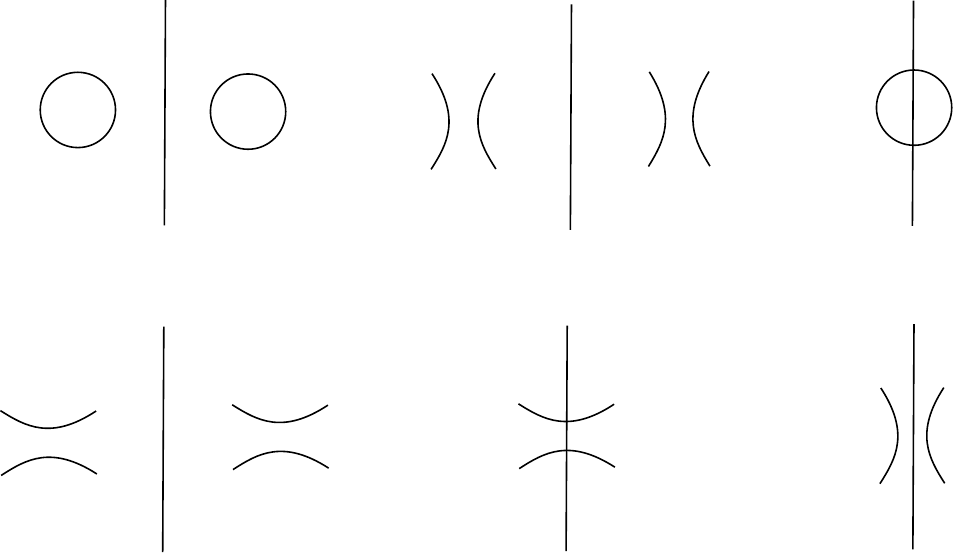}
	\caption{Set of elementary equivariant cobordisms. Note that there are death moves corresponding to each of the symmetric birth moves.}
	\label{fig:all_equi_moves_lobb_watson_cob}
\end{figure}

Once again, up to equivariant isotopy, any equivariant cobordism can be assumed to be equivariantly generic. However, there is a slight subtlety: in contrast to the non-equivariant case, this isotopy cannot in general be assumed to fix the boundary of $\Sigma$. For example, the I-move of Theorem~\ref{thm:equivalences-of-diagram} can be viewed as the trace of an equivariant isotopy $\Sigma$ which passes over $\{\infty\} \times I$. This cannot be equivariantly isotoped \textit{rel boundary} to miss the arc $\{\infty\} \times I$. However, if the outgoing end of the cobordism is allowed to move, then it is clear that $\Sigma$ is equivariantly isotopic in $S^3 \times I$ to the product cobordism. In general, to see that $\Sigma$ can be equivariantly isotoped to miss $\{\infty\} \times I$, choose a path $\gamma$ in $S^3 \times I$ such that $\gamma$ is disjoint from $\Sigma$ and intersects each $S^3 \times \{t\}$ at a single (non-constant) point on $\Fix(\tau)$. There is then an equivariant isotopy of $S^3 \times I$ which moves $\gamma$ into $\{\infty\} \times I$. 

In light of the above, we will usually assume that our equivariant cobordisms have been equivariantly isotoped (possibly moving the boundary) to lie in $\mathbb{R}^3 \times I$, and are equivariantly generic. This relegates all discussion of the I-move to the setting of invariance of diagram. While this is not essential, we would otherwise have to consider the I-move when constructing cobordism maps, which is not usually how cobordism maps are discussed in the Khovanov literature. We thus assume that all of our non-equivariant and equivariant cobordisms lie in $\mathbb{R}^3 \times I$.

\subsection{Algebraic preliminaries}

In no particular order, we collect here some algebraic results which will be useful throughout the paper.

As indicated in the introduction, we will often be dealing with chain complexes over different polynomial rings. In order to pass between these theories, it will be helpful for us to have the following lemma:

\begin{lem}\label{lem:quasi-iso-move} Let $R$ be a commutative ring and $C_1$ and $C_2$ be free, finitely-generated complexes over an $\F$-algebra $R[h]$. A chain map $f\colon C_1 \rightarrow C_2$ is a quasi-isomorphism if and only if the induced map $f/h \colon C_1/h \rightarrow C_2/h$ is a quasi-isomorphism. 
\end{lem}
\begin{proof}
  The $R[h]$-module $R_h = R[h]/h$ admits a length-one free resolution $0\to R[h]\xrightarrow{\cdot h} R[h]\to R_h\to 0$. That is, the projective dimension of $R_h$ is $1$ (see \cite[Chapter 4]{Weibel_book}), so all the $\Tor$ groups involving $R_h$ vanish except the zeroth and the first.
  Consequently, for any 
chain complex $C$ of free $R[h]$-modules, the universal coefficient spectral sequence computing the homology of $C\otimes R_h$
  degenerates, yielding a natural family of short exact sequences:
  \[
  0\to H_i(C)\otimes_{R[h]} R_h\to H_i(C\otimes_{R[h]} R_h)\to\Tor_1(H_{i+1}(C),R_h)\to 0.\footnote{We write homology rather than cohomology here since the Khovanov homology convention is that the differential increases grading by one.}
  \]
  Suppose $f$ is a chain map from $C_1$ to $C_2$. We have the commutative diagram
\[
  \begin{tikzcd}
    0 \ar[r] & H_i(C_1)\otimes_{R[h]}R_h\ar[r]\ar[d,"f_{1i}"] & H_i(C_1\otimes_{R[h]} R_h) \ar[r]\ar[d,"f_{2i}"] & \Tor_1(H_{i+1}(C_1),R_h) \ar[r]\ar[d,"f_{3i}"] & 0\\
    0 \ar[r] & H_i(C_2)\otimes_{R[h]}R_h\ar[r] & H_i(C_2\otimes_{R[h]} R_h) \ar[r] & \Tor_1(H_{i+1}(C_2),R_h) \ar[r]& 0
  \end{tikzcd}
\]
Suppose $f$ is a quasi-isomorphism. Then, the maps $f_{1i}$ and $f_{3i}$ are isomorphisms for all $i$.
The
statement that $f/h$ is a quasi-isomorphism is equivalent to saying that $f_{2i}$ is an isomorphism for all $i$. This follows
clearly from the five lemma.

To prove the converse, consider the filtration on $C_i$ given by powers of $h$, so that $C_i^\ell = h^\ell C_i$. The associated spectral sequence has $E^1$-page given by $H_*(C_i/h)\otimes \f[h]$; by finite generation, this converges to $H_*(C_i)$.  
As the complexes $C_1$ and $C_2$ are free, the morphism $f$ induces an isomorphism of $E^1$-pages,
by the hypothesis of the lemma, and so also induces an isomorphism $f_*\colon H_*(C_1)\to H_*(C_2)$. 
\end{proof}

As a special case, note the following:

\begin{lem} \label{lem:iso-after-set-zero}
Let $R$ be a commutative ring and $C_1$ and $C_2$ be free, finitely-generated complexes over an $\F$-algebra $R[h]$. A chain map $f\colon C_1 \rightarrow C_2$ is an isomorphism of chain complexes if and only if the induced map $f/h \colon C_1/h \rightarrow C_2/h$ is an isomorphism.
\end{lem}
\begin{proof}
This follows immediately from Lemma~\ref{lem:quasi-iso-move} by considering $C_1$ and $C_2$ as complexes with zero differential.
\end{proof}

\begin{rmk}
Lemmas~\ref{lem:quasi-iso-move} and \ref{lem:iso-after-set-zero} do not hold if $C_1$ and $C_2$ are not free. For example, if $C_1=R[h]/h^3$
  and $C_2=R[h]/h^2$ (with trivial differential), then there is a map from $C_1$ to $C_2$ inducing an isomorphism on $C_1/h\cong R$
  and $C_2/h\cong R$, but the complexes are not quasi-isomorphic.
\end{rmk}

Finally, it will be useful to translate between quasi-isomorphism and homotopy equivalence. The following is well-known:

\begin{lem} \label{lem:qitohe}
  Let $C_1$ and $C_2$ be free, bounded chain complexes. Then any quasi-isomorphism $f \colon C_1 \rightarrow C_2$ admits a homotopy inverse. That is, if $f \colon C_1 \rightarrow C_2$ is a quasi-isomorphism, then there exists a chain map $g \colon C_2 \rightarrow C_1$ such that $g \circ f \simeq \id$ and $f \circ g \simeq \id$. In particular, $C_1$ and $C_2$ are homotopy equivalent.
\end{lem}
\begin{proof}
This follows from \cite[Chapter 10.4]{Weibel_book}.
\end{proof}


\section{Bar-Natan Homology}\label{sec:khovanov-homology-preliminaries}

In this section, we discuss Bar-Natan homology \cite{bar-natan-tangle} and establish notation.  

\subsection{Bar-Natan homology}

Let $L$ be a knot or link. The Bar-Natan complex $\Kcm(L)$ of $L$ is a bigraded chain complex over the one-variable polynomial ring $\F[u]$, where $\F = \Z/2\Z$. We denote the two components of the bigrading by $(\gr_h, \gr_q)$; these are referred to as the \textit{homological grading} and the \textit{quantum grading}, respectively. Our conventions are that
\[
\gr(\partial) = (1, 0) \quad \text{and} \quad \gr(u) = (0, -2).
\]
In other literature (for example, \cite{bar-natan-tangle,dunfield-lipshitz-schuetz}), the variable $u$ is replaced by $h$. Here, we use $u$ in analogy with the variable $U$ from Heegaard Floer homology. Note that with respect to multiplication by $u$, it is the quantum grading which plays the role of the homological grading from Heegaard Floer theory. We review the pieces needed for the construction of $\Kcm(L)$ below.

\subsubsection*{The cube of resolutions} For $n \in \mathbb{N}$, we refer to $\{0,1\}^n$ as the \textit{$n$-dimensional cube}. We denote a typical vertex of $\{0,1\}^n$ by $v = (v_1, \ldots, v_n)$. If $v$ and $w$ are two vertices in $\{0,1\}^n$ that differ in exactly one coordinate $i$ for which $v_i = 0$ and $w_i = 1$, then we say that $w$ is an \emph{immediate successor} of $v$. In this situation, we draw an (ordered) edge from $v$ to $w$, which we denote by $\phi_{v, w}$. 

Now let $D$ be a diagram for $L$ with $n$ crossings, which we assume have been enumerated from $1$ to $n$. For each vertex in the $n$-dimensional cube, let $D_v$ be the complete resolution of $D$ formed by taking the $0$-resolution at the $i$th crossing if $v_i=0$, and the $1$-resolution if $v_i = 1$.  The diagram $D_v$ is a planar diagram of embedded circles; we refer to this set of circles by $Z(D_v)$. For each edge $\phi_{v, w}$, we construct an embedded $1$-handle cobordism in $\R^2 \times [0, 1]$ from $D_v$ to $D_w$. This either \textit{merges} two circles in $D_v$ or it \textit{splits} a single circle from $D_v$ into two. The set of all such resolutions and $1$-handle cobordisms between them is referred to as the \textit{cube of resolutions}. 

\begin{figure}[h!]
\includegraphics[scale = 0.9]{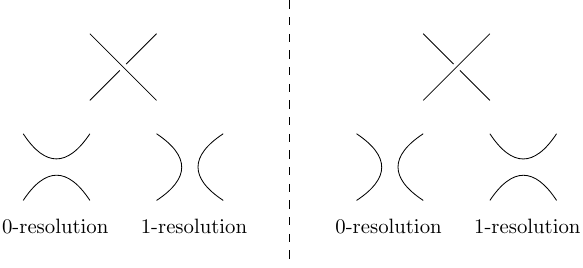}
	\caption{The $0$- and $1$-resolutions of a crossing.}
	\label{fig:resolutions}
\end{figure}

\subsubsection*{The Frobenius algebra} Let $\mathcal{A} = \spa_{\F[u]}\{1, x\}$ be the two-dimensional Frobenius algebra over $\F[u]$ with basis $\{1, x\}$ and multiplication $m\colon\cA\otimes\cA\to\cA$ given by
\[
m(1 \otimes 1)=1,\qquad m(1 \otimes x)=x,\qquad m(x \otimes 1)=x,\qquad m(x \otimes x)=ux.
\]
We define a comultiplication $\Delta\colon\cA\to\cA\otimes\cA$ by
\[
\Delta(1)=1\otimes x + x\otimes 1 + u(1\otimes 1),\qquad \Delta(x)=x\otimes x.
\]
Here, all instances of $\otimes$ are understood over $\F[u]$. The unit and counit are given by 
\[
i(1)=1,\qquad \epsilon(1)=0,\qquad \epsilon(x)=1.
\]
Bar-Natan homology comes from applying the (1+1)-dimensional TQFT corresponding to $\mathcal{A}$ to the cube of resolutions. We recall the meaning of this below.

\subsubsection*{Bar-Natan homology} 

For each $v \in \{0,1\}^n$, define an $\F[u]$-module $\AbFunc(v)$ by taking the tensor product of one copy of $\mathcal{A}$ for each circle $c \in Z(D_v)$:
\[
\AbFunc(v) = \bigotimes_{c \in Z(D_v)} \mathcal{A}.
\]
This tensor product is taken over $\F[u]$. A basis element of $\AbFunc(v)$ is thus obtained by labeling each $c \in Z(D_v)$ by either $1$ or $x$; $\AbFunc(v)$ is the $\F[u]$-span of these elements. For each edge $\phi_{v, w}$, define a $\F[u]$-equivariant map
\[
\AbFunc(\phi_{v,w}) \colon \AbFunc(v) \rightarrow \AbFunc(w).
\]
on basis elements as follows:
\begin{enumerate}
\item Suppose $\phi_{v, w}$ corresponds to merging $c_\alpha$ and $c_\beta$ to $c_\gamma$. Define $\AbFunc(\phi_{v,w})$ by carrying forward all labels on $Z(D_v) \setminus \{c_\alpha, c_\beta\}$ and labeling $c_\gamma$ by $m(a_\alpha \otimes a_\beta)$, where $a_\alpha$ and $a_\beta$ are the labels on $c_\alpha$ and $c_\beta$, respectively. In the case $a_\alpha = a_\beta = x$, note that the image has a coefficient of $u$.
\item Suppose $\phi_{v, w}$ corresponds to splitting $c_\alpha$ into $c_\beta$ and $c_\gamma$. Define $\AbFunc(\phi_{v,w})$ by carrying forward all labels on $Z(D_v) \setminus \{c_\alpha\}$ and labeling $c_\beta$ and $c_\gamma$ by the two factors of $\Delta(a_\alpha)$, where $a_\alpha$ is the label on $c_\alpha$. This is done in an $\f[u]$-linear manner: for example, if $a_\alpha = 1$, the image has three terms.
\end{enumerate}
Extend this map $\F[u]$-linearly. Define the \textit{Bar-Natan complex} associated to $D$ by setting
\[
\Kc^-(D) = \bigoplus_{v \in \{0,1\}^n} \AbFunc(v) \qquad \text{and} \qquad \partial = \sum_{\text{edges } \phi_{v,w}} \AbFunc(\phi_{v,w}).
\]
As we will see, the homotopy type of $\Kc^-(D)$ is an invariant of $L$; we correspondingly refer to this homotopy type by $\Kc^-(L)$. The notation $\Kc^-$ is meant to suggest Heegaard Floer homology $HF^-$. As in Heegaard Floer theory, we may form
\[
\Kchat(D) = \Kcm(D)/(u = 0) \quad \text{and} \quad \Kcinf(D) = \Kcm(D)\otimes_{\f[u]}\f[u,u^{-1}].
\]
Note that setting $u = 0$ gives the Frobenius algebra used in the usual definition of Khovanov homology \cite{khovanov}. The $u = 0$ theory is thus the ordinary Khovanov complex $\CKh(D)$, whose homology is the ordinary Khovanov homology $\Kh(L)$.

\subsubsection*{Gradings}\label{sec:gradings}

Suppose now that each component of $L$ is oriented and let $n_+$ and $n_-$ be the number of positive and negative crossings, respectively. Let $v \in \{0,1\}^n$ and $y \in \AbFunc(v)$ be a basis element given by a labeling of the circles in $D_v$. Let $|v|$ be the number of coordinates of $v$ equal to $1$ and $\ell=\ell(y)$ be the number of circles in $D_v$ which are labeled $x$. Define
\[
\gr_h(y) \; = \; |v|-n_- \quad \text{and} \quad \gr_q(y) \; = \; |Z(D_v)|-2\ell+|v|+n_+-2n_-.
\]
We extend this to the whole complex by setting $\gr(u) = (0, -2)$. It is easily checked that with respect to this bigrading, $\gr(\partial) = (1, 0)$. Note that in order to define the absolute grading, it is necessary to orient each component of $L$ to obtain $n_+$ and $n_-$. However, in the unoriented setting, it is still possible to disregard these and define a relative $\Z \oplus \Z$-grading.

\subsection{Reduced Bar-Natan homology}\label{sec:reduced-bar-natan}
We now review the construction of the \textit{reduced} Bar-Natan complex. We first define the reduced complex using a choice of basepoint on $L$, as in \cite{wigderson, khovanov}. This will be necessary to apply various results in the literature regarding the reduced Bar-Natan homology. For several other applications, however, it will be more convenient to have a definition of the reduced complex which does not require a choice of basepoint; this can be found in \cite{ORS}. We thus also review the unpointed construction and explain why the pointed and unpointed constructions are isomorphic. 


\subsubsection*{Reduced complex, pointed construction} Let $L$ be equipped with a basepoint $p$, so that $D$ is a pointed link diagram. Every resolution $D_v$ of $D$ then comes equipped with a single distinguished circle $c_p$ containing the point $p$. The standard definition of the reduced complex is as follows:

\begin{defn}\label{def:pointed-reduced}
The \textit{(pointed) reduced Bar-Natan complex} is given by
\[
\Kcrm_p(D) = \spa_{\F[u]} \{\text{generators labeling the circle containing } p \text{ with }x\}.
\]
It is straightforward to check that this is a subcomplex of $\Kc^-(D)$. We use the subscript $p$ to emphasize the dependence on the choice of basepoint, although (as we will see) different choices of basepoint result in isomorphic reduced complexes. 
\end{defn}




\subsubsection*{Reduced complex, unpointed construction} In order to give an unpointed version of the reduced complex, it will be useful to have a slight algebraic reformulation of the overall Bar-Natan complex. Let $D_v$ be a resolution of $D$ and suppose the circles of $D_v$ have been enumerated from $1$ to $m$. Define an $\F[u]$-module isomorphism
\begin{equation}\label{eqn:reducediden}
\AbFunc(v) \cong \F[x_1, \ldots, x_m]/(x_i^2 = ux_i \text{ for } i = 1, \ldots, m)
\end{equation}
as follows. Consider a basis element of $\AbFunc(v)$ in which the circles $c_{i_1}, \ldots, c_{i_k}$ are labeled $x$ and all other circles have been labeled $1$. We map this basis element to the monomial $x_{i_1} \cdots x_{i_k}$ and extend $\F[u]$-linearly. It is clear that this is an isomorphism of $\F[u]$-modules. However, the right-hand side additionally has an algebra structure which will be convenient presently.

We now re-write the edge maps $\AbFunc(\phi_{v, w})$ under the identification (\ref{eqn:reducediden}). There are two possibilities:
\begin{enumerate}
\item Suppose $\phi_{v, w}$ corresponds to merging $c_\alpha$ and $c_\beta$ to $c_\gamma$. Let $\pi$ be the corresponding map from the index set of $D_v$ to the index set of $D_w$; this is a bijection except that $\pi(\alpha) = \pi(\beta) = \gamma$. It is straightforward to check that under the identification (\ref{eqn:reducediden}), $\AbFunc(\phi_{v,w})$ is defined on monomials by
\begin{equation}\label{eq:rewrite1}
\AbFunc(\phi_{v,w})(x_{i_1} \cdots x_{i_k}) = x_{\pi(i_1)} \cdots x_{\pi(i_k)},
\end{equation}
extending $\F[u]$-linearly. (In the case that $c_\alpha$ and $c_\beta$ and labeled $x$, use the fact that $x_\gamma^2 = u x_\gamma$.) Note that $\AbFunc(\phi_{v,w})$ thus corresponds to the algebra map induced by the (non-injective) relabeling $\pi$.
\item Suppose $\phi_{v, w}$ corresponds to splitting $c_\alpha$ into $c_\beta$ and $c_\gamma$. Choose either one of the two resulting circles, say $c_\beta$. Let $\iota$ be the injection from the index set of $D_v$ to the index set of $D_w$ which sends $\alpha$ to $\beta$ and otherwise sends circles in $D_v$ bijectively to their images in $D_w$. We claim that under the identification (\ref{eqn:reducediden}), $\AbFunc(\phi_{v,w})$ is defined on monomials by
\begin{equation}\label{eq:rewrite2}
\AbFunc(\phi_{v,w})(x_{i_1} \cdots x_{i_k}) = (x_\beta + x_\gamma + u \cdot 1)(x_{\iota(i_1)} \cdots x_{\iota(i_k)}).
\end{equation}
If $x_\alpha$ does not appear in the monomial $x_{i_1} \cdots x_{i_k}$, then this follows immediately from the comultiplication $\Delta(1) =  1 \otimes x + x \otimes 1 + u(1 \otimes 1)$. If $x_\alpha$ does appear in the monomial $x_{i_1} \cdots x_{i_k}$ (say as $x_{i_1}$), then the comultiplication $\Delta(x) = x \otimes x$ indicates that under the identification (\ref{eqn:reducediden}), we have
\[
\AbFunc(\phi_{v,w})(x_{i_1} \cdots x_{i_k}) = x_\beta x_\gamma x_{\iota(i_2)} \cdots x_{\iota(i_k)}.
\]
To verify that this coincides with (\ref{eq:rewrite2}), we compute
\begin{align*}
(x_\beta + x_\gamma + u \cdot 1)(x_{\iota(\alpha)} x_{\iota(i_2)} \cdots x_{\iota(i_k)}) &= (x_\beta + x_\gamma + u \cdot 1)(x_{\beta} x_{\iota(i_2)} \cdots x_{\iota(i_k)}) \\
&= x_\beta^2 x_{\iota(i_2)} \cdots x_{\iota(i_k)} + x_\gamma x_\beta x_{\iota(i_2)} \cdots x_{\iota(i_k)} + u x_\beta x_{\iota(i_2)} \cdots x_{\iota(i_k)} \\
&= x_\gamma x_\beta x_{\iota(i_2)} \cdots x_{\iota(i_k)},
\end{align*}
where in the last line we have used the fact that $x_\beta^2 = u x_\beta$ and canceled the first and third terms. Note that $\AbFunc(\phi_{v,w})$ thus corresponds to first applying the algebra map induced by the (non-surjective) relabeling $\iota$, and then multiplying by $(x_\beta + x_\gamma + u \cdot 1)$.
\end{enumerate}
\noindent
We refer to the above version of the Bar-Natan complex -- equipped with its additional algebraic structure at each resolution -- as $\Kc'(D)$. Note that $\Kc'(D)$ has an algebra for each vertex $v$, which we denote by $\Kc'(D_v)$, together with $\f[u]$-module maps from $\Kc'(D_v)$ to $\Kc'(D_w)$ for each edge $\phi_{v, w}$ in $\{0, 1\}^n$. These module maps are not algebra maps: in the case of (\ref{eq:rewrite2}), the map is an algebra map followed by multiplication by $x_\beta + x_\gamma + u \cdot 1$. The isomorphism (\ref{eqn:reducediden}) provides the explicit isomorphism between $\Kcm(D)$ and $\Kc'(D)$.

This re-interpretation of $\Kcm(D)$ allows us to define an unpointed version of the reduced complex.

\begin{defn}
For each $v \in \{0, 1\}^n$, let $\Kcrm_{un}(D_v)$ be the subalgebra of $\Kc'(D_v)$ generated (as a $\f[u]$-subalgebra) by the elements $1$ and $x_i + x_j$, running over all $i$ and $j$. Define the \textit{unpointed reduced complex} $\Kcrm_{un}(D)$ to be the $\f[u]$-submodule
\[
\Kcrm_{un}(D)=\bigoplus_{v\in  \{0, 1\}^n} \Kcrm_{un}(D_v)\subset \Kc'(D).
\]
In fact, $\Kcrm_{un}(D)$ is a subcomplex of $\Kc'(D)$; see \cite{ORS} for the $u=0$ case. To see this, we simply check that $\Kcrm_{un}(D)$ is sent to itself by the edge maps. For merges this is clear, since in this case $\AbFunc(\phi_{v, w})$ is an algebra map which sends $x_i + x_j$ to $x_{\pi(i)} + x_{\pi(j)}$. For splits, first observe that the algebra map induced by the relabeling $\iota$ likewise maps $\Kcrm_{un}(D_v)$ into $\Kcrm_{un}(D_w)$. We obtain $\AbFunc(\phi_{v, w})$ by multiplying through by $(x_\beta + x_\gamma + u \cdot 1)$, which is an element of $\Kcrm_{un}(D_w)$.
\end{defn}

As usual, the complexes $\widehat{\Kcr}_{p}(D)$ and $\widehat{\Kcr}_{un}(D)$ are obtained from the corresponding $\Kcrm$ complexes by setting $u = 0$, and likewise for $\Kcr_p^{\infty}$ and $\Kcr_{un}^{\infty}$. In the $u = 0$ case, it is well-known that the pointed and unpointed reduced complexes are isomorphic; see e.g.\ \cite{ORS}. The proof extends to give the following result for Bar-Natan homology:

\begin{lem}\label{lem:reduced-theories-equivalent}
For any basepoint $p\in D$, there is an isomorphism $\Kcrm_{un}(D) \cong \Kcrm_{p}(D)$ as bigraded chain complexes over $\F[u]$. This isomorphism is canonical in the sense that it depends only on $D$ and $p$.
\end{lem}
\begin{proof}
Let $p$ be a basepoint on $D$. Every resolution $D_v$ of $D$ comes equipped with a single distinguished circle $c_p$ containing the point $p$. Define a map
\[
\mu_p \colon \Kc'(D) \rightarrow \Kc'(D)
\]
by multiplying by $x_p$ in each resolution. It is straightforward to check that $\mu_p$ commutes with the edge maps and is hence a chain map. Note that under the isomorphism of (\ref{eqn:reducediden}), the pointed reduced complex $\Kcrm_p(D)$ is identified with the subcomplex $\mathrm{im}(\mu_p) \subseteq \Kc'(D)$.

Restricting the domain of $\mu_p$ gives a chain map
\[
\mu_p |_{\Kcrm_{un}(D)} \colon \Kcrm_{un}(D) \rightarrow \mathrm{im}(\mu_p) \cong \Kcrm_{p}(D).
\]
We claim that this restriction is an isomorphism. To see this, set $u = 0$. We obtain a chain map
\[
\widehat{\Kcr}_{un}(D) \rightarrow \widehat{\Kcr}_{p}(D).
\]
This is easily seen to be surjective. Indeed, consider a generator of $\widehat{\Kcr}_{p}(D)$ which labels the circles $c_p$, $c_{i_1}, \ldots, c_{i_k}$ with $x$. This corresponds to the element $x_p x_{i_1} \cdots x_{i_k}$ in $\Kc'(D)$. If $u = 0$, then $x_p x_{i_1} \cdots x_{i_k} = x_p(x_p + x_{i_1}) \cdots (x_p + x_{i_k})$, and hence is in the image of $\mu_p |_{\Kcrm_{un}(D)}$.

On the other hand, it is straightforward to see that both complexes have half the dimension of $\Kchat(D)$. For the pointed reduced complex this is obvious from the definition; while after setting $u = 0$, $\Kcrm_{un}(D_v)$ is easily seen to be a half-dimensional subspace of the exterior algebra on $Z(D_v)$. (The exterior algebra on $Z(D_v)$ has dimension $2^m$, where $m$ is the number of circles in $Z(D_v)$. The subspace in question is the exterior algebra on $x_1 + x_i$ for $2 \leq i \leq m$ and thus has dimension $2^{m-1}$.) By Lemma~\ref{lem:iso-after-set-zero}, a chain map between free, finitely-generated $\F[u]$-complexes which is an isomorphism after setting $u = 0$ is itself an isomorphism.
\end{proof}

In particular, it follows that the chain isomorphism type of the reduced complex does not depend on the choice of basepoint $p$. We thus write $\Kcrm(D)$ to refer to this isomorphism type, without fixing a basepoint or even specifying whether we are using the pointed or unpointed construction.

\subsubsection*{Splitting of the unreduced complex} It is well-known that the unreduced theory splits into two copies of the (pointed) reduced theory. We record this below. Let $\Kcrm_1(D) = \Kcm(D)/\Kcrm_p(D)$. As an $\f[u]$-module, this is spanned by the generators that label $c_p$ by $1$. We write 
\[
0 \rightarrow \Kcrm_p(D) \xrightarrow{i_p} \Kcm(D) \xrightarrow{j_p} \Kcrm_1(D) \rightarrow 0.
\]
It follows from the work of Wigderson \cite{wigderson} that we can find a section $\sigma_p$ of the quotient map $j_p$. (See the discussion after \cite[Proposition 2.23]{Sano}.) Moreover, it is shown in \cite{wigderson} that $\Kcrm_1(D) \cong \Kcrm_p(D)$. Hence we obtain:


\begin{thm}[{\cite{wigderson, schutz-integral}}]\label{thm:wigderson}
    \label{thm:bar-natan-theory-splits-nonequivariantly}
    We have
    \[
    \Kcm(D) = \Kcrm_p(D) \oplus \sigma_p(\Kcrm_1(D)) \cong \Kcrm_p(D) \oplus \Kcrm_p(D).
    \]
\end{thm}
\noindent
We let
\[
i_p \colon \Kcrm_p(D)\to \Kcm(D) \quad \text{and} \quad \pi_p \colon \Kcm(D) \rightarrow \Kcrm_p(D)
\]
be inclusion and projection onto the first factor with respect to this splitting, so that $\pi_p \circ i_p = \id$. These will be useful for defining cobordism maps on the reduced complex.

\subsection{Behavior under localization}
It is well-known that $\Kc^\infty(L)$ is standard. We review the precise formulation of this here; our discussion here is taken from \cite[Section 3]{Sano}.

\begin{defn}
Let $D$ be a link diagram of $n$ components. For each orientation $o$ on $D$, define the \textit{canonical generator} $\frs_o$ as follows. Let $D_v$ be the oriented resolution of $D$ corresponding to $o$.
Color the planar regions of $\mathbb{R}^2 - D_v$ in the checkerboard manner, with the unbounded region colored white. For each circle $c \in Z(D_v)$, label $c$ with $x$ if there is a black region to the left of $c$ (with respect to its induced orientation), and with $x + u$ otherwise. It can be checked that $\frs_o$ is a cycle; see \cite[Section 3]{Sano}.
\end{defn}

\begin{prop}[{\cite[Theorem 2.1]{lipshitz-sarkar-mixed}}]\label{prop:unreduced-infty-splitting}
	Let $D$ be a link diagram for an $n$-component link $L$. Denote by $o(L)$ the set of orientations on $L$. Then the $2^{n}$ classes $[\frs_o]$ form a basis for $H_*(\Kc^{\infty}(D))$:
	\[
	H_*(\Kc^{\infty}(D)) = \spa_{\f[u, u^{-1}]}\{ [\frs_o] \colon o \in o(L) \} \cong \f[u, u^{-1}]^{2^{n}}.
	\]
\end{prop}


Moreover, exactly half of the generators of Theorem \ref{prop:unreduced-infty-splitting} lie in the reduced subcomplex. Indeed, the canonical generators $\frs_o$ and $\frs_{\overline{o}}$ are related by changing each label $x$ to $x + u$, and vice versa. The following result states that the reduced complex is spanned by the $2^{n-1}$ canonical generators which label the appropriate circle $c_p$ with $x$. Moreover, although the other generators do not lie in the reduced complex, they project isomorphically onto the reduced complex via $\pi_p$: 

\begin{prop}[{\cite[Proposition 3.12]{Sano}}]\label{prop:reduced-bar-natan-calculation}
Let $(D,p)$ be a pointed link diagram for an $n$-component link $L$ and let $o'(L) = \{o \in o(L) \colon \frs_o \text{ labels } c_p \text{ with } x\}$. Then
\[
H_*(\Kcr^{\infty}_{p}(D)) = \spa_{\f[u, u^{-1}]}\{ [\frs_o] \colon o \in o'(L) \} \cong \f[u, u^{-1}]^{2^{n-1}}.
\]
Moreover, $o \in o'(L, p)$ if and only if $\overline{o} \notin o'(L, p)$. However, $\pi_p([\frs_o]) = \pi_p([\frs_{\overline{o}}])$ for all $o \in o(L)$.
\end{prop}
\begin{rmk}
  The description of $\Kcr^\infty$ in Theorem~\ref{prop:reduced-bar-natan-calculation} is due to \cite{lee,Rasmussen}.
  However, for us the last part of Theorem~\ref{prop:reduced-bar-natan-calculation} regarding the behavior of $[\frs_o]$ and $[\frs_{\overline{o}}]$ will be essential. To the best of the authors' knowledge, this first appears explicitly in \cite{Sano}.
\end{rmk}

\begin{example}
Let $D$ be the trivial unknot diagram equipped with a basepoint $p$. In this case, it can be checked that the Wigderson splitting is given by the obvious splitting of $\Kcm(D)$:
\[
\Kcm(D) = (\f[u] \cdot x) \oplus (\f[u] \cdot 1). 
\]
The two canonical generators of $\Kc^{\infty}(D)$ are given by $[x]$ and $[x + u]$, which both project to $[x]$.
\end{example}

\subsection{Invariance and cobordism maps for unreduced theory}\label{sub:cobordism_maps}
We now briefly review the proof of invariance and the existence of cobordism maps for Bar-Natan homology. Although this material is standard, setting up the appropriate framework will be important when we discuss these notions in the presence of an involution. 

\subsubsection*{Invariance for the unreduced theory} 
Let $L$ be a link in $S^3$. Perturbing $L$ if necessary, we obtain a link diagram $D$ for $L$. We claim that the homotopy type of $\Kcm(D)$ is a well-defined invariant of $L$. Indeed, if $D$ and $D'$ are link diagrams obtained by two perturbations of $L$, then there exists a sequence of Reidemeister moves going from $D$ to $D'$. A Reidemeister move from $D$ to $D'$ induces a homotopy equivalence $\Kcm(D) \simeq \Kcm(D')$, with homotopy inverse given by the reverse Reidemeister move. (Strictly speaking, in addition to Reidemeister moves, there are also planar isotopies, which induce obvious isomorphisms.) Composing these if necessary, we see that the homotopy type of $\Kcm(D)$ is a well-defined invariant associated to $L$, which we denote by $\Kcm(L)$. More generally, this shows that $\Kcm(L)$ is an invariant of the isotopy class of $L$. 

\subsubsection*{Cobordism maps for the unreduced theory}
Let $\Sigma$ be a cobordism from $L_1$ to $L_2$ in $S^3 \times I$. Assume $\Sigma$ is generic with respect to the standard projection, so that $\Sigma$ is diagrammatically depicted by a sequence $\{S_i\}_{i = 1}^k$ of Reidemeister moves and elementary cobordism moves. To define the cobordism map $\Kcm(\Sigma)$, we define elementary cobordism maps associated to each Reidemeister move and elementary cobordism move. The maps $\Kcm(S_i)$ for the Reidemeister moves are just the homotopy equivalences discussed above. Elementary cobordism moves come in the following forms: births, deaths, and saddles. For each of these, there is a well-known elementary cobordism map $\Kcm(S_i)$; see for example \cite{blanchet}. We then define:
\begin{defn}\label{def:unreduced-cobordism}
Let $\Sigma$ be a cobordism in $S^3 \times I$ from $L_1$ to $L_2$. Assume $\Sigma$ is generic, so that $\Sigma$ gives rise to sequence $\{S_i\}_{i = 1}^k$ of Reidemeister moves and elementary cobordism moves from a diagram $D_1$ of $L_1$ to a diagram $D_2$ of $L_2$. Define the associated cobordism map to be the composition
\[
\Kcm(\Sigma) = \Kcm(S_k) \circ \cdots \circ \Kcm(S_1) \colon \Kcm(D_1) \rightarrow \Kcm(D_2)
\]
where the $\Kcm(S_i)$ are the maps associated to each elementary move. If $\Sigma$ is an oriented cobordism between oriented links, it can be checked that
\[
\gr(\Kcm(\Sigma)) = (0, -2g(\Sigma)).
\]
See for example \cite{blanchet}.
\end{defn}

Note that even if $L_1$ and $L_2$ are generic with respect to the standard projection, $\Sigma$ may not itself be generic. In this case, we apply a small perturbation of $\Sigma$ rel boundary in order to define $\Kcm(\Sigma)$. We may then ask whether $\Kcm(\Sigma)$ is independent of the choice of perturbation; or, more generally, whether the chain homotopy type of $\Kcm(\Sigma)$ is independent of isotopy rel boundary. The following theorem is due to Morrison-Walker-Wedrich \cite{MWW} building on work of Jacobsson \cite{jacobsson}; see \cite[Proposition 3.7]{lipshitz-sarkar-mixed} for the claim in the Bar-Natan setting:

\begin{thm}\label{thm:cobordism-naturality}
Let $\Sigma$ and $\Sigma'$ be two cobordisms from $L_1$ to $L_2$ in $S^3 \times I$. Assume that $\Sigma$ and $\Sigma'$ are both generic with respect to the standard projection. If $\Sigma$ and $\Sigma'$ are isotopic rel boundary in $S^3 \times I$, then $\Kcm(\Sigma)$ and $\Kcm(\Sigma')$ are chain homotopic.
\end{thm}

Thus, if we fix the ends of $\Sigma$, then $\Kcm(\Sigma)$ is independent of the choice of perturbation. More generally, the chain homotopy type of $\Kcm(\Sigma)$ is unchanged by isotopy rel boundary.

\begin{rmk}
Note that we may view the homotopy equivalence associated to a sequence of Reidemeister moves as a special case of a cobordism map. Indeed, any sequence of Reidemeister moves corresponds to an isotopy in $\mathbb{R}^3$; the cobordism in question is just the trace of this isotopy.
\end{rmk}


Finally, we discuss the behavior of cobordism maps with respect to the canonical generators. We say that a cobordism $\Sigma$ from $L_1$ to $L_2$ is \emph{weakly connected} if every component of $\Sigma$ has a boundary component in $L_1$. Note that if $L_1$ is given an orientation $o_1$ and $\Sigma$ weakly connected, then there is a unique orientation on $\Sigma$ which extends $o_1$; this induces a corresponding orientation $o_2$ on $L_2$. The following fact is well-known; see for example \cite[Theorem 4.1]{Rasmussen}, \cite[Theorem 2.4]{Lipshitz-Sarkar-refinement}, or \cite[Proposition 4.29]{Sano}:

\begin{prop}\label{thm:unreduced-local}
\label{prop:cob}
Let $L_1$ and $L_2$ be equipped with orientations $o_1$ and $o_2$ and let $\Sigma$ be a weakly connected, oriented cobordism from $(L_1, o_1)$ to $(L_2, o_2)$. Then $\Kc^\infty(\Sigma)$ maps $[\mathfrak{s}_{o_1}]$ to a nonzero $u$-multiple of $[\mathfrak{s}_{o_2}]$:
\[
\Kc^\infty(\Sigma)([\mathfrak{s}_{o_1}]) = u^{n} [\mathfrak{s}_{o_2}]
\]
for some $n \geq 0$.
\end{prop}

\subsection{Invariance and cobordism maps for reduced theory}\label{sub:cobordism_maps_reduced}
We now discuss what modifications need to be made when discussing the reduced theory.

\subsubsection*{Invariance for the reduced theory}
Let $L$ be a link in $S^3$. Perturbing $L$ if necessary, we obtain a link diagram $D$ for $L$. This allows us to form the isomorphism class $\Kcrm(D)$, either by choosing any basepoint $p$ on $D$ or by invoking the unpointed construction. To establish invariance, first suppose that $D$ and $D'$ are two diagrams for $L$ which differ by a planar isotopy. Fix a basepoint $p$ on $D$ and let $p'$ be the image of $p$ under the planar isotopy from $D$ to $D'$. Then it is clear that the isomorphism $\Kcm(D) \cong \Kcm(D')$ restricts to an isomorphism of subcomplexes $\smash{\Kcrm_p(D) \cong \Kcrm_{p'}(D')}$. Hence (keeping Lemma~\ref{lem:reduced-theories-equivalent} in mind) the isomorphism type of $\Kcrm(D)$ is preserved. 

Now suppose $D$ and $D'$ differ by a single Reidemeister move. Choose a basepoint $p$ which is unaffected by this Reidemeister move, in the sense that $p$ lies outside of the ball in which the Reidemeister move occurs. An examination of the maps associated to Reidemeister moves \cite{vogel} shows that the homotopy equivalences $f$ and $g$ between $\Kcm(D)$ and $\Kcm(D')$ preserve the reduced subcomplexes corresponding to $p$. This implies that we have a homotopy equivalence
\[
\Kcrm_p(D) \simeq \Kcrm_{p}(D').
\]
Indeed, suppose $gf$ is a map from $\Kcm(D)$ itself which is homotopic to $\id$ via a homotopy $H$. If $gf$ preserves $\smash{\Kcrm_p(D)}$, then $gf$ -- as a map from $\smash{\Kcrm_p(D)}$ to itself -- is homotopic to $\id$ via the homotopy $\pi_p H$, where $\pi_p$ is the Wigderson projection. The homotopy type of $\Kcrm(D)$ (which is itself an isomorphism class) is thus an invariant of $L$, which we denote by $\Kcrm(L)$. The same argument proves $\Kcrm(L)$ is an invariant of the isotopy class of $L$.

\begin{rmk}\label{rem:basepointjumping1}
Note that if we are given a sequence of Reidemeister moves from $D$ to $D'$, we do \textit{not} assume that the same basepoint $p$ is chosen throughout. Indeed, we might have to abruptly change our choice of basepoint between one Reidemeister move and the next. Thus we are not claiming to provide a natural homotopy equivalence between $\Kcrm_p(D)$ and $\Kcrm_{p'}(D')$; simply that (in light of Lemma~\ref{lem:reduced-theories-equivalent}) these homotopy types must abstractly be the same.
\end{rmk}

\subsubsection*{Cobordism maps for the reduced theory}
We now discuss how to define cobordism maps on the reduced complex. To truly make a functorial theory, we would have to equip each cobordism with a horizontal arc running between basepoints. This is done for example in \cite{baldwin-levine-sarkar}. However, requiring such an arc turns out to be inconvenient when we move to the equivariant setting. Instead of honestly tacking this problem, we introduce a less-natural version of the cobordism map defined directly in terms the projection and inclusion maps $i_p$ and $\pi_p$ of the Wigderson splitting. 

\begin{defn}\label{def:reduced-cobordism-map}
Let $\Sigma$ be a cobordism in $S^3 \times I$ from $L_1$ to $L_2$ which is generic with respect to the standard projection. For any pair of basepoints $p_1$ on $D_1$ and $p_2$ on $D_2$, define
\[
\Kcrm(\Sigma) \colon \Kcrm_{p_1}(D_1) \rightarrow \Kcrm_{p_2}(D_2)
\]
to be the composition
\[
\Kcrm_{p_1}(D_1) \xrightarrow{i_{p_1}} \Kcm(D_1) \xrightarrow{\Kcm(\Sigma)} \Kcm(D_2) \xrightarrow{\pi_{p_2}} \Kcrm_{p_2}(D_2).
\]
Here, the middle term $\Kcm(\Sigma)$ is the cobordism map for the \textit{unreduced} complex. We stress that the above cobordism map does not even require the \textit{existence} of a path from $p_1$ to $p_2$: the basepoints $p_1$ and $p_2$ are used only to define the Wigderson splittings of the ends. 
\end{defn}

\begin{rmk}
We stress that $\Kcrm(\Sigma)$ depends on a choice of basepoint and perturbation data. Indeed, without a choice of basepoints, the domain and codomain of $\Kcrm(\Sigma)$ are defined only up to isomorphism. Definition~\ref{def:reduced-cobordism-map} is thus assuredly unnatural: $\Kcrm(\Sigma)$ will depend drastically on the choice of basepoints $p_1$ and $p_2$, as these change the Wigderson splitting. If, for example, a different basepoint $p_2$ is chosen on the outgoing end, then $\Kcm(\Sigma)$ will be projected onto a completely different subcomplex of $\Kcm(D_2)$. Even though this new subcomplex is isomorphic to the original, there is no reason to think that -- even under this identification -- the chain homotopy type of $\Kcrm(\Sigma)$ will be preserved. Thus, at minimum, we should really denote $\Kcrm(\Sigma)$ by $\Kcrm(\Sigma, p_1, p_2)$. However, all claims that we make regarding the reduced cobordism maps will hold independent of the choice of basepoints, so we suppress these from the notation and write
\[
\Kcrm(\Sigma) \colon \Kcrm(L_1) \rightarrow \Kcrm(L_2).
\]
Note, however, that if we fix the ends of $\Sigma$ and the choice of basepoints, then $\Kcrm(\Sigma)$ will tautologically satisfy the analogue of Theorem~\ref{thm:cobordism-naturality}. Indeed, if $\Kcm(\Sigma)$ and $\Kcm(\Sigma')$ are chain homotopic, then we immediately have that $\Kcrm(\Sigma)$ and $\Kcrm(\Sigma')$ are chain homotopic, as the latter are obtained from the former by precomposition and postcomposition.
\end{rmk}

\subsection{Knotlike complexes}
We now specialize to the case of knots. If $K$ is a knot, then $\Kcrm(K)$ has especially simple behavior under inverting $u$. We record this in the following:

\begin{defn}\label{def:knotlike}
A \textit{knotlike complex} $C$ is a free, finitely-generated chain complex over $\F[u]$ with a $\Z \oplus \Z$-grading for which $\deg(\partial) = (1, 0)$ and $\deg(u) = (0, -2)$. As usual, we write this as $(\gr_h, \gr_q)$ and refer to the components as the \textit{homological grading} and \textit{quantum grading}, respectively. We require that $C$ satisfy the \textit{localization condition}
\[
u^{-1}H_*(C) \cong \F[u, u^{-1}],
\]
where $\F[u, u^{-1}]$ is supported in homological grading zero and even quantum grading. We say that a relatively graded chain map $f$ between two knotlike complexes $C_1$ and $C_2$ is \textit{local} if the induced map
\[
f \colon u^{-1}H_*(C_1) \cong \F[u, u^{-1}] \rightarrow u^{-1}H_*(C_2) \cong \F[u, u^{-1}]
\]
is an isomorphism. Note that $f$ is not required to be $(\gr_h, \gr_q)$-preserving, but instead preserves $\gr_h$ and shifts $\gr_q$ by an even constant. We refer to $\gr_q(f)$ simply as the associated grading shift.
\end{defn}

The following observation is crucial. We write $C[(a, b)]$ to indicate a grading shift by $(a, b)$, so that an element $x \in C$ of bigrading $\gr(x)$ now has bigrading $\gr(x) - (a, b)$.

\begin{thm}\label{thm:knotlike-locality}
If $K$ is a knot, then $\smash{\Kcrm_p(K)[(0, -1)]}$ is a knotlike complex. Moreover, if $\Sigma$ is a connected cobordism from $K_1$ to $K_2$, then $\Kcrm(\Sigma)$ is a local map of grading shift $- 2g(\Sigma)$.
\end{thm}
\begin{proof}
The claim that $\smash{\Kcrm_p(K)[(0, -1)]}$ is a knotlike complex is clear from Theorem~\ref{prop:reduced-bar-natan-calculation}; we shift grading by one so that $\smash{H_*(\Kcr_p^{\infty}(K))}$ is supported in even gradings. The fact that $\Kcrm(\Sigma)$ is local follows from Theorem~\ref{thm:unreduced-local} combined with the behavior of $\pi_p$ described in Theorem~\ref{prop:reduced-bar-natan-calculation}. Indeed, choose any pair of basepoints $p_1$ on $K_1$ and $p_2$ on $K_2$. Let $o_1 \in o(K_1)$ be the orientation on $K_1$ for which $[\frs_{o_1}]$ generates $H_*(\Kcr_{p_1}^{\infty}(K_1))$. Consider the orientation on $\Sigma$ that extends $o_1$ and let the induced orientation on $K_2$ be $o_2$. By Theorem~\ref{thm:unreduced-local}, $\Kc^\infty(\Sigma)$ takes $[\frs_{o_1}]$ to a nonzero $u$-multiple of $[\frs_{o_2}]$. Hence we have
\[
\Kcr^\infty(\Sigma)([\frs_{o_1}]) = (\pi_{p_2} \circ \Kc^\infty(\Sigma) \circ i_{p_1})([\frs_{o_1}]) = \pi_{p_2}(u^n[\frs_{o_2}]).
\]
Now, by Theorem~\ref{prop:reduced-bar-natan-calculation}, we have that $\pi_{p_2}([\frs_{o_2}]) = \pi_{p_2}([\frs_{\overline{o}_2}])$. However, (exactly) one of $[\frs_{o_2}]$ and $[\frs_{\overline{o}_2}]$ already lies in the reduced subcomplex $\Kcr_{p_2}^{\infty}(K_2)$. It follows that $\pi_{p_2}([\frs_{o_2}]) = \pi_{p_2}([\frs_{\overline{o}_2}])$ indeed generates $\Kcr_{p_2}^{\infty}(K_2)$, which is one-dimensional over $\f[u, u^{-1}]$. This establishes locality. Note that our conclusion holds independent of the choice of $p_1$ and $p_2$.
\end{proof}

\begin{rmk}
  We suppress the grading shift and write $\smash{\Kcrm(K)}$ in place of $\smash{\Kcrm(K)[(0, -1)]}$. The grading shift is chosen so that $\smash{\Kcrm_p(U)}$ consists of a single $\f[u]$-tower starting in bigrading $(0, 0)$.  
\end{rmk}

\begin{example}
Consider the cobordism $\Sigma$ from the unknot $U$ to itself displayed in Figure~\ref{fig:examplecobordism}. This consists of the trace of an isotopy from $U$ to $-U$ which effects a $\pi$-rotation about the vertical axis. (In Figure~\ref{fig:examplecobordism}, we decompose this into two successive Reidemeister R1-moves, applied to the top half and bottom half of the knot, respectively.) One can check, either using Theorem~\ref{thm:unreduced-local} or by direct calculation using the Reidemeister move maps, that $\Kc^\infty(\Sigma)$ sends $x$ in the domain to $x + u$ in the codomain. 

Regardless of the choice of basepoints, the Wigderson splitting is given by $(\f[u] \cdot x) \oplus (\f[u] \cdot 1)$ in both the domain and the codomain. Hence $\Kc^\infty(\Sigma)$ does \textit{not} map the reduced complex into the reduced complex. However, that the composition $\smash{\pi_{p_2} \circ \Kc^\infty(\Sigma) \circ i_{p_1}}$ is still an isomorphism.

\begin{figure}[h!]
\includegraphics[scale = 0.8]{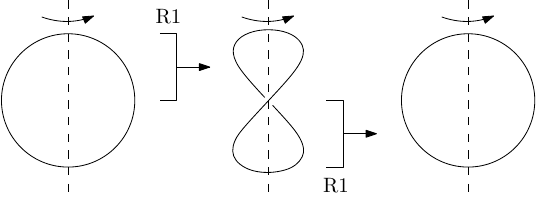}
	\caption{An example of a reduced cobordism map.}
	\label{fig:examplecobordism}
\end{figure}
\end{example}


\section{Involutions on Bar-Natan Homology}\label{sec:involutions-on-bar-natan}
Following \cite{Sano}, we now discuss the simplest way to study the action of $\tau$, which will set the stage for the introduction of the Borel complex in the sequel. We thus construct the action of $\tau$ on $\Kcm(L)$ and store this information via the homotopy type of the pair $(\Kcm(L), \tau)$. If an equivariant basepoint $p$ on $L$ is chosen, we moreover show that the reduced complex $\smash{\Kcrm_p(K)}$ is preserved by $\tau$. We prove that the equivariant isomorphism class of the reduced complex is independent of the choice of (equivariant) basepoint, so that we may unambguously refer to $(\Kcrm(K), \tau)$. We formalize our construction in the following:

\begin{thm}\label{thm:taucomplex}
The homotopy type of the pair $(\Kcm(L), \tau)$ is an invariant of $(L, \tau)$ up to Sakuma equivalence. Moreover, if $\Sigma$ is an isotopy-equivariant cobordism from $(L_1, \tau_1)$ to $(L_2, \tau_2)$, then the cobordism map $\Kcm(\Sigma)$ homotopy commutes with $\tau$:
\[
\Kcm(\Sigma) \circ \tau_1 \simeq \tau_2 \circ \Kcm(\Sigma).
\]
The homotopy type of the pair $(\Kcrm(L), \tau)$ is likewise an invariant of $(L, \tau)$, and we similarly have
\[
\Kcrm(\Sigma) \circ \tau_1 \simeq \tau_2 \circ \Kcrm(\Sigma).
\]
In both cases, the cobordism map has grading shift $(0, -2g(\Sigma))$. Moreover, if $L_1$ and $L_2$ are strongly invertible knots and $\Sigma$ is connected, then $\Kcrm(\Sigma)$ is local.\footnote{The claims regarding $\Kcrm(\Sigma)$ should be interpreted as holding for any equivariant basepoint data $p_1$ and $p_2$. We again stress that we do not require that $\Sigma$ be decorated with a path from $p_1$ to $p_2$, either to construct $\Kcrm(\Sigma)$ or to prove that it homotopy commutes with $\tau$.}
\end{thm}

Using locality, we define an (isotopy-)equivariant concordance invariant $\wt{s}(K, \tau) = \wt{s}(K)$. This bounds the isotopy-equivariant genus:

\begin{thm}\label{thm:s-bound}
Let $\Sigma$ be an isotopy-equivariant knot cobordism from $(K_1, \tau_1)$ to $(K_2, \tau_2)$. Then
\[
\wt{s}(K_1) - 2g(\Sigma) \leq \wt{s}(K_2).
\]
In particular, if $(K, \tau)$ is a strongly invertible knot, then $\ieg(K) \geq |\wt{s}(K)| / 2$.
\end{thm}

Again, Theorems~\ref{thm:taucomplex} and \ref{thm:s-bound} are a re-phrasing of the work of Sano \cite{Sano}. However, we give an alternative discussion of them here to motivate our results on Borel complexes in the next section. Theorems~\ref{thm:taucomplex} and \ref{thm:s-bound} should be viewed as introductory precursors to Theorems~\ref{thm:full_borel_invariant} and \ref{thm:equivariant_s_borel_intro} from the Borel setting.

\subsection{The action of $\tau$} 
Let $D$ be a transvergent diagram for $L$. Clearly, $\tau$ induces an involution on the set of resolutions of $D$, where a resolution $D_v$ is taken to its image $\tau(D_v)$. This gives a bijection from $Z(D_v)$ to $Z(\tau(D_v))$, which we again denote by $\tau$. If $D_v = \tau(D_v)$, this is an automorphism of $Z(D_v)$. See Figure~\ref{fig:tau-example}.

\begin{figure}[h!]
\includegraphics[scale = 0.65]{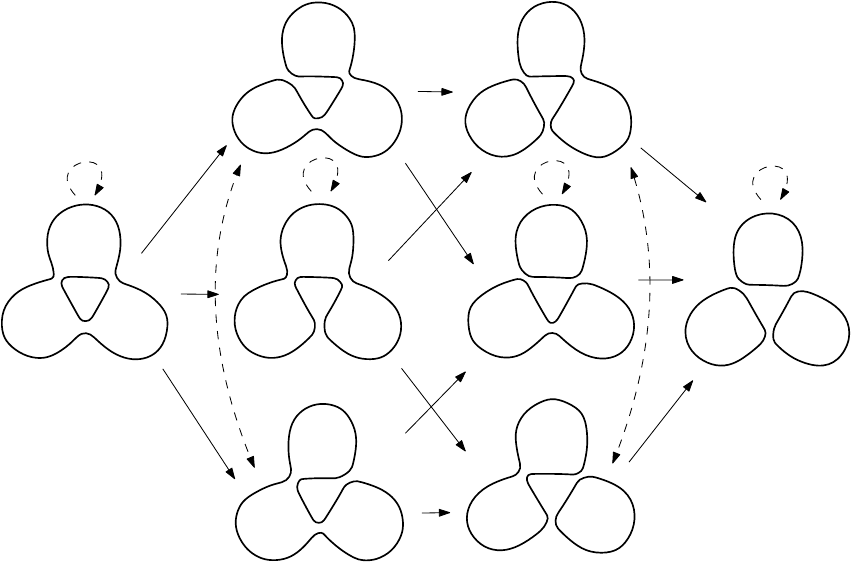}
	\caption{The cube of resolutions for the transvergent diagram of the trefoil from Figure~\ref{fig:transvergent-diagram}. The differentials going between different resolutions are indicated by solid arrows. The dashed arrows indicate the action of $\tau$ on the set of resolutions.}
	\label{fig:tau-example}
\end{figure}

\begin{defn}\label{def:diagrammatic_tau}
Let $D$ be a transvergent diagram for $L$ and assume that $D$ is oriented. Construct a map
\[
\tau \colon \Kcm(D) \rightarrow \Kcm(D)
\]
as follows. Let $y$ be a basis element from $D_v$ consisting of a label $a_c \in \{1, x\}$ on each circle $c \in Z(D_v)$. The image $\tau(y)$ is the basis element from $\tau(D_v)$ constructed by labeling each circle $\tau(c) \in Z(\tau(D_v))$ by $a_c$. We extend $\tau$ such that it is $\F[u]$-linear; clearly $\tau$ is a chain map and that $\tau^2 = \id$. We may likewise define a map
\[
\tau \colon \Kc'(D) \rightarrow \Kc'(D)
\]
for the version of the Bar-Natan complex used to construct the unpointed reduced complex. This sends $\Kc'(D_v)$ to $\Kc'(\tau(D_{v}))$ via the algebra map induced by the index bijection $\tau \colon Z(D_v) \rightarrow Z(\tau(D_v))$. It is clear that the isomorphism $\Kcm(D) \cong \Kc'(D)$ of (\ref{eqn:reducediden}) is $\tau$-equivariant.
\end{defn}



We stress that in Definition~\ref{def:diagrammatic_tau}, both the domain and codomain of $\tau$ are given the same orientation. The action of $\tau$ in Definition~\ref{def:diagrammatic_tau} is then grading-preserving. To see this, observe that $D_v$ and $\tau(D_v)$ have the same number of 1-resolutions and that the quantity $\ell$ of Section~\ref{sec:gradings} is preserved under $\tau$. We also remind the reader that we are working over $\F$, so that there is no sign assignment needed for the differential on $\Kcm(D)$; if we were attempting to work over $\Z$, we would need to use equivariant sign assignments as explained in \cite{bpyozgyur}. (This applies both for the definition of $\tau$, as well as the definition of the cobordism maps.)

The following is what we mean when we speak of the homotopy type of a pair:

\begin{defn}\label{def:homotopy-equivalence-pair}
Let $C_1$ and $C_2$ be two chain complexes equipped with homotopy involutions $\tau_1$ and $\tau_2$. We say that $(C_1, \tau_1)$ and $(C_2, \tau_2)$ are homotopy equivalent if there exist homotopy inverses $f \colon C_1 \rightarrow C_2$ and $g \colon C_2 \rightarrow C_1$ which homotopy commute with $\tau$; that is,
\[
  g \circ f =\id + [\partial, F] \quad \text{and} \quad f \circ g = \id + [\partial, G]
\]
and
\[
  f \circ \tau_1 = \tau_2 \circ f + [\partial, f_1] \quad \text{and} \quad g \circ \tau_2 =\tau_1 \circ g + [\partial, g_1],
\]
where $F\colon C_1\to C_1$, $G\colon C_2\to C_2$, $f_1\colon C_1\to C_2$ and $g_1\colon C_2\to C_1$ are sometimes referred to as \emph{witnessing}
homotopies
\end{defn}

\begin{rmk}
There are several items to emphasize regarding Definition~\ref{def:homotopy-equivalence-pair}:
\begin{enumerate}
\item Even though we constructed $\tau$ to be an actual involution in Definition~\ref{def:diagrammatic_tau}, in the context of Definition~\ref{def:homotopy-equivalence-pair}, we only require $\tau$ to be a homotopy involution. Indeed, relaxing the condition on $\tau$ allows us to find, in many cases, a smaller and simpler model of a Khovanov chain complex, which is more convenient for computations.

\item We do not require that $f$ and $g$ commute with $\tau$ on the nose, but only up to homotopy. We stress moreover that \emph{a priori} the homotopy $f_1,g_1$ need
  not be $\tau$-equivariant. Later, we will observe that (in the cases at hand) $f_1,g_1$ satisfy additional conditions with respect to $\tau$. This will lead to the invariance of the Borel construction.
\item Likewise, we require that $f$ and $g$ be homotopy inverses, but we do not place any restriction on the homotopies $F,G$ which witness the fact that $g \circ f \simeq \id$ and $f \circ g \simeq \id$. (For example, that they interact in a particular way with $\tau$.)
\end{enumerate}
\end{rmk}

\subsubsection*{Invariance and cobordism maps for the unreduced theory.} 
We now discuss invariance and cobordism maps in the context of an involution. Let $L$ be an involutive link in the standard model for $S^3$ and $D$ be a transvergent diagram for $L$. We must show that the homotopy type of the pair $(\Kcm(D), \tau)$ is invariant under the following:

\begin{enumerate}
\item Equivariant planar isotopies and equivariant Reidemeister moves; 
\item The $I$-move of passing a strand over the point at $\infty$ in $\mathbb{R}^2$; and,
\item The $R$-move obtained by rotating the diagram by $\pi$ about the origin.
\end{enumerate}
The proof of the first of these will utilize the following more general claim:

\begin{lem}\label{lem:tau-complex-isotopy-equivariant}
Let $L_1$ and $L_2$ be a pair of involutive links with transvergent diagrams $D_1$ and $D_2$. Let $\Sigma$ be a cobordism in $S^3 \times I$ from $L_1$ to $L_2$ and assume that $\Sigma$ is (non-equivariantly) generic, so that we have
\[
\Kcm(\Sigma) \colon \Kcm(D_1) \rightarrow \Kcm(D_2).
\]
If $\Sigma$ is isotopy-equivariant, then 
\[
\Kcm(\Sigma) \circ \tau_1 \simeq \tau_2 \circ \Kcm(\Sigma).
\]
\end{lem}
\begin{proof}
We first claim that we have the tautological commutation
\[
\Kcm(\tau(\Sigma)) \circ \tau_1 = \tau_2 \circ \Kcm(\Sigma).
\]
Here, to interpret the left-hand side, note that the rotated cobordism $\tau(\Sigma)$ will also be generic. (For conciseness, we write $\tau$ in place of $\tau_{S^3 \times I}$.) Hence $\tau(\Sigma)$ also yields a set of cobordism moves, which are the cobordism moves for $\Sigma$ but rotated by $\tau$. Composing these as usual gives the cobordism map $\Kcm(\tau(\Sigma))$.

To establish the claim, let $S_i$ be a cobordism move from $\Sigma$ that goes between two (possibly) non-equivariant diagrams $D$ and $D'$. The rotated move $\tau(S_i)$ then goes from $\tau(D)$ to $\tau(D')$. In this situation, $\tau$ no longer induces an automorphism of $\Kcm$ but instead yields tautological isomorphisms $\tau \colon \Kcm(D) \rightarrow \Kcm(\tau(D))$ and $\tau \colon \Kcm(D') \rightarrow \Kcm(\tau(D'))$. It is obvious that 
\[
\Kcm(\tau(S_i)) \circ \tau = \tau \circ \Kcm(S_i).
\]
Since $\Kcm(\Sigma) = \Kcm(S_k) \circ \cdots \circ \Kcm(S_1)$ and $\Kcm(\tau(\Sigma)) = \Kcm(\tau(S_k)) \circ \cdots \circ \Kcm(\tau(S_1))$, the tautological commutation follows. We now simply apply Theorem~\ref{thm:cobordism-naturality}: by isotopy invariance rel boundary, we have that the maps $\Kcm(\tau(\Sigma))$ and $\Kcm(\Sigma)$ are chain homotopic. Hence
\[
\Kcm(\Sigma) \circ \tau_1 \simeq \Kcm(\tau(\Sigma)) \circ \tau_1 = \tau_2 \circ \Kcm(\Sigma).
\]
This completes the proof.
\end{proof}

We now prove the desired claim:

\begin{lem}\label{lem:tau-complex-invariant}
Let $L$ be an involutive link with transvergent diagram $D$. The homotopy type of $(\Kcm(D), \tau)$ is a well-defined invariant of $L$ up to Sakuma equivalence.
\end{lem}

\begin{proof}
First suppose that $D$ and $D'$ differ by a sequence of equivariant planar isotopies and equivariant Reidemeister moves. This gives an equivariant cobordism $\Sigma$ in $\mathbb{R}^3 \times I$. In general, $\Sigma$ will be equivariantly generic but not generic in the sense of Definition~\ref{def:std-projection-cobordism}. However, we may perturb $\Sigma$ to a non-equivariant cobordism $\Sigma'$ which is generic with respect to the standard projection. Then $\Kcm(\Sigma')$ constitutes a homotopy equivalence $f$ from $\Kcm(D)$ to $\Kcm(D')$. On the other hand, $\tau(\Sigma')$ is also a generic perturbation of $\Sigma$, and clearly $\Sigma'$ and $\tau(\Sigma')$ are isotopic rel boundary. By Lemma~\ref{lem:tau-complex-isotopy-equivariant} we thus have that $f \circ \tau \simeq \tau \circ f$. This shows that the homotopy type of $(\Kcm(D), \tau)$ is invariant under equivariant isotopy in $\mathbb{R}^3$. 

We now deal with the I-move and the R-move. This is actually quite straightforward: if $D$ and $D'$ differ by the I-move, then there is an obvious isomorphism from $\Kcm(D)$ to $\Kcm(D')$, as displayed in Figure~\ref{fig:ircorrespondence}. This isomorphism is clearly $\tau$-equivariant. Likewise, if $D$ and $D'$ differ by the R-move, then we obtain a $\tau$-equivariant isomorphism from $\Kcm(D)$ to $\Kcm(D')$. This completes the proof.
\end{proof}

\begin{figure}[h!]
\includegraphics[scale = 0.7]{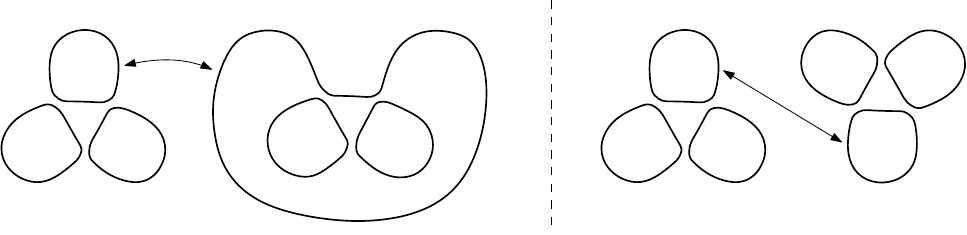}
	\caption{Left: if $D$ and $D'$ are two diagrams that differ by an I-move, then they have the same set of crossings. For each $v \in \{0, 1\}^n$, there is an obvious correspondence between $Z(D_v)$ and $Z(D'_v)$, where the circle containing the moving strand before the I-move is sent to the circle containing the moving strand after the I-move. Right: a similar claim holds for the R-move.}
	\label{fig:ircorrespondence}
\end{figure}

\subsubsection*{Invariance and cobordism maps for the reduced theory.} 
We now deal with the reduced case. Let $D$ be a transvergent diagram for $L$ and $p$ be any fixed point of $\tau$ on $D$. The circle $c_p$ containing $p$ in each resolution is then fixed by $\tau$. It follows that the action of $\tau$ preserves the subcomplex $\Kcrm_p(D)$. Likewise, it is clear that $\smash{\Kcrm_{un}(D) \subseteq \Kc'(D)}$ is preserved by $\tau$.

\begin{lem}\label{lem:reducedtauindependent}
Let $D$ be a transvergent diagram for $L$. For any choice of equivariant basepoint, $\smash{(\Kcrm_p(D), \tau)}$ is equivariantly isomorphic to $\smash{(\Kcrm_{un}(D), \tau)}$.
\end{lem}
\begin{proof}
If the basepoint $p$ is on the axis of symmetry, then the map $\mu_p$ in the proof of Lemma~\ref{lem:reduced-theories-equivalent} is $\tau$-equivariant since $\tau(x_p) = x_p$. It follows that the isomorphism of $\Kcrm_p(D) \cong \Kcrm_{un}(D)$ is $\tau$-equivariant.
\end{proof}

We write $(\Kcrm(D), \tau)$ for the isomorphism class of the reduced complex with its action of $\tau$. We now have the analogue of Lemma~\ref{lem:tau-complex-isotopy-equivariant} for the reduced complex. Recall that the reduced cobordism map is defined to be the composition
\[
\Kcrm(\Sigma) = \Kcrm_{p_1}(D_1) \xrightarrow{i_{p_1}} \Kcm(D_1) \xrightarrow{\Kcm(\Sigma)} \Kcm(D_2) \xrightarrow{\pi_{p_2}} \Kcrm_{p_2}(D_2).
\]
The crucial observation here is that the Wigderson splitting is $\tau$-equivariant, which follows from a direct examination of the section map $\sigma_p$; see \cite[Proposition 2.23]{Sano}. This implies that the inclusion and projection maps $i_p$ and $\pi_p$ are $\tau$-equivariant. Hence we obtain:

\begin{lem}\label{lem:reduced-isotopy-equivariant}
Let $L_1$ and $L_2$ be a pair of involutive links with transvergent diagrams $D_1$ and $D_2$. Fix equivariant basepoints $p_1$ and $p_2$ on $D_1$ and $D_2$. Let $\Sigma$ be a cobordism in $S^3 \times I$ from $L_1$ to $L_2$ and assume that $\Sigma$ is (non-equivariantly) generic, so that we have
\[
\Kcrm(\Sigma) \colon \Kcrm_{p_1}(D_1) \rightarrow \Kcrm_{p_2}(D_2).
\]
If $\Sigma$ is isotopy-equivariant, then 
\[
\Kcrm(\Sigma) \circ \tau_1 \simeq \tau_2 \circ \Kcrm(\Sigma).
\]
\end{lem}
\begin{proof}
Since $i_{p_1}$ and $\pi_{p_2}$ are $\tau$-equivariant, this follows immediately from Lemma~\ref{lem:tau-complex-isotopy-equivariant}.
\end{proof}

As before, we show that the homotopy type of $(\Kcrm(D), \tau)$ an invariant of $L$ up to Sakuma equivalence.

\begin{lem}\label{lem:reduced-invariant}
Let $L$ be an involutive link with transvergent diagram $D$. The homotopy type of $(\Kcrm(D), \tau)$ is a well-defined invariant of $L$ up to Sakuma equivalence.
\end{lem}

\begin{proof}
We combine the proof of Lemma~\ref{lem:tau-complex-invariant} with our discussion in Section~\ref{sub:cobordism_maps_reduced}. First suppose $D$ and $D'$ differ by an equivariant planar isotopy that sends an equivariant basepoint $p$ on $D$ to an equivariant basepoint $p'$ on $D'$. This clearly induces an equivariant isomorphism $\Kcrm_{p}(D) \cong \Kcrm_{p'}(D')$. Next, suppose that $D$ and $D'$ differ by a single equivariant Reidemeister move and let $p$ be an equivariant basepoint outside of the ball in which this move takes place. The Reidemeister move in question corresponds to an equivariant cobordism $\Sigma$ in $\mathbb{R}^3 \times I$. Perturb this cobordism rel boundary to a non-equivariant cobordism $\Sigma'$ which is generic with respect to the standard projection. This can be done away from the fixed basepoint $p$. As we showed in Section~\ref{sub:cobordism_maps_reduced}, $\Kcm(\Sigma')$ is a homotopy equivalence from $\Kcm(D)$ to $\Kcm(D')$ which restricts to a homotopy equivalence from $\smash{\Kcrm_p(D)}$ to $\smash{\Kcrm_{p}(D')}$. On the other hand, $\Kcm(\Sigma')$ also homotopy commutes with $\tau$ by Lemma~\ref{lem:tau-complex-isotopy-equivariant}. Thus the homotopy type of $(\Kcrm(D), \tau)$ is invariant under equivariant isotopy in $\mathbb{R}^3$. 

The I-move and R-move are similarly dealt with as before. For the I-move, we may (for example) choose our equivariant basepoint $p$ to be a basepoint which is not on the arc that moves over $\infty$; then the I-move isomorphism clearly sends $\smash{\Kcrm_p(D)}$ to $\smash{\Kcrm_{p}(D')}$. For the R-move, we let $p$ be any equivariant basepoint on $D$; rotation by $\pi$ about the origin sends this to an equivariant basepoint $p'$ of $D'$. Clearly, the R-move isomorphism likewise sends $\smash{\Kcrm_p(D)}$ to $\smash{\Kcrm_{p'}(D')}$.
\end{proof}

\begin{rmk}\label{rem:basepointjumping2}
We stress again that we do not attempt to prove any sort of naturality for the maps between pairs $(\Kcrm_p(D), \tau)$ for different choices of $p$ and $D$, just that these homotopy types are abstractly the same. As in Remark~\ref{rem:basepointjumping1}, if $D$ and $D'$ are related by a sequence of equivariant Reidemeister moves, then we may have to abruptly change our choice of basepoint in between Reidemeister moves so as to be outside the region where each Reidemeister move is occurring. Note that by Lemma~\ref{lem:reducedtauindependent}, the isomorphism type of $(\Kcrm_p(D), \tau)$ is independent of $p$.
\end{rmk}

\subsection{$\tau$-complexes and the $\wt{s}$-invariant}\label{sec:tau-complex}

We now specialize to the case of a strongly invertible knot. In this setting the algebraic formalism is exactly the same as that of an $\iota$-complex or $\tau$-complex from \cite{HM, HMZ, DHM}.

\begin{defn}\label{def:tau-complex}
A \textit{$\tau$-complex} $(C, \tau)$ is a knotlike complex $C$ equipped with a grading-preserving, $\f[u]$-equivariant homotopy involution $\tau \colon C \rightarrow C$. A \textit{homotopy equivalence} between $(C_1, \tau)$ and $(C_2, \tau_2)$ is defined as in Definition~\ref{def:homotopy-equivalence-pair}. A \textit{local map} from $(C_1, \tau)$ to $(C_2, \tau_2)$ is a local map $f \colon C_1 \rightarrow C_2$ in the sense of Definition~\ref{def:knotlike} which additionally homotopy commutes with $\tau$. Two $\tau$-complexes are \textit{locally equivalent} if there exist local maps between them of grading shift zero.
\end{defn}

It is immediate from Theorem~\ref{thm:knotlike-locality} that if $K$ is a strongly invertible knot, then $(\Kcrm(K), \tau)$ is a $\tau$-complex. Moreover, if $\Sigma$ is an isotopy-equivariant cobordism from $K_1$ to $K_2$, then $\Kcrm(\Sigma)$ is local by Theorem~\ref{thm:knotlike-locality} and Lemma~\ref{lem:reduced-isotopy-equivariant}. We thus finally obtain:

\begin{proof}[Proof of Theorem~\ref{thm:taucomplex}]
The unreduced claims follow from Lemmas~\ref{lem:tau-complex-invariant} and \ref{lem:tau-complex-isotopy-equivariant}. The reduced claims follow from \ref{lem:reduced-invariant} and \ref{lem:reduced-isotopy-equivariant}, together with Theorem~\ref{thm:knotlike-locality}.
\end{proof}

We emphasize here that Theorem~\ref{thm:taucomplex} is a consequence of the formal naturality and functoriality properties of Bar-Natan homology discussed in the previous section. This is analogous to the setup used in \cite{HM, HMZ, DHM, DMS}. As in \cite{HM}, we may likewise form auxiliary constructions such as the mapping cone of the map $1 + \tau$. This is done in \cite{Sano}, where Sano defines
\[
\mathit{CKhI}(L)  = \mathrm{Cone}(\Kcm(L) \xrightarrow{Q \cdot (1 + \tau)} Q \cdot \Kcm(L)).
\]
As in \cite{HM}, a chain map $f \colon \Kcm(L_1) \rightarrow \Kcm(L_2)$ which homotopy commutes with $\tau$ induces a chain map between mapping cones. Likewise, a quasi-isomorphism $f \colon \Kcm(L_1) \rightarrow \Kcm(L_2)$ which homotopy commutes with $\tau$ induces a quasi-isomorphism between mapping cones. Indeed, let $f\tau + \tau f = \partial H + H \partial$; then $f + QH$ is a chain map between mapping cones. This is a quasi-isomorphism after setting $Q = 0$, so it is a quasi-isomorphism by Lemma~\ref{lem:quasi-iso-move}. In particular, Theorem~\ref{thm:taucomplex} shows that the quasi-isomorphism class of $\mathit{CKhI}(L)$ is an invariant of $(L, \tau)$. 

\begin{rmk}\label{rmk:sanoresults}
By Lemma~\ref{lem:qitohe}, this proves that the homotopy equivalence class of $\mathit{CKhI}(L)$ is an invariant of $(L, \tau)$. 
In fact, in \cite{Sano} these homotopy equivalences are constructed explicitly: given an equivariant Reidemeister move, Sano constructs homotopy equivalences $f \colon \Kcm(D) \rightarrow \Kcm(D')$ and $g \colon \Kcm(D') \rightarrow \Kcm(D)$ such that the homotopies witnessing $g \circ f \simeq \id$ and $f \circ g \simeq \id$ satisfy a certain additional condition with respect to $\tau$. The maps $f$ and $g$, together with these homotopies, can then be used to construct the explicit homotopy equivalences between $\mathit{CKhI}(D)$ and $\mathit{CKhI}(D')$. 
\end{rmk}

We now come to the most basic version of the equivariant $s$-invariant. Let $(C, \tau)$ be a $\tau$-complex. Define
\[
\iota_* \colon H_*(C) \xrightarrow{\cong} \F[u,u^{-1}]
\]
to be the map induced by including $H_*(C)$ into its localization. (Note that $C$ is a knotlike complex, so $u^{-1}H_*(C) \cong \f[u, u^{-1}]$.) This is well-defined since the only (grading-preserving) endomorphism of $\F[u,u^{-1}]$ as an $\F[u,u^{-1}]$-module is the identity.

\begin{defn}\label{def:tilde-s}
Let $(C, \tau)$ be a $\tau$-complex. Define the even integer
\begin{align*}
\wt{s}(C, \tau) = \max_i \{&\text{there exists } [x] \in H_*(C) \mid \iota_*[x]=u^{-i/2} \in u^{-1}H_*(C) \text{ and } \tau_*[x] = [x] \}.
\end{align*}
Equivalently, $\wt{s}(C, \tau)$ is the maximum $i$ such that $H_*(C)$ has a $\tau$-invariant, $u$-nontorsion class in bigrading $(0, i)$. For a strongly invertible knot $(K, \tau)$, we define
\[
\wt{s}(K, \tau) = \wt{s}(\Kcrm(K), \tau).
\]
We often suppress writing $\tau$ and refer to this as $\wt{s}(K)$. Note that we always have $\wt{s}(K) \leq s(K)$, where $s(K)$ is the usual $s$-invariant. 
\end{defn}

\begin{rmk}\label{rmk:involutive-version}
One can also define $\wt{s}$ using the mapping cone, as done in \cite{Sano}. From this point of view there are two natural invariants, which play the role of $\underline{d}$ and $\overline{d}$ from involutive Heegaard Floer homology. However, these are related to each other by dualizing: $\wt{s}(C)$ performs the role of $\underline{d}$, while $\wt{s}(C^\vee)$ performs the role of $-\overline{d}$. For simplicity, we thus confine ourselves to a single numerical invariant. In general, $\wt{s}(K)$ and $\wt{s}(-K)$ cannot be determined from each other and contain different information.
\end{rmk}

If $f$ is a local map from $(C_1, \tau_1)$ to $(C_2, \tau_2)$ of grading shift $\gr_q(f)$, then it is clear that
\[
\wt{s}(C_1, \tau_1) + \gr_q(f) \leq \wt{s}(C_2, \tau_2).
\]
Indeed, if $[x]$ is a $u$-nontorsion class in $H_*(C_1)$ with $\tau([x]) = [x]$, then the locality condition combined with the fact that $f$ homotopy commutes with $\tau$ shows that $f([x])$ is such a class in $H_*(C_2)$. We thus obtain:

\begin{proof}[Proof of Theorem~\ref{thm:s-bound}]
If $\Sigma$ is an isotopy-equivariant cobordism from $(K_1,\tau_1)$ to $(K_2,\tau_2)$, then $\Kcrm(\Sigma)$ is a local map from $(\Kcrm_{p_1}(K_1), \tau_1)$ to $(\Kcrm_{p_2}(K_2), \tau_2)$ by Theorem~\ref{thm:taucomplex}. This has grading shift $-2 g(\Sigma)$. The first part is thus immediate. The second part follows from the fact that $\wt{s}(U)=0$ for $U$ the unknot, as then we have both $-2g(\Sigma) \leq \wt{s}(K)$ and $\wt{s}(K) - 2g(\Sigma) \leq 0$.
\end{proof}

\subsection{Examples}\label{subsec:6-examples}

We now present some sample calculations. A program aimed at computing the action of $\tau$ may be found at \cite{Khocomp}. This calculates the Khovanov homology $\Kh(K)$ (over $\f$) of a strongly invertible knot, together with the action of $\tau$ on $\Kh(K)$. The program additionally computes the second page of the Bar-Natan spectral sequence on $\Kh(K)$. As is well-known, this spectral sequence contains the same information as the Bar-Natan homology (over $\f[u]$) and in fact allows us to write down a homotopy equivalent model for $\Kcm(K)$. While the action of $\tau$ on this model is not entirely determined, it is determined modulo $u$ by the action of $\tau$ on $\Kh(K)$. Some simple examples of this procedure are given below:

\begin{example}\label{ex:9-46}
  Let $K$ be the strongly invertible slice knot $9_{46}$; see Figure~\ref{fig:9-46}. The spectral sequence for $\Khr(K)$ is shown on the left in Figure~\ref{fig:tau-946}, giving us the model for $\Kcrm(K)$ shown on the right. The action of $\tau$ on $\Kcrm(K)/u$ is determined by the action of $\tau$ on $\Khr(K)$; in this case, the fact that $\tau$ preserves homological and quantum grading shows that no further action of $\tau$ is possible. Note that the pair $(\Kcrm(K), \tau)$ is locally equivalent to the three-generator subcomplex contained inside the dashed box.

\begin{figure}
  \includegraphics[width=4cm]{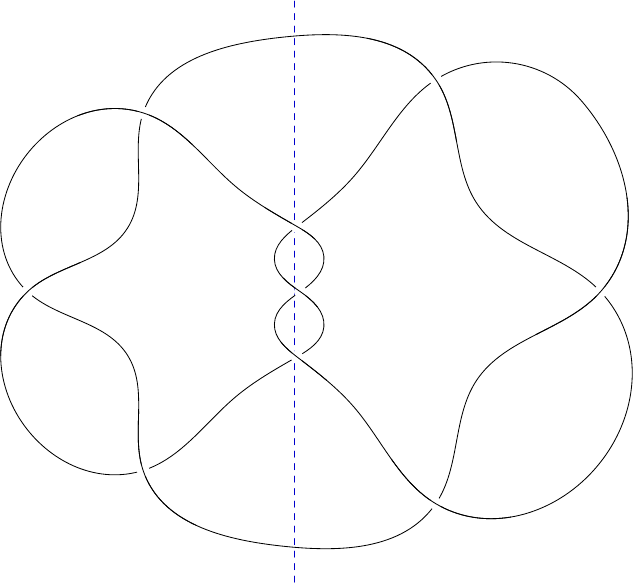}
  \caption{A symmetric diagram of the $9_{46}$ knot.}\label{fig:9-46}
\end{figure}

A quick examination shows $\wt{s}(K) = -2$. We correspondingly have $\ieg(K) \geq 1$, and it is easily seen that $K$ indeed has an equivariant slice surface (in fact, Seifert surface) of genus one. This calculation appeared previously in work of Sundberg-Swann \cite{swann-sundberg}; see also \cite{DMS} for the analogous computation in knot Floer homology.

\begin{figure}[h!]
\includegraphics[scale = 0.7]{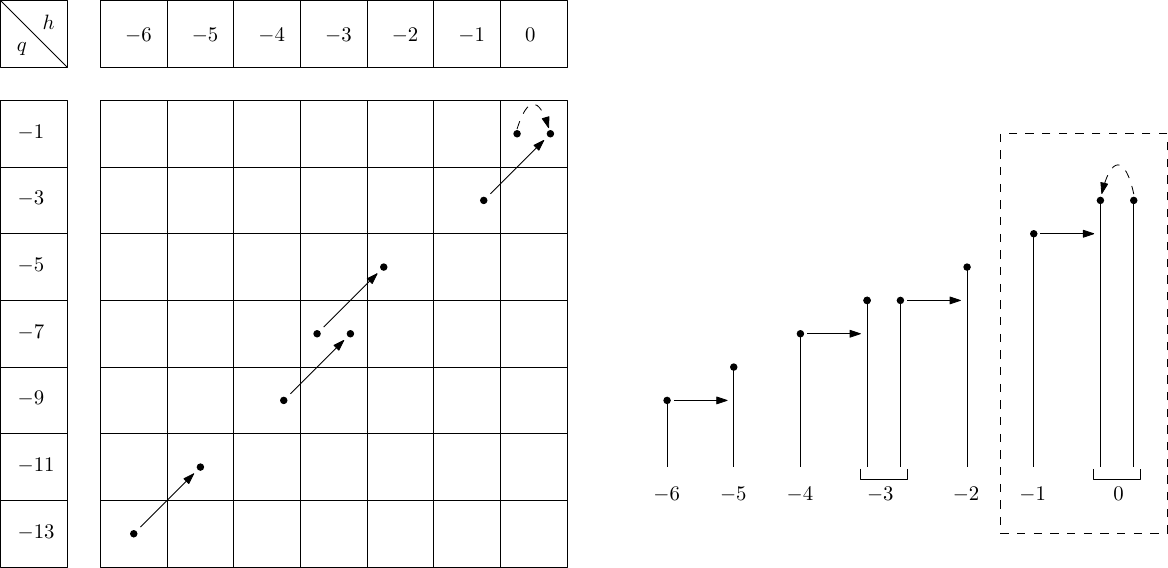}
	\caption{Left: the $E^1$-page of the Bar-Natan spectral sequence for $9_{46}$. Solid arrows are the differential; dashed arrows are the action of $1 + \tau$. Right: homotopy equivalent model for the Bar-Natan complex $\Kcrm(K)$. Solid arrows are the differential; dashed arrows are the action of $1 + \tau$ after setting $u = 0$. (In this case, no further action of $\tau$ is possible.) Each vertical line represents a $\f[u]$-tower in $\Kcrm(K)$. Quantum grading is represented by height, while homological grading is indicated along the horizontal axis.} \label{fig:tau-946}
\end{figure}
\end{example}


\begin{example}\label{ex:kyle_knot}
Let $J = 17nh_{74}$ be the strongly invertible knot from Theorem~\ref{thm:kyle_knot_intro} and $\bar{J}$ be its mirror. (We consider the mirror for computational convenience.) The spectral sequence for $\Khr(\bar{J})$ in homological gradings $-1$ through $3$ is shown on the left in Figure~\ref{fig:tau-J}, giving us the partial model for $\Kcrm(\bar{J})$ shown in the upper right. The action of $\tau$ on $\Kcrm(\bar{J})/u$ is determined by the action of $\tau$ on $\Khr(\bar{J})$. Unlike in Example~\ref{ex:9-46}, however, this does not fully determine the action of $\tau$ on $\Kcrm(\bar{J})$. The simplest possibility is that $\tau$ does not have any components landing in the image of $u$, so that the upper right complex in Figure~\ref{fig:tau-J} indeed displays the full action of $\tau$. However, several other extensions are possible. It is a straightforward but tedious exercise to enumerate all possible extensions of $\tau$ to $\Kcrm(\bar{J})$ which commute with $\partial$, preserve homological and quantum grading, and satisfy $(1 + \tau)^2 \simeq 0$. Once this is done, the ambitious reader may check that there are only two possible local equivalence classes (in the sense of Definition~\ref{def:tau-complex}); these are displayed in the lower right of Figure~\ref{fig:tau-J}. 


A quick examination shows $\smash{\wt{s}(\bar{J}) = -2}$. We correspondingly have $\smash{\ieg(J)} = \smash{\ieg(\bar{J}) \geq 1}$, and it is easily seen that $J$ indeed has an equivariant slice surface (in fact, Seifert surface) of genus one. Example~\ref{ex:kyle_knot} is connected to the study of exotic slice disks, as first shown by Hayden \cite{Hayden} and Hayden-Sundberg \cite{hayden-sundberg}. Compressing along the cores of the obvious $1$-handles in Figure~\ref{fig:kyle} gives a symmetric pair of $\Z$-disks for $J$. In \cite{hayden-sundberg}, it is (essentially) shown that the Khovanov cobordism maps distinguish these disks. Hence $\Sigma$ and $\tau_{B^4}(\Sigma)$ are not smoothly isotopic rel boundary, while work of Conway and Powell~\cite{ConwayPowell} shows that they are \textit{topologically} isotopic rel boundary. Theorem \ref{thm:s-bound} in fact proves that \textit{any} pair of symmetric disks for $J$ cannot be smoothly isotopic rel boundary; this was first shown in \cite{DMS} using knot Floer homology. 

\begin{figure}[h!]
\includegraphics[scale = 0.8]{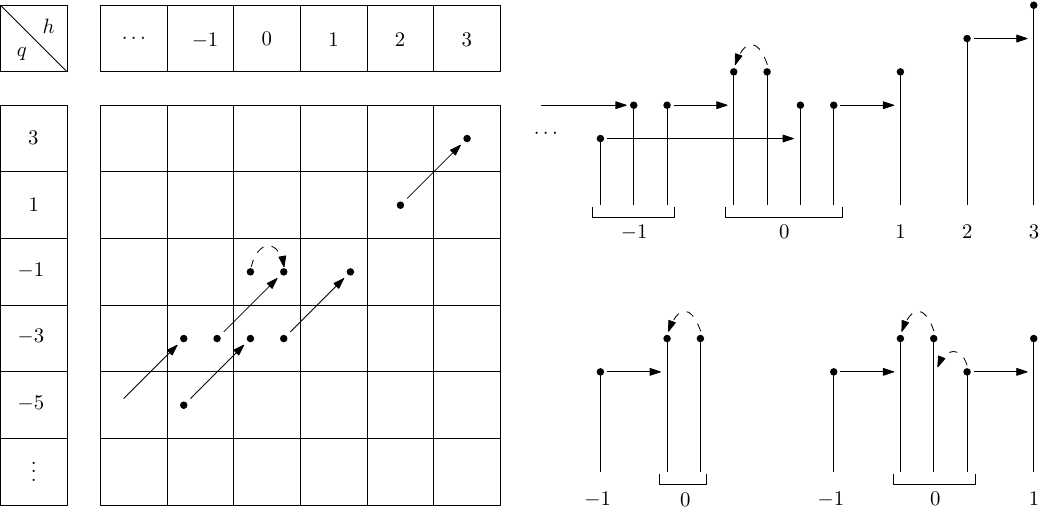}
	\caption{Left: the $E^1$-page of the Bar-Natan spectral sequence for $\bar{J}$. Solid arrows are the differential; dashed arrows are the action of $1 + \tau$. Upper right: homotopy equivalent model for the Bar-Natan complex $\Kcrm(\bar{J})$. Solid arrows are the differential; dashed arrows are the action of $1 + \tau$ after setting $u = 0$. Each vertical line represents a $\f[u]$-tower in $\Kcrm(\bar{J})$. Quantum grading is represented by height, while homological grading is indicated along the horizontal axis. Lower right: possible local equivalence classes for $(\Kcrm(\bar{J}), \tau)$.} 
	\label{fig:tau-J}
\end{figure}
\end{example}

The authors also conducted a general computer-based search for strongly invertible knots with $\wt{s}\neq s$ using \cite{Khocomp}. We identified diagrammatic symmetries using the following:

\begin{defn}
Let $D$ be a diagram for a knot whose PD code enumerates the edges from $1$ to $n$. We say that $D$ exhibits a \textit{PD code symmetry} if there exists some $k$ such that replacing each label $i$ with $((k - i) \ \mathrm{ mod } \ n) + 1$ leaves the PD code invariant. If such a symmetry occurs, then the knot is strongly invertible and $D$ is a transvergent diagram up to isotopy.
\end{defn}

For all knots up to thirteen crossings, we checked the associated PD code stored in SnapPy for a PD code symmetry.\footnote{The absence of a PD code symmetry does not, of course, indicate that the knot is not strongly invertible -- just that the diagram stored in SnapPy via the PD code does not exhibit the strong inversion.} After finding all listed PD codes with such a symmetry, we ran the computer program \cite{Khocomp} to determine the symmetry action on Khovanov homology. In many cases, we found that $\Khr(K)$ does not admit any $\tau$-invariant generator that survives to the $E^\infty$-page of the Bar-Natan spectral sequence. This implies $\wt{s}(K) < s(K)$. In fact, in all such examples, we additionally found that $\Khr(K)$ degenerates by the $E^2$-page, which means that $H_*(\Kcrm(K))$ has $u$-torsion order one. Since $s(K) - \wt{s}(K)$ is clearly bounded above by (twice) the $u$-torsion order of $H_*(\Kcrm(K))$, we have that $\wt{s}(K) = s(K) - 2$ for these knots. Our calculations are listed in Appendix~\ref{sec:different}. 

\subsection{Further applications}
We now give some further applications of studying the action of $\tau$.

\subsubsection*{Distinguishing mutant pairs}
In \cite{lobb-watson} Lobb-Watson present an example of a pair of mutant knots $K_1$ and $K_2$ that have isomorphic Khovanov and knot Floer homologies, but can nevertheless be distinguished by studying their triply-graded refinement. The idea of their proof is the following: both $K_1$ and $K_2$ are hyperbolic, and it is a result of Sakuma \cite[Proposition 3.1]{Sakuma} that a hyperbolic knot has at most two strong inversions. Fix any strong inversion $\tau_2$ of $K_2$. If $K_1$ and $K_2$ were isotopic, then the pair $(K_2, \tau_2)$ would be Sakuma equivalent to $(K_1, \tau_1)$ for some strong inversion $\tau_1$ of $K_1$. It thus suffices to find a strong inversion $\tau_2$ of $K_2$ such that the invariants for $(K_2, \tau_2)$ are not equal to the invariants for $(K_1, \tau_1)$ for either of the two strong inversions $\tau_1$ of $K_1$. In \cite[Section 6.1]{lobb-watson}, Lobb-Watson carry this out using their triply-graded refinement of Khovanov homology. Here, we show that we may instead simply calculate the action of $\tau$ on $\Kh(K)$ using \cite{Khocomp}.


\begin{proof}[Proof of Theorem~\ref{thm:lw_knots}]
The knot $K_1$ admits two transvergent diagrams corresponding to two strong inversions, which we denote $\tau_1$ and $\tau_1'$. These are displayed in Figure~\ref{fig:LWa1} and Figure~\ref{fig:LWa2}, respectively. The knot $K_2$ also admits two strong inversions, but we use only one of these, which we denote by $\tau_2$ and display on the right in Figure~\ref{fig:lobbwatson}. In each case, we computed the action of $\tau$ on the Khovanov homology $\Kh(K_i)$ using \cite{Khocomp} and recorded the dimension of $\ker(1 + \tau)$ in each bigrading. For the two involutions on $K_1$, these ranks were in fact the same, but the involution $\tau_2$ on $K_2$ yielded a different set of ranks. Our results are recorded in Table~\ref{fig:LWa1_hom}. It thus only remains to ensure that $\tau_1$ and $\tau_1'$ are actually distinct; this can be done by computing their $\eta$-polynomials according to the algorithm of \cite{Sakuma}. See Lemma~\ref{lem:distinguished} below.
\end{proof}



\begin{figure}
  \begin{tikzpicture}
    \node at (-5,0) {\includegraphics[width=3cm]{lwa1-eps-converted-to.pdf}};
    \node at (-0,0) {\includegraphics[width=2cm]{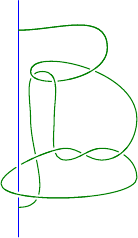}};
    \node at (5,0) {\includegraphics[width=3cm]{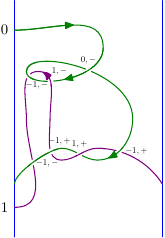}};
  \end{tikzpicture}
  \caption{Left: the diagram of $K_1$ with the strong inversion $\tau_{11}$. Middle and right: subsequent steps for computing
  the Sakuma $\eta$ polynomial for $(K_1,\tau_{11})$. }\label{fig:LWa1}
\end{figure}
\begin{figure}
  \begin{tikzpicture}
    \node at (-5,0) {\includegraphics[width=3cm]{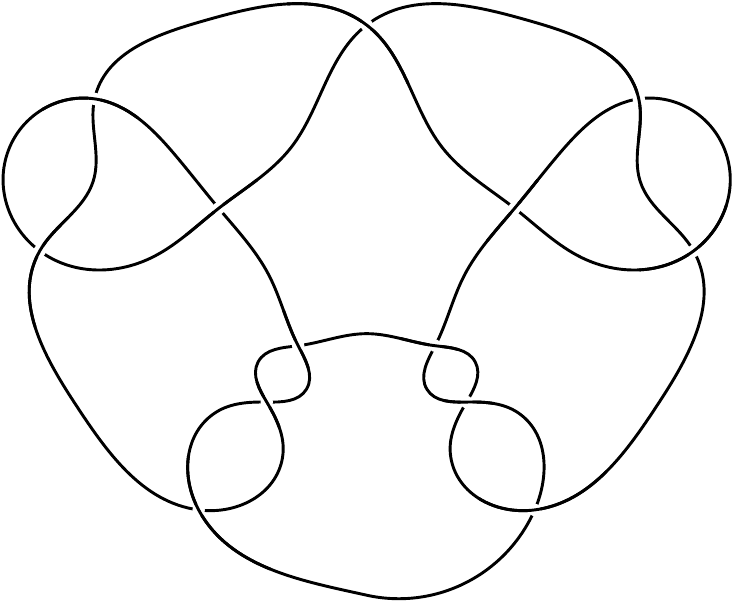}};
    \node at (0,0) {\includegraphics[width=2cm]{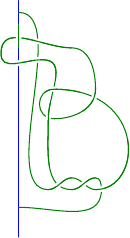}};
    \node at (5,0) {\includegraphics[width=3cm]{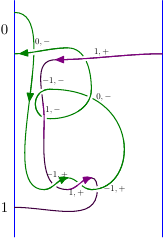}};
  \end{tikzpicture}
  \caption{Left: the diagram of $K_1$ with the strong inversion $\tau_{12}$. Middle and right: subsequent steps for computing
  the Sakuma $\eta$ polynomial for $(K_1,\tau_{12})$. }\label{fig:LWa2}
\end{figure}

\begin{table}
  \tiny
  \begin{tabular}{|c|c|c|c|c|c|c|c|c|c|c|c|c|c|c|}
\hline
-9 &-4 &-3 &-2 &-1 &0 &1 &2 &3 &4 &5 &6 &7 &8 &9\\
\hline
21 & &  &  &  &  &  &  &  &  &  &  &  &  & 1,1 \\
\hline
19 & &  &  &  &  &  &  &  &  &  &  &  & 3,2 & 1,1 \\
\hline
17 & &  &  &  &  &  &  &  &  &  &  & 6,4 & 3,2 &  \\
\hline
15 & &  &  &  &  &  &  &  &  &  & 14,8 & 6,4 &  &  \\
\hline
13 & &  &  &  &  &  &  &  &  & 22,12 & 14,8 &  &  &  \\
\hline
11 & &  &  &  &  &  &  &  & 29,16 & 22,12 &  &  &  &  \\
\hline
9 & &  &  &  &  &  &  & 37,20 & 29,16 &  &  &  &  &  \\
\hline
7 & &  &  &  &  &  & 37,21 & 37,20 &  &  &  &  &  &  \\
\hline
5 & &  &  &  &  & 33,17 & 37,21 &  &  &  &  &  &  &  \\
\hline
3 & &  &  &  & 28,16 & 33,17 &  &  &  &  &  &  &  &  \\
\hline
1 & &  &  & 18,9 & 28,16 &  &  &  &  &  &  &  &  &  \\
\hline
-1 & &  & 10,6 & 18,9 &  &  &  &  &  &  &  &  &  &  \\
\hline
-3 & & 4,2 & 10,6 &  &  &  &  &  &  &  &  &  &  &  \\
\hline
-5 &1,1 & 4,2 &  &  &  &  &  &  &  &  &  &  &  &  \\
\hline
-7 &1,1 &  &  &  &  &  &  &  &  &  &  &  &  &  \\
\hline
\end{tabular}

  \normalsize
  \vspace{0.5cm}
  \caption{The Khovanov homology of the $K_1$ knot. The first number denotes the rank. The second number is the dimension
  of $\ker(\id+\tau_1)$, which turns out to be equal to the dimension of $\ker(\id+\tau_1')$.}\label{fig:LWa1_hom}
\end{table}

\begin{table}
  \tiny
  \begin{tabular}{|c|c|c|c|c|c|c|c|c|c|c|c|c|c|c|}
\hline
-9 &-4 &-3 &-2 &-1 &0 &1 &2 &3 &4 &5 &6 &7 &8 &9\\
\hline
21 & &  &  &  &  &  &  &  &  &  &  &  &  & 1,1 \\
\hline
19 & &  &  &  &  &  &  &  &  &  &  &  & 3,3 & 1,1 \\
\hline
17 & &  &  &  &  &  &  &  &  &  &  & 6,5 & 3,3 &  \\
\hline
15 & &  &  &  &  &  &  &  &  &  & 14,10 & 6,5 &  &  \\
\hline
13 & &  &  &  &  &  &  &  &  & 22,15 & 14,10 &  &  &  \\
\hline
11 & &  &  &  &  &  &  &  & 29,19 & 22,15 &  &  &  &  \\
\hline
9 & &  &  &  &  &  &  & 37,23 & 29,19 &  &  &  &  &  \\
\hline
7 & &  &  &  &  &  & 37,24 & 37,23 &  &  &  &  &  &  \\
\hline
5 & &  &  &  &  & 33,21 & 37,24 &  &  &  &  &  &  &  \\
\hline
3 & &  &  &  & 28,18 & 33,21 &  &  &  &  &  &  &  &  \\
\hline
1 & &  &  & 18,12 & 28,18 &  &  &  &  &  &  &  &  &  \\
\hline
-1 & &  & 10,7 & 18,12 &  &  &  &  &  &  &  &  &  &  \\
\hline
-3 & & 4,3 & 10,7 &  &  &  &  &  &  &  &  &  &  &  \\
\hline
-5 &1,1 & 4,3 &  &  &  &  &  &  &  &  &  &  &  &  \\
\hline
-7 &1,1 &  &  &  &  &  &  &  &  &  &  &  &  &  \\
\hline
\end{tabular}

  \normalsize
  \vspace{0.5cm}
  \caption{The Khovanov homology of the $K_2$ knot. The first number denotes the rank. The second number is the dimension
  of $\ker(\id+\tau_2)$.}\label{fig:LWb1_hom}
\end{table}

\begin{lem}\label{lem:distinguished}
  The involutions $\tau_1$ and $\tau_1'$ are distinguished by their Sakuma polynomials. 
\end{lem}
\begin{proof}
We use the algorithm described in \cite{Sakuma}. Figures~\ref{fig:LWa1} and \ref{fig:LWa2} 
show passing from the knot to the collection of arcs in the ``pseudo-fundamental cycle"; compare \cite[Algorithm on pages 181--182]{Sakuma}. To simplify the notation, we will refer to the $\eta$ functions as $\eta_1$ for $\tau_1$ and $\eta_2$ for $\tau_1'$. 
We begin with computation of $\wt{\eta}_1$ and $\wt{\eta}_2$ of Step~5 of Sakuma's algorithm. From 
Figures~\ref{fig:LWa1} and \ref{fig:LWa2}, we read off: 
\begin{align*}
  \wt{\eta}_1&=-x_0\\
  \wt{\eta}_2&=x_1-2x_0+x_{-1}.
\end{align*}
Step~6 of Sakuma's algorithm replaces each $x_i$ by $t^{i-1}-2t^i+t^{i+1}$, so that
\begin{align*}
  \eta'_1&=-t+2-t^{-1}=[2,-1\\
  \eta'_2&=t^2-4t+6-4t^{-1}+t^{-2}=[6,-4,1
\end{align*}
where following Sakuma, we write $[a_0,a_1,\dots$ for the polynomial $a_0+a_1(t+t^{-1})+a_2(t^2+t^{-2})+\cdots$. From the primed polynomials $\eta$, we recover the $\eta$ polynomials via Step~6 of Sakuma's algorithm, namely
\begin{align*}
  \eta_1&=[0,0,0=0\\
  \eta_2&=[-2,0,1=t^2-2+t^{-2}.
\end{align*}
\end{proof}

\subsubsection*{Equivariant squeezedness} We now give an application to equivariant squeezedness. We first recall the definition of a \textit{squeezed knot} from \cite{feller-lewark-lobb}:

\begin{defn}\label{def:squeeze}
A knot $K \subset S^3$ is \emph{squeezed} if it appears as a slice of a genus-minimizing, connected cobordism from a positive torus knot $T^+$ to a negative torus knot $T^-$.
\end{defn}

Although at first Definition \ref{def:squeeze} may seem rather unnatural, \cite{feller-lewark-lobb} show that many classes of familiar knots are squeezed, among them quasipositive knots and alternating knots.  From our present point of view however, squeezed knots are those knots for which there is a direct geometric reason for vanishing of quantum-type invariants (an interpretation also pointed out by \cite{feller-lewark-lobb}) - for example, $4_1$ is squeezed but not slice, though its quantum invariants are compatible with it being slice (for this point of view, it is important to note that torus knots have very constrained local type in quantum-type invariants).  

Note that the minimal genus of a cobordism from $T^+$ to $T^-$ is $g_4(T^+) + g_4(T^-)$. Such a cobordism is realized (for example) by the standard minimal-genus cobordism from $T^+$ to the unknot, followed by the cobordism from the unknot to $T^-$. Minimality can established by using the usual $s$-invariant, since
\[
s(T^+) = 2g_4(T^+) \quad \text{and} \quad s(T^-) = - 2g_4(T^-).
\]
The notion of a squeezed knot admits a natural generalization to the equivariant setting: 

\begin{defn}\label{def:equivar-squeezed}
We say that an equivariant knot $K\subset S^3$ is \emph{equivariantly squeezed} if it appears as a slice of an \textit{equivariant} genus-minimizing, connected cobordism from a positive torus knot $T^+$ to a negative torus knot $T^-$.  
\end{defn}

Note that every torus knot has a unique strong inversion (see for example \cite[Theorem 3.1]{Sakuma}). To ensure that Definition~\ref{def:equivar-squeezed} is not vacuous, we claim:

\begin{lem} 
For each pair $(T^+, T^-)$, there exists at least one equivariant connected cobordism from $T^+$ to $T^-$ that is genus-minimizing.\footnote{We stress, however, that we do not claim \textit{every} genus-minimizing cobordism from $T^+$ to $T^-$ is isotopic rel boundary to an equivariant cobordism.}
\end{lem}
\begin{proof}
In light of the fact that the minimal genus of a cobordism from $T^+$ to $T^-$ is $g_4(T^+) + g_4(T^-)$, it suffices to show that there is a genus-minimizing slice surface for $T_{p, q}$ which is equivariant. Indeed, this holds even in the setting of Seifert surfaces: let $T_{p, q}$ be the link of the singularity $f(z, w) = z^p + w^q$ and consider the involution sending $(z, w) \mapsto (\bar{z}, \bar{w})$. It is clear that this is a strong inversion on $T_{p, q}$ since the intersection of $z^p + w^q = 0$ and $|z|^2 + |w|^2 = \epsilon$ in $\mathbb{R}^2$ consists of two points. Clearly, the Seifert surface $(f/|f|)^{-1}(1)$ is equivariant. Pushing this into the 4-ball and using the fact that the slice genus and Seifert genus of $T_{p, q}$ coincide, we see that the slice genus of $T_{p, q}$ is realized by an equivariant slice surface. 
\end{proof}

Torus knots are especially convenient since their $\tau$-complexes are locally trivial up to an overall shift. Let $\f[u]$ denote the trivial complex with the identity action of $\tau$, and let $\f[u]_{i}$ denote the complex $\f[u]$, shifted in grading so that $\gr(1) = i$. Then: 

\begin{lem}\label{lem:torus-calculation}
We have local equivalences of $\tau$-complexes
\[
(\Kcrm(T^+), \tau) \sim \f[u]_{2g_4(T^+)} \quad \text{and} \quad (\Kcrm(T^-), \tau) \sim \f[u]_{-2g_4(T^-)}.
\]
\end{lem}
\begin{proof}
First consider $T^+$. It is easily checked that we have an equivariant diagram $D$ for $T^+$, all of whose crossings are positive. In the resulting cube of resolutions, there is a unique resolution $D_v$ of homological grading zero, which is the minimal homological grading. Let $\frs_o$ be the canonical generator lying in the reduced complex. Clearly, $\frs_o$ lies in $D_v$ and is not homologous to any other cycle, since there are no chains of homological grading $-1$. Thus, $\frs_o$ necessarily generates the $\f[u]$-tower in the homology of $\Kcrm(T^+)$. In particular, $\frs_o$ lies in quantum grading $s(T^+) = 2g_4(T^+)$. It is easy to see that $\frs_o$ must be $\tau$-invariant, since $D_v$ is sent to itself by $\tau$ (for grading reasons) and symmetric pairs of circles are clearly assigned the same label in $\frs_o$. Sending $1$ to $\frs_o$ thus gives a local map
\[
\f[u]_{2g_4(T^+)} \rightarrow \Kcrm(T^+).
\]
The local map in the other direction is given by the minimal-genus equivariant cobordism from $T^+$ to the unknot, where we shift the codomain so that the map has grading shift zero. The claim for $T^-$ follows from the evident behavior of $\tau$ under mirroring. 
\end{proof}

We thus obtain:

\begin{proof}[Proof of Theorem~\ref{thm:equivariantly-squeezed}]
Suppose we have equivariant cobordisms from $T^+$ to $K$ and $K$ to $T^-$ as in Definition~\ref{def:equivar-squeezed}. Let these have genus $g^+$ and $g^-$. Then we have local maps
\[
\Kcrm(T^+) \rightarrow \Kcrm(K) \quad \text{and} \quad \Kcrm(K) \rightarrow \Kcrm(T^-)
\]
of grading shift $-2g^+$ and $-2g^-$, respectively. Using Lemma~\ref{lem:torus-calculation} and shifting the grading in the domain and codomain, we have local maps
\[
\f[u]_{2g_4(T^+) - 2g^+} \rightarrow \Kcrm(K) \quad \text{and} \quad \Kcrm(K) \rightarrow \f[u]_{-2g_4(T^-) + 2g^-}.
\]
On the other hand, by minimality of the genus, we have $g_4(T^+) + g_4(T^-) = g^+ + g^-$, so the left- and right-hand complexes are the same. Hence $(\Kcrm(K), \tau)$ is locally equivalent to a trivial complex, which necessarily starts in grading $s(K)$. 
\end{proof}

In particular:

\begin{proof}[Proof of Corollary~\ref{cor:10_141}]
  For $K = 10_{141}$ (equipped with the strong inversion of Figure~\ref{fig:10141}), we have $s(K) \neq \wt{s}(K)$, as displayed in Appendix~\ref{sec:different}.
\end{proof}

\section{The Borel Construction} \label{sec:borel} 

In this section, we introduce the Borel construction and prove Theorem~\ref{thm:full_borel_invariant}. 

\subsection{Algebraic setup} \label{sec:borel-setup}

We begin with the algebraic formalism, which will already be familiar to readers acquainted with Borel equivariant homology.

\begin{defn}\label{def:borel}
A \textit{Borel complex} $C_Q$ is a free, finitely-generated chain complex over the ring $\f[u, Q]$ with a $\Z \oplus \Z$-grading for which
\[
\gr(\partial) = (1, 0),  \quad \gr(u) = (0, -2), \quad \text{and} \quad \gr(Q) = (1, 0).
\]
Two Borel complexes are \textit{homotopy equivalent} if they are homotopy equivalent in the usual sense; that is, as complexes over $\f[u, Q]$.
\end{defn}

The principal way to generate a Borel complex is the following. Suppose that $C$ is a bigraded chain complex over $\f[u]$ equipped with an involution $\tau$. Then we may form a Borel complex by setting 
\[
C_Q = C \otimes_{\f} \f[Q] \quad \text{and} \quad \partial_Q = \partial + Q \cdot (1 + \tau).
\]
All of the Borel complexes in this paper will be formed in such a manner. When we denote a Borel complex by $C_Q$, we thus often mean that $C_Q$ is constructed from a complex $C$ as above.

\begin{rmk}\label{rmk:higher-homotopy1}
Strictly speaking, we do not need $\tau$ to be an actual involution in order to obtain a Borel complex from the data of $C$ and $\tau$. For instance, suppose $\tau$ is only a homotopy involution, but we can find a choice of homotopy $\tau^2 + \id = \partial H + H \partial$ such that $H$ is $\tau$-equivariant. Then it is straightforward to check that the map $\partial_Q = \partial + Q(1 + \tau) + Q^2 H$ is a differential on $C_Q$. Likewise, if $H$ is $\tau$-equivariant only up to homotopy, but we can find a specified choice of homotopy $\tau H + H \tau = \partial H' + H' \partial$ such that $H'$ is $\tau$-equivariant, then the map $\partial_Q = \partial + Q(1 + \tau) + Q^2 H + Q^3 H'$ defines a differential on $C_Q$, and so on. 
\end{rmk}

Now suppose that $C_1$ and $C_2$ are two chain complexes over $\f[u]$ equipped with involutions $\tau_1$ and $\tau_2$, respectively. If $f \colon C_1 \rightarrow C_2$ is a $\tau$-equivariant chain map, then $f_Q = f \otimes \id$ commutes with $\partial_Q$. More generally, if $f$ is only $\tau$-equivariant up to homotopy, but we can choose a witnessing homotopy $H$ (i.e. $f\circ\tau_1=\tau_2\circ f+[d,H]$) which itself is $\tau$-equivariant, then the map
\[
f_Q = f + Q H
\]
is easily seen to be a chain map from $(C_1)_Q$ to $(C_2)_Q$.

\begin{rmk}\label{rmk:higher-homotopy2}
Once again, we do not need the witnessing homotopy to be strictly $\tau$-equivariant. If $H$ is only $\tau$-homotopy-equivariant, but we can find a choice of homotopy $H'$ such that $H \tau_1 + \tau_2 H = \partial H' + H'\partial$ such that $H'$ is $\tau$-equivariant, then $f_Q = f + Q H + Q^2 H'$ defines a map of Borel complexes, and so on. The reader should think of the Borel construction as capturing the higher-order homotopy commutation behavior of $\tau$. 
\end{rmk}

As in Section~\ref{sec:tau-complex}, in the case of a strongly invertible knot we obtain a special type of Borel complex. This is captured by the following notion:

\begin{defn}\label{def:knotlike-borel}
A \textit{knotlike Borel complex} is a Borel complex $C_Q$ that additionally satisfies the localization condition
\[
u^{-1}H_*(C_Q) \cong \f[u, u^{-1}, Q],
\]
where $1 \in \F[u, u^{-1}, Q]$ has bigrading $(0, 0)$. We say that a relatively graded, $\f[u, Q]$-equivariant chain map $f$ between two knotlike Borel complexes $(C_1)_Q$ and $(C_2)_Q$ is \textit{local} if the induced map
\[
f \colon u^{-1}H_*((C_1)_Q) \cong \f[u, u^{-1}, Q] \rightarrow u^{-1}H_*((C_2)_Q) \cong \f[u, u^{-1}, Q]
\]
is an isomorphism. Note that $f$ is not required to be $(\gr_h, \gr_q)$-preserving, but instead necessarily preserves $\gr_h$ and shifts $\gr_q$ by an even constant. We refer to $\gr_q(f)$ as the associated grading shift. Two Borel complexes are \textit{locally equivalent} if there exist local maps between them of grading shift zero. It is not difficult to check that if $C_Q$ is constructed out of a knotlike complex $C$ in the sense of Definition~\ref{def:knotlike}, then $C_Q$ will be a knotlike Borel complex.
\end{defn}

We now give several examples of Borel complexes to illustrate Definition~\ref{def:borel}. Although not strictly necessary, each of our examples is constructed out of a pair $(C, \tau)$ as above. We begin with the two simplest non-trivial examples:

\begin{example}\label{ex:borel-ex1}
Consider the chain complex $C = \mathrm{span}_{\f[u]}\{x, y, z\}$ over $\f[u]$ with differential $\partial y = uz$ and gradings $\gr(x) = (0, 0)$, $\gr(y) = (0,0)$ and $\gr(z) = (1, 2)$. Define $(1 + \tau)y = x$. Then $C$, $C_Q$, and $H_*(C_Q)$ are as in Figure~\ref{fig:borel-ex1}.
\begin{figure}[h!]
\includegraphics[scale = 1]{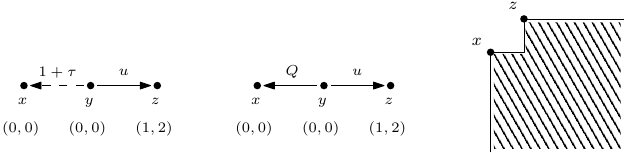}
\caption{Left: $(C, \tau)$. Dashed lines denote the action of $1 + \tau$, while solid lines denote the differential $\partial$. The decoration above the solid line indicates the coefficient of the differential; that is, $\partial y = uz$. Middle: $C_Q$. Solid lines denote the differential $\partial_Q$; here, $\partial_Q y = Qx + uz$. Right: $H_*(C_Q)$. Multiplication by $u$ and $Q$ are schematically represented by vertical and horizontal translation, respectively. In this case, $H_*(C_Q)$ has a presentation with two generators $x$ and $z$ and the single relation $Qx = uz$.}\label{fig:borel-ex1}
\end{figure}
\end{example}

\begin{example}\label{ex:borel-ex2}
Consider the chain complex $C = \mathrm{span}_{\f[u]}\{x, y, z\}$ over $\f[u]$ with differential $\partial z = uy$ and gradings $\gr(x) = (0, 0)$, $\gr(y) = (0,0)$ and $\gr(z) = (-1, -2)$. Define $(1 + \tau)x = y$. Then $C$, $C_Q$, and $H_*(C_Q)$ are as in Figure~\ref{fig:borel-ex2}.
\begin{figure}[h!]
\includegraphics[scale = 1]{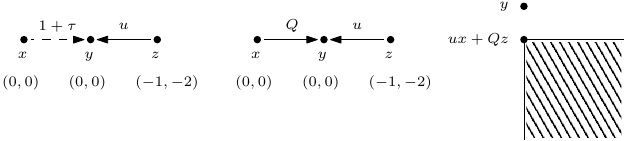}
\caption{Left: $(C, \tau)$. Middle: $C_Q$. Right: $H_*(C_Q) \cong \f[u, Q]_{(0, -2)} \oplus \f[u, Q]/(u, Q)_{(0,0)}$. Here, we write the subscript $(a, b)$ indicates a grading shift so that the element $1$ in the relevant module has grading $(a, b)$.}\label{fig:borel-ex2}
\end{figure}
\end{example}

\begin{rmk}
Borel complexes are quite similar to knot Floer complexes over the ring $\f[U, V]$. The reader who is familiar with knot Floer homology will easily generalize the Borel complexes of Examples~\ref{ex:borel-ex1} and \ref{ex:borel-ex2} to the analogue of staircase complexes in knot Floer theory.
\end{rmk}

In Examples~\ref{fig:borel-ex1} and \ref{fig:borel-ex2}, the Borel construction does not fundamentally contain more information than $(C, \tau)$. Indeed, in these cases the Borel complex can be determined from $H_*(C)$ together with the action of $\tau$ on $H_*(C/u)$. (We make this precise in Lemma~\ref{lem:minimal-model-2} below.) However, in general it is possible for $(C_1, \tau_1)$ and $(C_2, \tau_2)$ to be homotopy equivalent in the sense of Definition~\ref{def:homotopy-equivalence-pair}, but for their Borel complexes be distinct. The simplest example is probably the following:

\begin{example}\label{ex:borel-ex3}
Consider the chain complex $C = \mathrm{span}_{\f[u]}\{x, r, r', y, z \}$ over $\f[u]$ with differential
\[
\partial r' = r \quad \text{and} \quad \partial z = u y
\]
and gradings
\[
\gr(x) = (0, 0), \ \gr(r) = (0, 0), \ \gr(r') = (-1, 0), \ \gr(y) = (-1, 0), \text{ and } \gr(z) = (-2, -2).
\]
Define two different involutions $\tau_1$ and $\tau_2$ on $C$ for which
\[
(1 + \tau_1) r' = y \quad \text{while} \quad (1 + \tau_2) x = r \quad \text{and} \quad (1 + \tau_2) r' = y.
\]
These are depicted in Figure~\ref{fig:borel-ex3-1}. Note that the actions on $\tau_1$ and $\tau_2$ on homology are both trivial. In fact, it is not difficult to check that as $\tau$-complexes, $(C, \tau_1)$ and $(C, \tau_2)$ are homotopy equivalent. To see this, note that the identity map between them is not $\tau$-equivariant, but it is $\tau$-equivariant up to homotopy: if we set $H(x) = r'$, then $\tau_1 + \tau_2 = \partial H + H \partial$. Importantly, $H$ itself is \textit{not} $\tau$-equivariant -- in fact, $H$ is not even $\tau$-equivariant up to homotopy, since $(H\tau_1 + \tau_2H)(x) = y$, which is a homologically nontrivial cycle in $C$.
\begin{figure}[h!]
\includegraphics[scale = 1]{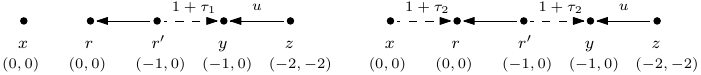}
\caption{Left: $(C, \tau_1)$. Right: $(C, \tau_2)$.}\label{fig:borel-ex3-1}
\end{figure}

It is straightforward to compute the Borel complexes $(C_1)_Q$ and $(C_2)_Q$, where we use $C_1$ and $C_2$ to denote $C$ equipped with the involutions $\tau_1$ and $\tau_2$, respectively. Using the canceling pair $\partial r' = r$, we may observe that $(C_1)_Q$ and $(C_2)_Q$ are both homotopy equivalent to smaller complexes, each spanned by three generators. These are displayed in Figure~\ref{fig:borel-ex3-2}. It is then straightforward to check that $(C_1)_Q$ and $(C_2)_Q$ do not even have the same homology; hence $(C_1)_Q$ and $(C_2)_Q$ are not homotopic over $\f[u, Q]$.
\begin{figure}[h!]
\includegraphics[scale = 1]{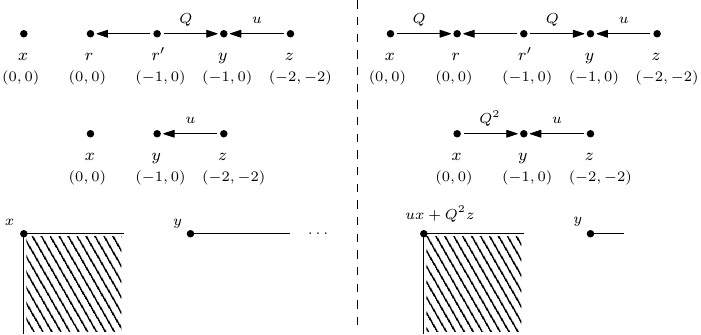}
\caption{Top: $(C_1)_Q$ and $(C_2)_Q$. Middle: simplified models for $(C_1)_Q$ and $(C_2)_Q$; abusing notation, we label these generators also by $x$, $y$, and $z$. Bottom: $H_*((C_1)_Q) \cong \f[u, Q]_{(0,0)} \oplus \f[u, Q]/(u)_{(-1,0)}$ and $H_*((C_2)_Q) \cong \f[u,Q]_{(0, -2)} \oplus \f[u, Q]/(u, Q^2)_{(-1, 0)}$.}\label{fig:borel-ex3-2}
\end{figure}
\end{example}

We give one final example, which although unmotivated at present will turn out to be useful later:

\begin{example}\label{ex:borel-ex4}
Consider the chain complex $C$ and involution $\tau$ displayed in Figure~\ref{fig:borel-ex4}. The Borel complex of this is easily computed; using the canceling generators $\partial r' = r$ and $\partial s' = s$, $C_Q$ is homotopy equivalent to a smaller complex spanned by five generators. This complex and its homology are likewise displayed in Figure~\ref{fig:borel-ex4}. 
\begin{figure}[h!]
\includegraphics[scale = 1]{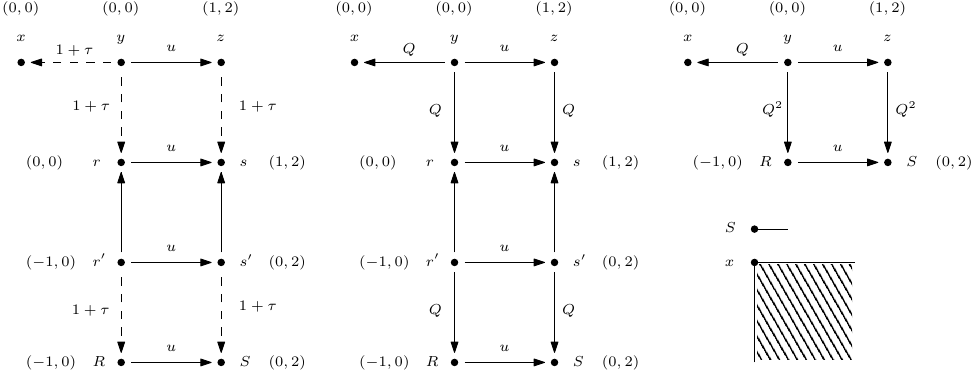}
\caption{Left: $(C, \tau)$. Middle: $C_Q$. Upper right: a simplified model for $C_Q$; by abuse of notation, we label these generators also by $x$, $y$, $z$, $R$, and $S$. Lower right: $H_*(C_Q) \cong \f[u, Q]_{(0,0)} \oplus \f[u, Q]/(u, Q^2)_{(0,2)}$.}\label{fig:borel-ex4}
\end{figure}

It is not difficult to check that $(C, \tau)$ is locally equivalent (in the sense of Definition~\ref{def:tau-complex}) to the pair from Example~\ref{ex:borel-ex1}. In one direction, the local map is given by quotienting out by the subcomplex spanned by $\{r, s, r', s', R, S\}$; in the other direction, the local map is given by the obvious inclusion of $\f[u]$-complexes. As in Example~\ref{ex:borel-ex3}, this latter map is not $\tau$-equivariant, but it is $\tau$-homotopy-equivariant. However, as we will see, $C_Q$ is not locally equivalent (in an appropriate sense for Borel complexes) to the Borel complex from Example~\ref{ex:borel-ex1}.
\end{example}

\subsection{Borel complexes for transvergent diagrams}\label{sec:borel-transvergent}

We now introduce the central construction of the section. 

\begin{defn}\label{def:kcq}
Let $D$ be a transvergent diagram for $L$ and $\tau$ be the involution on $\Kcm(D)$ constructed in Definition~\ref{def:diagrammatic_tau}. Define the \textit{Borel complex of $D$} to be
\[
\Kcm_Q(D) =  \Kcm(D) \otimes \f[Q] \quad \text{with} \quad \partial_Q = \partial + Q \cdot (1 + \tau).
\]
If $L$ is a strongly invertible knot, then $\tau$ preserves the reduced complex $\Kcrm(D)$ (and all versions of the reduced complex are equivariantly isomorphic). In this case, we define the (isomorphism class of the) \textit{reduced Borel complex of $D$} to be
\[
\Kcrm_Q(D) =  \Kcrm(D) \otimes \f[Q] \quad \text{with} \quad \partial_Q = \partial + Q \cdot (1 + \tau).
\]
As usual, we write $\Kcrm(D)$ in place of $\Kcrm(D)[(0, -1)]$. Using the fact that $u^{-1} \Kcrm(D) \simeq \f[u, u^{-1}]$, it is not difficult to check that $\Kcrm_Q(D)$ is a knotlike Borel complex.
\end{defn}

\subsubsection*{Invariance of the Borel complex}
We now claim that the homotopy type of $\smash{\Kcm_Q(D)}$ is an invariant of $L$. This does not immediately follow from the $\tau$-complex formalism of the previous section. Indeed, in Theorem~\ref{thm:taucomplex} we showed that if $D$ and $D'$ are two transvergent diagrams for the same involutive link, then $(\Kcm(D), \tau)$ and $(\Kcm(D'), \tau)$ are homotopy equivalent pairs. However, \textit{a priori} the maps between these complexes only homotopy commute with $\tau$, so this does not suffice to construct maps between $\Kcm_Q(D)$ and $\Kcm_Q(D')$. We thus check that these maps can be made $\tau$-equivariant on the nose, or at least that they homotopy commute with $\tau$ up to a $\tau$-equivariant homotopy. The following lemma gives the cases which are essentially formal; see also \cite[Remark 2.16]{Sano}: 

\begin{lem}\label{lem:invariant-formal-case}
Let $D$ and $D'$ be two transvergent diagrams that differ by one of the following:
\begin{enumerate}
\item An equivariant planar isotopy.
\item An equivariant Reidemeister move IR1, IR2, IR3, R1, or R2.
\item An I-move.
\item An R-move.
\end{enumerate}
Then there exist maps $f \colon \Kcm(D) \rightarrow \Kcm(D')$ and $g \colon \Kcm(D') \rightarrow \Kcm(D)$ such that:
\begin{enumerate}
\item $f$ and $g$ are both $\tau$-equivariant.
\item $f$ and $g$ are homotopy inverse via a $\tau$-equivariant homotopy. That is, there exist $\tau$-equivariant homotopies $F$ and $G$ such that
\[
g \circ f + \id = \partial F + F \partial \quad \text{and} \quad f \circ g + \id = \partial G + G \partial.
\]
\end{enumerate}
Moreover, for an appropriate choice of $p \in D$ and $p' \in D'$, the maps $f$, $g$, $F$, and $G$ preserve $\smash{\Kcrm_{p}(D)}$ and $\smash{\Kcrm_{p'}(D')}$, as appropriate.
\end{lem}
\begin{proof}

The first, third, and fourth cases are obvious: these follow from observing that the corresponding isomorphisms from Lemma~\ref{lem:tau-complex-invariant} are $\tau$-equivariant on the nose. In these cases, the homotopies $F$ and $G$ are zero. We thus consider the second case. Each of IR1, IR2, and IR3 consists of a pair of usual Reidemeister moves occurring on opposite sides of $\Fix(\tau)$. Let $f$ be the usual map for the Reidemeister move to the left of the axis, followed by the map for the move to the right. Let $f'$ likewise be the right-hand map followed by the left-hand map. Clearly, $f \circ \tau = \tau \circ f'$. On the other hand, two Reidemeister maps which occur in disjoint planar disks commute with each other. This follows from Bar-Natan's extension of Khovanov homology to tangles \cite{bar-natan-tangle}; see for example \cite[Section 2.2.2]{Sano}. Thus $f = f'$ and hence $f$ is $\tau$-equivariant. We define $g$ using the reverse pair of Reidemeister moves; we then have homotopies $F$ and $G$ such that
\[
g \circ f + \id = \partial F + F \partial \quad \text{and} \quad f \circ g + \id = \partial G + G \partial.
\]
These are $\tau$-equivariant for the same reason that $f$ and $g$ are $\tau$-equivariant: using Bar-Natan's extension of Khovanov homology to tangles \cite{bar-natan-tangle}, $F$ and $G$ consist of a symmetric pair of homotopies occurring in disjoint planar disks.
For future reference in Section~\ref{sub:mixed_links}, we present an example for the IR1 move in Figure~\ref{fig:IR1}.
\begin{figure}
  \includegraphics[width=12cm]{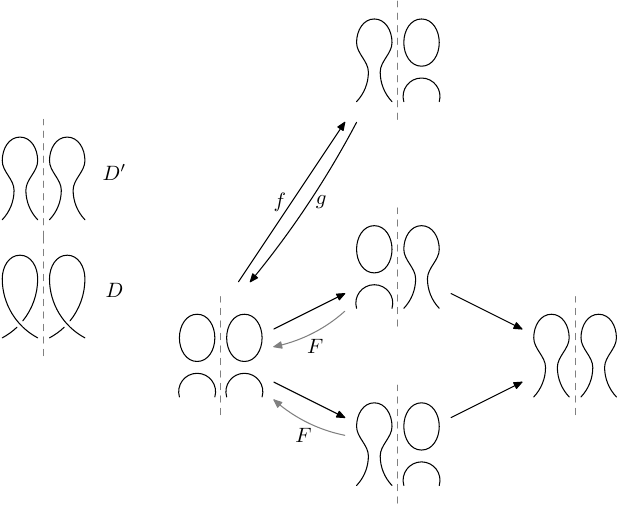}
  \caption{The Bar-Natan maps for the IR1-move as the tensor products of the maps for a classical Reidemeister moves.
    The maps are tensor products of the relevant Bar-Natan maps. The map $F$ is the homotopy on one component and
    the identity on the other. Compare Figures~\ref{fig:new_R1_f} and~\ref{fig:new_R1_gh}.}\label{fig:IR1}
\end{figure}


\begin{figure}
  \includegraphics[width=8cm]{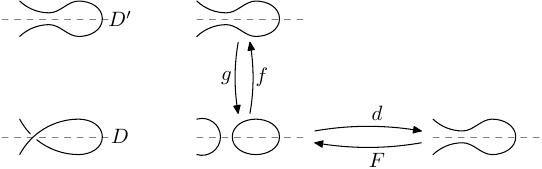}
  \caption{The Bar-Natan maps for the R1 Reidemeister move. The symmetry axis is horizontal.}\label{fig:bn_R1}.
\end{figure}
Now consider an equivariant R1 Reidemeister move. In this case, it turns out the Bar-Natan homotopy equivalences used in the proof of the ordinary R1 Reidemeister move is in fact equivariant, and moreover the respective homotopies used are equivariant. To see this, we recall the maps
in \cite[Section 4.3]{bar-natan-tangle}. These can be seen in Figure~\ref{fig:bn_R1}. The map $d$ is the Bar-Natan differential. The maps $f$,
$g$ and $G$ are drawn in Figures~\ref{fig:new_R1_f} and~\ref{fig:new_R1_gh} with the convention that the source is on the top, and the target
is on the bottom. 
It was proved in \cite{bar-natan-tangle} that $fg + \id =[\partial,G]$ and $gf=\id$. The construction of cobordisms indicates that all the maps $f,g,G$
commute with the $\tau$-action.

\begin{figure}
  \includegraphics[trim = 0 50 0 0,clip,width=6cm]{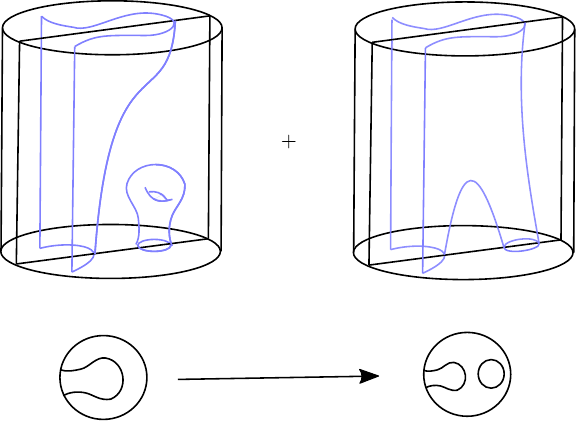}
  \caption{The cobordism map $f$ of the diagram in Figure~\ref{fig:bn_R1}, which is the sum of two cobordisms. The symmetry plane is drawn to indicate that the cobordism map is indeed
  $\tau$-invariant.}\label{fig:new_R1_f}
\end{figure}
\begin{figure}
  \begin{tikzpicture}
    \node at (-1.8,0) {\includegraphics[trim = 0 50 0 0, clip, width =5cm]{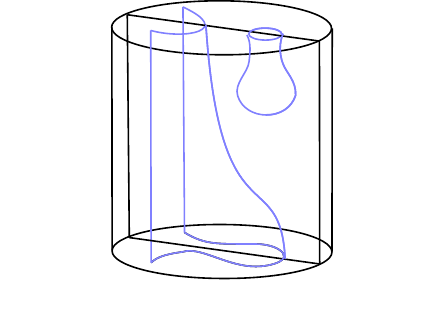}};
    \node at (2,0.2) {\includegraphics[trim = 0 0 0 0, clip, width =2.5cm]{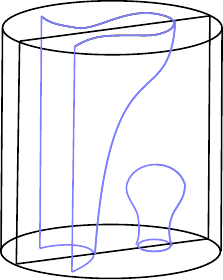}};
  \end{tikzpicture}
  \caption{The maps $g$ (left) and $G$ (right) of the diagram in Figure~\ref{fig:bn_R1}.}\label{fig:new_R1_gh}
\end{figure}






\begin{figure}
  \includegraphics[width=12cm]{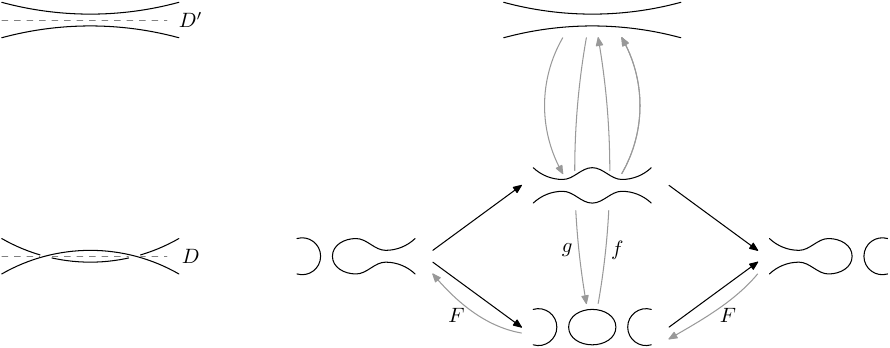}
  \caption{The Bar-Natan maps for the R2 Reidemeister move. The symmetry axis is horizontal, indicated
  on the base diagrams for $D$ and $D'$.}\label{fig:bn_R2}.
\end{figure}
Finally, consider an R2 Reidemeister move. The diagram in Figure~\ref{fig:bn_R2} is taken from \cite[Section 4.3]{bar-natan-tangle}.
Once again, in this case we let $f$ and $g$ be the usual Reidemeister move maps: these are the sums of the identity maps and
the maps drawn in Figure~\ref{fig:R2_1}. It is straightforward to check that $f$ and $g$ are $\tau$-equivariant. 
\begin{figure}[h!]
\includegraphics[trim = 0 50 0 0,clip, scale = 0.59]{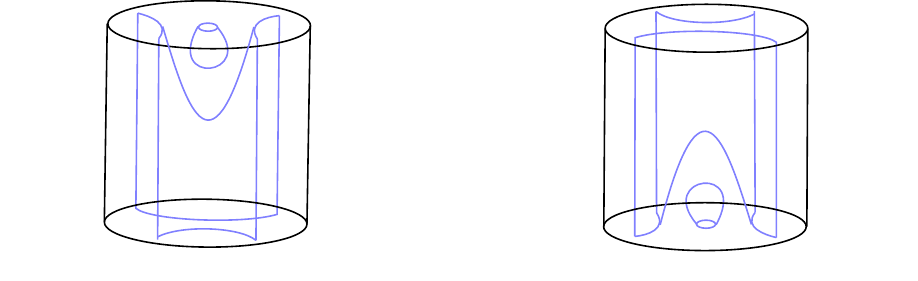}
\caption{Non-trivial part of the homotopy equivalences $g$ (left) and $f$ (right) for R2.}\label{fig:R2_1}
\end{figure}

It follows from the sphere relation that $fg$ is trivial. Hence $G$ is zero. In Figure~\ref{fig:R2_2}, we display the homotopy $F$ between $gf$ and identity. It was shown in \cite{bar-natan-tangle} that $gf + \id =[\partial,F]$ as required. Also, $F$ is immediately seen to be $\tau$-equivariant.
\begin{figure}[h!]
\includegraphics[trim = 0 50 0 0,clip, scale = 0.7]{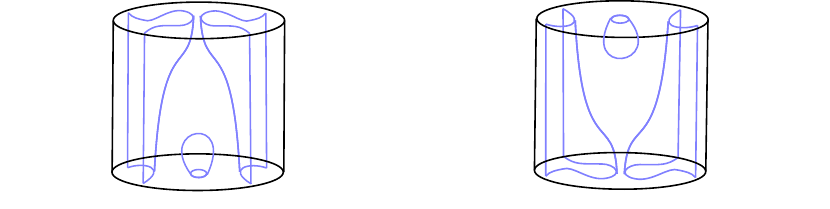}
\caption{Components of the only non-trivial homotopy $F$ for R2.}\label{fig:R2_2}
\end{figure}

The last part of the lemma follows by choosing basepoints as in Lemma~\ref{lem:reduced-invariant}
\end{proof}

We thus immediately obtain the following:

\begin{lem}\label{lem:invariant-formal-case-ii}
Let $D$ and $D'$ be two transvergent diagrams that differ by one of the moves in Lemma~\ref{lem:invariant-formal-case}. Then we have homotopy equivalences $\smash{\Kcm_Q(D)} \simeq \smash{\Kcm_Q(D')}$ and $\smash{\Kcrm_Q(D)} \simeq \smash{\Kcrm_Q(D')}$.
\end{lem}
\begin{proof}
Let $f$ and $g$ be the maps constructed in Lemma~\ref{lem:invariant-formal-case}. Since these commute with $\tau$, we may promote them to chain maps $f_Q = f \otimes \id$ and $g_Q = g \otimes \id$ between $\smash{\Kcm_Q(D)}$ and $\smash{\Kcm_Q(D')}$. Using the fact that $F$ and $G$ also commute with $\tau$, it is easily checked that $f_Q$ and $g_Q$ are homotopy inverse to each other via the homotopies $F_Q = F \otimes \id$ and $G_Q = G \otimes \id$.
\end{proof}

It remains to check invariance under the equivariant Reidemeister moves M1, M2, and M3. These are where the real distinction between the Borel construction and the mapping cone formalism occurs. Roughly speaking, we will follow the construction of Lobb-Watson \cite[Section 5]{lobb-watson} and construct a map $f_0 \colon \Kcm(D) \rightarrow \Kcm(D')$ which is not necessarily $\tau$-equivariant, but satisfies $\tau f_0 + f_0 \tau = \partial f_1 + f_1 \partial$ for a $\tau$-equivariant $f_1$. As discussed in Section~\ref{sec:borel-setup}, this gives a map $\smash{f_Q = f_0 + Q f_1 \colon \Kcm_Q(D) \rightarrow \Kcm_Q(D')}$, which we will show is a homotopy equivalence of Borel complexes. Moreover, $f_0$ will be chain homotopic to a sequence of (non-equivariant) Reidemeister move maps. Since the construction of $f_0$ and $f_1$ is slightly involved, we delay discussion of this until the next subsection.

\subsubsection*{Cobordism maps for the Borel complex}
We now construct cobordism maps for the Borel complex. Let $\Sigma$ be an equivariant cobordism from $L_1$ to $L_2$ and assume that $\Sigma$ is equivariantly generic. Then $\Sigma$ gives a sequence of equivariant Reidemeister moves and equivariant cobordism moves from $D_1$ to $D_2$. We have already argued that each equivariant Reidemeister move induces a homotopy equivalence between successive diagrams $\smash{\Kcm_Q(D)}$ and $\smash{\Kcm_Q(D')}$. We thus consider an equivariant cobordism move. These come in two kinds: equivariant pairs of births, deaths, and saddles, and on-axis births, deaths, and saddles. 

For an equivariant pair, consider the map $f$ from $\Kcm(D)$ to $\Kcm(D')$ given by the composition of the two regular cobordism maps occurring on either side of the axis. As in the proof of Lemma~\ref{lem:invariant-formal-case}, this composition may be taken in either order and is consequently $\tau$-equivariant. Thus $f$ lifts to a map $f \otimes \id$ from $\smash{\Kcm_Q(D)}$ to $\smash{\Kcm_Q(D')}$. For the on-axis cobordism moves, it is evident that the usual cobordism map $f$ is itself $\tau$-equivariant and thus lifts to a map from $\smash{\Kcm_Q(D)}$ to $\smash{\Kcm_Q(D')}$. 


\begin{defn}
Let $\Sigma$ be a cobordism in $S^3 \times I$ from $L_1$ to $L_2$. Assume $\Sigma$ is equivariantly generic, so that $\Sigma$ gives rise to sequence $\{S_i\}_{i = 1}^k$ of equivariant Reidemeister moves and equivariant elementary cobordism moves from $D_1$ to $D_2$. Define the associated cobordism map to be the composition
\[
\Kcm_Q(\Sigma) = \Kcm_Q(S_k) \circ \cdots \circ \Kcm_Q(S_1) \colon \Kcm_Q(D_1) \rightarrow \Kcm_Q(D_2)
\]
where the $\Kcm_Q(S_i)$ are the maps associated to each elementary move. 
\end{defn}

We now define a Borel map for the reduced complex. This is done in exactly the same way as in Definition~\ref{def:reduced-cobordism-map}. First observe that we have a Wigderson splitting on the level of Borel complexes:

\begin{lem}
Let $K$ be a strongly invertible knot with transvergent diagram $D$ and let $p$ be an equivariant basepoint on $D$. The Wigderson maps $i_p$ and $\pi_p$ extend to inclusion and projection maps
\[
i_p \otimes \id \colon \Kcrm_Q(D) \rightarrow \Kcm_Q(D) \quad \text{and} \quad \pi_p \otimes \id \colon \Kcm_Q(D) \rightarrow \Kcrm_Q(D).
\]
\end{lem}
\begin{proof}
This follows immediately from the fact that $i_p$ and $\pi_p$ are $\tau$-equivariant.
\end{proof}

We may thus define:

\begin{defn}
Let $\Sigma$ be an equivariantly generic cobordism in $S^3 \times I$ from $L_1$ to $L_2$. For any pair of basepoints $p_1$ on $D_1$ and $p_2$ on $D_2$, define
\[
\Kcrm_Q(\Sigma) \colon \Kcrm_{p_1}(D_1) \rightarrow \Kcrm_{p_2}(D_2)
\]
to be the composition
\[
\Kcrm_Q(D_1) \xrightarrow{i_{p_1} \otimes \id} \Kcm(D_1) \xrightarrow{\Kcm_Q(\Sigma)} \Kcm(D_2) \xrightarrow{\pi_{p_2} \otimes \id} \Kcrm_Q(D_2)
\]
where the domain and codomain are $\Kcrm_{p_1}(D_1) \otimes \f[Q]$ and $\Kcrm_{p_2}(D_2) \otimes \f[Q]$, respectively.
\end{defn}

Modulo our discussion of the M1, M2, and M3 moves, this completes the proof of the main theorem:

\begin{proof}[Proof of Theorem~\ref{thm:full_borel_invariant}]
Invariance follows from Lemma~\ref{lem:invariant-formal-case-ii}, together with our discussion of the M1, M2, and M3 moves in the next subsection. It remains to check that if $K_1$ and $K_2$ are strongly invertible knots and $\Sigma$ is connected, then $\smash{\Kcrm_Q(\Sigma)}$ is local. To see this, reduce modulo $Q$. We obtain the map 
\[
\smash{\pi_{p_2} \circ (\Kcm_Q(\Sigma) \bmod Q) \circ i_{p_1}} \colon \Kcrm_{p_1}(D_1) \rightarrow \Kcrm_{p_2}(D_2). 
\]
Now, each elementary cobordism map used to define $\smash{\Kcm_Q(\Sigma)}$ has a mod $Q$ reduction which is either just a usual (non-equivariant) Reidemeister or cobordism move map, or is chain homotopic to a composition of such maps. Thus $\Kcm_Q(\Sigma) \bmod Q$ is chain homotopic to $\Kcm(\Sigma)$ as given in Definition~\ref{def:unreduced-cobordism}. Hence $\smash{\Kcrm_Q(\Sigma)} \bmod Q$ is chain homotopic to $\smash{\Kcrm(\Sigma)}$ as given in Definition~\ref{def:reduced-cobordism-map}. By Theorem~\ref{thm:knotlike-locality}, we know that $\smash{\Kcrm(\Sigma)}$ induces an isomorphism on homology after inverting $u$. According to Lemma~\ref{lem:quasi-iso-move}, the same is then true for $\smash{\Kcrm_Q(\Sigma)}$, as desired.
\end{proof}

\subsection{Invariance under M1, M2, and M3 moves}\label{sec:miinvariance}

We now complete the proof of Theorem~\ref{thm:full_borel_invariant} by establishing invariance of $\Kcm_Q(D)$ and $\Kcrm_Q(D)$ under the M1, M2, and M3 moves. The first two of these closely follow the argument given in \cite[Section 5.3]{lobb-watson}.


\subsubsection*{The M1 move} Now let $D$ and $D'$ be two transvergent diagrams which differ by an M1 move, where $D$ is as on the left of the M1 in Figure \ref{fig:all_equi_moves_lobb_watson} and $D'$ is on the right. We begin by constructing maps $f_0$ and $f_1$ from $\Kcm(D)$ to $\Kcm(D')$ such that:
\begin{enumerate}
\item $f_0$ is a chain map; 
\item $\tau f_0 + f_0 \tau = \partial f_1 + f_1 \partial$; and,
\item $f_1$ is $\tau$-equivariant.
\end{enumerate}
Non-equivariantly, the M1 move is actually just a Reidemeister R3 move which moves the overstrand of the diagram to the lower right. We thus let $f_0$ be the associated Reidemeister R3 map on the Bar-Natan complex. This is not $\tau$-equivariant, but by construction it is a chain map (it may be regarded as a variant of \cite[Figure 9]{bar-natan-tangle} for different crossing signs).  Moreover, it homotopy commutes with $\tau$ up to a homotopy $f_1$. The maps $f_0$ and $f_1$ are displayed in Figure~\ref{fig:M1_borel_1}. The analogous maps $g_0$ and $g_1$ in the opposite direction are depicted in Figure~\ref{fig:M1_borel_2}. Here the cobordisms used to define the $g_0$ maps are similar to those used for $f_0$.

\begin{figure}[h!]
\includegraphics[scale = 1.0]{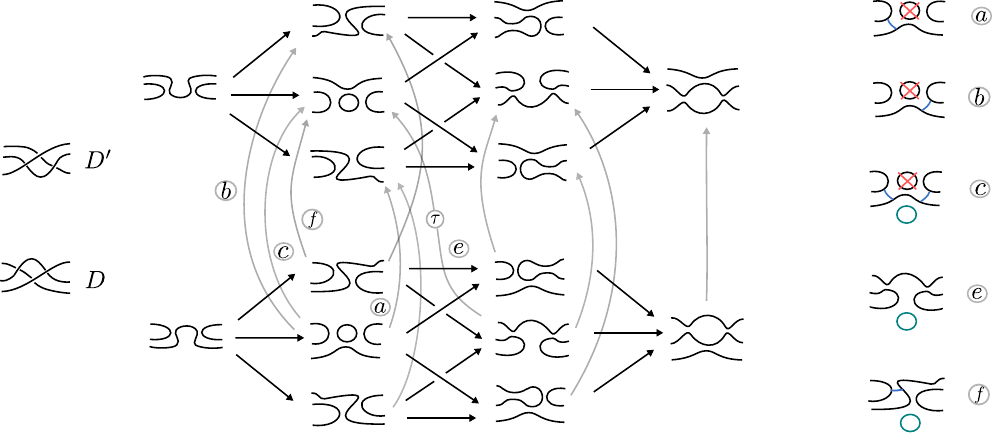}
\caption{Depiction of the maps $f_0$ and $f_1$. In this and subsequent figures, an arrow decorated with $\tau$ means that the corresponding cobordism map is precomposed with $\tau$. The arrow labeled with $\tau$ is the only non-zero component of $f_1$; the rest of the arrows define $f_0$. The alphabetical labels beside the arrows represent the corresponding cobordisms depicted on the right hand side. An arrow without any label is the identity cobordism. The cobordism labeled $(c)$ consist of a first a death, then two saddles, followed by a birth. Note our convention for birth is a green circle, and death is a circle with a red cross.}\label{fig:M1_borel_1}
\end{figure}

\begin{figure}[h!]
\includegraphics[scale = 1.0]{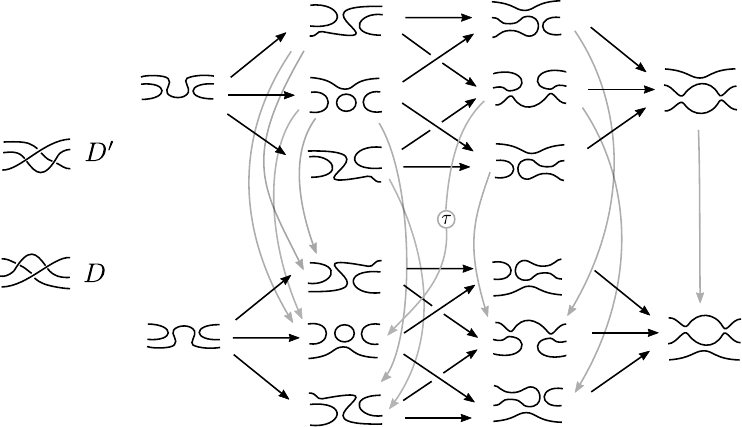}
\caption{The map $g_0$ and $g_1$. As before, the map $g_1$ is defined by only one arrow that is decorated by $\tau$.}\label{fig:M1_borel_2}
\end{figure}

As discussed in Section~\ref{sec:borel-setup}, this gives chain maps $f_Q = f_0 + Qf_1$ and $g_Q = g_0 + Qg_1$ between the Borel complexes $\smash{\Kcm_Q(D)}$ and $\smash{\Kcm_Q(D')}$. We claim that these are homotopy inverse to each other. For this, we produce a map $F_0$ from $\Kcm(D)$ to itself such that: 

\begin{figure}[h!]
\includegraphics[scale = 1.2]{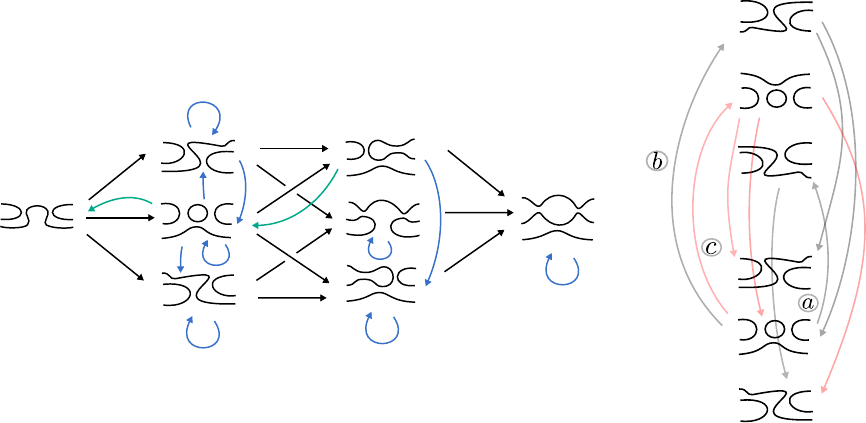}
\caption{On the left, the map $F_0$ in green and $g_0 f_0$ in blue. On the right, a sample computation of $g_0 f_0$, maps in red compose to zero.}\label{fig:M1_borel_3}
\end{figure}
\begin{enumerate}
\item $g_0 f_0 + \id = \partial F_0 + F_0 \partial$;
\item $g_1 f_0 +  g_0 f_1 = \tau F_0 + F_0 \tau $; and,
\item $g_1 f_1 = 0$.
\end{enumerate}
Given this, it is not difficult to verify the identity
\[
(\partial + Q(1 + \tau))F_0 + F_0(\partial + Q(1 + \tau)) = (g_0 + Q g_1)(f_0 + Qf_1) + \id,
\]
demonstrating that $g_Q f_Q$ is homotopic to the identity as a map from $\smash{\Kcm_Q(D)}$ to itself via the homotopy $F_Q = F_0$.

We first begin by defining $F_0$ in Figure~\ref{fig:M1_borel_3}. Now we compute the composition $g_0 f_0$. On the right of Figure~\ref{fig:M1_borel_3}, we show one such sample computation. Indeed the composition of the red arrows is zero since all of those composed cobordisms have a sphere (due to birth followed by a death in the same circle). On the other hand, the rest of the composed cobordisms (coming from the grey arrows) are all non-trivial. The entire map $g_0 f_0$ is shown in Figure~\ref{fig:M1_borel_3}. It immediately follows from Figure~\ref{fig:M1_borel_3} that  $g_0 f_0 + \id = \partial F_0 + F_0 \partial$.

We now note that $g_1 f_1$ is zero, by a direct computation.  The map $g_1 f_0 + g_0 f_1$ is shown in Figure~\ref{fig:M1_borel_4}. Note that again in this computation, due to the sphere relation, $g_0 f_1$ is zero. It follows from Figure~\ref{fig:M1_borel_3} that $g_1 f_0 + g_0 f_1= \tau F_0 + F_0 \tau$, by a further computation. This verifies all the three described conditions mentioned previously, showing $g_Qf_Q$ is homotopic to identity as a map from $\smash{\Kcm_Q(D)}$ to itself via the homotopy 
$F_Q = F_0 $.

Likewise, we claim that there exists a map $G_0$ from $\Kcm(D')$ to itself which satisfy the above conditions with $f$ and $g$ interchanged, showing that $f_Q g_Q$ is homotopic to the identity via $G_Q = G_0 $. 
  To describe $G_Q$ conveniently, we note that $f$ and $g$ are related as follows --- there is a symmetry taking the tangle $D$ to $D'$ by rotating by $\pi$ about the center of the disk containing $D,D'$.  This symmetry induces a bijection $\beta$ between the resolutions pictured in the cubes in Figure \ref{fig:M1_borel_1}; the map $g$ is $\beta\circ f\circ \beta$, where this term requires some interpretation to parse, which we leave to the reader.  The map $G_0$ is similarly given by $\beta \circ F_0\circ \beta$.  

\begin{figure}[h!]
\includegraphics[scale = 1.2]{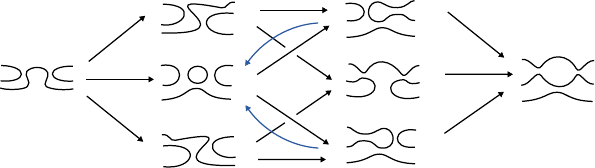}
\caption{The composition $g_1f_0 + g_0f_1$ in blue, which is same as $\tau F_0 + F_0 \tau$.}\label{fig:M1_borel_4}
\end{figure}

If the equivariant basepoint $p$ of the pointed link $(D,p)$ is outside the disk in which the M1 move takes place, then all maps in question preserve $\Kcrm_p$.  Recall all versions of the reduced complex $(\Kcr^-(D),\tau)$ are isomorphic, namely, the pointed complex $\Kcrm_{p}(D)$ via any equivariant basepoint $p$, or defined by the unpointed construction $\Kcrm_{un}(D)$.  Thus, to show $\Kcrm_Q(D)\simeq \Kcrm_Q(D')$, we can use the reduced complexes on both sides that are defined with a basepoint \emph{away} from the M1 region, so that the chain-homotopy type of $\Kcrm_Q$ is indeed unchanged by M1 moves; see for example Remark~\ref{rem:basepointjumping2}. 

\subsubsection*{The M2 move} Now let $D$ and $D'$ be two transvergent diagrams which differ by an M2 move, where $D$ is as on the left of the M2 in Figure \ref{fig:all_equi_moves_lobb_watson} and $D'$ is on the right. We once again construct maps $f_0$ and $f_1$ from $\Kcm(D)$ to $\Kcm(D')$ such that:
\begin{enumerate}
\item $f_0$ is a chain map;
\item $\tau f_0 + f_0 \tau = \partial f_1 + f_1 \partial$; and,
\item $f_1$ is $\tau$-equivariant.
\end{enumerate}
Non-equivariantly, the M2 move is just a Reidemeister R2 move which cancels the left and middle crossings. We thus let $f_0$ be the associated Reidemeister R2 map on the Bar-Natan complex. The maps $f_0$ and $f_1$ are displayed in Figure \ref{fig:m2-f0-f1}. We take the convention that the crossings in $D'$ are ordered from left to right, so, that the $101$-resolution $D'_{101}$ is the diagram with two closed circles. (This is the middle diagram in the third row of Figure~\ref{fig:m2-f0-f1}.) All but one of the cobordisms in the diagram is determined by quantum degree.  The cobordism from $D_1$ to $D'_{101}$ is not determined by quantum degree -  this cobordism is given by birthing the left circle, and splitting the other circle from the \emph{left}-hand side of the tangle (in particular, the right-hand circle in $D'_{101}$ comes from the saddle, not a birth).  

\begin{figure}[h!]
	\includegraphics[scale=1.4]{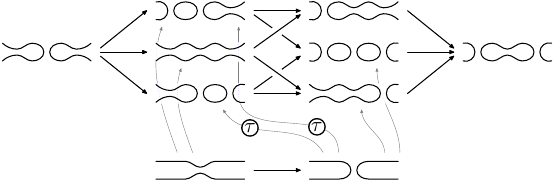}
	\caption{The morphisms from $D$ (bottom) to $D'$ (top).  Here, the part preserving homological degree is $f_0$, while the morphism $f_1$ is the part  decreasing homological grading by $1$.}\label{fig:m2-f0-f1}
\end{figure}

There is a map $g_0$ in the opposite direction, which is the ordinary R2 move map, together with a map $g_1$, both pictured in Figure \ref{fig:m2-g0-g1}.

\begin{figure}[h!]
	\includegraphics[scale=1.6]{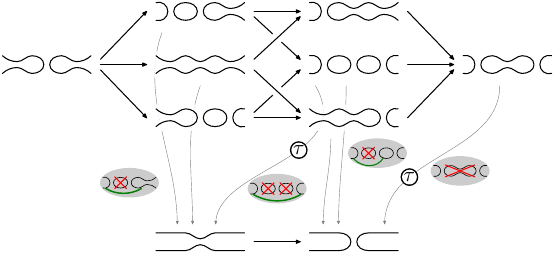}
	\caption{The morphisms from $D'$ (top) to $D$ (bottom).  Here, the part preserving homological degree is $g_0$, while the morphism $g_1$ is the part decreasing homological grading by $1$.  In this figure, we have indicated several of the cobordisms maps, in the shaded gray regions.}\label{fig:m2-g0-g1}
\end{figure}

The maps $g_0,g_1$ satisfy the analogs of the conditions for $f_0,f_1$, namely:
\begin{enumerate}
	\item $g_0$ is a chain map;
	\item $\tau g_0 + g_0 \tau = \partial g_1 + g_1 \partial$; and,
	\item $g_1$ is $\tau$-equivariant.
\end{enumerate}

As in the previous case, to establish that $f_Q = f_0 + Qf_1$ and $g_Q = g_0 + Qg_1$ are homotopy inverse to each other, we need to perform several additional checks. In one direction, this turns out to be trivial. Indeed, in \cite[Section 4.3]{bar-natan-tangle} it is shown that the usual Bar-Natan homotopy for $f_0$ followed by $g_0$ is zero; that is, $g_0f_0 = \id$. Further inspection shows that $f_0g_1+f_1g_0=0$ and $g_1f_1=0$. Hence $g_Qf_Q = \id$. 

For the other direction, we argue as follows. Let $G_0$ be the Bar-Natan homotopy associated to $g_0$ followed by $f_0$, so that $f_0g_0 + \id = [\partial, G_0]$. Explicitly, $G_0$ is given by four components, taking the $D'_{11i}$ resolution to $D'_{10i}$ and $D'_{10i}$ to $D'_{00i}$ for $i=0,1$, as described in Figure \ref{fig:m2-homotopy-1}.  

\begin{figure}
	\centering
	\includegraphics[width=14cm]{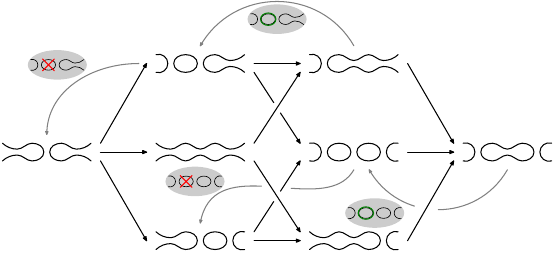}
	\caption{The morphism $G_0$.}\label{fig:m2-homotopy-1}
\end{figure}

Viewing $G_0$ as a chain map on Borel complexes by extending $Q$-linearly, and writing $\partial_Q$ for the Borel differential, we then have $f_Qg_Q+\id = [\partial_Q,G_0] \bmod{Q}$. Hence we may define $E \colon \Kcm(D')\to \Kcm(D')$ to be the map $E$ such that
\[
QE = f_Qg_Q+ \id + [\partial_Q, G_0].
\]
Note $E$ commutes with $\partial_Q$, since the right-hand side of the above expression commutes with $\partial_Q$. Set
\begin{equation}\label{eq:m2-g-homotopy}
G_Q = (1 + QE)^{-1} G_0( f_Qg_Q + \id).
\end{equation}
Here, by $(1 + QE)^{-1}$, we mean the usual infinite series in $QE$. However, since $E$ necessarily decreases homological grading by one and $\Kcm(D')$ is finitely-generated, this expression is well-defined when evaluated on any element of $\Kcm(D')$. We claim that $G_Q$ is the desired homotopy. For this, we compute
\begin{align*}
[\partial_Q, G_Q] &= (1 + QE)^{-1}[\partial_Q, G_0]( f_Qg_Q + \id) \\
&= (1 + QE)^{-1}(f_Q g_Q + \id + QE)(f_Q g_Q + \id) \\
&= (1 + QE)^{-1}((f_Q g_Q + \id)(f_Q g_Q + \id) + QE(f_Q g_Q + \id)) \\
&= (1 + QE)^{-1}((f_Q g_Q f_Q g_Q + \id) + QE(f_Q g_Q + \id)) \\
&= (1 + QE)^{-1}((f_Q g_Q + \id) + QE(f_Q g_Q + \id)) \\
&= f_Qg_Q + \id,
\end{align*}
as desired. Here, in the first line we have used the fact that $E$ commutes with $\partial_Q$, while in the fifth line we have used the fact that $g_Q f_Q = \id$.

Note that if an equivariant basepoint $p$ is chosen outside the disk in which the M2 move takes place, then all maps in question preserve $\Kcrm_p$.

\subsubsection*{The M3 move}
For invariance under M3, we diverge from the treatment in Lobb-Watson and give a more topological argument. (The M3 move in \cite[Section 5.3]{lobb-watson} is significantly more complicated than the previous two cases.) Let $D$ and $D'$ differ by an M3 move. The trick will be to consider the (partial) diagram $D_0$ equipped with the two arcs $\alpha$ and $\beta$ displayed in Figure~\ref{fig:m3cones}. Let $D_1$ and $D_1'$ be the diagrams obtained from $D_0$ by surgery along $\alpha$ and $\beta$, respectively, and let $D_2$ be the diagram obtained by surgery along both. 

\begin{figure}[h!]
\includegraphics[scale = 0.9]{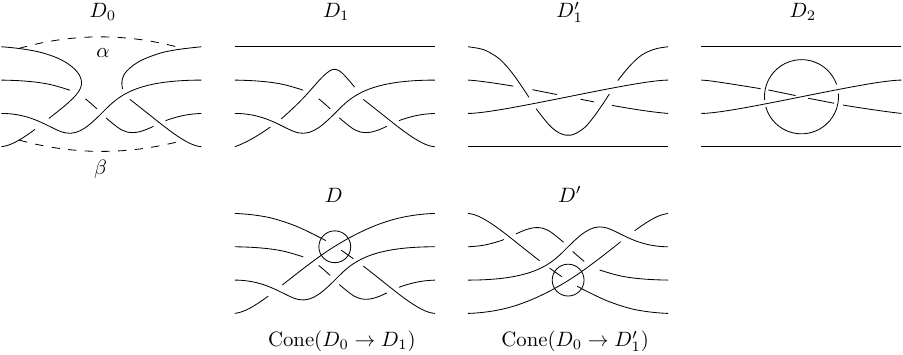}
\caption{Top row: the diagram $D_0$ with two dotted arcs $\alpha$ and $\beta$. Up to planar isotopy, the diagrams $D_1$ and $D_1'$ are obtained from $D_0$ by surgery along $\alpha$ and $\beta$, respectively. Surgery along both $\alpha$ and $\beta$ gives $D_2$. Bottom row: the diagrams $D$ and $D'$ before and after the M3 move. Each diagram has a circled crossing. The 0-resolution in each case is $D_0$, while the 1-resolutions are $D_1$ and $D_1'$, respectively.}\label{fig:m3cones}
\end{figure}

Let $S_\alpha$ and $S_\beta$ be the saddle cobordisms corresponding to surgery along $\alpha$ and $\beta$, and denote their corresponding Bar-Natan cobordism maps by $s_\alpha$ and $s_\beta$. These maps are clearly $\tau$-equivariant and hence may be promoted to maps (which we also denote by $s_\alpha$ and $s_\beta$) of Borel complexes. Note that since $S_\alpha$ and $S_\beta$ occur in disjoint disks, we have the commutative diagram:
\[
\begin{tikzcd}
	\Kcm_Q(D_0) & \Kcm_Q(D_1) \\
	\Kcm_Q(D_1') & \Kcm_Q(D_2)
	\arrow["s_\alpha", from=1-1, to=1-2]
	\arrow["s_\beta"', from=1-1, to=2-1]
	\arrow["s_\beta", from=1-2, to=2-2]
	\arrow["s_\alpha"', from=2-1, to=2-2]
\end{tikzcd}
\]
Moreover, by resolving the crossings indicated in Figure~\ref{fig:m3cones}, we see that we may identify
\[
\Kcm_Q(D) = \mathrm{Cone}(\Kcm_Q(D_0) \xrightarrow{s_\alpha} \Kcm_Q(D_1)) \quad \text{and} \quad \Kcm_Q(D') = \mathrm{Cone}(\Kcm_Q(D_0) \xrightarrow{s_\beta} \Kcm_Q(D_1')).
\]

Now consider the auxiliary sequence of cobordism moves displayed in Figure~\ref{fig:m3aux}. This starts at $D_1$ and consists of the surgery $S_\beta$, followed by an M1 move, an IR2 move, and a death cobordism. For the latter three cobordisms, let the induced maps on the Borel complex be denoted by $F_{\mathrm{M1}}$, $F_{\mathrm{IR2}}$ and $F_{\mathrm{death}}$. Write
\[
F_{\mathrm{aux}} = F_{\mathrm{death}} \circ F_{\mathrm{IR2}} \circ F_{\mathrm{M1}}. 
\]
Observe that the surgery along $\beta$ occurs in a disk which is disjoint from the M1 and IR2 moves. Hence $s_\beta$ commutes with $F_{\mathrm{M1}}$ and $F_{\mathrm{IR2}}$, so that
\[
F_{\mathrm{aux}} \circ s_\beta = F_{\mathrm{death}} \circ s_\beta \circ F_{\mathrm{IR2}} \circ F_{\mathrm{M1}}.
\]
On the right-hand side, $F_{\mathrm{death}} \circ s_\beta$ splits off a circle and then immediately deletes it. This induces the identity map on the Borel complex. Thus $F_{\mathrm{aux}} \circ s_\beta = F_{\mathrm{IR2}} \circ F_{\mathrm{M1}}$; by Section~\ref{sec:borel-transvergent}, this is a homotopy equivalence of Borel complexes.

\begin{figure}[h!]
\includegraphics[scale = 0.85]{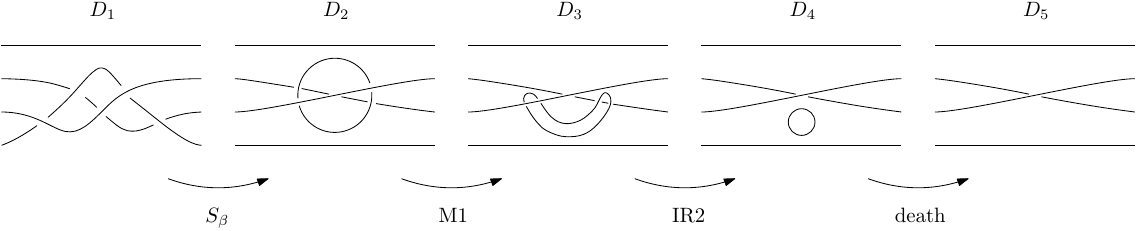}
\caption{An auxiliary sequence of cobordisms moves from $D_1$ to $D_5$.}\label{fig:m3aux}
\end{figure}

Since $F_{\mathrm{aux}} \circ s_\beta$ is a homotopy equivalence, we obtain a homotopy equivalence of mapping cones
\[
\mathrm{Cone}(\Kcm_Q(D_0) \xrightarrow{s_\alpha} \Kcm_Q(D_1)) \simeq \mathrm{Cone}(\Kcm_Q(D_0) \xrightarrow{(F_{\mathrm{aux}} \circ s_\beta) \circ s_\alpha} \Kcm_Q(D_5)),
\]
where $D_5$ is as in Figure~\ref{fig:m3aux}. On the other hand, we may similarly consider the sequence of cobordism moves which starts from $D_1'$ and consists of $S_\alpha$, followed by the same M1 move, IR2 move, and death cobordism. The same argument as before shows that 
\[
\mathrm{Cone}(\Kcm_Q(D_0) \xrightarrow{s_\beta} \Kcm_Q(D_1')) \simeq \mathrm{Cone}(\Kcm_Q(D_0) \xrightarrow{(F_{\mathrm{aux}} \circ s_\alpha) \circ s_\beta} \Kcm_Q(D_5)).
\]
But $s_\beta \circ s_\alpha = s_\alpha \circ s_\beta$, so the two right-hand mapping cones are identical. We conclude $\Kcm_Q(D) \simeq \Kcm_Q(D')$, as desired.

We moreover claim that our homotopy equivalence reduces modulo $Q$ to the ordinary Bar-Natan Reidemeister map from $\Kc^-(D)$ to $\Kc^-(D')$ by an isotopy from $D$ to $D'$ (for instance, this isotopy can be realized as the composite of four Reidemeister 3 moves).

Let $R,R'$ be the tangles for $D,D'$, respectively, in the disk region $\mathcal{R}$ where the M3 takes place, similarly write $R_0,\ldots,R_5$ and $R_i'$ for the active tangles in Figures \ref{fig:m3cones}, \ref{fig:m3aux}.  The tangles $R$ and $R'$ are both \emph{simple} in the sense of \cite{bar-natan-tangle}, as are the tangles $R_1,R_5$ and $R_1'$.  Recall that Bar-Natan \cite{bar-natan-tangle} defines, for a disk $B$ with some finite collection of marked points a category $\Kob(B)$ whose objects are ``complexes of tangles" in $B$ with those boundary points, and for which there is a functor $\Kc^-\colon \Kob(\emptyset)\to \f[u]-\mathrm{gMod}$ so that $\Kc^-$ applied to the cube of resolutions for a link $L$ gives $\Kc^-(L)$.  A tangle is \emph{simple} if it has a unique automorphism in $\Kob$.  Since $R,R'$ (etc.) are simple, there is a unique, up-to-homotopy, homotopy equivalence $\varphi\colon \Kc^-(D)\to \Kc^-(D')$ that is homotopic to a map of the form $\Kc^-(\phi\Id_{D\cap(\mathbb{R}^2-\mathcal{R})})$ where $\phi$ is a homotopy-equivalence in $\Kob$ from $D$ to $D'$.  In what follows, we will show that our morphism $\Kc^-(D)\to \Kc^-(D')$ constructed above, call it $F$, is indeed $\Kc^-$ applied to a morphism in Bar-Natan's tangle category, from $D$ to $D'$ - since there is only one such morphism $\Kc(D)\to \Kc(D')$, we have that the homotopy-equivalence constructed by a sequence of Reidemeister move maps from $D$ to $D'$ agrees with $F$.  

First, we will show that the morphisms $\Kc^-(D_1)\to \Kc^-(D_5)$ and $\Kc^-(D_1')\to \Kc^-(D_5)$ come from morphisms in $\Kob(\mathcal{R})$.  We then claim that the homotopy equivalence that these morphisms induce,
\begin{align}
\begin{split}
\mathrm{Cone}(\Kc^-(D_0)\rightarrow \Kc^-(D_1))&\to \mathrm{Cone}(\Kc^-(D_0)\rightarrow \Kc^-(D_5)) \label{eq:kob-check}\\ \mathrm{Cone}(\Kc^-(D_0)\rightarrow \Kc^-(D_1'))& \to \mathrm{Cone}(\Kc^-(D_0)\rightarrow \Kc^-(D_5))
\end{split}
\end{align}
respectively, are induced by morphisms in $\Kob$.  

Indeed, for the claim that $\Kc^-(D_1)\to \Kc^-(D_5)$ arises from a morphism in $\Kob(\mathcal{R})$ we observe that each of the $S_{\beta}$, M1, IR2, and death maps in Figure \ref{fig:m3aux} are $\Kc^-$ applied to a $\Kob$ morphism.  It is by definition that $S_\beta$ and the death map arise from $\Kob$.  We have already seen that M1 is an R3 map nonequivariantly, and IR2 is nonequivariantly a composite of two R2 maps, so these morphisms also arise from $\Kob(\mathcal{R})$.  Now, the morphisms between mapping cones in (\ref{eq:kob-check}) are $\Kc^-$ applied to the corresponding cones in $\Kob(\mathcal{R})$ (tensored with the tangle $D_0\cap (\mathbb{R}^2-\mathcal{R})$ in $\mathcal{R}^2-\mathcal{R}$).  Finally, the identification of $\Kc^-(D)$ with $\mathrm{Cone}(\Kc^-(D_0)\rightarrow \Kc^-(D_1)$ also is at the level of $\Kob$, by construction.  Thus the sequence of equivalences
\begin{align*}
\Kc^-(D)\to \mathrm{Cone}(\Kc^-(D_0)\rightarrow \Kc^-(D_1))\to \ &\mathrm{Cone}(\Kc^-(D_0)\rightarrow \Kc^-(D_5))\to \\
&\mathrm{Cone}(\Kc^-(D_0)\rightarrow \Kc^-(D_1'))\to \Kc^-(D')
\end{align*}
all come from tensoring equivalences in $\Kob(\mathcal{R})$ with $D\cap (\mathbb{R}^2-\mathcal{R})$, and so $F\colon \Kc^-(D)\to \Kc^-(D')$ that we have constructed arises from a $\Kob(\mathcal{R})$ map.  Thus $F$ agrees with the ordinary nonequivariant isotopy map $\Kc^-(D)\to \Kc^-(D')$, as needed.  
%
%

As before, note that if an equivariant basepoint $p$ is chosen outside the disk in which the M3 move takes place, then all maps in question preserve $\Kcrm_p$.

\begin{rmk}
The discussion of the M3 move completes the proof of Theorem~\ref{thm:full_borel_invariant}. In fact, we have established a slightly stronger claim by explicitly promoting the usual Bar-Natan Reidemeister maps to maps of Borel complexes which are homotopy inverse to each other. Strictly speaking, the latter claim is more than is required by Theorem~\ref{thm:full_borel_invariant}. Indeed, a shorter proof of the homotopy invariance of $\smash{\Kcm_Q(D)}$ can be given as follows: in each case, we construct a homotopy equivalence $f_0$ of Bar-Natan complexes. We then promote this to a map of Borel complexes of the form $f_Q = f_0 \otimes \id$ or $f_Q = f_0 + Q f_1$. Since reducing modulo $Q$ gives a quasi-isomorphism of Bar-Natan complexes, Lemma~\ref{lem:quasi-iso-move} implies that $f_Q$ is also a quasi-isomorphism. Lemma~\ref{lem:qitohe} then guarantees that $\Kcm_Q(D)$ and $\Kcm_Q(D')$ are homotopy equivalent. 

As in Remark~\ref{rmk:sanoresults}, this argument does not require us to construct the homotopies $F_Q$ and $G_Q$ used throughout the last two sections. It should be noted that although Lemma~\ref{lem:qitohe} allows us to conclude that  $f_Q$ and $g_Q$ each admit a homotopy inverse, it does not necessarily follow that $f_Q$ and $g_Q$ are homotopy inverse to each other, even though $f_0$ and $g_0$ are homotopy inverses as maps of Bar-Natan complexes. Here we construct the homotopies in question since they will be useful in our discussion of mixed complexes.
\end{rmk}


\section{Applications and Examples} \label{sec:applications}
We now give the principal applications of the Borel formalism. First, we define the refined equivariant $s$-invariants $\wt{s}_{Q, A, B}$ and prove Theorem~\ref{thm:equivariant_s_borel_intro}. We then discuss techniques for partially computing the Borel complex from homological data. We use this to give several example computations and establish our results on the isotopy-equivariant genus.

\subsection{Knotlike complexes and the $\wt{s}_Q$-invariant}\label{sec:strong-s-examples}
Using the localization behavior of a knotlike Borel complex, we can refine the equivariant $s$-invariant $\wt{s}$ introduced in Definition~\ref{def:tilde-s}. Let 
\[
\iota_* \colon H_*(C_Q) \xrightarrow{\cong} \F[u,u^{-1},Q]
\]
be the map induced by including $H_*(C_Q)$ into its localization. This is well-defined since the only (grading-preserving) endomorphism of $\F[u,u^{-1},Q]$ as an $\F[u,u^{-1},Q]$-module is the identity.

\begin{defn}\label{def:strong-s}
Let $C_Q$ be a knotlike Borel complex. Define the even integer
\begin{align*}
\wt{s}_Q(C_Q) = \max_i \{&\text{there exists } [x] \in H_*(C_Q) \mid \iota_*[x]=u^{-i/2} \in u^{-1}H_*(C_Q)\}.
\end{align*}
Equivalently, $\wt{s}_Q(C_Q)$ is the maximum $i$ such that $H_*(C_Q)$ has a $u$-nontorsion class in bigrading $(0, i)$. We can likewise study $u$-nontorsion classes in other homological gradings. This leads to the following refinement: for each $A \geq 0$, define the even integer
\begin{align*}
\wt{s}_{Q, A}(C_Q) = \max_i \{&\text{there exists } [x] \in H_*(C_Q) \mid \iota_*[x]=u^{-i/2}Q^A \in u^{-1}H_*(C_Q)\}.
\end{align*}
Equivalently, $\wt{s}_{Q, A}(C_Q)$ is the maximum $i$ such that $H_*(C_Q)$ has a $u$-nontorsion class in bigrading $(A, i)$. Clearly, $\wt{s}_{Q, 0} = \wt{s}_Q$. The reader may gain some intuition for $\wt{s}_{Q, A}$ by considering Figure~\ref{fig:sq-example}.
\end{defn}

\begin{figure}[h!]
\includegraphics[scale = 1]{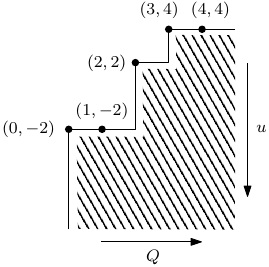}
\caption{A schematic example of the part of $H_*(C_Q)$ that contributes to the calculation of $\wt{s}_Q$. The corresponding $\wt{s}_Q$-invariants are given by: $\wt{s}_Q = \wt{s}_{Q, 0} = \wt{s}_{Q, 1} = -2$, $\wt{s}_{Q, 2} = 2$, and $\wt{s}_{Q, 3} = \wt{s}_{Q, 4} = 4$.}\label{fig:sq-example}
\end{figure}

It will also be useful to consider the truncation $C_Q/Q^B$ for $B > 0$, since will allow us to simplify certain calculations by disregarding high powers of $Q$. It is not difficult to check that if $C_Q$ is a knotlike Borel complex, then $u^{-1}H_*(C_Q/Q^B) \cong \f[u, u^{-1}, Q]/Q^B$. We thus let
\[
\iota_* \colon H_*(C_Q/Q^B) \xrightarrow{\cong} \F[u,u^{-1},Q]/Q^B
\]
be the map induced by including $H_*(C_Q/Q^B)$ into its localization. This gives the truncated version of the $\wt{s}_Q$-invariant:

\begin{defn}\label{def:strong-s-general}
Let $C_Q$ be a knotlike Borel complex and $0 \leq A < B$. Define the even integer
\begin{align*}
\wt{s}_{Q, A, B}(C_Q) = \max_i \{&\text{there exists } [x] \in H_*(C_Q/Q^B) \mid \iota_*[x]=u^{-i/2}Q^A \in u^{-1}H_*(C_Q/Q^B)\}.
\end{align*}
Equivalently, $\wt{s}_{Q, A, B}(C_Q)$ is the maximum $i$ such that $H_*(C_Q/Q^B)$ has a $u$-nontorsion class in bigrading $(A, i)$. Clearly, $\wt{s}_{Q, A, \infty} = \wt{s}_{Q, A}$, while $\wt{s}_{Q, 0, 1}$ is the non-equivariant $s$-invariant of $C_Q/Q$.
\end{defn}

As with the usual $\wt{s}$-invariant, it is clear that if $f_Q \colon (C_1)_Q \rightarrow (C_2)_Q$ is local, then for any $0 \leq A < B$ (including the cases of $0$ and $\infty$), we have 
\[
\wt{s}_{Q, A, B}((C_1)_Q) + \gr_q(f_Q) \leq \wt{s}_{Q, A, B}((C_2)_Q).
\]
This of course immediately gives Theorem~\ref{thm:equivariant_s_borel_intro}:

\begin{proof}[Proof of Theorem~\ref{thm:equivariant_s_borel_intro}]
If $\Sigma$ is an equivariant cobordism from $K_1$ to $K_2$, then by Theorem~\ref{thm:full_borel_invariant} there exists a cobordism map fro $\Kcrm_Q(K_1)$ to $\Kcrm_Q(K_2)$ which is local and has grading shift $-2 g(\Sigma)$.
\end{proof}

The reader should calculate the various flavors of $\wt{s}_Q$ for the examples discussed in Section~\ref{sec:borel-setup} by using the Borel complexes and homologies displayed in Figures~\ref{fig:borel-ex1}, \ref{fig:borel-ex2}, \ref{fig:borel-ex3-1} and \ref{fig:borel-ex4}. The results are given here:

\begin{example}\label{ex:sq-ex1}
For the complex of Example~\ref{ex:borel-ex1}, we have $\wt{s}(C, \tau) = 0$ and 
\[
\wt{s}_{Q, A}(C_Q) = 
\begin{cases}
0 & \text{for } A = 0 \\
2 & \text{for } A \geq 1.
\end{cases}
\] 
The calculation for $\wt{s}_{Q, A, B}(C_Q)$ is identical, except that if $B = 1$, we destroy all $Q$-information and the only remaining invariant is $\wt{s}_{Q, 0, 1}(C_Q) = 0$.
\end{example}

\begin{example}\label{ex:sq-ex2}
For the complex of Example~\ref{ex:borel-ex2}, we have $\wt{s}(C, \tau) = -2$ and 
\[
\wt{s}_{Q, A}(C_Q) = -2
\] 
for all $A$. The calculation for $\wt{s}_{Q, A, B}(C_Q)$ is identical, except that if $B = 1$, we destroy all $Q$-information and the only remaining invariant is $\wt{s}_{Q, 0, 1}(C_Q) = 0$.
\end{example}

\begin{example}\label{ex:sq-ex3}
All the invariants for the complex $C_1$ of Example~\ref{ex:borel-ex3} are zero. For the complex $C_2$, we have that $\wt{s}(C_2, \tau_2) = 0$ and
\[
\wt{s}_{Q, A}((C_2)_Q) = -2
\] 
for all $A$. The calculation for $\wt{s}_{Q, A, B}((C_2)_Q)$ is identical, except that if $B \leq 2$, we destroy all $Q$-information and the invariant collapses to zero.
\end{example}

\begin{example}\label{ex:sq-ex4}
For the complex $C$ of Example~\ref{ex:borel-ex4}, we have that $\wt{s}(C, \tau) = 0$ and
\[
\wt{s}_{Q, A}(C_Q) = 0
\]
for all $A$. For $B \geq 3$, the calculation for $\wt{s}_{Q, A, B}(C_Q)$ is identical. However, for $B = 2$, the complex $C_Q/Q^2$ is locally equivalent to that of Example~\ref{ex:sq-ex1}; thus $\wt{s}_{Q, 0, 2}(C_Q) = 0$ and $\wt{s}_{Q, 1, 2}(C_Q) = 2$. For $B = 1$, we destroy all $Q$-information and the invariant collapses to zero. 
\end{example}

Note that as shown by Example~\ref{ex:sq-ex4}, it is possible for the truncated $\wt{s}_{Q, A, B}$-invariants to be nontrivial even if the usual invariants $\wt{s}_{Q, A}$ are all trivial. Finally, we observe:

\begin{lem}
For any fixed $B$ (including $B = \infty$), we have
\[
\wt{s}_{Q, 0, B} \leq \wt{s}_{Q, 1, B} \leq \wt{s}_{Q, 2, B} \leq \cdots \leq \wt{s}_{Q, B - 1, B}.
\]
For any fixed $A$, we have
\[
\wt{s}_{Q, A, A + 1} \geq \wt{s}_{Q, A, A + 2} \geq \wt{s}_{Q, A, A + 3} \geq \cdots \geq \wt{s}_{Q, A, \infty} = \wt{s}_{Q, A}.
\]
\end{lem}
\begin{proof}
The first claim follows from the fact that
\[
\iota_* \colon H_*(C_Q/Q^B) \rightarrow u^{-1} H_*(C_Q/Q^B) \cong \f[u, u^{-1}, Q]/Q^B
\]
is $Q$-equivariant. Thus, if $\iota_*([x]) = u^{-i/2}Q^A$, then $\iota_*(Q[x]) = u^{-i/2}Q^{A+1}$. The second claim follows from the observation that the diagram
\[
\begin{tikzcd}
	{H_*(C_Q/Q^{B+1})} & {H_*(C_Q/Q^{B})} \\
	{\mathbb{F}[u,u^{-1},Q]/Q^{B+1}} & {\mathbb{F}[u,u^{-1},Q]/Q^{B}}
	\arrow[from=1-1, to=1-2]
	\arrow["{\iota_*}"', from=1-1, to=2-1]
	\arrow["{\iota_*}", from=1-2, to=2-2]
	\arrow[from=2-1, to=2-2]
\end{tikzcd}
\]
commutes, where the horizontal arrows are induced by the quotient. Thus, if $[x] \in H_*(C_Q/Q^{B+1})$ satisfies $\iota_*([x]) = u^{-i/2}Q^A$, then so does the image of $[x]$ in $H_*(C_Q/Q^B)$.
\end{proof}

\subsection{Constraining the Borel complex from homological data} \label{sec:borel-bootstrapping}

We now discuss some lemmas aimed at computing the Borel complex from partial data. Unfortunately, we will not generally have access to the full computation the action of $\tau$ on the Bar-Natan complex, as the size of the chain complex is very large even for relatively simple knots. Instead, as in Section~\ref{subsec:6-examples}, we will usually know:
\begin{enumerate}
\item[(a)] The Bar-Natan complex, via computing the Khovanov homology together with the Bar-Natan spectral sequence; and,
\item[(b)] The action of $\tau$ on the Khovanov homology $\Kh$. 
\end{enumerate}
While this is very far from determining the homotopy type of the Borel complex, we will attempt to enumerate the possible homotopy types compatible with such data.

We begin with the following general situation. Let $(C, \partial)$ be a complex over $\f[u]$. Suppose that we endow $C \otimes \f[Q]$ with an endomorphism $\partial_Q = \partial + Q\partial_1$ such that
\[
(C_Q = C \otimes \f[Q], \partial_Q = \partial + Q\partial_1)
\]
is a chain complex. Here, by $\partial$ we mean $\partial \otimes \id$, but we allow $\partial_1$ to be any (grading-preserving) endomorphism of $C_Q$. (That is, $\partial_1$ is not required to arise from an endomorphism of $C$, although in the initial case of interest we indeed have $\partial_1 = (1 + \tau) \otimes \id$.) Now suppose $(C', \partial')$ is an $\f[u]$-complex which is homotopy equivalent (over $\f[u]$) to $(C, \partial)$. We will usually think of $C'$ as a smaller model for $C$. We ask whether we can find an endomorphism $\partial_1'$ of $C' \otimes \f[Q]$ that makes
\[
(C_Q' = C' \otimes \f[Q], \partial_Q' = \partial' + Q\partial_1')
\]
into a chain complex which is homotopy equivalent to $(C_Q, \partial_Q)$. Roughly speaking, we wish to use the homotopy equivalence between $C$ and $C'$ to transfer the perturbed differential $\partial_Q$ on $C_Q$ to a perturbed differential on $C_Q'$.

\begin{rmk}\label{rmk:setuzero}
Note that every Borel complex $C_Q$ arises as $(C_Q = C \otimes \f[Q], \partial_Q = \partial + Q\partial_1)$ for some $\f[u]$-complex $(C, \partial)$ and endomorphism $\partial_1$. Indeed, if we simply fix a basis $\{x_1, \ldots, x_n\}$ for $C_Q$ over $\f[u, Q]$, then setting $C = \mathrm{span}_{\f[u]}\{x_1, \ldots, x_n\}$ and $\partial = \partial_Q \bmod Q$ provides the desired $\f[u]$-complex.
\end{rmk}

The desired endomorphism $\partial_1'$ is provided by the homological perturbation lemma:
 
\begin{lem}\label{lem:minimal-model-1}
  Let $C$ and $C'$ be chain complexes over $\F[u]$. Suppose that $f \colon C \rightarrow C'$ and $g \colon C' \rightarrow C$ are homotopy inverses and let $H \colon C \rightarrow C$ be a homotopy between $g \circ f$ and $\id$. Consider the formal expression
\[
\partial_1' = f(\id + Q\partial_1H)^{-1} \partial_1 g = f \left( \sum_{i = 0}^\infty (Q\partial_1 H)^i \right) \partial_1 g.
\]
Then $\partial_Q' = \partial' + Q \partial_1'$ is a differential on $C_Q'$ and $(C_Q, \partial_Q)$ and $(C_Q', \partial_Q')$ are homotopy equivalent.
\end{lem}
\begin{proof}
  First note that $\partial_1'$ is in fact a well-defined endomorphism of $\smash{C_Q'}$. This is because $\partial_1$ preserves the homological grading while $H$ decreases the homological grading. Hence $(\partial_1 H)^i$ decreases the homological grading by $i$; since $C_Q$ is bounded below in homological grading, this shows that $\partial_1'$ applied to any fixed element has only a finite number of nonzero terms. The claim is then a matter of algebraic verification. The explicit formulas are given in (for example) the work of Crainic \cite{crainic}.
\end{proof}

We now return to the case at hand. Let $(C, \partial)$ be a free, finitely generated complex over $\f[u]$ equipped with an involution $\tau$. As discussed in Section~\ref{subsec:6-examples}, $(C, \partial)$ is homotopy equivalent to $(C' = H_*(C/u) \otimes \f[u], \partial')$ for some differential $\partial'$ computed from the spectral sequence as in the proof of Lemma~\ref{lem:quasi-iso-move}. Using Lemma~\ref{lem:minimal-model-1}, we transfer $\partial_Q$ from $C_Q$ onto the smaller model $C' \otimes \f[Q] = H_*(C/u) \otimes \f[u, Q]$. Although the resulting perturbation $\partial_1'$ may be somewhat complicated, it turns out that $\partial_1' \bmod (u, Q)$ is quite tractable. Indeed, it is exactly what one would expect: modulo $(u, Q)$, the perturbation $\partial_1'$ is just the action of $1 + \tau_*$ on $H_*(C/u)$.

\begin{lem}\label{lem:minimal-model-2}
Let $C$ be a free $\f[u]$-complex with an action of an involution $\tau$. Then the Borel complex $C_Q$ is homotopy equivalent to a complex
\[
C_Q' = H_*(C/u) \otimes \f[u, Q] \quad \text{with} \quad \partial_Q' = \partial' + Q \partial_1'
\]
for which
\[
\partial_1' = (1 + \tau_*) \bmod (u, Q)
\]
where $\tau_*$ is the action of $\tau$ on $H_*(C/u)$.
\end{lem}
\begin{proof}
Choose homotopy inverses $f \colon C \rightarrow C' = H_*(C/u) \otimes \f[u]$ and $g \colon C' = H_*(C/u) \otimes \f[u] \rightarrow C$. Apply Lemma~\ref{lem:minimal-model-1} and quotient out by $u$. Note that after quotienting out by $u$, the map $g$ is an inclusion of $\f$-vector spaces from $H_*(C/u)$ into a subspace of cycles of $C/u$, and $f$ is some projection onto this subspace followed by taking the homology class. An examination of the formula for $\partial_1'$ in Lemma~\ref{lem:minimal-model-1} thus shows that modulo $(u, Q)$, we have $\partial_1' = f\partial_1g = f(1 + \tau)g = 1 + \tau_*$.
\end{proof}

In practice, there is thus an straightforward algorithm for constraining (a homotopy representative of) the Borel complex $C_Q$, given the data discussed at the beginning of the section. First, we simply write down a basis for the $\f$-vector space $H_*(C/u)$. Tensoring this with $\f[u, Q]$ gives a basis for $C_Q$. We can then fill in the components of the differential that have a coefficient of $u^i$ by using the spectral sequence, as well as the components of the differential that have a coefficient of $Q$ by using the homological action of $\tau$ on $H_*(C/u)$. While we cannot completely determine the other components of the differential in this basis, we can attempt to constrain the set of possibilities using grading reasons and the fact that $\partial_Q^2 = 0$.

To understand this explicitly, let us consider the examples of Section~\ref{sec:borel-setup}. In Examples~\ref{ex:borel-ex1} and \ref{ex:borel-ex2}, the differential on the Borel complex is completely determined by the differential on $C$ and the action of $1 + \tau_*$ on $H_*(C/u)$. This is in contrast to Example~\ref{ex:borel-ex3}, where $H_*(C/u) \otimes \f[u]$ provides a smaller, three-generator model for $C$. In both the cases $C_1$ and $C_2$ of Example~\ref{ex:borel-ex3}, the $Q$-component of the Borel differential on the smaller model \textit{to first order in $Q$} is zero. This is exactly as predicted by Lemma~\ref{lem:minimal-model-2}, since the action of $\tau$ on $H_*(C/u)$ is trivial. However, in the case of $C_2$, we see that there is an additional $Q^2$ term of the Borel differential. A similar phenomenon occurs with Example~\ref{ex:borel-ex4}: when constructing the Borel complex out of the smaller five-generator model $H_*(C/u) \otimes \f[u]$ for $C$, the Borel differential to first order in $Q$ is determined by the homological action of $\tau$ on $H_*(C/u)$. Roughly speaking, however, the Borel differential also remembers the behavior of $\tau$ involving homologically inessential pairs of generators; this information is stored in the higher $Q$-powers of the Borel differential.

\subsection{Examples} 

We now give some calculations and partial calculations following the procedure of Section~\ref{sec:borel-bootstrapping}.

\begin{example}
Let $K = 9_{46}$. By Lemma~\ref{lem:minimal-model-1}, $\Kcrm_Q(K)$ is homotopy equivalent to a complex with underlying module $\Khr(K) \otimes \f[u, Q]$. By Lemma~\ref{lem:minimal-model-2}, the arrows decorated by a power of $u$ are determined by the spectral sequence (reproduced in Figure~\ref{fig:borel-946}), while the arrows decorated with $Q$ are given by the action of $1 + \tau$ on $\Khr(K)$. Note that if an arrow with decoration $u^i Q^j$ points from $x$ to $y$, then 
\[
\gr(y) - \gr(x) = (1 - j, 2i). 
\]
An examination of Figure~\ref{fig:borel-946} then shows that further arrows are ruled out for grading reasons. The Borel complex $\Kcrm(K)$ is thus given as on the right. Note that this is locally equivalent to the complex from Example~\ref{ex:borel-ex2}.

\begin{figure}[h!]
\includegraphics[scale = 0.7]{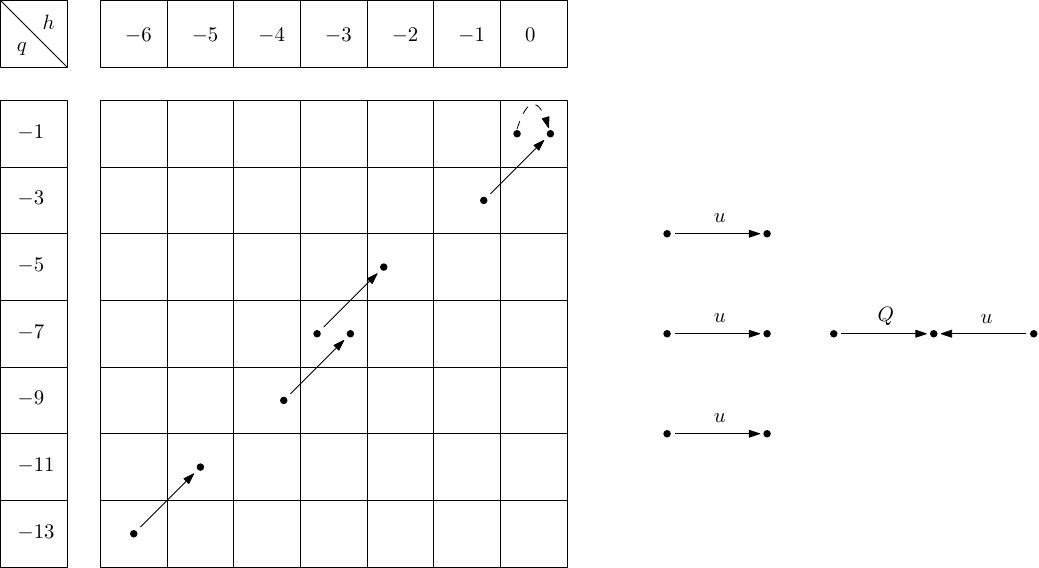}
	\caption{Left: $E^1$-page of the Bar-Natan spectral sequence for $9_{46}$. Right: the Borel complex $\Kcrm_Q(K)$.} \label{fig:borel-946}
\end{figure}
\end{example}

\begin{example}
Let $J = 17nh_{74}$. By Lemma~\ref{lem:minimal-model-1}, $\Kcrm(J)$ is homotopy equivalent to a complex with underlying module $\Khr(J) \otimes \f[u, Q]$. By Lemma~\ref{lem:minimal-model-2}, the arrows decorated by a power of $u$ are determined by the spectral sequence (reproduced in Figure~\ref{fig:borel-kyle}), while the arrows decorated with $Q$ are given by the action of $1 + \tau$ on $\Khr(J)$. Unfortunately, we will not be able to fully determine $\Kcrm(J)$ in this case. Instead, we show that there exists a local map of grading shift zero from the complex $C_Q$ of Example~\ref{ex:borel-ex4} -- displayed again in Figure~\ref{fig:borel-kyle} -- into $\smash{\Kcrm_Q(J)}$. While this is weaker than determining the homotopy type (or even local equivalence class) of $\smash{\Kcrm_Q(J)}$, the reader should think of this as bounding $\smash{\Kcrm_Q(J)}$ from below. Indeed, in particular note that we have $\smash{\wt{s}_{Q, 1, 2}(J) \geq \wt{s}_{Q, 1, 2}(C_Q) \geq 2}$.
\end{example}

\begin{figure}[h!]
\includegraphics[scale = 0.8]{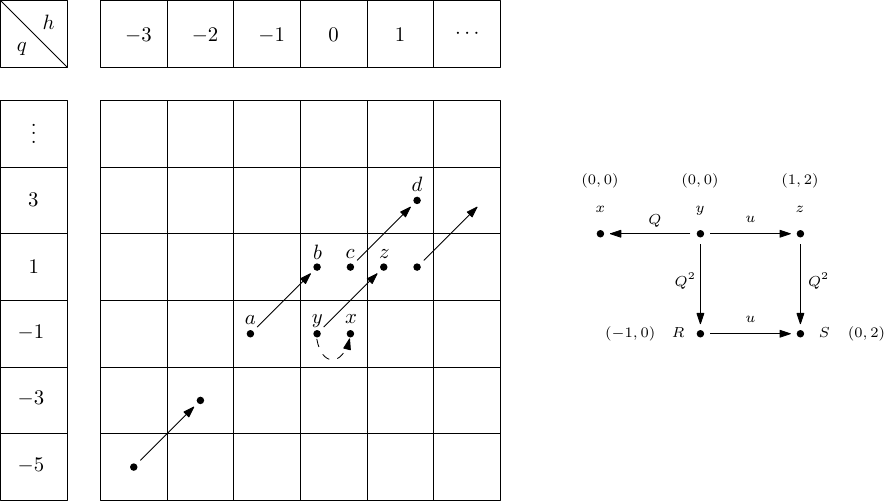}
\caption{Left: $E^1$-page of the Bar-Natan spectral sequence for $J$. Right: the Borel complex $C_Q$ of Example~\ref{ex:borel-ex4}, which bounds $\Kcrm_Q(J)$ from below.}\label{fig:borel-kyle} 
\end{figure}

\begin{lem}\label{lem:partial-computation-J}
There exists a local map of grading shift zero from the Borel complex $C_Q$ of Example~\ref{ex:borel-ex4} into $\Kcrm_Q(J)$.
\end{lem}
\begin{proof}
As discussed above, $\smash{\Kcrm_Q(J)}$ is homotopy equivalent to a complex with underlying module $\Khr(J) \otimes \f[u, Q]$. To prove the lemma we begin by constructing a particular subcomplex of $\smash{\Kcrm_Q(J)}$. Consider the basis elements $x$, $y$, and $z$ labeled in Figure~\ref{fig:borel-kyle}. The $u^i$- and $Q$-arrows going out of $x$, $y$, and $z$ are all determined by Lemma~\ref{lem:minimal-model-2}. We perform an exhaustive analysis of the possibilities for the other arrows. Write
\begin{equation}\label{eq:full-differential}
\partial_Q x = uQ X_1 + Q^2 X_2, \quad \partial_Q y = uz + Qx + uQ Y_1 + Q^2 Y_2, \quad \text{and} \quad \partial_Q z = uQ Z_1 + Q^2 Z_2
\end{equation}
where the $X_i$, $Y_i$, and $Z_i$ are $\f[u, Q]$-linear combinations of basis elements from $\Khr(J) = H_*(\Kcrm(J)/u)$. Examining the gradings of these elements, we see that
\begin{equation}\label{eq:enumerating-possibilities}
X_2, Y_2 \in \spa_{\f}\{a\}, \quad X_1, Y_1, Z_2 \in \spa_{\f}\{b, c\}, \quad \text{and} \quad Z_1 \in \spa_{\f}\{d\}.
\end{equation}
Note that even though the $X_i$, $Y_i$, and $Z_i$ are \textit{a priori} $\f[u, Q]$-linear combinations of basis elements, we see for grading reasons that in fact they must be $\f$-linear combinations of the basis elements listed above.

We now use the condition $\partial_Q^2 = 0$ to further constrain the possibilities for $X_i$, $Y_i$, and $Z_i$. First observe that $\partial_Q a = ub$ and $\partial_Q c = ud$. This can be read off from the spectral sequence and homological action of $\tau$, since there are no further arrows out of $a$, $b$, $c$, or $d$ for grading reasons. Applying the condition $\smash{\partial_Q^2 = 0}$ to $x$ shows that $X_2 = 0$ and $X_1 = 0$ or $b$. Applying the condition $\smash{\partial_Q^2 = 0}$ to $y$ then shows that
\begin{align*}
u(uQZ_1 + Q^2 Z_2) + Q(uQX_1) + uQ \partial_Q Y_1 + Q^2 \partial_Q Y_2 = 0.
\end{align*}
Since $Z_1$, $Z_2$, and $X_1$ are not in the image of $(u, Q)$ and $\partial_Q Y_1$ and $\partial_Q Y_2$ are either zero or a $u$-multiple of a basis element, we may collect coefficients to conclude
\begin{equation}\label{eq:simplified-relations}
uZ_1 = \partial_Q Y_1 \quad \text{and} \quad uZ_2 + uX_1 = \partial_Q Y_2.
\end{equation}
Now let us perform the change-of-basis $z = z + QY_1$. Modifying (\ref{eq:full-differential}) accordingly then gives
\[
\partial_Q x = uQX_1, \quad \partial_Q y = uz + Qx + Q^2 Y_2, \quad \text{and} \quad \partial_Q z = Q^2 Z_2.
\]
We know that $X_1 = 0$ or $b$. Reviewing the possibilities (\ref{eq:enumerating-possibilities}) for $Y_2$ and $Z_2$ and taking into account the second relation of (\ref{eq:simplified-relations}), we thus see that there are four situations:
\begin{enumerate}
\item $X_1 = 0$, $Y_2 = 0$, $Z_2 = 0$: This is the complex on the left in Figure~\ref{fig:borelex2}.
\item $X_1 = 0$, $Y_2 = a$, $Z_2 = b$: This is the complex on the right in Figure~\ref{fig:borelex2}.
\item $X_1 = b$, $Y_2 = 0$, $Z_2 = b$: Performing the change-of-basis $x = x + Qa$ gives the complex on the right in Figure~\ref{fig:borelex2}. 
\item $X_1 = b$, $Y_2 = a$, $Z_2 = 0$: Performing the change-of-basis $x = x + Qa$ gives the complex on the left in Figure~\ref{fig:borelex2}.
\end{enumerate}

\begin{figure}[h!]
\includegraphics[scale = 1]{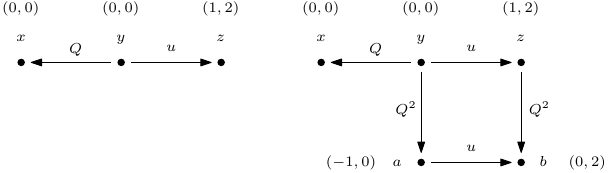}
\caption{Two possible subcomplexes of $\Kcrm_Q(J)$.}\label{fig:borelex2}
\end{figure}

We have thus identified a subcomplex of $\Kcrm_Q(J)$ which is either isomorphic to the complex on the left of Figure~\ref{fig:borelex2} or the complex on the right of Figure~\ref{fig:borelex2}. Moreover, note that in all cases, after reducing mod $Q$ the generator $x$ is sent to the Khovanov homology generator $x$ in $\Khr(J)$. Since this survives the Bar-Natan spectral sequence, the generator $x$ of our subcomplex has $u$-nontorsion homology class. If the subcomplex we have identified is as on the right-hand side of Figure~\ref{fig:borelex2}, then we are obviously done. If the subcomplex is as on the left-hand side, then we simply map in the complex on the right-hand side by sending $R$ and $S$ to zero.
\end{proof}

\subsection{Connected sums} \label{sec:connected-sum1}
We now study the behavior of our invariants under connected sums. Let $(K_1, \tau_1)$ and $(K_2, \tau_2)$ be two strongly invertible knots with transvergent diagrams $D_1$ and $D_2$. 

\begin{thm}\label{thm:connected-sum-tau}
There is an equivariant isomorphism
\[
(\Kcrm(D_1 \# D_2), \tau_1 \# \tau_2) \cong (\Kcrm(D_1), \tau_1) \otimes (\Kcrm(D_2), \tau_2).
\]
Here, the right-hand side is the tensor product complex $\Kcrm(D_1) \otimes_{\f[u]} \Kcrm(D_2)$ equipped with the tensor product involution $\tau_1 \otimes \tau_2$. The claim holds regardless of how the equivariant connected sum on the left is formed.
\end{thm}
\begin{proof}

Observe that there is a chain map
\[
\Phi \colon \Kc'(D_1 \sqcup D_2) \rightarrow \Kc'(D_1 \# D_2)
\]
corresponding to the saddle cobordism from $D_1 \sqcup D_2$ to $D_1 \# D_2$. This is the merge map discussed in Section~\ref{sec:reduced-bar-natan}, which restricts to preserve the unpointed reduced complex:
\[
\phi \colon \Kcrm_{un}(D_1 \sqcup D_2) \rightarrow \Kcrm_{un}(D_1 \# D_2).
\]
Moreover, we have an inclusion map
\[
\iota \colon \Kcrm_{un}(D_1) \otimes \Kcrm_{un}(D_2) \rightarrow \Kcrm_{un}(D_1 \sqcup D_2).
\]
It is clear that all of these maps are $\tau$-equivariant, where we put the diagonal action of $\tau$ on the domain of $\iota$.

Let the connected sum be formed along a choice of equivariant basepoints $p_1 \in D_1$ and $p_2 \in D_2$. Consider any pair of resolutions $(D_1)_v$ of $D_1$ and $(D_2)_w$ of $D_2$. Taking the connected sum of these resolutions gives a resolution of $D_1 \# D_2$ which we denote by $(D_1 \# D_2)_{v \# w}$. Note that for each $v$ and $w$, the aforementioned saddle map goes from $(D_1)_v \sqcup (D_2)_w$ to $(D_1 \# D_2)_{v \# w}$. Enumerate
\[
Z((D_1)_v) = \{c_1, \ldots, c_m\} \quad \text{and} \quad Z((D_2)_w) = \{c_{m+1}, \ldots, c_{m+n}\}
\]
so that $c_1$ is the circle containing $p_1$ and $c_{m+1}$ is the circle containing $p_2$. Then we have the enumerations
\[
Z((D_1)_v \sqcup (D_2)_w) = \{c_1, \ldots, c_m\} \cup \{c_{m+1}, \ldots, c_{m+n}\}
\]
and
\[
Z((D_1 \# D_2)_{v \# w}) = \{c_0\} \cup \left(\{c_1, \ldots, c_m\} - \{c_1\}\right) \cup \left(\{c_{m+1}, \ldots, c_{m+n}\} - \{c_{m+1}\}\right)
\]
where $c_0 = c_1 \# c_{m+1}$. As discussed in Section~\ref{sec:reduced-bar-natan}, the merge map is the algebra map induced by the map $\pi$ of index sets which is an identity except for $\pi(1) = \pi(m+1) = 0$.

We now claim that $\phi \circ \iota$ is an isomorphism. To see this, first observe that $\phi \circ \iota$ is surjective. Indeed, $\Kcrm_{un}((D_1 \# D_2)_{v \# w})$ is generated as an algebra by $1$, together with five kinds of generators: 
\begin{enumerate}
\item $x_i + x_j$ with $2 \leq i, j \leq m$,
\item $x_i + x_j$ with $m+2 \leq i, j \leq m + n$, 
\item $x_i + x_0$ with $2 \leq i \leq m$, 
\item $x_0 + x_j$ with $m+2 \leq j \leq m + n$, and
\item $x_i + x_j$ with $2 \leq i \leq m$ and $m + 2 \leq j \leq m + n$.
\end{enumerate}
The first four classes of generators are clearly in the image of $\varphi \circ \iota$. For example, for the third class we have
\[
(\phi \circ \iota)( (x_i + x_1) \otimes 1) = x_i + x_0.
\]
To deal with the fifth class, we observe that
\[
(\phi \circ \iota)( (x_i + x_1) \otimes 1 + 1 \otimes (x_j + x_{m+1}) ) = (x_i + x_0) + (x_j + x_0) = x_i + x_j.
\]
Since $\phi \circ \iota$ is an algebra map, this shows that $\phi \circ \iota$ is surjective.

We now observe that $\phi \circ \iota$ is an isomorphism modulo $u$. Indeed, we showed in the proof of Lemma~\ref{lem:reduced-theories-equivalent} that modulo $u$, $\Kcrm_{un}(D_v)$ has rank equal to $2^{m-1}$, where $m$ is the number of circles in $Z(D_v)$. Thus the rank of $\Kcrm_{un}((D_1)_v) \otimes \Kcrm_{un}((D_2)_w)$ after setting $u = 0$ is given by $2^{m-1} \cdot 2^{n-1} = 2^{m + n - 2}$. On the other hand, this is exactly the same as the rank of $\Kcrm_{un}((D_1 \# D_2)_{v \# w})$ after setting $u = 0$, as the latter has $m + n - 1$ circles. Thus $\phi \circ \iota$ is an isomorphism modulo $u$; by Lemma~\ref{lem:iso-after-set-zero}, we have that $\phi \circ \iota$ is an isomorphism, as desired.
\end{proof}

We caution the reader that the Borel complex associated to the tensor product is \textit{not} the tensor product of Borel complexes in the usual sense. Indeed, the Borel complex associated to the tensor product cannot easily be computed only from the Borel differentials on each of the factors. Instead, we have:

\begin{lem}\label{lem:modified-product}
Let $C_1$ and $C_2$ be complexes equipped with involutions $\tau_1$ and $\tau_2$. Let $C_\otimes = C_1 \otimes C_2$ be equipped with the involution $\tau_1 \otimes \tau_2$. Write $\partial_{Q, 1}$, $\partial_{Q, 2}$, and $\partial_{Q, \otimes}$ for the Borel differentials on the complexes $(C_1)_Q$, $(C_2)_Q$, and $(C_\otimes)_Q$, respectively. Then
\[
\partial_{Q, \otimes} = \partial_{Q, 1} \otimes 1 + 1 \otimes \partial_{Q, 2} + Q(1 + \tau_1) \otimes (1 + \tau_2).
\]
\end{lem}
\begin{proof}
By definition, $\partial_{Q, \otimes}$ is given by
\[
\partial_{Q, \otimes} = (\partial_1 \otimes 1 + 1 \otimes \partial_2) + Q( 1\otimes 1 + \tau_1 \otimes \tau_2 ).
\]
On the other hand,
\begin{align*}
\partial_{Q, 1} \otimes 1 + 1 \otimes \partial_{Q, 2} &= (\partial_1 + Q(1 + \tau_1)) \otimes 1 + 1 \otimes (\partial_2 + Q(1 + \tau_2)) \\
&= (\partial_1 \otimes 1 + 1 \otimes \partial_2) + Q(\tau_1 \otimes 1 + 1 \otimes \tau_2).
\end{align*}
The difference between these is precisely $Q(1 + \tau_1) \otimes (1 + \tau_2)$.
\end{proof}

\subsection{Koszul duality} \label{sec:koszul-borel}
The previous section indicates that $\Kcrm_Q(K_1 \# K_2)$ is difficult to calculate directly from $\smash{\Kcrm_Q(K_1)}$ and $\smash{\Kcrm_Q(K_2)}$. Instead, Lemma~\ref{lem:modified-product} shows that the necessary data to compute the tensor product consists of the underlying $\tau$-complexes of $K_1$ and $K_2$. It is thus natural to ask whether the latter can be determined from the former. 

\subsubsection{Koszul bimodules} We begin with a brief review of Koszul duality; most results here appear in (or can be deduced from) the seminal paper \cite{BGS}. As we will use Koszul duality in a rather specific setting, we provide some background.

\begin{defn}\label{def:koszul}
Let $R = \f[u]$. Denote by $\bKom_{R[\tau]/(\tau^2 + 1)}$ the category of bigraded complexes over $R[\tau]/(\tau^2 + 1)$ which are free over $R$ and homologically bounded below. Likewise, denote by $\bKom_{R[Q]}$ the category of bigraded complexes over $R[Q]$ which are free over $R$ and homologically bounded below. In each case, we refer to the two gradings as the homological and quantum gradings, with $\deg(\tau) = (0, 0)$ and $\partial$, $u$, and $Q$ having their usual degrees from Definition~\ref{def:borel}.
\end{defn}

If $D$ is a transvergent diagram for $L$, then $\Kcm(D)$ gives rise to an element of $\smash{\bKom_{R[\tau]/(\tau^2 + 1)}}$ simply by promoting the action of $\tau$ to a variable in the polynomial ring. This contains the same information as the complex $(\Kcm(D), \tau)$, but we emphasize the fact that $\tau$ is an actual involution. Note that a morphism between two complexes in $\smash{\bKom_{R[\tau]/(\tau^2 + 1)}}$ is a $R[\tau]/(\tau^2 + 1)$-module map, and hence commutes with $\tau$ on the nose. Both the objects and morphisms of $\smash{\bKom_{R[\tau]/(\tau^2 + 1)}}$ are thus more stringent than in our discussion of $\tau$-complexes and maps between them.

Our treatment of Koszul duality will proceed via tensoring with certain \textit{Koszul bimodules}:

\begin{defn}\label{def:bimodule}
  Define a bimodule-with-differential by setting
\[
{}_{R[\tau]/(\tau^2+1)}\Br_{R[Q]} = R[\tau]/(\tau^2+1)\otimes_R R[Q] \quad \text{and} \quad \partial(a \otimes b) = (1 + \tau) a \otimes Qb.
\]
Similarly, define a bimodule-with-differential by setting
\[
{}_{R[Q]}\Hh_{R[\tau]/(\tau^2+1)} = \Hom_R(R[\tau]/(\tau^2+1),R[Q]) \quad \text{and} \quad (\partial F)(x) = Q F((1 + \tau)x).
\]
These are depicted in Figure~\ref{fig:bimodules}. Note that $\smash{{}_{R[\tau]/(\tau^2+1)}\Br_{R[Q]}}$ and $\smash{{}_{R[Q]}\Hh_{R[\tau]/(\tau^2+1)}}$ are in fact isomorphic, but we distinguish them as they will be used slightly differently. The reader should think of each of these as the minimal complex which has homology $R$ and is free (as a complex) over $R[\tau]/(\tau^2 + 1)$ and free (as a complex) over $R[Q]$.
\end{defn}

\begin{figure}[h!]
\includegraphics[scale = 0.9]{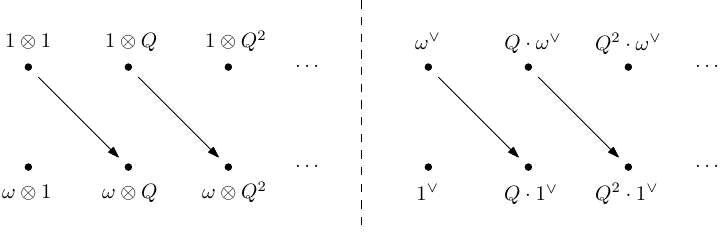}
\caption{The bimodules ${}_{R[\tau]/(\tau^2+1)}\Br_{R[Q]}$ (left) and ${}_{R[Q]}\Hh_{R[\tau]/(\tau^2+1)}$ (right), with arrows depicting the differential. It is convenient to consider the change-of-variables $\omega = 1 + \tau$, so that $R[\tau]/(\tau^2 + 1) = R[\omega]/\omega^2$. Over $R$, the ring $R[\omega]/\omega^2$ has basis $\{1, \omega\}$; we denote the dual basis of $\Hom_R(R[\omega]/\omega^2, R)$ by $\{1^\vee, \omega^\vee\}$. For both bimodules, multiplication by $Q$ is given by horizontal translation to the right, while multiplication by $\omega$ is vertical translation downwards.}\label{fig:bimodules}
\end{figure}

\begin{defn}
Define a functor
\[
\Br \colon \bKom_{R[\tau]/(\tau^2 + 1)} \rightarrow \bKom_{R[Q]} \quad \text{by} \quad X \mapsto \Br(X) = X \otimes {}_{R[\tau]/(\tau^2+1)}\Br_{R[Q]}.
\]
Likewise, define a functor
\[
\Hh \colon \bKom_{R[Q]} \rightarrow \bKom_{R[\tau]/(\tau^2 + 1)} \quad \text{by} \quad Y \mapsto \Hh(Y) = Y \otimes {}_{R[Q]}\Hh_{R[\tau]/(\tau^2+1)}.
\]
We will see that $\Br$ takes a complex over $R[\tau]/(\tau^2 + 1)$ and applies the Borel construction. We think of $\Hh$ as doing the reverse -- it takes a complex over $R[Q]$ and returns a complex equipped with an involution up to chain homotopy equivalence commuting with $\tau$.
\end{defn}


If we choose a basis $\{x_i\}$ for $X$ over $R$, then $\{x_i \otimes (1 \otimes 1)\}$ forms a basis for $X \otimes {}_{R[\tau]/(\tau^2+1)}\Br_{R[Q]}$ over $R[Q]$, since $x_i \otimes (\omega \otimes 1) = (\omega x_i) \otimes (1 \otimes 1)$. Moreover, the differential in this basis is given by
\[
\partial_Q(x_i \otimes (1 \otimes 1)) = (\partial x_i) \otimes (1 \otimes 1) + x_i \otimes (\omega \otimes Q) = (\partial x_i) \otimes (1 \otimes 1) + (\omega x_i) \otimes (1 \otimes Q).
\]
Under the identification that sends $x \otimes (1 \otimes Q^i)$ to $x \otimes Q^i$, this is precisely the Borel differential. Note that if $X$ is finitely generated over $R$, then $\Br(X)$ is finitely generated over $R[Q]$.

Now suppose $Y$ is free over $R[Q]$, as is the case for every Borel complex considered in this paper. Let $\{y_i\}$ be a basis for $Y$ over $R[Q]$. The tensor product $Y \otimes {}_{R[Q]}\Hh_{R[\tau]/(\tau^2+1)}$ has two towers for each $y_i$, given by $\{Q^j (y_i \otimes \omega^\vee)\}$ and $\{Q^j (y_i \otimes 1^\vee)\}$ for $j \geq 0$. Over $R[Q]$, these towers are of course spanned by $y_i \otimes \omega^\vee$ and $y_i \otimes 1^\vee$. However, over $R$ we instead have two infinite sets of linearly independent generators, each indexed by $j \geq 0$. The action of $\omega$ on $Y \otimes {}_{R[Q]}\Hh_{R[\tau]/(\tau^2+1)}$ maps the first of these onto the second with an index shift of one. Thus, even if $Y$ is finitely generated over $R[Q]$, $\Hh(Y)$ will \textit{not} be finitely generated over $R$. 

Given the above, it is clear that the composition of $\Br$ and $\Hh$ is not literally the identity. Nevertheless:

\begin{thm}\label{thm:derived-equivalence}
There exist quasi-isomorphisms of bimodules
\[
{}_{R[\tau]/(\tau^2+1)}\Id_{R[\tau]/(\tau^2+1)} \rightarrow {}_{R[\tau]/(\tau^2+1)}\Br_{R[Q]} \otimes {}_{R[Q]}\Hh_{\f[\tau]/(\tau^2+1)} 
\]
and
\[
\quad {}_{R[Q]}\Hh_{R[\tau]/(\tau^2+1)}\otimes {}_{R[\tau]/(\tau^2+1)}\Br_{R[Q]} \rightarrow {}_{R[Q]}\Id_{R[Q]}.
\]
In particular, $\Br$ and $\Hh$ induce equivalences of derived categories.
\end{thm}
\begin{proof}
  The statement is well-known, \cite{BGS}, although it is often used for $R=\f$. For the reader's convenience we show how to extend the
  result from $R=\f$ to more general rings. The method is the extension of scalars.
   To see that this implies $\Br$ and $\Hh$ induce equivalences of derived categories, we argue as follows. First, we claim that $\Br$ and $\Hh$ indeed constitute functors between derived categories; that is, they take quasi-isomorphisms to quasi-isomorphisms. Suppose $f \colon X_1 \rightarrow X_2$ is a quasi-isomorphism of $R[\tau]/(\tau^2+1)$-complexes. An examination of Figure~\ref{fig:bimodules} shows ${}_{R[\tau]/(\tau^2+1)}\Br_{R[Q]}$ is free as a complex over $R[\tau]/(\tau^2 + 1)$; that is, it decomposes into the direct sum of free modules over $R[\tau]/(\tau^2 + 1)$ with the differential going between successive summands. Since quasi-isomorphism is preserved under tensoring with free complexes, it follows that 
\[
f \otimes \id \colon \Br(X_1) = X_1 \otimes {}_{R[\tau]/(\tau^2+1)}\Br_{R[Q]} \rightarrow \Br(X_2) = X_2 \otimes {}_{R[\tau]/(\tau^2+1)}\Br_{R[Q]}
\]
is a quasi-isomorphism. The claim that $\Hh$ sends quasi-isomorphisms to quasi-isomorphisms is similar, using the fact that ${}_{R[Q]}\Hh_{R[\tau]/(\tau^2+1)}$ is free as a complex over $R[Q]$.

\begin{figure}[h!]
\includegraphics[scale = 1]{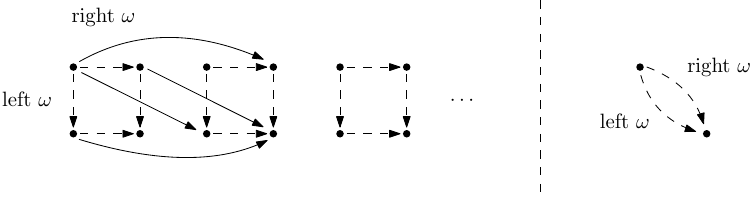}
\caption{Left: $\Br \otimes \Hh$. Over $R$, this consists of an infinite number of copies of a four-generator square spanned by $\{1, 1\cdot \omega, \omega \cdot 1, \omega \cdot 1 \cdot \omega\}$. The left and right $\omega$-actions are depicted by the vertical and horizontal dashed arrows, respectively. The differential takes $1$ in square $i$ to the sum of $\omega \cdot 1$ and $1 \cdot \omega$ in square $i+1$, and takes both $\omega \cdot 1$ and $1 \cdot \omega$ in square $i$ to $\omega \cdot 1 \cdot \omega$ in square $i+1$. This is partially depicted by the solid arrows. Right: the identity bimodule, spanned (over $R$) by two generators $1$ and $\omega$, with the left and right $\omega$-actions being the same. The quasi-isomorphism from right-to-left sends $1$ to the sum of $\omega \cdot 1$ and $1 \cdot \omega$ in the first square, and takes $\omega$ to $\omega \cdot 1 \cdot \omega$ in the first square. This has a homotopy inverse (which is only left $R[\tau]/(\tau^2 + 1)$-equivariant) that sends $1 \cdot \omega$ in the first square to $1$, takes $\omega \cdot 1 \cdot \omega$ in the first square to $\omega$, and sends everything else to zero.}\label{fig:boreltensor}
\end{figure}

\begin{figure}[h!]
\includegraphics[scale = 1]{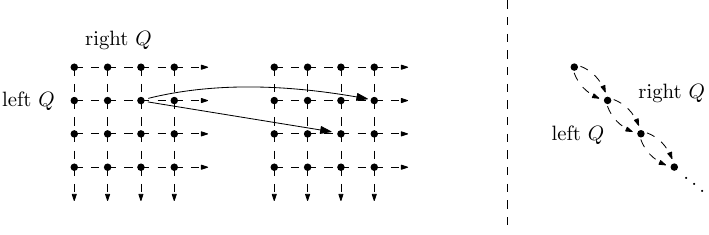}
\caption{Left: $\Hh \otimes \Br$. Over $R$, this consists of two copies of an infinite grid spanned by $\{Q^i \cdot 1 \cdot Q^j\}$ for $i, j \geq 0$. The left and right $Q$-actions are depicted by the vertical and horizontal dashed arrows, respectively. The differential takes $Q^i \cdot 1 \cdot Q^j$ in the first infinite grid to the sum of $Q^{i+1} \cdot 1 \cdot Q^j$ and $Q^i \cdot 1 \cdot Q^{j+1}$ in the second. This is partially depicted by the solid arrows.  Right: the identity bimodule, spanned (over $R$) by an infinite sequence of generators $Q^i$, with the left and right $Q$-actions being the same. The quasi-isomorphism from left-to-right takes $Q^i \cdot 1 \cdot Q^j$ in the second infinite grid to $Q^{i+j}$, and takes every generator in the first grid to zero. This has a homotopy inverse (which is only left $R[Q]$-equivariant) that sends $Q^i$ to $Q^i \cdot 1$ in the second grid.}\label{fig:boreltensor2}
\end{figure}

To see that $\Br$ and $\Hh$ are equivalences of derived categories, let $X \in \bKom_{R[\tau]/(\tau^2 + 1)}$. Consider the quasi-isomorphism
\[
{}_{R[\tau]/(\tau^2+1)}\Id_{R[\tau]/(\tau^2+1)} \rightarrow {}_{R[\tau]/(\tau^2+1)}\Br_{R[Q]} \otimes {}_{R[Q]}\Hh_{\f[\tau]/(\tau^2+1)} 
\]
It is straightforward to check that ${}_{R[\tau]/(\tau^2+1)}\Br_{R[Q]} \otimes {}_{R[Q]}\Hh_{R[\tau]/(\tau^2+1)}$ is free as a left $R[\tau]/(\tau^2+1)$-complex. By Lemma~\ref{lem:qitohe}, the above quasi-isomorphism thus has a homotopy inverse, although this inverse is only equivariant with respect to the left $R[\tau]/(\tau^2+1)$-action. (Again, it is easy to check this explicitly; see Figure~\ref{fig:boreltensor}.) The natural map
\[
X = X \otimes {}_{R[\tau]/(\tau^2+1)}\Id_{R[\tau]/(\tau^2+1)} \rightarrow X \otimes {}_{R[\tau]/(\tau^2+1)}\Br_{R[Q]} \otimes {}_{R[Q]}\Hh_{R[\tau]/(\tau^2+1)}
\]
is thus a homotopy equivalence of $R$-complexes, and in particular is a quasi-isomorphism. Hence we have a quasi-isomorphism from $X$ to $(\Hh \circ \Br)(X)$ which is (by construction) equivariant with respect to the right $R[\tau]/(\tau^2+1)$-action, showing that $\Hh \circ \Br$ is the identity map on the derived category. The claim for $\Br \circ \Hh$ is analogous.
\end{proof}

\subsubsection{Tensor products}
We now define the tensor product of $R[\tau]/(\tau^2 + 1)$-complexes and combine this with Theorem~\ref{thm:derived-equivalence} to define the tensor product of Borel complexes.

\begin{defn}\label{def:tensor-product-tau}
Let $X_1$ and $X_2$ be two complexes in $\bKom_{R[\tau]/(\tau^2+1)}$. Define
\[
X_1 \otimes X_2 = X_1 \otimes_{R} X_2
\]
by taking the tensor product of $X_1$ and $X_2$ over $R$ and giving $X_1 \otimes_{R} X_2$ the diagonal action $\tau \otimes \tau$. In terms of $\omega = 1 + \tau$, note that the action of $\omega$ on the tensor product is given by $1 \otimes \omega + \omega \otimes 1 + \omega \otimes \omega$.
\end{defn}

\begin{defn}\label{def:tensor-product-borel}
  Let $Y_1$ and $Y_2$ be two complexes in $\bKom_{R[Q]}$. Define 
\[
Y_1 \otimes_{\Br} Y_2 = \Br(\Hh(Y_1) \otimes \Hh(Y_2)),
\]
where the tensor product on the right is the product in $\bKom_{R[\tau]/(\tau^2+1)}$ from Definition~\ref{def:tensor-product-tau}.
\end{defn}

It is straightforward to see that Definitions~\ref{def:tensor-product-tau} and \ref{def:tensor-product-borel} are well-defined as operations on the relevant derived categories. Indeed, suppose that $X_1$ and $X_2$ are quasi-isomorphic to $X_1'$ and $X_2'$ via maps $f_1$ and $f_2$, respectively. Since all of our complexes are free over $R$, we have that $X_1 \otimes_R X_2$ and $X_1' \otimes_R X_2'$ are quasi-isomorphic via $f_1 \otimes f_2$. It is evident that $f_1 \otimes f_2$ intertwines the diagonal action of $\tau$ on $X_1 \otimes X_2$ with that on $X_1' \otimes X_2'$ and is thus a map in $\bKom_{R[\tau]/(\tau^2 + 1)}$. This establishes the claim for Definition~\ref{def:tensor-product-tau}. Since $\Hh$ and $\Br$ send quasi-isomorphisms to quasi-isomorphisms, the claim for Definition~\ref{def:tensor-product-borel} follows. 

Re-phrasing Theorem~\ref{thm:connected-sum-tau} in the present language, we obtain:

\begin{thm}\label{thm:borel-tensor-koszul}
Let $D_1$ and $D_2$ be transvergent diagrams for $(L_1, \tau_1)$ and $(L_2, \tau_2)$. Then we have a homotopy equivalence
\[
\Kcrm_{Q}(D_1 \# D_2) \simeq \Kcrm_{Q}(D_1) \otimes_{\Br} \Kcrm_{Q}(D_2).
\]
\end{thm}
\begin{proof}
Denote by $X_1$ and $X_2$ the complexes in $\smash{\bKom_{R[\tau]/(\tau^2 + 1)}}$ obtained from $(\Kcm(D_1), \tau_1)$ and $(\Kcm(D_2), \tau_2)$. Theorem~\ref{thm:connected-sum-tau} shows
\[
\Kcrm_{Q, W}(D_1 \# D_2) \cong \Br(X_1 \otimes X_2).
\]
Note that $\smash{\Br(X_i) = \Kcrm_{Q}(D_i)}$. Using the fact that $\Br$ and $\Hh$ induce equivalences of derived categories, it follows that $X_i$ is quasi-isomorphic to $\smash{\Hh(\Kcrm_{Q}(D_i))}$. Hence we have a quasi-isomorphism between
\[
\Kcrm_{Q}(D_1 \# D_2) \quad \text{and} \quad \Kcrm_{Q}(D_1) \otimes_{\Br} \Kcrm_{Q}(D_2).
\]
Since both of these complexes are free over $R$, invoking Lemma~\ref{lem:qitohe} completes the proof.
\end{proof}

Note that in fact the proof of Theorem~\ref{thm:borel-tensor-koszul} shows:

\begin{lem}\label{lem:simplifiedmodel}
Let $Y_1$ and $Y_2$ be complexes in $\bKom_{R[Q]}$. If $X_1$ and $X_2$ are any pair of complexes in $\bKom_{R[\tau]/(\tau^2 + 1)}$ such that $\Br(X_i)$ is quasi-isomorphic to $Y_i$, then $Y_1 \otimes_{\Br} Y_2$ is quasi-isomorphic to $\Br(X_1 \otimes X_2)$. 
\end{lem}

\begin{proof}
By Theorem~\ref{thm:derived-equivalence}, if $\Br(X_i)$ is quasi-isomorphic to $Y_i$, then $X_i$ is quasi-isomorphic to $\Hh(Y_i)$. The claim is then immediate from the fact that Definitions~\ref{def:tensor-product-tau} and \ref{def:tensor-product-borel}  respect quasi-isomorphism.
\end{proof}

\subsubsection{Local equivalence}
We close with a discussion relating $\Br$ and $\Hh$ to local equivalence. As usual, we say that a complex in $\bKom_{R[\tau]/(\tau^2 + 1)}$ is \textit{knotlike} if after inverting $u$, the resulting homology is isomorphic to $\f[u, u^{-1}]$ as an $\f[u]$-module. A map between knotlike $R[\tau]/(\tau^2 + 1)$-complexes is said to be \textit{local} if it induces an isomorphism on homology after inverting $u$. 
We note that in the present context, the notion of a local map differs from Definition~\ref{def:tau-complex}, where we consider maps homotopy commuting with $\tau$. The category $\bKom_{R[\tau]/(\tau^2+1)}$ consists of chain complexes, where both differentials and morphisms
commute with $\tau$ in the strict sense. 

\begin{lem}\label{lem:koszul-local}
The functors $\Br$ and $\Hh$ take local maps to local maps. 
\end{lem}
\begin{proof}
The results of this section hold equally well replacing $R = \f[u]$ with $R = \f[u, u^{-1}]$. In particular, the functors $\Br$ and $\Hh$ still take quasi-isomorphisms to quasi-isomorphisms after setting $R = \f[u, u^{-1}]$, which implies the claim.
\end{proof}

\begin{lem}\label{lem:local-tensor-product}
Suppose there exist local maps $f \colon Y_1 \rightarrow Y_1'$ and $g \colon Y_2 \rightarrow Y_2'$ in $\bKom_{R[Q]}$. Then there exists a local map $Y_1 \otimes Y_1' \rightarrow Y_2 \otimes Y_2'$ of grading shift $\gr(f) + \gr(g)$.
\end{lem}
\begin{proof}
Since $\Hh$ sends local maps to local maps, we obtain local maps $\Hh(Y_1) \rightarrow \Hh(Y_1')$ and $\Hh(Y_2) \rightarrow \Hh(Y_2')$. It is clear that the tensor product operation of Definition~\ref{def:tensor-product-tau} preserves locality. The claim then follows from the fact that $\Br$ sends local maps to local maps.
\end{proof}

\subsection{Equivariant genus}\label{subsec:bounding-strict}
We now prove Theorem~\ref{thm:kyle_knot_intro}. This will be based on the following fundamental algebraic calculation. In Lemma~\ref{lem:partial-computation-J} we constructed a local map from the complex $C_Q$ of Example~\ref{ex:borel-ex4} -- whose underlying complex $(C, \tau)$ we display again in Figure~\ref{fig:tensor-calculation} -- into $\Kcrm_Q(J)$. Recall that $C_Q$ has a nontrivial $\wt{s}_Q$-invariant; namely, $\wt{s}_{Q, 1, 2}(C_Q) = 2$. The first order of business will be to establish a similar calculation for the $m$-fold tensor product  $C_Q^m = \otimes_m C_Q$. Note that by Lemma~\ref{lem:simplifiedmodel}, we may calculate the quasi-isomorphism class of $C_Q^m$ by taking the self-tensor product of $(C, \tau)$ (in the sense of Definition~\ref{def:tensor-product-tau}) and applying the Borel construction.


\begin{figure}[h!]
\includegraphics[scale = 1]{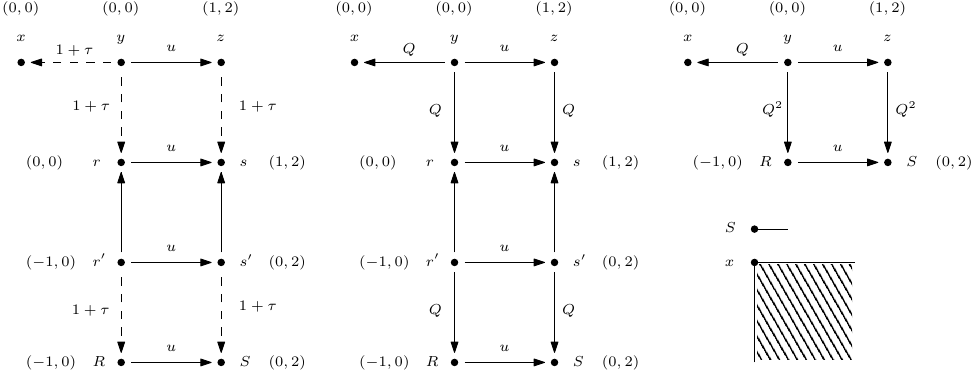}
\caption{Left: $(C, \tau)$. Middle: $C_Q$. Upper right: a simplified model for $C_Q$; by abuse of notation, we label these generators also by $x$, $y$, $z$, $R$, and $S$. Lower right: $H_*(C_Q) \cong \f[u, Q]_{(0,0)} \oplus \f[u, Q]/(u, Q^2)_{(0,2)}$.}\label{fig:tensor-calculation}
\end{figure}

\begin{lem}\label{lem:technical-calculation}
For all $m \geq 1$, 
\[
x^m = x \otimes \cdots \otimes x \in C_Q^m
\]
is a cycle in bigrading $(0, 0)$ with $u$-nontorsion homology class. Moreover, $Q^m x^m$ is homologous to a cycle of the form
\[
\sum_{i = 0}^{m} u^{m - i} Q^{2i} x_i.
\]
\end{lem}
\begin{proof}
We first verify that $x^m$ is a cycle. This is clear for $m = 1$; we proceed by induction. Indeed, more generally, suppose that $\zeta \in C_Q^{m-1}$ is a cycle and consider $\zeta \otimes x$. Following Lemma~\ref{lem:modified-product} and suppressing subscripts, we have
\begin{align*}
\partial_Q(\zeta \otimes x ) = (\partial_Q \zeta) \otimes x + \zeta \otimes (\partial_Q x) + Q (1 + \tau)\zeta \otimes (1 + \tau)x.
\end{align*}
Since $\partial_Q \zeta = 0$ by assumption, $\partial_Q x = 0$ by inspection, and $(1 + \tau)x = 0$ by inspection, this shows $\zeta \otimes x$ is a cycle. The fact that $x^m$ has $u$-nontorsion class is clear since this is obviously true after setting $Q = 0$. 

To prove the second claim, we also proceed by induction. For the case $m = 1$, we observe that
\[
\partial_Q(y + Qr') = uz + Qx + uQs' + Q^2R.
\]
and hence $Qx$ is homologous to the cycle
\[
u(z + Qs') + Q^2R
\]
as claimed. For the inductive step, consider in general the structure of $\zeta \otimes Qx$ for a cycle $\zeta$ of $C_Q^{m-1}$. We observe the following algebraic identities:
\begin{align*}
\partial_Q(\zeta \otimes (y + Qr')) & = (\partial_Q \zeta) \otimes (y + Qr') + \zeta \otimes \partial_Q(y + Qr') + Q (1 + \tau)\zeta \otimes (1 + \tau)(y + Qr') \\
& = 0 + \zeta \otimes \partial_Q(y + Qr') + Q (1 + \tau)\zeta \otimes (1 + \tau)(y + Qr') \\
& = \zeta \otimes (uz + Qx + uQs' + Q^2R) + Q(1 + \tau)\zeta \otimes (x + r + QR).
\end{align*}
and
\begin{align*}
\partial_Q((1 + \tau)\zeta \otimes y) & = \partial_Q((1 + \tau)\zeta) \otimes y + (1 + \tau)\zeta \otimes \partial_Q y + Q (1 + \tau)^2 \zeta \otimes (1 + \tau)y \\
& = 0 + (1 + \tau)\zeta \otimes \partial_Q y + 0 \\
&=  (1 + \tau)\zeta \otimes (uz + Qx + Qr).
\end{align*}
Summing these together and simplifying, we see that the boundary of $\zeta \otimes (y + Qr') + (1 + \tau)\zeta \otimes y$ is thus given by
\[
\zeta \otimes (uz + Qx + uQs' + Q^2R) + Q(1 + \tau) \zeta \otimes QR + (1 + \tau)\zeta \otimes uz.
\]
Rearranging, this means that $\zeta \otimes Qx$ is homologous to the cycle
\[
u( \zeta \otimes z + (1 + \tau)\zeta \otimes z  + Q \zeta \otimes s' ) + Q^2(\zeta \otimes R + (1 + \tau)\zeta \otimes R).
\]
Letting $\zeta = (Qx)^{m-1}$ and substituting the inductive hypothesis into the above expression gives the claim. 
\end{proof}

It follows immediately that $C_Q^m$ has a $\wt{s}_Q$-invariant that grows with $m$:

\begin{lem}\label{lem:s-tensor-bound}
For all $m \geq 1$, we have $\wt{s}_{Q, m, m+1}(C_Q^m) \geq 2 \lceil m/2 \rceil$.
\end{lem}
\begin{proof}
Since the $\wt{s}_Q$-invariants are preserved under quasi-isomorphism, without loss of generality we may use Lemma~\ref{lem:technical-calculation} to compute $\wt{s}_Q$. As we have seen, $Q^m x^m$ has $u$-nontorsion class and is homologous to a chain of the form
\[
\sum_{i = 0}^{m} u^{m - i} Q^{2i} x_i.
\]
Modulo $Q^{m+1}$, the above expression reduces to
\[
\sum_{i = 0}^{\lfloor m/2 \rfloor} u^{m - i} Q^{2i} x_i.
\]
The least power of $u$ appearing in this chain is $m - \lfloor m/2 \rfloor = \lceil m/2 \rceil$. Thus, modulo $Q^{m+1}$, the cycle $Q^m x^m$ is actually in the image of $u^{\lceil m/2 \rceil}$. It follows that $\wt{s}_{Q, m, m+1}(C_Q^m) \geq 2 \lceil m/2 \rceil$.
\end{proof}

We now use Lemma~\ref{lem:s-tensor-bound} to bound the equivariant genus of $\#_mJ$.

\begin{lem}\label{lem:eq-gen-lower-bound}
For all $m \geq 1$, we have $\eg(\#_mJ) \geq \lceil m/2 \rceil$.
\end{lem}
\begin{proof}
In Lemma~\ref{lem:partial-computation-J}, we showed that there is a local map of grading shift zero from the Borel complex $C_Q$ of Example~\ref{ex:borel-ex4} into $\smash{\Kcrm_Q(J)}$. By Lemma~\ref{lem:local-tensor-product}, this gives a local map of grading shift zero from $C_Q^m$ into $\Kcrm_Q(\#_mJ)$. Lemma~\ref{lem:s-tensor-bound} thus gives
\[
2\lceil m/2 \rceil \leq \wt{s}_{Q, m, m + 1}(C_Q^m) \leq \wt{s}_{Q, m, m + 1}(\#_mJ).
\]
The claim then follows from Theorem~\ref{thm:equivariant_s_borel_intro}.
\end{proof}

To prove Theorem~\ref{thm:kyle_knot_intro}, it remains to show that the \textit{isotopy}-equivariant genus of $\#_mJ$ remains bounded. This is a straightforward topological argument:

\begin{lem}\label{lem:isoeq-genus-one}
For all $m \geq 1$, we have $\ieg(\#_mJ) \leq 1$.
\end{lem}
\begin{proof}
  As discussed in Example~\ref{ex:kyle_knot}, we have a symmetric pair of slice disks $D$ and $\tau(D)$ for $J$ given by compressing along the curves on the Seifert surface shown in Figure~\ref{fig:kyle}. While these disks are not isotopic to each other, they are isotopic after a single stabilization. Indeed, un-compressing $D$ gives a slice surface isotopic to the pushed-in Seifert surface for $J$, as does un-compressing $\tau(D)$. Moreover, note that since $B^4 - D$ has fundamental group $\Z$, there is no ambiguity when we speak of stabilizing $D$; we claim that two stabilizing 1-handles are isotopic (if we allow the feet to move along $D$).
  
  The argument for this claim appears in \cite{CP,ConwayPowell}, but since their work is set in topological category, we quickly review their reasoning. A stabilization $\Sigma\# h$ of a surface $\Sigma$ depends on the choice of an isotopy class of the cocore of $h$, which is a curve, say $\gamma$ in $B^4$, whose boundary is in $\Sigma$ and whose interior is disjoint from $\Sigma$. Any two such curves $\gamma,\gamma'$ differ
  by an element in $\pi_1(B^4\setminus\Sigma)$: we choose the endpoints of $\gamma,\gamma'$ to agree. Now, $\pi_1(B^4\setminus\Sigma)=\Z$; by Seifert-van Kampen theorem one checks that the meridian of $\Sigma$ (the fiber of the sphere bundle associated with the normal bundle of $\Sigma$).
  Two arcs $\gamma$ and $\gamma'$ differing by the meridian, are easily seen to yield the isotopic stabilizations, see Figure~\ref{fig:two_arcs}. Therefore, we may always assume that $\gamma-\gamma'=0\in\pi_1(B^4\setminus\Sigma)$; that is, $\gamma$ and $\gamma'$ are homotopic
  rel boundary. Now homotopy rel boundary of arcs in a 4-dimensional manifold implies isotopy, proving the claim.
  \begin{figure}[h!]
    \includegraphics[width=6cm]{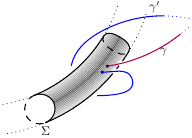}
    \caption{Proof of Lemma~\ref{lem:isoeq-genus-one}. The stabilization along two arcs $\gamma,\gamma'$ yields isotopic surfaces.}\label{fig:two_arcs}
  \end{figure}
%

In particular, we deduce that $D \# h \simeq \tau(D) \# h$ rel boundary, where we write $\# h$ for the unambiguous operation of stabilizing by a 1-handle. Now consider the boundary sum $\natural_m D$. Note that $B^4 - \natural_m D$ likewise has fundamental group $\Z$, so once again we may write $\# h$ for the unambiguous operation of stabilizing by any 1-handle. Observe that:
\begin{align*}
(\tau(D) \natural \tau(D) \natural \cdots \natural \tau(D)) \# h &\simeq (\tau(D) \# h) \natural \tau(D) \natural \cdots \natural \tau(D) \\
&\simeq (D \# h) \natural \tau(D) \natural \cdots \natural \tau(D) \\
&\simeq D \natural (\tau(D) \# h) \natural \cdots \natural \tau(D) \\
&\simeq D \natural (D \# h) \natural \cdots \natural \tau(D) \\
&\hspace{0.17cm} \vdots \\
&\simeq (D \natural D \natural \cdots \natural D) \# h.
\end{align*}
Roughly speaking, the idea is that we begin by moving the stabilizing $1$-handle $h$ onto the first copy of $\tau(D)$ in the boundary sum. As discussed in the previous paragraph, this allows us to effect the isotopy $\tau(D) \# h \simeq D \# h$. We then move the stabilizing $1$-handle onto the second factor of $\tau(D)$ in the boundary sum and again use the isotopy $\tau(D) \# h \simeq D \# h$. Continuing in this manner, we see that $\natural_m( \tau(D) ) \# h ) \simeq (\natural_m D) \# h$. Since (up to isotopy) these genus-one surfaces are related by $\tau$, this shows $\ieg(\#_mJ) \leq 1$. 
\end{proof}

We thus conclude:

\begin{proof}[Proof of Theorem~\ref{thm:kyle_knot_intro}]
Follows from Lemmas~\ref{lem:eq-gen-lower-bound} and \ref{lem:isoeq-genus-one}.
\end{proof}


\section{The Lobb-Watson Filtration}\label{sec:lobb-watson-filtration}

We now review the work of Lobb-Watson \cite{lobb-watson}. Our ultimate goal will be to merge the Lobb-Watson construction with the Borel formalism of the previous section. We first present a naive method for combining the two, before explaining several problems and subtleties that arise. Our discussion will lay the groundwork for the more sophisticated notion of a \textit{mixed complex} in the sequel.

\subsection{The work of Lobb-Watson}\label{sub:lw_work}
Let $L$ be an involutive link and $D$ be a transvergent diagram for $L$. The first input to the Lobb-Watson invariant is the construction of an additional half-integer grading $\deg_k$ on the Khovanov complex $\CKh(D)$. This is defined as follows. Let $o(c)$ and $u(c)$ be the half-integer functions on the set of crossings of $D$ given by:

\begin{enumerate}
	\item Suppose $c$ is off the axis of symmetry.  Then $o(c) = 0$ and $u(c) = \pm \frac{1}{2}$, according to whether $c$ is a positive or negative crossing, respectively.
	\item Suppose $c$ is on-axis and the involution reverses the orientation of $c$. Then $o(c) = 0$ and $u(c) = \pm 1$, according to whether $c$ is a positive or negative crossing, respectively.
	\item Suppose $c$ is on-axis and the involution preserves the orientation of $c$. Then $o(c) = \mp \frac{1}{2}$ and $v(c) = \pm \frac{1}{2}$, according to whether $c$ is a positive or negative crossing, respectively.
\end{enumerate}

Now consider a resolution $D_v$ of $D$. Within $D_v$, each crossing $c$ of $D$ is either smoothed in an oriented manner or an anti-oriented manner. Define 
\[
\deg_k(D_v) = \sum_{\substack{c \text{ smoothed}\\ \text{in an oriented} \\ \text{manner in } D_v}} o(c) \quad + \sum_{\substack{c \text{ smoothed}\\ \text{in an anti-oriented} \\ \text{manner in } D_v}} u(c).
\]
That is, each oriented smoothing in $D_v$ contributes the value $o(c)$ to $\deg_k$, while each anti-oriented smoothing contributes $u(c)$. We give each generator in $\mathfrak{F}(v)$ the $k$-grading $\deg_k(D_v)$, so that $\deg_k$ is constant within each vertex of the cube of resolutions.

Lobb-Watson then define a perturbed differential
\[
\partial_\mathrm{tot} = \partial + (1 + \tau)
\]
on $\CKh(D)$. Here, the map $\tau$ is as in Definition~\ref{def:diagrammatic_tau} (except on the Khovanov complex, rather than the Bar-Natan complex) and hence $\partial_\mathrm{tot}^2 = 0$. Note that $\partial_\mathrm{tot}$ is \textit{not} homogenous with respect to the homological grading -- in this setting, the homological grading on $\CKh(D)$ is relaxed to a homological filtration $\mathcal{F}$, with respect to which $\partial_\mathrm{tot}$ is non-decreasing. Lobb-Watson also show that $\partial_\mathrm{tot}$ is non-decreasing with respect to $\deg_k$, thus leading to a second filtration $\mathcal{G}$ on $\CKh(D)$. Taking the associated graded object with respect to this bifiltration (and keeping in mind the quantum grading) gives a triply-graded object $\mathbb{H}^{i, j, k}(D)$, which Lobb-Watson show is an invariant of the pair $(L, \tau)$.

For small knots, it is often the filtration which captures the behavior of $\tau$, rather than $\tau$ itself. Indeed, for many knots of low crossing number, one can check that the action of $\tau$ is the identity on Khovanov homology. Nevertheless, Lobb-Watson show that $\mathbb{H}^{i, j, k}$ can be interesting for many knots of low crossing number \cite[Section 6.5]{lobb-watson}, whereas the first instance of a non-trivial homological action of $\tau$ occurs for $9_{46}$. It is thus natural to bring together the Lobb-Watson filtration with the Borel construction so as to leverage the non-triviality of $\mathbb{H}^{i, k, j}$ for small knots. In the rest of this paper, we suggest an outline of this theory, leaving further development to future work.


\subsection{Encoding the filtration as an additional variable}

Our first goal is to promote the Lobb-Watson grading to a third grading on the Borel complex. Indeed, observe that there is no difficulty in defining $\deg_k$ on $\Kcm_Q(D)$, simply by setting $\deg_k(u) = \deg_k(Q) = 0$. The fact that $\partial_{\mathrm{tot}}$ is filtered with respect to $\deg_k$ on $\CKh(D)$ obviously implies the same for $\partial_Q$. We make $\partial_Q$ into a $\deg_k$-preserving map by introducing a new variable $W$ with $\deg_k(W) = -1/2$ and inserting powers of $W$ into the expression for $\partial_Q$ as appropriate. More precisely:

\begin{defn}\label{def:QWcomplex}
Let $C_Q$ be a Borel complex. Suppose that $C_Q$ is equipped with an additional half-integer grading $\deg_k$ with respect to which $\partial_Q$ is filtered and $\deg_k(u) = \deg_k(Q) = 0$. Set
\[
C_{Q, W} = C_Q \otimes_{\f} \f[W].
\]
Let $\deg_k(W) = -1/2$. To define a differential $\partial_{Q, W}$ on $C_{Q, W}$, pick a homogenous basis for $C_Q$ over $\f[u, Q]$. This gives a basis for $C_{Q, W}$ over $\f[u, Q, W]$. For each basis element $x$, write $\partial_Q x$ as a sum of basis elements $\sum_i c_i y_i$. We then let
\[
\partial_{Q, W} x = \sum_i W^{2(\deg_k(y_i) - \deg_k(x))} c_i y_i,
\]
extending linearly. Note here that since $\partial_Q$ is non-decreasing with respect to $\deg_k$, each $W$-exponent in this sum is non-negative. It is straightforward to check that this makes $C_{Q, W}$ a complex over $\f[u, Q, W]$ such that $\partial_{Q, W}$ is $\deg_k$-preserving.
\end{defn}

Carrying this out in the context of involutive links, we have:

\begin{defn}\label{def:QWdiagram}
Let $D$ be a transvergent diagram for an involutive link $L$. Define $\smash{\Kcm_{Q, W}(D)}$ by applying Definition~\ref{def:QWcomplex} to the Borel complex $\smash{\Kcm_Q(D)}$ equipped with the Lobb-Watson filtration. Likewise, define $\smash{\Kcrm_{Q, W}(D)}$ by applying Definition~\ref{def:QWcomplex} to the reduced Borel complex $\smash{\Kcrm_Q(D)}$. We refer to these as complexes as the \textit{$(Q, W)$-complex} and \textit{reduced $(Q, W)$-complex} associated to $D$, respectively.
\end{defn}

\begin{rmk}\label{rem:reducedQWdefined}
To establish independence of the choice of basepoint in the reduced case, note that for each $p \in D$, we obtain a reduced subcomplex of $\Kcm_Q(D)$ that inherits the Lobb-Watson filtration. It is clear that the unpointed reduced complex inherits the Lobb-Watson filtration, simply by defining $\deg_k$ on each resolution of $D$ in the usual way. The isomorphism between unpointed and pointed reduced sends generators in each given resolution of $D$ to generators in the same resolution, and is hence $\deg_k$-preserving. Thus we may unambiguously speak of the $\deg_k$-grading on the reduced complex. This shows that the isomorphism class of $\smash{\Kcrm_{Q, W}(D)}$ is well-defined, as desired.
\end{rmk}

For future reference, it will be useful to note that we can perform a similar operation whenever we have a chain map between two Borel complexes:

\begin{defn}\label{def:mapWlift}
Let $(C_1)_Q$ and $(C_2)_Q$ be two Borel complexes, each equipped with an additional filtration $\deg_k$ as in Definition~\ref{def:QWcomplex}. Let $f_Q \colon (C_1)_Q \rightarrow (C_2)_Q$ be a chain map. Pick homogenous bases for $(C_1)_Q$ and $(C_2)_Q$ over $\f[u, Q]$; these give bases for $(C_1)_{Q, W}$ and $(C_2)_{Q, W}$ over $\f[u, Q, W]$. For each basis element $x$, write $\partial_Q x$ as a sum of basis elements $f_Q(x) = \sum_i c_i y_i$. We then let
\[
f_{Q, W}(x) = \sum_i W^{2(\deg_k(y_i) - \deg_k(x))} c_i y_i,
\]
extending linearly. If $f$ is filtered, then this defines a map $f_{Q,W} \colon (C_1)_{Q,W} \rightarrow (C_2)_{Q,W}$. If $f$ is not filtered, then some of the powers of $W$ in the above equation may be negative. In this case, $f_{Q,W}$ only exists as a map $f_{Q,W} \colon W^{-1} (C_1)_{Q,W} \rightarrow W^{-1}(C_2)_{Q,W}$.
\end{defn}

\begin{rmk}
Strictly speaking, Definitions~\ref{def:QWcomplex} and \ref{def:mapWlift} are dependent on a choice of basis. However, we will always have in mind the standard basis coming from the construction of the Bar-Natan complex via the cube of resolutions. 
\end{rmk}

\subsection{Examples and non-invariance}\label{sub:bad_unknot}
It is clear that $\smash{\Kcm_{Q, W}(D)}$ encodes both the Borel complex and (in some sense) the Lobb-Watson filtration. Unfortunately, it turns out that the literal homotopy type of $\smash{\Kcm_{Q, W}(D)}$ is not an invariant of $D$, even for the unknot:

\begin{example}\label{ex:non_inv}
Let $U$ be the trivial equivariant unknot diagram. Then 
\[
\Kcm_{Q, W}(U) = \f[u, Q, W]_{(0, 1, 0)} \oplus \f[u, Q, W]_{(0, -1, 0)}.
\]
\end{example}

\begin{example}\label{ex:badunknot}
Let $U'$ be the equivariant unknot diagram displayed on the left in Figure~\ref{fig:badunknot}. Note that this differs from $U$ by a single R1-move. The complex $\smash{\Kcm_{Q, W}(U')}$ is displayed on the right in Figure~\ref{fig:badunknot}. It is easily checked that the homology of $\smash{\Kcm_{Q, W}(U')}$ is given by
\[
\f[u, Q, W]_{(0, 1, 0)} \oplus \f[u, Q, W]_{(0, -1, 0)} \oplus (\f[u, Q, W]/W^2)_{(1, 3, 1)} \oplus (\f[u, Q, W]/W^2)_{(1, 1, 1)}.
\]
Clearly, $\Kcm_{Q, W}(U)$ and $\Kcm_{Q, W}(U')$ are not homotopy equivalent. 
\begin{rmk}
  The diagram for $U'$ can be put in a braid form, showing that the complex $\Kcm_{Q,W}$ is not invariant even if we restrict to
  diagrams of palindromic braids
  of \cite{Merz}.
\end{rmk}

\begin{figure}[h!]
\includegraphics[scale = 1]{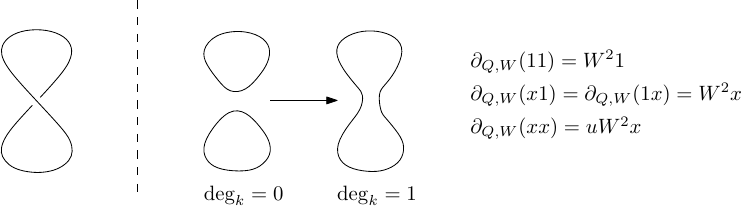}
\caption{Left: A non-trivial equivariant diagram for the unknot. The symmetry is given by rotation about the obvious vertical axis. Right: the resulting $(Q, W)$-complex with the value of $\deg_k$ labeled under each resolution. In $\smash{\Kcm_Q(U')}$, the cycles $1$ and $x$ in the right-hand resolution are boundaries; in $\smash{\Kcm_{Q, W}(U')}$, they are no longer boundaries but are instead $W^2$-torsion in homology.}\label{fig:badunknot}
\end{figure}
\end{example}

It turns out to be quite difficult to describe in what sense $\smash{\Kcm_{Q, W}(D)}$ is an invariant. An indication for the correct answer is suggested by the behavior of $\smash{\Kcm_{Q, W}(D)}$ under equivariant Reidemeister moves. The first question to ask is whether each Reidemeister map of Borel complexes is filtered with respect to $\deg_k$. While this is true for the IR1, IR2, IR3, R1, and R2 moves, it turns out that the maps for the M1, M2, and M3 moves defined in Section~\ref{sec:miinvariance} are \textit{not} filtered with respect to $\deg_k$. Hence in general if $D$ and $D'$ are two transvergent diagrams for the same knot, then Definition~\ref{def:mapWlift} only gives a map between $W$-localizations:
\[
f_Q \colon W^{-1} \Kcm_{Q, W}(D) \rightarrow W^{-1} \Kcm_{Q, W} (D').
\]
While this map is indeed a homotopy equivalence, it contains no more information than the usual homotopy equivalence between the Borel complexes for $D$ and $D'$. (This is why we again use the notation $f_Q$ in this situation, rather than $f_{Q, W}$.) Indeed, the reader may check that $\smash{W^{-1} \Kcm_{Q, W}(D)}$ is simply isomorphic to $\smash{\Kcm_Q(D) \otimes_{\f} \f[W, W^{-1}]}$.

Given this information, the reader may wonder how $\mathbb{H}^{i, j, k}$ can possibly be invariant of diagram. Lobb-Watson do establish constraints on the Reidemeister maps with respect to $\deg_k$, but the relation is somewhat subtle. In our context, their work shows that the maps associated to equivariant Reidemeister moves in Section~\ref{sec:miinvariance} can be \textit{homotoped} so that they are filtered. However, the catch is that this homotopy can only be done after inverting $Q$. If $D$ and $D'$ are two transvergent diagrams related by equivariant Reidemeister moves, we thus additionally obtain a map
\[
f_W \colon Q^{-1} \Kcm_{Q,W}(D) \rightarrow Q^{-1} \Kcm_{Q, W}(D')
\]
which is homotopic to $f_Q$ as a map from $Q^{-1}W^{-1} \Kcm_{Q, W}(D)$ to $Q^{-1}W^{-1} \Kcm_{Q, W}(D')$.

The mixed complex formalism will thus essentially consist of remembering the two localizations
\[
W^{-1} \Kcm_{Q, W}(D) \quad \text{and} \quad Q^{-1} \Kcm_{Q, W}(D).
\]
These are related inasmuch as they both arise from the same complex $\smash{\Kcm_{Q, W}}$. A map between two mixed complexes will be a pair of maps $(f_Q, f_W)$, each existing only on one of the two localizations, but which become homotopic after inverting \textit{both} $Q$ and $W$.


\section{Mixed Complexes}\label{sec:3}

We now introduce the mixed complex formalism and prove Theorem~\ref{thm:mixedcomplex}.

\subsection{Formal definitions} We begin with the definition of a mixed complex. Throughout, let $R = \f[u]$.

\begin{defn}\label{def:mixed}
A \emph{mixed complex} $M$ over $R$ is a tuple $(M_Q, M_W,\eta)$ such that:
      \begin{enumerate}
	\item $M_Q$ is a triply-graded, free finitely-generated complex over $R[Q,W,W^{-1}]$,
	\item $M_W$ is a triply-graded, free finite-generated complex over $R[Q,Q^{-1},W]$; and,
	\item $\eta\colon Q^{-1}M_Q\to W^{-1}M_W$ is a grading-preserving homotopy equivalence over $R[Q,Q^{-1},W,W^{-1}]$.
      \end{enumerate}
A mixed complex is thus a pair of complexes, together with a homotopy equivalence of their localizations. 
For both $M_Q$ and $M_W$, the grading lies in $\Z \times \Z \times \frac{1}{2} \Z$ and is denoted $(\deg_h, \deg_q, \deg_k)$. We refer to the first two as the homological and quantum grading, while in our application $\deg_k$ will be the Lobb-Watson grading. We require
\[
\deg(\partial) = (1, 0, 0), \quad \deg(u) = (0, -2, 0), \quad \deg(Q) = (1, 0, 0), \quad \deg(W) = (0, 0, -1/2).
\]
\end{defn}


\begin{example}\label{ex:simple_mixed}
As indicated in the previous section, all mixed complexes in our setting will arise out of an underlying complex over $R[Q, W]$. That is, let $C$ be a triply-graded complex over $R[Q,W]$ with the same grading properties as Definition~\ref{def:mixed}. Define a mixed complex $M(C)$ by setting
  \[
  M_Q = W^{-1} C \quad \text{and} \quad M_W = Q^{-1}C.
  \]
Then $Q^{-1} M_Q = Q^{-1} W^{-1} C = W^{-1} M_W$ and we let $\eta$ the identity map from $Q^{-1} W^{-1} C$ to itself. 
\end{example}

A morphism of mixed complexes is a pair of morphisms between the corresponding $M_Q$- and $M_W$-complexes, with a suitable agreement on localizations. This is made precise below:
\begin{defn}\label{def:mixed-morphism}
     A \emph{morphism between mixed complexes} $X$ and $Y$ is a triple $\frf = (f_Q,f_W,h)$ consisting of a pair of morphisms $f_Q\colon X_Q\to Y_Q$ and $f_W\colon X_W\to Y_W$, such that the diagram
      \[\begin{tikzcd}
	{Q^{-1}X_Q} && {W^{-1}X_W} \\
	\\
	{Q^{-1}Y_Q} && {W^{-1}Y_W}
	\arrow["{\eta^1}", from=1-1, to=1-3]
	\arrow["{f_Q}"', from=1-1, to=3-1]
	\arrow["{f_W}", from=1-3, to=3-3]
	\arrow["{\eta^2}", from=3-1, to=3-3]
\end{tikzcd}\]
commutes up to the homotopy $h\colon Q^{-1}X_Q\to W^{-1}Y_W$:
\[
\partial h+ h\partial =\eta^2f_Q+f_W\eta^1.
\]
 We require $f_Q$ and $f_W$ to preserve the homological and $\deg_k$-grading, and be homogenous (of the same degree) with respect to the quantum grading. We refer to $\deg_q(f_Q) = \deg_q(f_W)$ as the \textit{grading shift} of $\frf$. We say that $\frf$ is \textit{local} if $f_Q$ induces an isomorphism 
\[
f_Q \colon u^{-1} H_*(X_Q) \rightarrow u^{-1} H_*(Y_Q).
\]
\end{defn}

To compose morphisms we use the following recipe:
\begin{defn}\label{def:composition}
  Suppose $\frf^i = (f^i_Q,f^i_W,h^i)$ for $i=1,2$ are morphisms of mixed complexes from $X$ to $Y$ and $Y$ to $Z$, respectively.
  We define the composition
  \[
  \frf^2\circ \frf^1 = (f^{12}_Q,f^{12}_W,h^{12}) \colon X \rightarrow Z,
  \]
  where
  
  \[
  f^{12}_Q=f^2_Q\circ f^1_Q, \quad f^{12}_W=f^2_W\circ f^1_W, \quad \text{and} \quad h^{12}=f^2_Wh^1+h^2f^1_Q.
  \]
\end{defn}

\begin{lem}
  The composition of morphisms is associative.
\end{lem}
\begin{proof}
Suppose $\frf^1,\frf^2$, and $\frf^3$ are three maps between mixed complexes. Write $\frf^{12}=\frf^2\circ \frf^1$ and $\frf^{23}=\frf^{3}\circ \frf^2$, and let $\frf^{1(23)}=\frf^{32}\circ \frf^1$ and $\frf^{(12)3}=\frf^3\circ \frf^{21}$. As the composition of the $Q$- and $W$-maps are obviously associative, it remains to check that the homotopies $h^{1(23)}$ and $h^{(12)3}$ are the same. Explicitly,
   \begin{align*}
    h^{1(23)}&=f^{23}_Wh^1+h^{23}f^1_Q=f^3_Wf^2_Wh^1+f^3_Wh^2f^1_Q+h^3f^2_Qf^1_Q\\
    h^{(12)3}&=f^{3}_Wh^{12}+h^{3}f^{12}_Q=f^3_Wf^2_Wh^1+f^3_Wh^2f^1_Q+h^3f^2_Qf^1_Q
  \end{align*}
  as desired.
\end{proof}

We now define the principal notion of equivalence of mixed complexes which we will use in this paper:

\begin{defn}\label{def:Qequivalence}
 We say that $X$ and $Y$ are \emph{$Q$-equivalent} if there exist mixed complex morphisms $\frf \colon X \rightarrow Y$ and $\frg \colon Y \rightarrow X$ such that their $Q$-maps $f_Q$ and $g_Q$ satisfy
 \[
 g_Q \circ f_Q\simeq \id_{X_Q} \quad \text{and} \quad f_Q \circ g_Q \simeq \id_{Y_Q}.
 \]
We refer to $\frf$ and $\frg$ as \textit{$Q$-equivalences}.
\end{defn}

A $Q$-equivalence is thus a pair of mixed complex morphisms whose $Q$-components $f_Q$ and $g_Q$ are homotopy inverses of $R[Q, W, W^{-1}]$-complexes in the usual sense. Note that in order to conclude that a $Q$-equivalence exists, it does \textit{not} suffice to simply present a homotopy equivalence between $X_Q$ and $Y_Q$ \textit{in vacuo}. Instead, we must find homotopy equivalences $f_Q$ and $g_Q$ that can be completed to a pair of mixed complex morphisms between $X$ and $Y$. This means we must additionally produce the corresponding $W$-side maps $f_W$ and $g_W$ satisfying the homotopy commutative square of Definition~\ref{def:mixed-morphism}. Such maps cannot always be found -- indeed, it may be helpful for the reader to observe the following:

\begin{lem}\label{lem:Wtorsion}
Let $\frf$ and $\frg$ be $Q$-equivalences between $X$ and $Y$. Then their $W$-maps $f_W$ and $g_W$ induce isomorphisms modulo $W$-torsion:
  \[
  (f_W)_*\colon H_*(X_W)/\Torso_W\xrightarrow{\cong} H_*(Y_W)/\Torso_W \quad \text{and} \quad (g_W)_*\colon H_*(Y_W)/\Torso_W\xrightarrow{\cong} H_*(X_W)/\Torso_W.
  \]
\end{lem}
\begin{proof}
Since $(f_Q)_*$ and $(g_Q)_*$ are inverse as maps between $H_*(X_Q)$ and $H_*(Y_Q)$, the same is true after inverting $Q$. By the compatibility condition of Definition~\ref{def:mixed-morphism}, this means that the maps $(f_W)_*$ and $(g_W)_*$ are inverse as maps between $W^{-1} H_*(X_W)$ and $W^{-1} H_*(Y_W)$. It easily follows that $(f_W)_*$ and $(g_W)_*$ are isomorphisms between $H_*(X_W)$ and $H_*(Y_W)$ modulo $W$-torsion.
\end{proof}

Although Lemma~\ref{lem:Wtorsion} will not be used in this paper, it indicates that $Q$-equivalence places a non-trivial requirement on the relation between $X_W$ and $Y_W$. Definition~\ref{def:Qequivalence} is thus more stringent than simply having a homotopy equivalence of Borel complexes. In our context, note that if we have a homotopy equivalence of Borel complexes whose maps are filtered with respect to $\deg_k$, then these may be promoted to maps of $(Q, W)$-complexes. This yields the desired $Q$- and $W$-side maps after localizing at $W$ and $Q$, respectively.

It is still useful to define homotopy equivalences of mixed complexes. For this, we first consider the following notion: 

\begin{defn}\label{def:homotopy-of-mixed}
Let $\frf^i=(f^i_Q,f^i_W,h^i)$ for $i=1,2$ be two morphisms of mixed complexes from $X$ to $Y$. A \textit{chain homotopy} from $\frf^1$ to $\frf^2$ consists of:
\begin{enumerate}
\item A chain homotopy $H_Q$ between the $Q$-maps $f^1_Q \colon X_Q \rightarrow Y_Q$ and $f^2_Q \colon X_Q \rightarrow Y_Q$;
\item A chain homotopy $H_W$ between the $W$-maps $f^1_W \colon X_W \rightarrow Y_W$ and $f^2_W \colon X_W \rightarrow Y_W$;
\item A diagonal map $\Delta \colon Q^{-1} X_Q \rightarrow W^{-1} Y_W$ such that
\[
\partial \Delta + \Delta \partial = \eta_Y H_Q + H_W \eta_X + h^1 + h^2.
\]
after $H_Q$ and $H_W$ have been localized at $Q$ and $W$, respectively.
\end{enumerate}
It may be helpful to consider these maps as fitting into the homotopy coherent cube:

\[
\begin{tikzcd}
	& {Q^{-1}X_Q} &&&& {W^{-1}X_W} \\
	\\
	{Q^{-1}X_Q} &&&& {W^{-1}X_W} \\
	\\
	& {Q^{-1}Y_Q} &&&& {W^{-1}Y_W} \\
	\\
	{Q^{-1}Y_Q} &&&& {W^{-1}Y_W}
	\arrow["{\eta_X}"{description}, from=1-2, to=1-6]
	\arrow["{f^2_Q}"{description, pos=0.4}, dotted, from=1-2, to=5-2]
	\arrow["{h^2}"{pos=0.4}, dotted, from=1-2, to=5-6]
	\arrow["{f^2_W}"{description}, from=1-6, to=5-6]
	\arrow[from=3-1, to=1-2]
	\arrow["{\eta_X}"{description}, from=3-1, to=3-5]
	\arrow["{H_Q}"', dotted, from=3-1, to=5-2]
	\arrow["{f^1_Q}"{description}, from=3-1, to=7-1]
	\arrow["{h^1}"{description, pos=0.6}, from=3-1, to=7-5]
	\arrow[from=3-5, to=1-6]
	\arrow["{H_W}", shift left, from=3-5, to=5-6]
	\arrow["{f^1_W}"{description, pos=0.6}, from=3-5, to=7-5]
	\arrow["{\eta_Y}"{description}, dotted, from=5-2, to=5-6]
	\arrow[dotted, from=7-1, to=5-2]
	\arrow["{\eta_Y}"{description}, from=7-1, to=7-5]
	\arrow[from=7-5, to=5-6]
\end{tikzcd}
\]
	
Here, the front and back faces of the cube constitute the maps $\frf^1$ and $\frf^2$, respectively. The top and bottom faces are the identity, in the sense that each edge is assigned the identity map and the diagonal map is zero. The homotopies $H_Q$ and $H_W$ -- after localizing at $Q$ and $W$, respectively -- are the diagonal maps on the left and right faces. The map $D$ witnesses the fact that the sum of all two-step paths from the source to the sink of the cube is homotopic to zero. 
\end{defn}

We stress that in Definition~\ref{def:homotopy-of-mixed}, $H_Q$ and $H_W$ are required to be defined on $X_Q$ and $X_W$. However, it is helpful to also consider a slightly relaxed version of Definition~\ref{def:homotopy-of-mixed}: 

\begin{defn}\label{def:homotopy-delta}
Let $\frf^i=(f^i_Q,f^i_W,h^i)$ for $i=1,2$ be two mixed complex morphisms from $X$ to $Y$. We say that $\frf^1$ and $\frf^2$ are \textit{chain homotopic up to defect $\delta$} if there exist $H_Q$, $H_W$, and $\Delta$ as in Definition~\ref{def:homotopy-of-mixed}, but with the condition for $H_W$ replaced by the following:
\begin{enumerate}
\item[(2$^*$)] A map $H_W \colon W^{-1} X_W \rightarrow W^{-1} Y_W$ such that
\begin{enumerate}
\item $H_W$ is a chain homotopy between the $W$-localized maps 
\[
f^1_W \colon W^{-1}X_W \rightarrow W^{-1}Y_W \quad \text{and} \quad f^2_W \colon W^{-1}X_W \rightarrow W^{-1}Y_W;
\]
\item $H_W$ maps $X_W\subset W^{-1}X_W$ to $(W^{-2\delta})Y_W\subset W^{-1}Y_W$.
\end{enumerate}
\end{enumerate}
Here, $(W^{-2\delta})Y_W$ is the subcomplex of $W^{-1} Y_W$ consisting of elements such that multiplying by $W^{2 \delta}$ lands in $Y_W \subset W^{-1} Y_W$.
\end{defn}

The import of Definition~\ref{def:homotopy-delta} is that we no longer require the $W$-side maps $f^1_W$ and $f^2_W$ to be chain homotopic as maps from $X_W$ to $Y_W$, but instead as maps from $W^{-1}X_W$ to $W^{-1}Y_W$. However, we still exert some control over $H_W$ by requiring $H_W$ to map $X_W\subset W^{-1}X_W$ into the subcomplex $(W^{-2\delta})Y_W\subset W^{-1}Y_W$. Note that a chain homotopy up to defect $\delta = 0$ is a chain homotopy in the usual sense, using the fact that $X_W$ and $Y_W$ are free.

\begin{defn}\label{def:homotopy-equivalence-mixed}
We say two mixed complexes $X$ and $Y$ are \textit{homotopy equivalent} if there exist mixed complex morphisms $\frf$ and $\frg$ between them such that $\frg \circ \frf$ and $\frf \circ \frg$ are homotopic to the identity, as per Definition~\ref{def:homotopy-of-mixed}. We say $X$ and $Y$ are \textit{homotopy equivalent up to defect $\delta$} if there exist mixed complex morphisms $\frf$ and $\frg$ between them such that $\frg \circ \frf$ and $\frf \circ \frg$ are homotopic to the identity up to defect $\delta$, as per Definition~\ref{def:homotopy-delta}. Note that a homotopy equivalence up to defect $\delta = 0$ is a homotopy equivalence.
\end{defn}

Roughly speaking, the relation between Definiton~\ref{def:Qequivalence} and Definition~\ref{def:homotopy-equivalence-mixed} is the following. Consider a pair of mixed complexes $X$ and $Y$ such that $X_Q$ and $Y_Q$ are homotopy equivalent via maps $f_Q$ and $g_Q$. As we have seen, we can ask that $f_Q$ and $g_Q$ be promoted to maps $\frf$ and $\frg$ of mixed complexes. This results in the notion of $Q$-equivalence and requires finding compatible maps $f_W$ and $g_W$ on the $W$-side. Definition~\ref{def:homotopy-equivalence-mixed} requires the homotopies witnessing $g_Qf_Q \simeq \id$ and $f_Q g_Q \simeq \id$ to also lift to maps of mixed complexes. That is, we ask for $W$-side homotopies making $g_W f_W \simeq \id$ and $f_W g_W \simeq \id$ which satisfy the compatibility condition of Definition~\ref{def:homotopy-of-mixed}. It turns out, however, that finding such homotopies is often impossible. We thus sometimes allow our homotopies on the $W$-side to not quite be defined on $X_W$ and $Y_W$, and instead be defined on $W^{-1}X_W$ and $W^{-1}Y_W$ but with controlled behavior with respect to $W$.

We now list several technical results will be useful in our proof of M3 invariance. 

\begin{lem}\label{lem:composition-defect-delta}
Suppose $\frf^1$ and $\frf^2$ from $X$ to $Y$ are chain homotopic up to defect $\delta$, and likewise that $\frg^1$ and $\frg^2$ from $Y$ to $Z$ are chain homotopic up to defect $\delta$. Then $\frg^1 \circ \frf^1$ and $\frg^2 \circ \frf^2$ from $X$ to $Z$ are chain homotopic up to defect $\delta$.
\end{lem}
\begin{proof}
  Let $\frf^1$ and $\frf^2$ be chain homotopic (up to defect $\delta$) via the data $\smash{(H_Q^f, H_W^f, \Delta^f)}$ and $\frg^1$ and $\frg^2$ be chain homotopic (up to defect $\delta$) via the data $\smash{(H_Q^g, H_W^g, \Delta^g)}$. The chain homotopy between $\frg^1 \circ \frf^1$ and $\frg^2 \circ \frf^2$ can be computed by stacking the homotopy coherent cubes resulting from Definition~\ref{def:homotopy-of-mixed} above each other and then applying the hyperbox compression procedure of e.g.\ \cite[Section 5.2]{MOlink}. 

\[\begin{tikzcd}
	& {Q^{-1}X_Q} &&& {W^{-1}X_W} \\
	{Q^{-1}X_Q} &&& {W^{-1}X_W} \\
	\\
	& {Q^{-1}Y_Q} &&& {W^{-1}Y_W} \\
	{Q^{-1}Y_Q} &&& {W^{-1}Y_W} \\
	\\
	& {Q^{-1}Z_Q} &&& {W^{-1}Z_W} \\
	{Q^{-1}Z_Q} &&& {W^{-1}Z_W}
	\arrow[from=1-2, to=1-5]
	\arrow[dotted, from=1-2, to=4-2]
	\arrow[shift right, dotted, from=1-2, to=4-5]
	\arrow[from=1-5, to=4-5]
	\arrow[from=2-1, to=1-2]
	\arrow[from=2-1, to=2-4]
	\arrow[dotted, from=2-1, to=4-2]
	\arrow[from=2-1, to=5-1]
	\arrow[from=2-1, to=5-4]
	\arrow[from=2-4, to=1-5]
	\arrow[from=2-4, to=4-5]
	\arrow[from=2-4, to=5-4]
	\arrow[dotted, from=4-2, to=4-5]
	\arrow[dotted, from=4-2, to=7-2]
	\arrow[shift right, dotted, from=4-2, to=7-5]
	\arrow[from=4-5, to=7-5]
	\arrow[dotted, from=5-1, to=4-2]
	\arrow[from=5-1, to=5-4]
	\arrow[dotted, from=5-1, to=7-2]
	\arrow[from=5-1, to=8-1]
	\arrow[from=5-1, to=8-4]
	\arrow[from=5-4, to=4-5]
	\arrow[from=5-4, to=7-5]
	\arrow[from=5-4, to=8-4]
	\arrow[dotted, from=7-2, to=7-5]
	\arrow[dotted, from=8-1, to=7-2]
	\arrow[from=8-1, to=8-4]
	\arrow[from=8-4, to=7-5]
\end{tikzcd}\]

In particular, the $Q$- and $W$-side homotopies are given by
\[
H_Q^{gf} = H_Q^gf^1_Q + g^2_QH_Q^f \quad \text{and} \quad H_W^{gf} = H_W^gf^1_W + g^2_WH_W^f,
\]
where the $W$-side homotopy is defined after inverting $W$. Now, $\smash{f^1_W}$ maps $\smash{X_W \subset W^{-1}X_W}$ into $\smash{Y_W \subset W^{-1}Y_W}$, while $\smash{H_W^g}$ maps $\smash{Y_W \subset W^{-1}Y_W}$ into $\smash{(W^{-2\delta})Z_Q \subset W^{-1}Z_W}$. Hence we see that their composition $H_W^gf^1_W$ maps $\smash{X_W \subset W^{-1}X_W}$ into $\smash{(W^{-2\delta})Z_Q \subset W^{-1}Z_W}$. An analogous statement holds for $\smash{g^2_WH_W^f}$. But this shows that $\smash{H_W^{gf}}$ satisfies the requirements of Definition~\ref{def:homotopy-delta} with defect $\delta$, as desired.
\end{proof}

In particular:

\begin{lem}\label{lem:homotopy-defect-delta}
If $X$ and $Y$ are homotopy equivalent up to defect $\delta$ and $Y$ and $Z$ are homotopy equivalent up to defect $\delta$, then $X$ and $Z$ are homotopy equivalent up to defect $\delta$ via the composition of the relevant homotopy equivalences.
\end{lem}

\begin{proof}
Let $\frf$ and $\bar{\frf}$ be homotopy equivalences (up to defect $\delta$) between $X$ and $Y$ and $\frg$ and $\bar{\frg}$ be homotopy equivalences (up to defect $\delta$) between $Y$ and $Z$. Repeatedly applying Lemma~\ref{lem:composition-defect-delta} shows $(\bar{\frf} \circ \bar{\frg}) \circ (\frg \circ \frf) = \bar{\frf} \circ (\bar{\frg} \circ \frg) \circ \frf$ is homotopy equivalent to the identity up to defect $\delta$, and likewise for the composition in the other direction.
\end{proof}

We will sometimes precompose a chain homotopy up to defect $\delta$ with a mixed complex morphism whose $W$-side map lands in the image of $W^{2 \delta}$. This turns the chain homotopy up to defect $\delta$ into a genuine chain homotopy, as the following lemma indicates: 

\begin{lem}\label{lem:technical-composition}
Let $X$, $Y$, and $Z$ be three mixed complexes. Suppose:
\begin{enumerate}
\item $\frf$ is a mixed complex morphism from $X$ to $Y$ whose $W$-side map $f_W$ maps $X_W$ into $(W^{2\delta}) Y_W$.
\item $\frg^1$ and $\frg^2$ are mixed complex morphisms from $Y$ to $Z$ that are chain homotopic up to defect $\delta$.
\end{enumerate}
Then $\frg^1 \circ \frf$ and $\frg^2 \circ \frf$ are chain homotopic.
\end{lem}

\begin{proof}
We apply the argument of Lemma~\ref{lem:composition-defect-delta} to the pairs $\frf^1 = \frf^2 = \frf$ and $\frg^1$ and $\frg^2$. Then $\frg^1 \circ \frf$ and $\frg^2 \circ \frf$ are chain homotopic up to defect $\delta$, with $W$-side homotopy given by
\[
H_W^{gf} = H_W^g f_W.
\]
Here, we have used that for the pair $\frf^1 = \frf^2 = \frf$, the $W$-side homotopy is zero. By assumption, $f_W$ maps $X_W \subset W^{-1}X_W$ into $(W^{2\delta})Y_W \subset W^{-1} Y_W$ while $H_W^g$ maps $Y_W \subset W^{-1} Y_W$ into $(W^{-2\delta}) Z_W \subset W^{-1} Z_W$. Hence their composition in fact maps $X_W \subset W^{-1} X_W$ into $Z_W \subset W^{-1} Z_W$, showing that $\frg^1 \circ \frf$ and $\frg^2 \circ \frf$ are chain homotopic with zero defect.
\end{proof}

We close with some standard homological algebra which will be important in the sequel: 

\begin{defn}\label{def:mapping-cone}
The \emph{mapping cone} of a morphism $\frf=(f_Q,f_W,h)\colon X\to Y$ is given by
\[ 
\mathrm{Cone}(\frf) = (\mathrm{Cone}(f_Q), \mathrm{Cone}(f_W),\eta_{\mathrm{Cone}(\frf)}).
\] 
Here, $\mathrm{Cone}(f_Q)$ is the $R[Q, W, W^{-1}]$-complex given by the usual mapping cone of $f_Q \colon X_Q \rightarrow Y_Q$, and likewise for $\mathrm{Cone}(f_W)$. The homotopy equivalence $\eta_{\mathrm{Cone}(\frf)}$ is defined by
\begin{equation}\label{eq:eta_equiv}
  \eta_{\mathrm{Cone}(\frf)}=\left(\begin{array}{cc}
      \eta_{X} & 0 \\ h & \eta_{Y}
\end{array}\right)
\end{equation}
as a map from $Q^{-1} (X_Q \oplus Y_Q)$ to $W^{-1} (X_W \oplus Y_W)$. 
\end{defn}

\color{black}

\begin{defn}\label{def:composition-cone}
Let $\frf=(f_Q,f_W,h^f)\colon X\to Y$ and $\frg = (g_Q, g_W, h^g) \colon Y \rightarrow Z$ be morphisms of mixed complexes. Then we may define the \textit{composition morphism} from $\mathrm{Cone}(\frf)$ to $\mathrm{Cone}(\frg \circ \frf)$ via the triple
\[
\left(\begin{array}{cc}\id & 0 \\0 & g_Q\end{array}\right), \quad  \left(\begin{array}{cc}\id& 0 \\0 & g_Q\end{array}\right), \quad \text{and} \quad \left(\begin{array}{cc}0 & 0 \\0 & h^g\end{array}\right).
\]
It is tedious but straightforward to check that this is a morphism of mixed complexes.
\end{defn}

\begin{lem}\label{lem:isomorphic-cones}
Suppose that $\frf^1 = (f^1_Q, f^1_W, h^1)$ and $\frf^2 = (f^2_Q, f^2_W, h^2)$ are chain homotopic morphisms from $X$ to $Y$. Then $\mathrm{Cone}(\frf^1)$ and $\mathrm{Cone}(\frf^2)$ are isomorphic.
\end{lem}

\begin{proof}
Let $(H_Q, H_W, \Delta)$ constitute the data of the chain homotopy. It is tedious but straightforward to check that the following triple defines a mixed complex morphism from $\mathrm{Cone}(\frf^1)$ to $\mathrm{Cone}(\frf^2)$:
\[
\left(\begin{array}{cc}\id & 0 \\H_Q & \id\end{array}\right), \quad  \left(\begin{array}{cc}\id& 0 \\H_W & \id\end{array}\right), \quad \text{and} \quad \left(\begin{array}{cc}0 & 0 \\ \Delta & 0\end{array}\right).
\]
In fact, the same triple defines a mixed complex morphism from $\mathrm{Cone}(\frf^2)$ to $\mathrm{Cone}(\frf^1)$; it is straightforward to check that these are inverse to each other.
\end{proof}
\begin{rmk}
  We present
  a concept explaining the definition of a mixed complex (but not a map between them). Suppose $\F$ is a field, and $M_Q,M_W$ are two projective modules over $\F[Q,W,W^{-1}]$, respectively $\F[Q,Q^{-1},W]$. Then $M_Q,M_W$ can be regarded as vector bundles over $\F\times \F^*$
  and $\F^*\times\F$. Identification of $Q^{-1}M_Q$ with $W^{-1}M_W$ is gluing these two vector bundles over $\F^*\times\F^*$. In other
  words, two projective modules $M_Q$ with $M_W$ together with an identification of $Q^{-1}M_Q\cong W^{-1}M_W$ specify
  a vector bundle over $\F\times\F\setminus\{(0,0)\}$.
\end{rmk}

\subsection{Mixed complexes for links}\label{sub:mixed_links}
We now define mixed complexes for links and establish invariance up to $Q$-equivalence. We then discuss the construction of cobordism maps for mixed complexes.

\begin{defn}
Let $D$ be a transvergent diagram for an involutive link $L$. Define 
\[
\Mkc(D) = M(\Kcm_{Q, W}(D)) \quad \text{and} \quad \Mkcr(D) = M(\Kcrm_{Q, W}(D))
\]
by applying the procedure of Example~\ref{ex:simple_mixed} to the $(Q, W)$-complex and reduced $(Q, W)$-complex associated to $D$ in Definition~\ref{def:QWdiagram}.
\end{defn}

We aim to show that $\Mkc(D)$ (up to $Q$-equivalence) is actually a link invariant, unlike $\Kcrm_{Q,W}(D)$ itself.
Our general strategy will be to lift the Borel maps defined in Sections~\ref{sec:borel-transvergent} and \ref{sec:miinvariance} to morphisms of mixed complexes. More precisely, each time we encounter a map $f_Q$ of Borel complexes, we check that $f_Q$ is filtered with respect to $\deg_k$. (In the case of the M1, M2, and M3 moves, it will be necessary to first homotope $f_Q$ over $Q^{-1}W^{-1}\Kcm_Q$ in order to make it filtered.) We then apply Definition~\ref{def:mapWlift} to obtain a corresponding $W$-side map $f_W$, thus completing $f_Q$ to a morphism of mixed complexes. Likewise, each time we encounter a homotopy $H_Q$ between Borel maps, we check whether $H_Q$ is filtered with respect to $\deg_k$. In general, the homotopies constructed in Sections~\ref{sec:borel-transvergent} and \ref{sec:miinvariance} will not be filtered, but will decrease $\deg_k$ by at most some constant $2\delta$. Applying Definition~\ref{def:mapWlift} then results in a chain homotopy up to defect $\delta$ as defined in Definition~\ref{def:homotopy-delta}. While an examination of the homotopies of Sections~\ref{sec:borel-transvergent} and \ref{sec:miinvariance} is not strictly necessary for establishing $Q$-equivalence, this additional analysis will be useful in our discussion of the M3 move.

\subsubsection*{Invariance of the mixed complex}

The proof of invariance will be split into several different cases. 




\begin{lem}\label{lem:simpleQequivalence}
If $D$ and $D'$ differ by:
\begin{enumerate}
\item an R-move or I-move, then $\Mkc(D)$ and $\Mkc(D')$ are homotopy equivalent;
\item an IR1, IR2, or IR3 move, then $\Mkc(D)$ and $\Mkc(D')$ are homotopy equivalent up to defect $\delta = 1/2$;
\item an R1 or R2 move, then $\Mkc(D)$ and $\Mkc(D')$ are homotopy equivalent up to defect $\delta = 1$.
\end{enumerate}
Analogous statements hold for $\Mkcr(D)$ and $\Mkcr(D')$.
\end{lem}
\begin{proof}
Let $D$ and $D'$ differ by an IR1, IR2, IR3, R1, R2, R-move, or I-move. Let
\[
f_Q \colon \Kcm_Q(D) \rightarrow \Kcm_Q(D') \quad \text{and} \quad g_Q \colon \Kcm_Q(D') \rightarrow \Kcm_Q(D)
\]
be the homotopy equivalences between $\Kcm_Q(D)$ and $\Kcm_Q(D')$ constructed in the proofs of Lemmas~\ref{lem:invariant-formal-case} and \ref{lem:invariant-formal-case-ii}. It is straightforward to check that in each case these are filtered with respect to $\deg_k$. For the R-move or I-move this is obvious, since the maps in question are $\deg_k$-preserving isomorphisms. For the R1 and R2 moves, we reproduce Figures~\ref{fig:bn_R1}, and~\ref{fig:bn_R2} with the additional data of the $k$-grading. See Figures~\ref{fig:bn_R3}, \ref{fig:bn_R4}. The case of IR1, IR2, and IR3
is straightforward, see  Figure~\ref{fig:bn_R5}.

\begin{figure}
  \includegraphics[width=12cm]{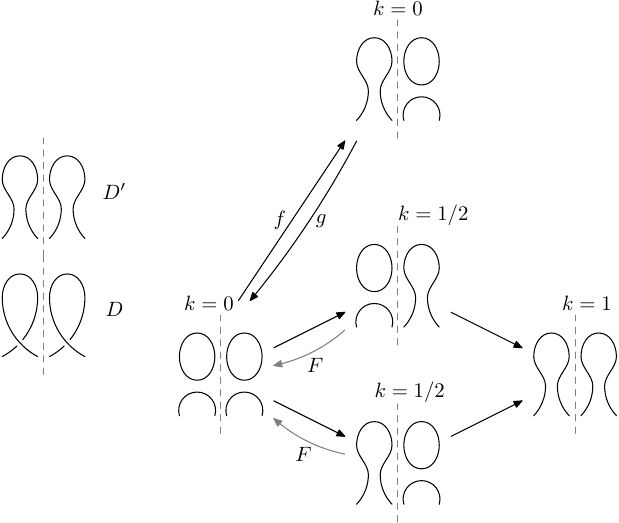}
  \caption{Figure~\ref{fig:IR1} with the $k$-grading indicated.}\label{fig:bn_R5}.
\end{figure}

\begin{figure}
  \includegraphics[width=10cm]{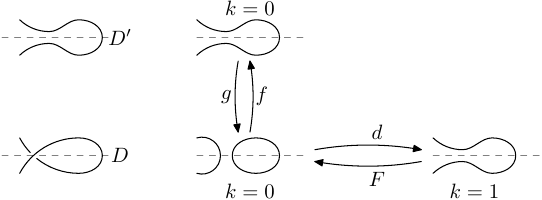}
  \caption{Figure~\ref{fig:bn_R1} with the $k$-grading indicated.}\label{fig:bn_R3}
\end{figure}
\begin{figure}
  \includegraphics[width=12cm]{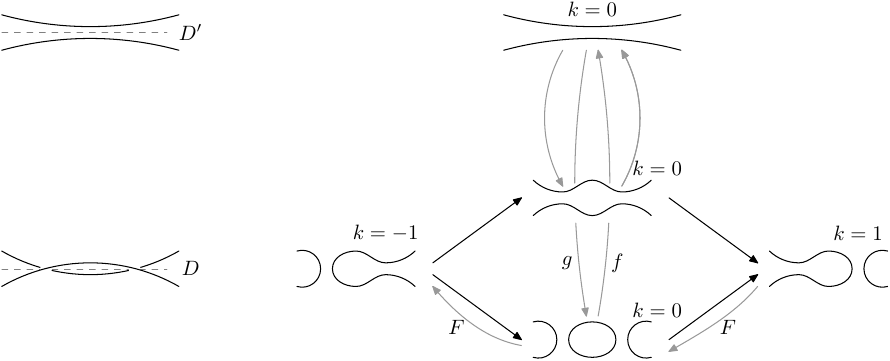}
  \caption{Figure~\ref{fig:bn_R2} with the $k$-grading indicated.}\label{fig:bn_R4}
\end{figure}

Using the same procedure as in Definition~\ref{def:mapWlift}, we may promote $f_Q$ and $g_Q$ to maps between $\smash{\Kcm_{Q, W}(D)}$ and $\smash{\Kcm_{Q, W}(D')}$. Localizing these maps with respect to $W$ and $Q$ gives us mixed complex morphisms 
\[
\frf = (f_Q, f_W = f_Q, 0) \quad \text{and} \quad \frg = (g_Q, g_W = g_Q, 0)
\]
between $\Mkc(D)$ and $\Mkc(D')$, showing that $\Mkc(D)$ and $\Mkc(D')$ are $Q$-equivalent. 

In order to establish $\Mkc(D)$ and $\Mkc(D')$ are homotopy equivalent, we must further examine the homotopies making $g_Q f_Q \simeq \id$ and $f_Q g_Q \simeq \id$ from the proofs of Lemmas~\ref{lem:invariant-formal-case} and \ref{lem:invariant-formal-case-ii}. For the R-move and I-move, these homotopies are of course zero.

For the moves R1, R2, IR1, IR2, IR3, the homotopy goes one step back in the resolution, see Figure~\ref{fig:IR1} for the IR1 move.
Each such map, by the definition, changes the $k$-grading by the amount given by the change between $o(c)$ and $u(c)$ on
the relevant crossing, see Subsection~\ref{sub:lw_work}. This change is at most $1/2$ for crossings off-the-axis,
and at most $1$ for crossings on the axis. That is, the homotopy drops the $k$ grading by at most $1/2$ for any of the IR1, IR2, IR3 moves,
while it drops the $k$ grading at most by $1$ for the R1 and R2 moves.

In each case, applying Definition~\ref{def:mapWlift} gives a $W$-side homotopy defined on $Q^{-1}W^{-1} \Kcm_{Q, W}$ that has controlled behavior with respect to $W$. We thus obtain a homotopy equivalence up to the appropriate defect, as desired. (Explicitly, to verify the third condition of Definition~\ref{def:homotopy-of-mixed}, take $\Delta = 0$ and recall that in this situation $\eta_X = \eta_Y = \id$ and $h^1 = h^2 = 0$.) The reduced case follows from treating basepoints as in the proof of Lemma~\ref{lem:reduced-invariant}.

\end{proof}

The proof of invariance for the M1, M2, and M3 moves is somewhat more complicated due to the fact that the Borel map $f_Q$ is not filtered with respect to $\deg_k$. Instead, we first homotope $f_Q$ over $Q^{-1}W^{-1}\Kcm_Q$ to make it filtered. As the resulting analysis is slightly more involved, we delay this to the next subsection.

\subsubsection*{Cobordism maps for mixed complexes}
Construction of the cobordism maps in the mixed complex setting is entirely straightforward. It is easy to check that each elementary cobordism move is filtered with respect to $\deg_k$. We can thus directly apply Definition~\ref{def:mapWlift} to lift these to morphisms of mixed complexes, as in the proof of Lemma~\ref{lem:simpleQequivalence}. As usual, the overall cobordism map for mixed complexes is then obtained by composing the mixed complex morphisms for elementary cobordisms and equivariant Reidemeister moves, as appropriate.

Defining cobordism maps for the reduced mixed complexes follows the same strategy as in the Bar-Natan and Borel cases. For this, note that the Wigderson splitting of Borel complexes extends to a splitting of $(Q, W)$-complexes. Indeed, both $i_p$ and the section map $\sigma_p$ send generators within a resolution of $D$ to generators in the same resolution, and are hence $\deg_k$-preserving. As in Remark~\ref{rem:reducedQWdefined}, this shows that $i_p$ and $\sigma_p$ can be promoted to ($\deg_k$-preserving) maps of $(Q, W)$-complexes, from which we obtain mixed complex morphisms $\mathfrak{i}_p$ and $\mathfrak{s}_p$ after localizing. We thus define
\[
\Mkcr(\Sigma) = \mathfrak{s}_p \circ \Mkc(\Sigma) \circ \mathfrak{i}_p.
\]

Modulo our discussion of the M1, M2, and M3 moves, we thus obtain a proof of Theorem~\ref{thm:mixedcomplex}:

\begin{proof}[Proof of Theorem~\ref{thm:mixedcomplex}]
Invariance up to $Q$-equivalence follows from Lemma~\ref{lem:simpleQequivalence}, combined with our discussion of the M1, M2, and M3 moves in the next subsection. To establish locality, observe that the $Q$-side part of the reduced mixed complex cobordism map is just the usual reduced Borel map, localized at $W$. This follows from examining the composition rule for mixed complex morphisms. Since the reduced Borel map is local, it easily follows that $\Mkcr(\Sigma)$ is local. 
\end{proof}

\subsection{Invariance under M1, M2, and M3 moves}
We now turn to the M1, M2, and M3 moves. 

\subsubsection*{Invariance for M1.} Let $D$ and $D'$ differ by an M1 move. We begin by promoting the map $f_Q$ of Borel complexes from Section~\ref{sec:miinvariance} to a mixed complex morphism.  The grading $\deg_k$ for $D,D'$ depend on the orientation of the underlying link (in fact, only on the orientation up to reversing the orientations of all components, that is, $\deg_k$ depends on the quasi-orientation).  However, by inspecting the rules in Section \ref{sub:lw_work}, we see that $\deg_k$ only depends on the quasi-orientation up to an overall shift.  In particular, for a given diagram $D$, the mixed complex $\MKc(D)$ is well-defined as a relatively $\deg_k$-graded complex independent of quasi-orientation.

 For M1, there are four separate quasi-orientations of the tangle in the relevant regions; we will define below maps $\frf,\frg$ between the mixed complexes $\MKc(D),\MKc(D')$, for a particular orientation on the underlying link.  It will turn out that the maps $\frf,\frg$ remain maps of mixed complexes for the other underlying orientations (where we identify the different complexes $\MKc(D)$ for different orientations, up to a grading shift), to give that for any orientation on the underlying link $L$, there is a homotopy equivalence of mixed complexes with defect at most $1$.

In order to define the $\deg_k$ grading on $D,D'$, we choose an orientation for which $L$ is strongly-invertible; this puts the resolution with a disjoint circle in $k$-grading zero.

Unlike in the previous cases, $f_Q$ is not filtered with respect to $\deg_k$. Instead, following the work of Lobb-Watson, we define a homotopy
\[
h \colon Q^{-1} W^{-1}\Kcm_{Q, W}(D) \rightarrow Q^{-1} W^{-1}\Kcm_{Q, W}(D')
\]
as shown in Figure \ref{fig:M1mixed_1}. It is a direct computation to verify that:
\begin{enumerate}
\item $\tau h + h \tau = 0$; and, 
\item $f_Q + \partial_Q h + h \partial_Q$ sends $\smash{Q^{-1} \Kcm_{Q, W}(D)}$ to $\smash{Q^{-1} \Kcm_{Q, W}(D')}$. 
\end{enumerate}
Given this, we obtain a mixed complex morphism
\[
\frf = (f_Q, f_W = f_Q + \partial_Q h + h \partial_Q, h).
\]
The analogous construction for $\frg$ is obtained by rotating Figure \ref{fig:M1mixed_1} by 180 degrees.  To be more precise, each entry in the top cube of resolutions \ref{fig:M1mixed_1}, under rotation is identified with an entry in the bottom cube; we define $\frg$ by first performing this identification, applying $\frf$, and then rotating back.  We obtain a mixed complex morphism
\[
\frg = (g_Q, g_W = g_Q + \partial_Q h' + h' \partial_Q, h').
\]
See Figure~\ref{fig:M1mixed_2}.
This shows that $\Mkc(D)$ and $\Mkc(D')$ are $Q$-equivalent.

\begin{figure}[h!]
\includegraphics[scale = 1.3]{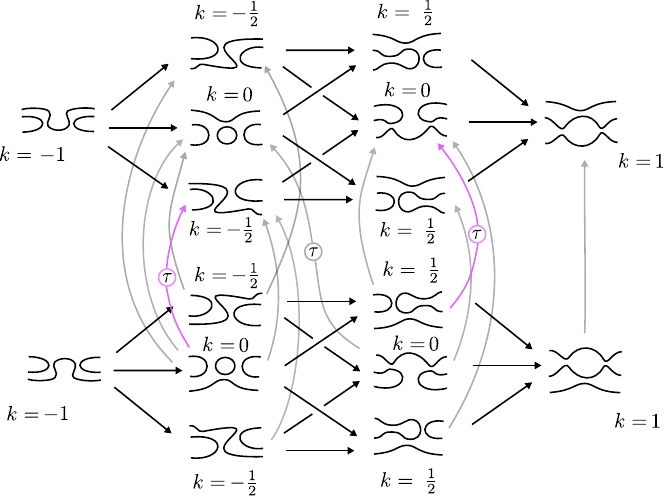}
\caption{M1 move maps. The map $f_Q$ in grey and the homotopy $h$ in pink.}\label{fig:M1mixed_1}
\end{figure}

As before, we now attempt to establish homotopy equivalence up to some defect $\delta$. An analysis of the chain homotopies $F_Q$ and $G_Q$ constructed in the proof of M1 invariance shows that these decrease $\deg_k$ by at most $1$; see Figure~\ref{fig:M1mixed_3}.

\begin{figure}[h!]
\includegraphics[scale = 1.3]{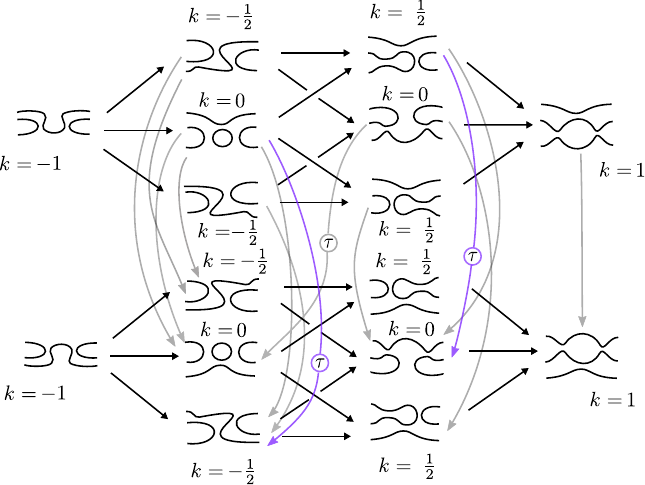}
\caption{M1 move maps. The map $g_Q$ in grey and the homotopy $h'$ in purple.}\label{fig:M1mixed_2}
\end{figure}

\begin{figure}[h!]
\includegraphics[scale = 1.3]{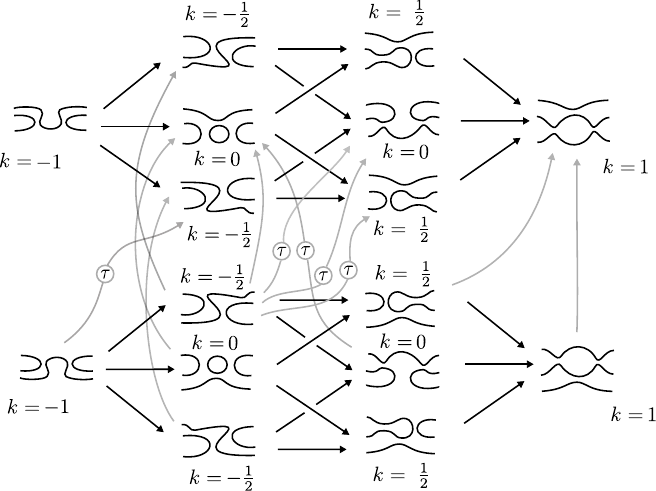}
\caption{M1 move maps. The map $f_W$ is grey, and the map $F_Q$ is green.}\label{fig:M1mixed_3}
\end{figure}

However, the $W$-side maps $f_W$ and $g_W$ are not quite equal to $f_Q$ and $g_Q$ after localizing, but are rather chain homotopic to these. Thus (for example) using Definition~\ref{def:mapWlift} to promote $F_Q$ to a $W$-side map does not quite give a chain homotopy from $g_W f_W$ to the identity. Instead, it is straightforward to verify that
\begin{align*}
g_Wf_W &= g_W(f_Q + \partial_Q h + h \partial_Q) \\
&= (g_Q + \partial_Q h' + h' \partial_Q) f_Q + \partial_Q (g_Wh) + (g_Wh) \partial_Q \\
&= g_Qf_Q + \partial_Q(h'f_Q + g_Wh) + (h'f_Q + g_Wh)\partial_Q.
\end{align*}
We thus let
\[
F_W = F_Q + (h'f_Q + g_Wh)
\]
as a map from $Q^{-1} W^{-1}\Kcm_{Q, W}(D)$ to itself. Then $F_W$ is a chain homotopy from $g_W f_W$ to the identity. We check from Figures \ref{fig:M1mixed_1}, \ref{fig:M1mixed_2}, and \ref{fig:M1mixed_3} that:
\begin{enumerate}
\item $F_Q$ decreases $\deg_k$ by at most $1$;
\item $f_Q$ and $h'$ decrease $\deg_k$ by at most $1/2$, and hence $h'f_Q$ decreases $\deg_k$ by at most $1$; and,
\item $h$ decreases $\deg_k$ by at most $1/2$, and hence $g_Wh$ decreases $\deg_k$ by at most $1/2$.
\end{enumerate}
Thus $F_W$ decreases $\deg_k$ by at most one. It follows that $F_Q$ and $F_W$ together make up a chain homotopy of mixed complexes up to defect $\delta = 1$. (Explicitly, to verify the third condition of Definition~\ref{def:homotopy-of-mixed}, take $\Delta = 0$ and recall that in this situation $\eta_X = \eta_Y = \id$, $h^1 = h'f_Q + g_Wh$ by Definition~\ref{def:composition}, and $h^2 = 0$.) The construction for $f_Wg_W$ is analogous. We conclude that $\Mkc(D)$ and $\Mkc(D')$ are homotopy equivalent up to defect $\delta = 1$. The reduced case is analogous. 

One may check directly that $\frf$ and $\frg$ so-defined, together with $F_Q,F_W$, also define a homotopy equivalence of defect $1$ for the other orientations on M1.


\subsubsection*{Invariance for M2.} Let $D$ and $D'$ differ by an M2 move, as in \ref{fig:all_equi_moves_lobb_watson}, with $D$ on the right-hand side of M2 and $D'$ on the left. 

Once again, we must check M2 for each of the orientations on the underlying link $L$.  In the following discussion, we use $\deg_k$ corresponding to the strongly-invertible orientation on $L$.  For the other quasi-orientation, the same maps and homotopies, define a homotopy equivalence of mixed complexes of defect $1$, just as in the M1 case, and we will not comment further on the other case.

We begin by promoting the map $f_Q$ of Borel complexes from Section~\ref{sec:miinvariance} to a mixed complex morphism. Once again following the work of Lobb-Watson, we define a homotopy
\[
h \colon Q^{-1} W^{-1}\Kcm_{Q, W}(D) \rightarrow Q^{-1} W^{-1}\Kcm_{Q, W}(D')
\]
as shown in Figure \ref{fig:m2-h-homotopy}.  
\begin{figure}
	\centering
	\includegraphics[width=14cm]{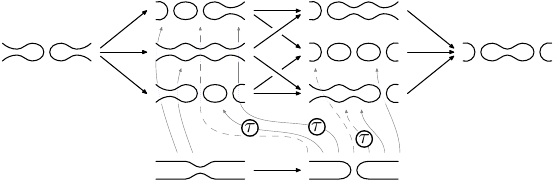}
	\caption{M2 move maps. The homotopy $h$ for $\Kcm(D)\to \Kcm(D')$.  The solid arrows indicate the maps $f_0,f_1$ from Section \ref{sec:borel}, while the dashed arrows constitute the map $h$.  The component of $h$ decreasing homological degree is given by a birth, the component $D_1$ to $D'_{011}$ is given by the identity cobordism (composed with $\tau$, as indicated).  Finally, the dashed arrow from $D_1$ to $D'_{101}$ is given by precomposition with $\tau$, birthing the two circles, and acting by a dot on the leftmost part of the tangle. }
	\label{fig:m2-h-homotopy}
\end{figure}

It is a direct computation to verify that:
\begin{enumerate}
\item $\tau h + h \tau = 0$; and, 
\item $f_Q + \partial_Q h + h \partial_Q$, viewed as a map from $\smash{Q^{-1} W^{-1}\Kcm_{Q, W}(D)}$ to $\smash{Q^{-1} W^{-1}\Kcm_{Q, W}(D')}$, has that the restriction sends $\smash{Q^{-1} \Kcm_{Q, W}(D)}$ to $\smash{Q^{-1}\Kcm_{Q,W}(D')}$. 
\end{enumerate}
Given this, we obtain a mixed complex morphism
\[
\frf = (f_Q, f_W = f_Q + \partial_Q h + h \partial_Q, h).
\]
It is readily checked that $g_Q$, as defined in Section \ref{sec:borel} is non-decreasing in $\deg_k$; to see this we reproduce here, in Figure \ref{fig:m2-g0-g1-with-k}, the Figure \ref{fig:m2-g0-g1} defining $g_Q$, now with $k$-gradings pictured.  
\begin{figure}
	\centering
	\includegraphics[width=14cm]{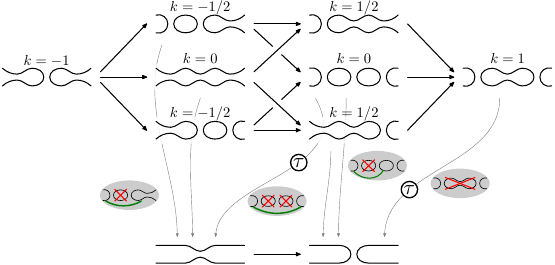}
	\caption{M2 move. The $k$-gradings for the components of $g$.  }
	\label{fig:m2-g0-g1-with-k}  
\end{figure}

Since $g_Q$ is non-decreasing in $\deg_k$, we can set $h'=0$ to define a mixed complex morphism for $f_Q$.  Namely, we set
\[
\frg = (g_Q,g_Q,0).
\]
The existence of maps $\frf$ and $\frg$ lifting $f,g$, respectively, implies that $\Mkc(D)$ and $\Mkc(D')$ are $Q$-equivalent.

Once again, we can actually prove something slightly stronger.  We show that $\frf$ and $\frg$ above are homotopy-inverse with defect $1$.  

We will write down homotopies from both $\frg \frf$ and $\frf\frg$ to the identity morphism (of mixed complexes), with defect as above. 

We start with $\frg\frf$.
As discussed in Section~\ref{sec:miinvariance}, $g_Qf_Q=\mathrm{Id}$, but $g_Wf_W$ is not the identity - there is a nonzero term $D_1$ to $D_0$ by a split, call this $\alpha$, and $g_Wf_W=\mathrm{Id}+\alpha$.  The homotopy $H$ from $g_Qf_Q$ to $g_Wf_W$ is given by the identity map $\Kcm_{Q,W}(D_1)$ to itself, and vanishes on $\Kcm_{Q,W}(D_0)$.   That is,
\[
\frg\frf = (\mathrm{Id},\mathrm{Id}+\alpha,H ), 
\]
with $H$ described as above.  

We now need a homotopy $\mathfrak{H}$ from $\frg \frf$ to $\mathrm{Id}$.  On the $Q$-side, we can use the $0$ map, as $g_Qf_Q=\mathrm{Id}$.  On the $W$-side, we use $H$ itself.  In this case, we may set the diagonal map in Definition \ref{def:homotopy-of-mixed} to be zero.   Since $H$ preserves $\deg_k$, it follows that $\frg\frf$ is homotopic to $\mathrm{Id}$ with defect zero.

We now construct a homotopy $\frH'$ between $\frf \frg $ and $\mathrm{Id}$.  We first write out $\frf \frg$:
\[
\frf\frg=(f_Qg_Q,f_Qg_Q+[\partial_Q,h]g_Q,hg_Q).
\]
In the $Q$-side of $\frH'$, we use the homotopy $G_Q$ from (\ref{eq:m2-g-homotopy}).  There exists a homotopy $\frH'$ with defect at most $1$, where the $Q$-side of $\frH'$ is $G_Q$, exactly when
\begin{equation}\label{eq:error-term}
G_Q+hg_Q
\end{equation}
is homotopic, over $Q^{-1}W^{-1}\Kcm_{Q,W}(D')\to Q^{-1}W^{-1}\Kcm_{Q,W}(D')$, to a map which decreases $\deg_k$ by at most $1$.  We now show that there exists such a homotopy from (\ref{eq:error-term}) to a map decreasing $\deg_k$ by at most $1$.


This argument consists of two parts.  First, we will show that $G_Q$ has only one entry where $\deg_k$ is decreased by more than $1$, and similarly we will find a single such term for $hg_Q$; once these terms are identified it is easy to find a homotopy $H'$ so that 
\begin{equation}\label{eq:k-right}
	[\partial_Q,H']+G_Q+hg_Q
\end{equation}  
decreases $\deg_k$ by at most $1$.

Let us show that $G_Q$ has the desired property.  Our first observation is that $EG_0$, evaluated on terms $D'_{i}$ with $i\neq 000$, can never have support on $D'_{000}$.  Indeed, this is because the only term of $E$ which has image supported in $D'_{000}$ is $D'_{100}$, by a death, but the image of $G_0$ is only supported on births to $D'_{100}$.  So, we need only find terms of $G_Q$ that go from $D'_{111}$ to either $D'_{100}$ or $D'_{001}$, as these are only remaining pairs of resolutions with a difference of $\deg_k$ greater than $1$.  First, to get to $D'_{100}$ or $D'_{001}$ by terms of the form $\frac{1}{1+QE}G_0$, we need only consider $(1+QE)G_0$, as $E$ takes a generator at any resolution $D'_{i}$ only to resolutions $D'_{j}$ with $j\leq i$ (in the sense that each entry of $j$ is at most as large as the corresponding entry of $i$). So we are left to consider $QEG_0$.   The component of this which decreases $\deg_k$ by more than $1$ is readily calculated to be the map by precomposition with $\tau$, death of the central circle, and then birth of a new central circle from $D'_{111}$ to $D'_{001}$, as depicted in Figure \ref{fig:m2-first-error-term}. 

\begin{figure}
	\centering
	\includegraphics[width=10cm]{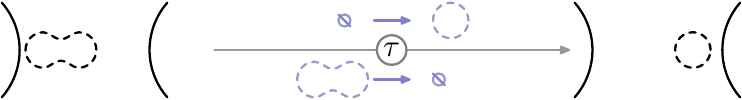}
	\caption{M2 move. Pictured is the term of $QEG_0$ which decreases $\deg_k$ by more than $1$.  This is a morphism from $\Kcm(D_{111}')$ to $\Kcm(D'_{001})$.}
	\label{fig:m2-first-error-term}
\end{figure}

A priori there is also the contribution $\frac{1}{1+QE}G_0f_Qg_Q$ to $G_Q$ as well.  Since $g_Q^0|_{\Kcm(D'_{111})}=0$, this is $G_0f_Q^0g_Q^1$ (no $E$ terms as these move resolutions to the left, as before).  A direct calculation shows that this map is zero.

Finally, we consider $hg_Q$, which happens to agree with $hg_W$, as $g_Q=g_W$.  A direct computation shows that $D'_{111}$ is sent to $D'_{100}$ by precomposing with $\tau$, death on the central circle and a birth of a new central circle, as depicted in Figure \ref{fig:m2-second-error-term}, and there are no other components decreasing $\deg_k$ by more than $1$.   

\begin{figure}
	\centering
	\includegraphics[width=10cm]{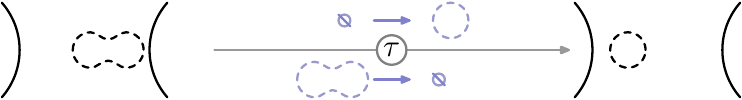}
	\caption{M2 move. Pictured is the term of $hg_Q$ which decreases $\deg_k$ by more than $1$.  This is a morphism from $\Kcm(D_{111}')$ to $\Kcm(D'_{100})$.}
	\label{fig:m2-second-error-term}
\end{figure}

Let $\zeta \colon Q^{-1}W^{-1}\Kcm_{Q,W}(D')\to Q^{-1}W^{-1}\Kcm_{Q,W}(D')$ be the map which sends $\Kcm(D_{111}')$ to $\Kcm(D_{100}')$ by death on the central circle, followed by birth of a central circle and is zero on the other resolutions.  This map does not preserve $\deg_k$, however, $[\partial_Q,\zeta]=hg_Q+G_Q$ modulo terms that decrease $\deg_k$ by at most $1$.  Thus $\zeta$ is the needed homotopy.   

That is, $\frH'=(G_Q,G_Q+[\partial_Q,\zeta],\zeta)$ is a homotopy from $\frf\frg$ to $\mathrm{Id}$, with defect at most $1$, as needed.

\subsubsection*{Invariance for M3.}  
We now turn to the actual proof of M3 invariance. Let $D$ and $D'$ differ by an M3 move.  We fix an orientation on the underlying link, which will be mostly suppressed from the notation, in order to define $\deg_k$.  It will turn out, as in the other cases, that the mixed complex morphisms (and homotopies) we define for a particular orientation will induce Q-equivalences for the other orientations.  Recall from Section~\ref{sec:miinvariance} that we identified $\Kcm_Q(D)$ and $\Kcm_Q(D')$ with the mapping cones of the saddle maps $s_\alpha$ and $s_\beta$, respectively, of Figure~\ref{fig:m3cones}. We then argued that
\[
\mathrm{Cone}(\Kcm_Q(D_0) \xrightarrow{s_\alpha} \Kcm_Q(D_1)) \simeq \mathrm{Cone}(\Kcm_Q(D_0) \xrightarrow{(F_{\mathrm{aux}} \circ s_\beta) \circ s_\alpha} \Kcm_Q(D_5))
\]
and
\[
\mathrm{Cone}(\Kcm_Q(D_0) \xrightarrow{s_\beta} \Kcm_Q(D_1')) \simeq \mathrm{Cone}(\Kcm_Q(D_0) \xrightarrow{(F_{\mathrm{aux}} \circ s_\alpha) \circ s_\beta} \Kcm_Q(D_5)).
\]
This was done by showing that (for example) $F_{\mathrm{aux}} \circ s_\beta = F_\mathrm{IR2} \circ F_\mathrm{M1}$ and thus constitutes (half of) a homotopy equivalence of Borel complexes. Since $s_\beta \circ s_\alpha = s_\alpha \circ s_\beta$, this completed the proof.

In the setting of mixed complexes, the initial and final steps still hold. That is, the mixed complexes for $D$ and $D'$ are still the mapping cones of the mixed complex morphisms associated to $s_\alpha$ and $s_\beta$. Moreover, it is still true that $s_\alpha$ and $s_\beta$ commute, since this holds for the underlying Borel maps. Our principal claim is thus that we have $Q$-equivalences between 
\[
\mathrm{Cone}(\Mkc(D_0) \xrightarrow{s_\alpha} \Mkc(D_1)) \quad \text{and} \quad \mathrm{Cone}(\Mkc(D_0) \xrightarrow{(F_{\mathrm{aux}} \circ s_\beta) \circ s_\alpha} \Mkc(D_5)),
\]
as well as between
\[
\mathrm{Cone}(\Mkc(D_0) \xrightarrow{s_\beta} \Mkc(D_1')) \quad \text{and} \quad \mathrm{Cone}(\Mkc(D_0) \xrightarrow{(F_{\mathrm{aux}} \circ s_\alpha) \circ s_\beta} \Mkc(D_5)).
\]
Proving the claim will conclude the proof.

We consider the former. Let $\frf$ denote the mixed complex morphism 
\[
\frf \colon X = \Mkc(D_0) \rightarrow Y = \Mkc(D_1)
\]
associated to $s_\alpha$ and let $\frg$ denote the mixed complex morphism 
\[
\frg \colon Y = \Mkc(D_1) \rightarrow Z = \Mkc(D_5)
\]
associated to $F_{\mathrm{aux}} \circ s_\beta = F_\mathrm{IR2} \circ F_\mathrm{M1}$. Doing the IR2 and M1 moves in reverse gives a mixed complex morphism 
\[
\bar{\frg} \colon Z = \Mkc(D_5) \rightarrow Y = \Mkc(D_1)
\]
such that $\frg$ and $\bar{\frg}$ are homotopy equivalences up to defect $\delta = 1$. This follows from our calculation of the homotopy defect for the IR2 and M1 moves combined with Lemma~\ref{lem:homotopy-defect-delta}. Consider the sequence of mixed complex morphisms
\[
\mathrm{Cone}(\frf) \rightarrow \mathrm{Cone}(\frg \circ \frf) \rightarrow \mathrm{Cone}(\bar{\frg} \circ \frg \circ \frf) \xrightarrow{\cong} \mathrm{Cone}(\frf).
\]
The first two of these are the composition morphisms of Definition~\ref{def:composition-cone}. To understand the final morphism, note that by direct inspection, the $Q$-side map for $\frf$ lands in the image of $W$. Since $\frg$ and $\bar{\frg}$ are homotopy equivalences up to defect $\delta = 1$, by Lemma~\ref{lem:technical-composition} we have that $\bar{\frg} \circ \frg \circ \frf$ and $\frf$ are chain homotopic with zero defect. Hence Lemma~\ref{lem:isomorphic-cones} shows that $\mathrm{Cone}(\bar{\frg} \circ \frg \circ \frf)$ and $\mathrm{Cone}(\frf)$ are isomorphic.

We thus have mixed complex morphisms
\[
\mathrm{Cone}(\frf) \rightarrow \mathrm{Cone}(\frg \circ \frf) \quad \text{and} \quad \mathrm{Cone}(\frg \circ \frf) \rightarrow \mathrm{Cone}(\bar{\frg} \circ \frg \circ \frf) \xrightarrow{\cong} \mathrm{Cone}(\frf).
\]
We claim that these are $Q$-equivalences. To see this, consider the $Q$-side maps associated to these morphisms. Let $H_Q$ be the Borel homotopy such that $\bar{g}_Q g_Q + \id = \partial_Q H_Q + H_Q \partial_Q$, so that $H_Qf_Q$ gives the Borel homotopy from $\bar{g}_Q g_Q f_Q$ to $f_Q$. Examining Definition~\ref{def:composition-cone} and the proof of Lemma~\ref{lem:isomorphic-cones}, the $Q$-side maps for the above morphisms are given by
\[
\left(\begin{array}{cc}\id & 0 \\ 0 & g_Q\end{array}\right) \quad \text{and} \quad \left(\begin{array}{cc}\id & 0 \\ H_Q f_Q& \id\end{array}\right)\left(\begin{array}{cc}\id & 0 \\ 0 & \bar{g}_Q\end{array}\right) = \left(\begin{array}{cc}\id & 0 \\ H_Q f_Q& \bar{g}_Q\end{array}\right),
\]
respectively. We claim that these are homotopy equivalences. It is straightforward to check that the composition in one direction is chain homotopic to the identity by verifying that
\[
\left(\begin{array}{cc}\id & 0 \\ H_Q f_Q& \bar{g}_Q\end{array}\right)\left(\begin{array}{cc}\id & 0 \\ 0 & g_Q\end{array}\right) + \id = \left[ \partial_{\mathrm{Cone}(f_Q)}, \left(\begin{array}{cc}0 & 0 \\ 0& H_Q\end{array}\right)\right],
\]
where we use $[\cdot,\cdot]$ to denote the commutator of maps.
To verify that the composition in the other direction is chain homotopic to the identity, let $\overline{H}_Q$ be the Borel homotopy such that $g_Q \bar{g}_Q + \id = \partial_Q \overline{H}_Q + \overline{H}_Q \partial_Q$. Then it is an extremely tedious but straightforward exercise to show that
\begin{align*}
\left(\begin{array}{cc}\id & 0 \\ 0 & g_Q\end{array}\right)&\left(\begin{array}{cc}\id & 0 \\ H_Q f_Q& \bar{g}_Q\end{array}\right) + \id = \\
&\left[ \partial_{\mathrm{Cone}(g_Qf_Q)}, \left(\begin{array}{cc}0 & 0 \\ (g_QH_Q + \overline{H}_Qg_Q)H_Qf_Q& (g_QH_Q + \overline{H}_Qg_Q)\bar{g}_Q\end{array}\right) + \left(\begin{array}{cc}0 & 0 \\ 0& \overline{H}_Q\end{array}\right)  \right].
\end{align*}
This establishes $Q$-equivalence for the mapping cone of $s_\alpha$. The mapping cone of $s_\beta$ is similar, completing the proof of $Q$-equivalence under the M3 move. The reduced case is analogous.

\subsection{Connected sums}
We close with a discussion of the connected sum formula for mixed complexes. Let $(K_1, \tau_1)$ and $(K_2, \tau_2)$ be two strongly invertible knots with transvergent diagrams $D_1$ and $D_2$. The essential observation is the following:

\begin{lem}\label{lem:graded-connected-sum}
The equivariant isomorphism
\[
(\Kcrm(D_1 \# D_2), \tau_1 \# \tau_2) \cong (\Kcrm(D_1), \tau_1) \otimes (\Kcrm(D_2), \tau_2)
\]
of Theorem~\ref{thm:connected-sum-tau} is $\deg_k$-preserving.
\end{lem}
\begin{proof}
Recall that the isomorphism in question is given by 
\[
\Kcrm_{un}(D_1) \otimes \Kcrm_{un}(D_2) \rightarrow \Kcrm_{un}(D_1 \sqcup D_2) \rightarrow \Kcrm_{un}(D_1 \# D_2)
\]
with the first map being the obvious inclusion and the second map being the saddle cobordism. It is straightforward to check that these are $\deg_k$-preserving. Indeed, if $(D_1)_v$ and $(D_2)_w$ are resolutions of $D_1$ and $D_2$, respectively, then an examination of the Lobb-Watson grading from Section~\ref{sec:lobb-watson-filtration} shows that
\[
\deg_k((D_1)_v) + \deg_k((D_2)_w) = \deg_k((D_1)_v \sqcup (D_2)_w) = \deg_k((D_1 \# D_2)_{v \# w}).
\]
Since $\deg_k$ on $\Kcrm(D_1) \otimes \Kcrm(D_2)$ is defined by $\deg_k(x \otimes y) = \deg_k(x) + \deg_k(y)$, this gives the claim.
\end{proof}

As in Section~\ref{sec:koszul-borel}, we can re-phrase the connected sum formula using Koszul duality. For this, we replace the ground ring $R = \f[u]$ throughout Section~\ref{sec:koszul-borel} with $R = \f[u, W]$ and consider all of our complexes to be triply-graded, rather than bigraded. Define the Koszul bimodules $\smash{{}_{R[\tau]/(\tau^2+1)}\Br_{R[Q]}}$ and $\smash{{}_{R[\tau]/(\tau^2+1)}\Hh_{R[Q]}}$ exactly as before -- the bases in Figure~\ref{fig:bimodules} remains bases over $\f[u, W]$, lying in $\deg_k$-grading zero. Tensoring with $\smash{{}_{R[\tau]/(\tau^2+1)}\Br_{R[Q]}}$ and $\smash{{}_{R[\tau]/(\tau^2+1)}\Hh_{R[Q]}}$ gives functors
\[
\Br \colon \bKom_{R[\tau]/(\tau^2 + 1)} \rightarrow \bKom_{R[Q]}
\]
and
\[
\Hh \colon \bKom_{R[Q]} \rightarrow \bKom_{R[\tau]/(\tau^2 + 1)}.
\]
These induce equivalences of derived categories, as in Section~\ref{sec:koszul-borel}. We define the tensor product as in Definitions~\ref{def:tensor-product-tau} and \ref{def:tensor-product-borel}: for $X_1$ and $X_2$ in $\bKom_{R[\tau]/(\tau^2 + 1)}$, let
\[
X_1 \otimes X_2 = X_1 \otimes_R X_2,
\]
where we give the right-hand-side the diagonal action $\tau \otimes \tau$. For $Y_1$ and $Y_2$ in $\bKom_{R[Q]}$, let
\[
Y_1 \otimes_{\Br} Y_2 = \Br(\Hh(Y_1) \otimes \Hh(Y_2)).
\]
The same argument as in Section~\ref{sec:koszul-borel} shows that these respect quasi-isomorphisms.

To see the relevance to the present situation, let $D$ be a transvergent diagram for $(L, \tau)$. Applying Definition~\ref{def:mapWlift} to the Bar-Natan differential and (separately) to the action of $\tau$ on $\Kcrm(D)$ gives a complex $X$ in $\smash{\bKom_{R[\tau]/(\tau^2 + 1)}}$. Note that $\Br(X) = \Kcrm_{Q, W}(D)$. We thus have:

\begin{thm}
Let $D_1$ and $D_2$ be transvergent diagrams for $(L_1, \tau_1)$ and $(L_2, \tau_2)$. Then we have a homotopy equivalence
\[
\Kcrm_{Q, W}(D_1 \# D_2) \simeq \Kcrm_{Q, W}(D_1) \otimes_{\Br} \Kcrm_{Q, W}(D_2).
\]
\end{thm}
\begin{proof}
Denote by $X_1$ and $X_2$ the complexes in $\smash{\bKom_{R[\tau]/(\tau^2 + 1)}}$ obtained from $(\Kcm(D_1), \tau_1)$ and $(\Kcm(D_2), \tau_2)$ via Definition~\ref{def:mapWlift}. Lemma~\ref{lem:graded-connected-sum} then implies
\[
\Kcrm_{Q, W}(D_1 \# D_2) \cong \Br(X_1 \otimes X_2).
\]
Note that $\smash{\Br(X_i) = \Kcrm_{Q, W}(D_i)}$. Using the fact that $\Br$ and $\Hh$ induce equivalences of derived categories, it follows that $X_i$ is quasi-isomorphic to $\smash{\Hh(\Kcrm_{Q, W}(D_i))}$. Hence we have a quasi-isomorphism between
\[
\Kcrm_{Q, W}(D_1 \# D_2) \quad \text{and} \quad \Kcrm_{Q, W}(D_1) \otimes_{\Br} \Kcrm_{Q, W}(D_2).
\]
Since both of these complexes are free over $R = \f[Q, W]$, invoking Lemma~\ref{lem:qitohe} completes the proof.
\end{proof}

\appendix
\section{Invariance of $\eta$ under topological cobordism}\label{sec:appendixa}
Let $K\subset S^3$ be a strongly invertible knot. In \cite{Sakuma} Sakuma associates to $K$ an invariant $\eta\in\Z[t,t^{-1}]$, which is an invariant
of smooth equivariant concordance. It is well-known to the experts that $\eta$ is invariant under topological (locally flat) concordance, however no formal proof of that fact seems to have appeared in the literature; compare the discussion in \cite{MillerPowell}. In this appendix we fill this gap.

First, we recall from \cite{Sakuma} the formal definition of $\eta$.
Suppose $K$ is strongly invertible and $\tau\colon S^3\to S^3$ the strong inversion. Let $O$ be the fixed point set, and let $\pi\colon S^3\to S^3/\tau\cong S^3$ be the projection. Let $O_\tau=\pi(O)$.
Set $\lambda$ to be a preferred longitute of $K$, i.e., $\lk(\lambda,K)=0$. Assume that $\lambda$ is disjoint from $O$.
Set $\lambda_\tau=\pi(\lambda_\tau)$. This is a knot in $S^3\setminus O_\tau$.
\begin{defn}
  The $\eta$ invariant of $K$ is the Kojima-Yamasaki polynomial associated with the pair $(\lambda_\tau,O_\tau)$, equivalently
  the $\Z[t,t^{-1}]$ equivariant self-linking of $\lambda_\tau$ in $S^3\setminus O_\tau$.
\end{defn}
We refer the reader to \cite{KojimaYamasaki} for the properties of the $\eta$-polynomial. Also, \cite[Section 3]{BorodzikFriedl}
gives a detailed 
detailed study of equivariant linking numbers.

The proof of invariance of $\eta$ under topological concordance adapts the proof of Sakuma, which in turn relies on invariance of the Kojima-Yamasaki polynomial.
We begin with clarify the assumptions on the action of $\Z_2$ on $S^3\times[0,1]$.
\begin{defn}
  Suppose a finite group $G$ acts on a topological space $X$. Suppose $P\subset X$ 
  is an orbit type $G/H$ and $V$ is a finite dimensional linear space
  on which $H$ acts orthogonally. A \emph{linear tube} of $P$ is a $G$-equivariant embedding into open neighborhood of $P$, $\phi\colon G\times_H V\to X$ mapping $G\times_H\{0\}$ to $P$. We say that the action of $G$ is \emph{locally linear} if each orbit $P$ admits a linear tube.
\end{defn}
\begin{rmk}
  In some sources, like \cite{Bredon}, the notion `locally smooth' is used instead of `locally linear'.
\end{rmk}
We can now state the main result of the appendix.
\begin{thm}\label{thm:eta}
  Suppose $K_0,K_1$ are two strongly invertible knots in $S^3$. Suppose there exists a locally smooth
  $\Z_2$ action $\tau_4\colon S^3\times[0,1]$ extending the action of $\tau$ on $S^3\times\{0\}$ and $S^3\times\{1\}$. Assume also that there is a $\tau_4$-invariant
  locally flat annulus $A\in S^3\times[0,1]$ such that $\partial A=K_1\times\{1\}\sqcup K_0\times\{0\}$. Then, $\eta(K_0)=\eta(K_1)$.
\end{thm}
\begin{proof}
  The proof adapts the proof of Sakuma. We need first to show the following.
  \begin{lem}\label{lem:topological}
    The quotient $S^3\times[0,1]/\tau_4$ is a topological manifold.
  \end{lem}
  \begin{proof}[Proof of Lemma~\ref{lem:topological}]
    Suppose $y\in S^3\times[0,1]/\tau_4$. 
    If $y\notin O_\tau$, then a neighborhood $U_y$ has preimage $\wt{U}_y\in S^3\times[0,1]$ disjoint from $O_\tau$. 
    Now, if $U_y$
    is sufficiently small, then $\wt{U}_y$ is a union of two open subsets of $S^3\times[0,1]$.

    Suppose $y\in O_4$ and let $\wt{y}$ be the preimage. As the action is locally smooth, 
    there is an open subset $\wt{U}\subset X$,
    such that $\wt{U}=\Z_2\times_{\Z_2}V$,where $\Z_2$ acts orthogonally on $V$. As the dimension of the fixed point set is 2, we infer
    that $V$ has dimension $4$ with a fixed point set of dimension $2$. That is, $V=\R^2\oplus\R^2_-$, where $\Z_2$ acts on $\R^2_-$
    by symmetry. We have $\wt{U}/\Z_2=V/\Z_2=\R^2\oplus(\R^2_-/\Z_2)$. Now $\R^2_-/\Z_2\cong \R^2$.
    So $\wt{U}/\Z_2$ is a neighborhood of $y$, which is homeomorphic to $\R^4$.
  \end{proof}
    We will need the following fundamental result of Freedman and Quinn.
  \begin{thm}[\cite{FQ}] \label{thm:fq}
    Suppose $A$ is a locally flat submanifold of a 4-dimensional topological manifold. Then $A$ has a tubular neighborhood and a normal microbundle.
  \end{thm}
  Lemma~\ref{lem:topological} allows us to deduce the following statement from Theorem~\ref{thm:fq}.
\begin{corollary}
  If $A$ is a concordance in $S^3\times[0,1]$ and $\tau_4 A =\tau_4$, then $A$ has an equivariant tubular neighborhood.
\end{corollary}
\begin{proof}
  Consider $A/\Z_2\subset S^3\times[0,1]/\Z_2$. As $\tau_4$ acts locally linearly,
  the quotient $A/\Z_2$ is a locally flat submanifold of $S^3\times[0,1]/\Z_2$. By Theorem~\ref{thm:fq}, it has a tubular neighborhood.
  The pull-back of that neighborhood via the quotient map yields an equivariant tubular neighborhood for $A$.
\end{proof}
The remaining part of the proof of Theorem~\ref{thm:eta} follows the line of Sakuma \cite[Proof of Theorem III]{Sakuma}. 
Namely, an equivariant tubular neighborhood $\nu(A)$ of $A$
allows us to find an annulus $A'$ in $\nu(A)$ with the property that $\tau_4A'\cap A'=\emptyset$ and $A'\cap S^3\times\{i\}$, $i=0,1$,
is the preferred longitude for $K_i$.

Let $O$ be 
the fixed point set of $\tau_4$. As the action is locally linear, $O$ is a topological manifold. The boundary of $O$ is equal to two circles (one on $S^3\times\{0\}$, another on $S^3\times\{1\}$), in particular, $\dim O=2$. By Smith theorem, see e.g. \cite[Theorem III.5.1]{Bredon}, $O$ is a $\Z_2$-homology sphere, which means it is homeomorphic to an annulus. 

The quotient $S^3\times[0,1]/\Z_2$ is simply-connected by e.g. \cite[Corollary II.6.3]{Bredon}, and it is easy to see that it is a $\Z HS^3$.
Let $O_\tau$ be the image of $O$ under the projection. Let $A'_\tau$ be the image of $A'$ under the projection. Then, $A'_\tau$ and $O_\tau$
are annuli. We invoke \cite[Theorem 2]{KojimaYamasaki} to show that the Kojima-Yamasaki polynomials associated with $(A'_\tau,O_\tau)\cap S^3\times\{i\}/\Z_2$ are equal for $i=0$ and $1$, but these polynomials are, by definition, the $\eta$ polynomials of $K_0$ and $K_1$.
\end{proof}

We apply Theorem~\ref{thm:eta} to one of the main knots in the article, $J=17nh_{74}$.
\begin{lem}\label{lem:J_eta}
  The knot $J=17nh_{74}$ with the strong inversion of Figure~\ref{fig:kyle} is not topologically equivariantly slice.
\end{lem}
\begin{figure}
  \begin{tikzpicture}
    \node at (-3,0) {\includegraphics[width=3cm]{Kyle.pdf}};
    \node at (3,0) {\includegraphics[width=2.6cm]{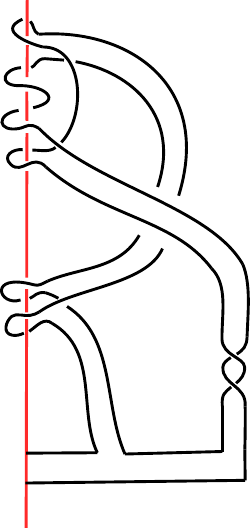}};
  \end{tikzpicture}
  \caption{Left: The knot $J$. Right: The intermediate step for computing the fundamental region for $J$,
  see \cite[Figure 2.2]{Sakuma}.}\label{fig:sakuma_kyle}
\end{figure}
\begin{figure}
  \includegraphics[width=8cm]{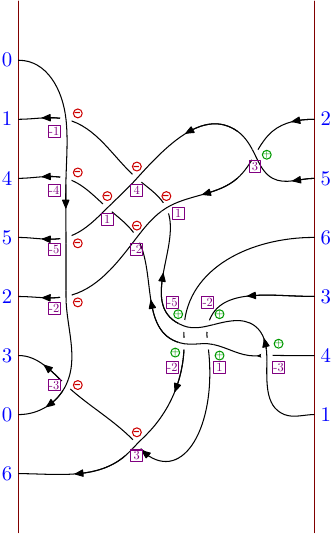}
\caption{The fundamental region for the knot $J$. The arcs are assigned numbers according to Sakuma's algorithm. The signs
of crossings are also marked. Boxed there are integers $d_p$ of Sakuma: these are differences of labels of the two arcs in the crossing.}\label{fig:sakuma_kyle2} 
\end{figure}
\begin{proof}
  We show that $\eta(J)\neq 0$. To this end, we reproduce the steps in Sakuma's algorithm. In Figure~\ref{fig:sakuma_kyle}, we draw the knot $J$ and its `right half', where each crossing on-the-axis is replaced by a small hook. Next, we present the fundamental region for $J$. After some
  simplifications, it is given in Figure~\ref{fig:sakuma_kyle2}.

From Figure~\ref{fig:sakuma_kyle2} we calculate the $\eta$ polynomial. We count $\sum \epsilon_p x_{d_p}$, where $p$ runs over all crossings. The number $\epsilon_p$ is the sign of the crossing, $d_p$ is the integer assigned to it, and $x_{d_p}$ is a formal variable. We see that the sum
is $x_1+x_{-1}-x_4-x_{-4}$. Next, we substitute $t^{i+1}-2t^i+t^{i-1}$ for each occurrence of $x_i$. We obtain
$\wt{\eta}=-(t^5+t^{-5})+2(t^4+t^{-4})-(t^3+t^{-3})+(t^2+t^{-2})-2(t+t^{-1})+2$. That is,
\[\eta(t)=-6-2(t+t^{-1})+(t^2+t^{-2})-(t^3+t^{-3})+2(t^4+t^{-4})-(t^5+t^{-5}).
\]
\end{proof}
\begin{corollary}\label{cor:J_topo}
  In topological category,
  the knot $J=17nh_{74}$ has different equivariant genus and isotopy equivariant genus.
\end{corollary}
\begin{proof}
  We have shown that $\eta(J)\neq 0$, so that the equivariant genus is at least $1$. On the other hand, $J$ bounds a (non-equivariant) slice
  disk $D\subset B^4$ such that $\pi_1(B^4\setminus D)=\Z$. The action $\tau$ on $B^4$ takes $D$ to $\tau(D)$, whose boundary is still $J$.
  As $\pi_1(B^4\setminus D)=\Z$, by \cite{ConwayPowell} $D$ and $\tau(D)$ are topologically isotopic rel boundary. That is $D$ is isotopy-equivariant. Hence the isotopy equivariant genus of $J$ is zero.
\end{proof}

\section{Table of knots with $\wt{s}$ different than $s$.}\label{sec:different}
Here, we provide the table of knots for which $\wt{s}$ is different than $s$.
The tables should be read as follows. The PD code entry gives the PD code of the diagram of the knot. The column `action' provides the
action. For example, $\tau=17$ for $9_{46}$ means that strand 17 of the PD code is fixed. Then, strands 18 and 16 are swapped, strands 1 and 15 are swapped, and so on.

  \begin{longtable}{|c|>{\tiny} p{9cm}<{\normalsize}|c|c|}\hline
    knot & \normalsize PD code & action & $s$ \\ \hline
    $9\_46$ &
    [11,18,12,1], [1,10,2,11], [9,2,10,3], [14,4,15,3], [4,14,5,13], [12,6,13,5], [17,6,18,7], [7,16,8,17], [15,8,16,9] &
    17 & 0 \\ \hline
    $10\_141$ &
[20,7,1,8], [1,14,2,15], [9,2,10,3], [16,3,17,4], [4,11,5,12], [18,6,19,5], [6,14,7,13], [15,9,16,8], [10,18,11,17], [12,19,13,20] &
3 & 0 \\ \hline
$10\_160$ &
[20,16,1,15], [8,1,9,2], [13,3,14,2], [3,13,4,12], [17,5,18,4], [10,6,11,5], [6,20,7,19], [14,7,15,8], [16,10,17,9], [11,19,12,18] & 
15 & 4 \\ \hline
$10\_163$ &
[7,20,8,1], [1,14,2,15], [2,9,3,10], [10,3,11,4], [4,18,5,17], [16,6,17,5], [6,16,7,15], [13,9,14,8], [18,11,19,12], [12,19,13,20] &
20 & -2\\ \hline
$12n744$ &
[5,1,6,24], [1,17,2,16], [2,10,3,9], [20,4,21,3], [13,5,14,4], [15,6,16,7], [22,8,23,7], [8,22,9,21], [19,11,20,10], [11,19,12,18], [17,13,18,12], [23,14,24,15] &
5 & 4\\ \hline 
$12n829$ &
[5,1,6,24], [8,1,9,2], [2,9,3,10], [20,3,21,4], [4,21,5,22], [15,6,16,7], [7,16,8,17], [10,18,11,17], [18,12,19,11], [12,20,13,19], [13,22,14,23], [23,14,24,15] & 
5 & -2\\ \hline
$12n832$ &
[5,1,6,24], [8,1,9,2], [2,9,3,10], [20,3,21,4], [4,21,5,22], [15,6,16,7], [7,16,8,17], [10,20,11,19], [18,12,19,11], [12,18,13,17], [13,22,14,23], [23,14,24,15] & 
5 & -4\\ \hline 
$12n847$ &
[5,1,6,24], [1,17,2,16], [11,2,12,3], [3,18,4,19], [13,5,14,4], [15,6,16,7], [22,8,23,7], [8,22,9,21], [20,10,21,9], [10,20,11,19], [17,13,18,12], [23,14,24,15] & 
5 & 2\\ \hline 
$12n867$ &
[5,1,6,24], [10,1,11,2], [2,11,3,12], [18,3,19,4], [4,19,5,20], [15,6,16,7], [7,22,8,23], [21,8,22,9], [9,16,10,17], [12,18,13,17], [13,20,14,21], [23,14,24,15] & 5 & -4\\ \hline
12n882 &
[5,1,6,24], [18,1,19,2], [2,19,3,20], [10,3,11,4], [4,11,5,12], [15,6,16,7], [7,16,8,17], [17,8,18,9], [20,10,21,9], [21,12,22,13], [13,22,14,23], [23,14,24,15] & 5 & -6\\ \hline
13n2400 &
[3,1,4,26], [1,9,2,8], [9,3,10,2], [4,15,5,16], [16,5,17,6], [6,17,7,18], [14,7,15,8], [10,23,11,24], [24,11,25,12], [19,13,20,12], [13,21,14,20], [21,19,22,18], [22,25,23,26] & 21 & 0\\ \hline
13n3659 &
[3,1,4,26], [1,15,2,14], [15,3,16,2], [4,9,5,10], [12,5,13,6], [6,23,7,24], [22,7,23,8], [8,13,9,14], [10,19,11,20], [18,11,19,12], [16,21,17,22], [24,17,25,18], [20,25,21,26] & 3 & -2\\ \hline
13n3720 &
[3,1,4,26], [1,15,2,14], [15,3,16,2], [4,21,5,22], [12,5,13,6], [6,11,7,12], [22,7,23,8], [8,25,9,26], [16,9,17,10], [10,19,11,20], [20,13,21,14], [24,17,25,18], [18,23,19,24] & 3 & -2\\ \hline
13n4370 &
[5,1,6,26], [10,2,11,1], [2,10,3,9], [16,3,17,4], [4,17,5,18], [11,7,12,6], [20,7,21,8], [8,21,9,22], [19,12,20,13], [24,14,25,13], [14,24,15,23], [22,16,23,15], [25,18,26,19] & 11 & 0\\ \hline
13n4374 &
[5,1,6,26], [10,2,11,1], [2,10,3,9], [3,17,4,16], [17,5,18,4], [11,7,12,6], [7,21,8,20], [21,9,22,8], [19,12,20,13], [24,14,25,13], [14,24,15,23], [22,16,23,15], [25,18,26,19] & 11 & 6\\ \hline
13n4381 &
[5,1,6,26], [10,2,11,1], [2,10,3,9], [16,3,17,4], [4,17,5,18], [11,7,12,6], [20,7,21,8], [8,21,9,22], [25,13,26,12], [13,25,14,24], [14,20,15,19], [22,16,23,15], [18,24,19,23] & 11 & 2\\ \hline
13n4906 &
[5,1,6,26], [1,12,2,13], [11,2,12,3], [3,20,4,21], [19,4,20,5], [6,15,7,16], [24,8,25,7], [17,8,18,9], [9,18,10,19], [21,10,22,11], [13,22,14,23], [23,14,24,15], [16,25,17,26] & 5 & -4\\ \hline
13n4959 &
[5,1,6,26], [1,20,2,21], [13,3,14,2], [3,19,4,18], [11,4,12,5], [6,15,7,16], [24,8,25,7], [8,24,9,23], [22,10,23,9], [17,10,18,11], [19,13,20,12], [21,14,22,15], [16,25,17,26] & 5 & 0\\ \hline
13n4963 &
[5,1,6,26], [1,20,2,21], [13,3,14,2], [3,19,4,18], [11,4,12,5], [25,7,26,6], [7,25,8,24], [8,15,9,16], [22,10,23,9], [17,10,18,11], [19,13,20,12], [21,14,22,15], [16,23,17,24] & 5 & 0\\ \hline
13n4982 &
[5,1,6,26], [1,18,2,19], [13,2,14,3], [3,14,4,15], [17,4,18,5], [19,7,20,6], [7,24,8,25], [21,8,22,9], [16,10,17,9], [10,16,11,15], [11,22,12,23], [23,12,24,13], [25,21,26,20] & 25 & -2\\ \hline
13n5007 &
[5,1,6,26], [12,1,13,2], [2,13,3,14], [18,3,19,4], [4,19,5,20], [25,7,26,6], [20,7,21,8], [8,17,9,18], [16,9,17,10], [10,21,11,22], [24,11,25,12], [14,23,15,24], [22,15,23,16] & 5 & -4\\ \hline
13n5084 &
[5,1,6,26], [1,12,2,13], [21,3,22,2], [3,11,4,10], [19,4,20,5], [6,15,7,16], [24,8,25,7], [17,8,18,9], [9,18,10,19], [11,21,12,20], [13,22,14,23], [23,14,24,15], [16,25,17,26] & 5 & -2\\ \hline
  \caption{Knots for which $s\neq\wt{s}$.}\label{tab:le}
\end{longtable}

  \begin{longtable}{|c|>{\tiny} p{9cm}<{\normalsize}|c|c|}
    \hline
    knot & \normalsize PD code & action & $s$  \\ \hline
10\_157 & 
[20,15,1,16], [8,2,9,1], [2,10,3,9], [3,17,4,16], [17,5,18,4], [5,10,6,11], [6,14,7,13], [14,8,15,7], [11,19,12,18], [19,13,20,12]  & 
15 & -4 
\\ \hline
10\_166 & 
[20,7,1,8], [6,1,7,2], [15,3,16,2], [3,19,4,18], [9,5,10,4], [5,13,6,12], [13,9,14,8], [17,11,18,10], [11,17,12,16], [19,15,20,14]  & 
7 & -2
\\ \hline
10n25 & 
[3,1,4,20], [1,9,2,8], [9,3,10,2], [4,13,5,14], [14,5,15,6], [6,15,7,16], [12,7,13,8], [17,11,18,10], [11,19,12,18], [19,17,20,16]  & 
19 & 0
\\ \hline
11n156 & 
[5,1,6,22], [1,18,2,19], [2,10,3,9], [16,3,17,4], [11,5,12,4], [13,6,14,7], [7,20,8,21], [15,9,16,8], [10,18,11,17], [21,12,22,13], [19,15,20,14]  & 
19 & 0
\\ \hline
11n162 & 
[5,1,6,22], [1,18,2,19], [2,10,3,9], [16,3,17,4], [11,5,12,4], [21,7,22,6], [7,13,8,12], [15,9,16,8], [10,18,11,17], [13,21,14,20], [19,15,20,14]  & 
19 & -2
\\ \hline
11n71 & 
[3,1,4,22], [1,7,2,6], [7,3,8,2], [13,5,14,4], [5,13,6,12], [8,19,9,20], [20,9,21,10], [15,11,16,10], [11,17,12,16], [17,15,18,14], [18,21,19,22]  & 
17 & -2
\\ \hline
11n74 & 
[3,1,4,22], [1,7,2,6], [7,3,8,2], [4,13,5,14], [12,5,13,6], [8,19,9,20], [20,9,21,10], [15,11,16,10], [11,17,12,16], [17,15,18,14], [18,21,19,22]  & 
17 & 0
\\ \hline
12n225 & 
[3,1,4,24], [1,7,2,6], [7,3,8,2], [4,13,5,14], [12,5,13,6], [8,22,9,21], [20,10,21,9], [15,11,16,10], [11,17,12,16], [17,15,18,14], [18,23,19,24], [22,19,23,20]  & 
17 & -2
\\ \hline
12n571 & 
[3,1,4,24], [1,11,2,10], [11,3,12,2], [4,15,5,16], [16,5,17,6], [6,17,7,18], [18,7,19,8], [8,19,9,20], [14,9,15,10], [21,13,22,12], [13,23,14,22], [23,21,24,20]  & 
23 & 2
\\ \hline
12n573 & 
[3,1,4,24], [1,11,2,10], [11,3,12,2], [4,15,5,16], [18,5,19,6], [6,17,7,18], [16,7,17,8], [8,19,9,20], [14,9,15,10], [21,13,22,12], [13,23,14,22], [23,21,24,20]  & 
23 & 0
\\ \hline
12n612 & 
[3,1,4,24], [1,12,2,13], [11,2,12,3], [19,5,20,4], [5,21,6,20], [6,16,7,15], [22,8,23,7], [17,9,18,8], [9,19,10,18], [13,11,14,10], [23,14,24,15], [16,21,17,22]  & 
13 & -2
\\ \hline
12n630 & 
[3,1,4,24], [1,14,2,15], [13,2,14,3], [4,9,5,10], [16,6,17,5], [21,7,22,6], [7,21,8,20], [8,16,9,15], [10,17,11,18], [22,12,23,11], [12,20,13,19], [18,23,19,24]  & 
3 & 0
\\ \hline
12n671 & 
[5,1,6,24], [1,7,2,6], [7,3,8,2], [3,11,4,10], [11,5,12,4], [8,15,9,16], [14,9,15,10], [19,13,20,12], [13,21,14,20], [21,17,22,16], [17,23,18,22], [23,19,24,18]  & 
23 & -6
\\ \hline
12n722 & 
[5,1,6,24], [1,9,2,8], [9,3,10,2], [3,11,4,10], [11,5,12,4], [6,17,7,18], [16,7,17,8], [19,13,20,12], [13,21,14,20], [21,15,22,14], [15,23,16,22], [23,19,24,18]  & 
23 & -6
\\ \hline
12n726 & 
[5,1,6,24], [1,11,2,10], [9,3,10,2], [3,9,4,8], [11,5,12,4], [6,17,7,18], [16,7,17,8], [19,13,20,12], [13,23,14,22], [21,15,22,14], [15,21,16,20], [23,19,24,18]  & 
23 & -2
\\ \hline
12n759 & 
[5,1,6,24], [1,17,2,16], [2,10,3,9], [20,4,21,3], [13,5,14,4], [23,7,24,6], [7,15,8,14], [21,8,22,9], [10,19,11,20], [18,11,19,12], [12,17,13,18], [15,23,16,22]  & 
5 & -2
\\ \hline
12n814 & 
[5,1,6,24], [1,10,2,11], [9,2,10,3], [20,4,21,3], [4,20,5,19], [11,7,12,6], [16,8,17,7], [8,16,9,15], [12,23,13,24], [18,14,19,13], [14,22,15,21], [22,18,23,17]  & 
11 & -2
\\ \hline
12n815 & 
[5,1,6,24], [1,10,2,11], [9,2,10,3], [20,4,21,3], [4,20,5,19], [11,7,12,6], [16,8,17,7], [8,16,9,15], [23,13,24,12], [13,18,14,19], [21,14,22,15], [17,22,18,23]  & 
11 & 0
\\ \hline
12n817 & 
[5,1,6,24], [1,10,2,11], [9,2,10,3], [3,19,4,18], [19,5,20,4], [11,7,12,6], [7,17,8,16], [17,9,18,8], [12,21,13,22], [22,13,23,14], [14,23,15,24], [20,15,21,16]  & 
11 & 0
\\ \hline
12n818 & 
[5,1,6,24], [1,10,2,11], [9,2,10,3], [3,19,4,18], [19,5,20,4], [11,7,12,6], [7,17,8,16], [17,9,18,8], [12,23,13,24], [22,13,23,14], [14,21,15,22], [20,15,21,16]  & 
11 & -2
\\ \hline
12n838 & 
[5,1,6,24], [20,2,21,1], [2,11,3,12], [3,19,4,18], [9,4,10,5], [6,16,7,15], [22,8,23,7], [13,9,14,8], [19,11,20,10], [12,18,13,17], [23,14,24,15], [16,21,17,22]  & 
13 & 0
\\ \hline
12n868 & 
[5,1,6,24], [1,21,2,20], [2,11,3,12], [10,3,11,4], [21,5,22,4], [19,6,20,7], [7,18,8,19], [13,9,14,8], [9,17,10,16], [17,13,18,12], [14,23,15,24], [22,15,23,16]  & 
1 & 0
\\ \hline
12n869 & 
[5,1,6,24], [1,21,2,20], [2,11,3,12], [10,3,11,4], [21,5,22,4], [19,6,20,7], [7,18,8,19], [13,9,14,8], [9,17,10,16], [17,13,18,12], [23,15,24,14], [15,23,16,22]  & 
13 & -2
\\ \hline
13n1281 & 
[3,1,4,26], [1,7,2,6], [7,3,8,2], [4,15,5,16], [14,5,15,6], [8,21,9,22], [22,9,23,10], [10,23,11,24], [24,11,25,12], [17,13,18,12], [13,19,14,18], [19,17,20,16], [20,25,21,26]  & 
19 & 2
\\ \hline
13n1284 & 
[3,1,4,26], [1,7,2,6], [7,3,8,2], [15,5,16,4], [5,15,6,14], [8,21,9,22], [22,9,23,10], [10,23,11,24], [24,11,25,12], [17,13,18,12], [13,19,14,18], [19,17,20,16], [20,25,21,26]  & 
19 & 0
\\ \hline
13n2402 & 
[3,1,4,26], [1,9,2,8], [9,3,10,2], [4,15,5,16], [16,5,17,6], [6,17,7,18], [14,7,15,8], [23,11,24,10], [11,25,12,24], [19,13,20,12], [13,21,14,20], [21,19,22,18], [25,23,26,22]  & 
21 & -2
\\ \hline
13n3527 & 
[3,1,4,26], [1,13,2,12], [13,3,14,2], [4,18,5,17], [18,6,19,5], [6,20,7,19], [24,8,25,7], [8,15,9,16], [14,9,15,10], [21,11,22,10], [11,23,12,22], [16,26,17,25], [23,21,24,20]  & 
23 & -6
\\ \hline
13n3661 & 
[3,1,4,26], [1,15,2,14], [15,3,16,2], [4,9,5,10], [12,5,13,6], [6,23,7,24], [22,7,23,8], [8,13,9,14], [19,11,20,10], [11,19,12,18], [16,21,17,22], [24,17,25,18], [20,25,21,26]  & 
3 & 2
\\ \hline
13n3665 & 
[3,1,4,26], [1,15,2,14], [15,3,16,2], [9,5,10,4], [5,13,6,12], [6,23,7,24], [22,7,23,8], [13,9,14,8], [10,19,11,20], [18,11,19,12], [21,17,22,16], [17,25,18,24], [25,21,26,20]  & 
3 & -2
\\ \hline
13n3666 & 
[3,1,4,26], [1,15,2,14], [15,3,16,2], [9,5,10,4], [5,13,6,12], [6,23,7,24], [22,7,23,8], [13,9,14,8], [19,11,20,10], [11,19,12,18], [21,17,22,16], [17,25,18,24], [25,21,26,20]  & 
3 & -4
\\ \hline
13n4025 & 
[5,1,6,26], [1,21,2,20], [2,10,3,9], [22,4,23,3], [11,5,12,4], [15,7,16,6], [7,15,8,14], [19,8,20,9], [10,21,11,22], [23,12,24,13], [13,18,14,19], [25,17,26,16], [17,25,18,24]  & 
5 & 0
\\ \hline
13n4393 & 
[5,1,6,26], [1,10,2,11], [9,2,10,3], [3,17,4,16], [17,5,18,4], [11,7,12,6], [7,21,8,20], [21,9,22,8], [12,25,13,26], [24,13,25,14], [14,20,15,19], [22,16,23,15], [18,24,19,23]  & 
11 & -4
\\ \hline
13n4658 & 
[5,1,6,26], [18,1,19,2], [2,11,3,12], [3,20,4,21], [9,4,10,5], [6,16,7,15], [24,8,25,7], [13,9,14,8], [19,10,20,11], [12,21,13,22], [25,14,26,15], [16,23,17,24], [22,17,23,18]  & 
13 & 4
\\ \hline
13n4660 & 
[5,1,6,26], [1,19,2,18], [2,11,3,12], [20,4,21,3], [9,4,10,5], [15,6,16,7], [7,24,8,25], [13,9,14,8], [10,20,11,19], [21,13,22,12], [25,14,26,15], [23,17,24,16], [17,23,18,22]  & 
13 & 0
\\ \hline
13n4672 & 
[5,1,6,26], [1,19,2,18], [2,11,3,12], [20,4,21,3], [9,4,10,5], [17,6,18,7], [7,22,8,23], [13,9,14,8], [10,20,11,19], [21,13,22,12], [23,14,24,15], [15,24,16,25], [25,16,26,17]  & 
13 & 2
\\ \hline
13n4675 & 
[5,1,6,26], [1,19,2,18], [2,11,3,12], [20,4,21,3], [9,4,10,5], [17,6,18,7], [7,22,8,23], [13,9,14,8], [10,20,11,19], [21,13,22,12], [25,14,26,15], [15,24,16,25], [23,16,24,17]  & 
13 & 0
\\ \hline
13n4690 & 
[5,1,6,26], [1,19,2,18], [2,11,3,12], [20,4,21,3], [9,4,10,5], [23,7,24,6], [7,17,8,16], [13,9,14,8], [10,20,11,19], [21,13,22,12], [25,14,26,15], [15,24,16,25], [17,23,18,22]  & 
13 & -2
\\ \hline
13n4695 & 
[5,1,6,26], [1,19,2,18], [2,11,3,12], [20,4,21,3], [9,4,10,5], [25,7,26,6], [7,15,8,14], [13,9,14,8], [10,20,11,19], [21,13,22,12], [15,23,16,22], [23,17,24,16], [17,25,18,24]  & 
13 & -4
\\ \hline
13n4698 & 
[5,1,6,26], [1,19,2,18], [2,11,3,12], [20,4,21,3], [9,4,10,5], [25,7,26,6], [7,15,8,14], [13,9,14,8], [10,20,11,19], [21,13,22,12], [15,25,16,24], [23,17,24,16], [17,23,18,22]  & 
13 & -2
\\ \hline
13n966 & 
[3,1,4,26], [1,7,2,6], [7,3,8,2], [13,5,14,4], [5,13,6,12], [8,21,9,22], [22,9,23,10], [15,11,16,10], [11,17,12,16], [17,15,18,14], [18,23,19,24], [24,19,25,20], [20,25,21,26]  & 
17 & 0
\\ \hline
13n968 & 
[3,1,4,26], [1,7,2,6], [7,3,8,2], [4,13,5,14], [12,5,13,6], [8,21,9,22], [22,9,23,10], [15,11,16,10], [11,17,12,16], [17,15,18,14], [18,25,19,26], [24,19,25,20], [20,23,21,24]  & 
17 & 0
\\ \hline
\caption{Knots which mirror satisfies $s\neq \wt{s}$.}\label{el:bat}
  \end{longtable}

\printindex

\bibliographystyle{amsalpha}
\def\MR#1{}
\bibliography{bib}

\end{document}